%% file: KCL_Arxiv_TonyHill_April17.tex
\documentclass[a4paper,12pt]{report}

\usepackage{amsmath,amsfonts,amssymb, mathtools}
\usepackage{amsthm}
\usepackage{color,multicol}
\usepackage{graphicx}
\usepackage{cases}
\usepackage{enumerate}
\usepackage{csquotes}
\usepackage{textcomp}
\usepackage{float}
\usepackage{tocloft}

\newtheorem{lemma}{Lemma}[chapter]
\newtheorem{theorem}[lemma]{Theorem}
\newtheorem{corollary}[lemma]{Corollary}

\newtheorem{definition}[lemma]{Definition}

\newtheorem{remark}[lemma]{Remark}
\DeclareMathOperator{\sgn}{sgn}

\begin{document}
\begin{titlepage}
\begin{center}
\vspace*{1in}
\begin{center}

\textbf{\Large Mellin and Wiener-Hopf operators in a non-classical boundary value problem describing a L\'evy process} \\
\end{center}

\par
\vspace{1.5in}
{\Large Anthony Christopher Hill}
\par
\vfill
A thesis presented for the degree of \\
Doctor of Philosophy
\par
\vspace{0.5in}
Department of Mathematics
\par
\vspace{0in}
King's College London
\par
\vspace{0.5in}
March 2017
\end{center}
\end{titlepage}

\begin{abstract}
Markov processes are well understood in the case when they take place in the whole Euclidean space. However, the situation becomes much more complicated if a Markov process is restricted to a domain with a boundary, and then a satisfactory theory only exists for processes with continuous trajectories. This research, into non-classical boundary value problems, is motivated by the study of stochastic processes, restricted to a domain, that can have discontinuous trajectories. We demonstrate that the singularities, for example delta functions, that might be expected at the boundary, are mitigated, using current probability theory, by what amounts to the addition of a carefully chosen potential. \\

To make this general problem more tractable, we consider a particular operator, $\mathcal{A}$, which is chosen to be the generator of a certain stable L\'evy process restricted to the positive half-line. We are able to represent $\mathcal{A}$ as a (hyper-) singular integral and, using this representation, deduce simple conditions for its boundedness, between Bessel potential spaces. Moreover, from energy estimates, we prove that, under certain conditions, $\mathcal{A}$ has a trivial kernel. \\

A central feature of this research is our use of Mellin operators to deal with the leading singular terms that combine, and cancel, at the boundary. Indeed, after considerable analysis, the problem is reformulated in the context of an algebra of multiplication, Wiener-Hopf and Mellin operators, acting on a Lebesgue space. The resulting generalised symbol is examined and, it turns out, that a certain transcendental equation, involving gamma and trigonometric functions with complex arguments, plays a pivotal role. Following detailed consideration of this transcendental equation, we are able to determine when our operator is Fredholm and, in that case, calculate its index. Finally, combining information on the kernel with the Fredholm index, we establish precise conditions for the invertibility of $\mathcal{A}$.

\end{abstract}

\renewcommand{\abstractname}{Acknowledgements}
\begin{abstract}
I would like to thank King's College London, for the opportunity, in my retirement, to study for a Ph.D. in pure mathematics. However, above all, I am indebted to my inspirational supervisor, Prof. Eugene Shargorodsky, for so freely sharing his huge mathematical experience, and for the many hours he has invested in helping me begin to understand this challenging, but very rewarding, problem.
\end{abstract}
\newpage
\setcounter{tocdepth}{1}
\tableofcontents 
\newpage
\setlength\cftbeforefigskip{3pt}
\listoffigures
All figures were produced in ${\textit{Mathematica}}\textsuperscript{\textregistered} \, 10$ Student Edition.
\newpage	

\include{KCL_Thesis_Preamble_v5}     
\include{KCL_Thesis_ProbDef_v5}       
\include{KCL_Thesis_KeyResults_v5}  
\include{KCL_Thesis_Chapter2_v5}     
\include{KCL_Thesis_Chapter3_v5}     
\include{KCL_Thesis_Chapter4_v5}     
\include{KCL_Thesis_Chapter5_v5}     
\include{KCL_Thesis_Chapter6_v5}     
\include{KCL_Thesis_Chapter8_v5}     
\include{KCL_Thesis_Chapter9_v5}     
\include{KCL_Thesis_Chapter7_v5}     
\include{KCL_Thesis_Chapter10_v5} 
\appendix
\include{KCL_Thesis_Appendices_v5} 

\include{KCL_Thesis_Refs_v5}            

\end{document}

%% file: KCL_Thesis_Preamble_v5.tex
\chapter{Introduction}
\section {Preamble} \label{preamble}
We begin by defining some key concepts central to this research. References for all the results stated in this section without proof can be found in a combination of \cite{Es} and \cite{Roch}. \\

The \textit{Fourier Transform} on the Schwartz space, $S(\mathbb{R})$, of rapidly decaying infinitely differentiable functions $u$ is given by
\begin{equation} \label{FTdefinition}
(\mathcal{F}u)(\xi) := \dfrac{1}{{\sqrt{2\pi}}} \int_{\mathbb{R}} e^{+i\xi  x} \, u(x) \, dx, \quad \xi \in \mathbb{R}.
\end{equation}
The Fourier transform is invertible on the Schwartz space, and its inverse $\mathcal{F}^{-1}$ is given by
\begin{equation*}
(\mathcal{F}^{-1}v)(x) := \dfrac{1}{{\sqrt{2\pi}}} \int_{\mathbb{R}}  e^{-ix \xi} \, v(\xi) \, d\xi,\quad x \in \mathbb{R}. 
\end{equation*}
$\mathcal{F}^{\pm 1}$ can be extended to $S'(\mathbb{R})$, the space of tempered distributions corresponding to the Schwartz space. \\

For any $y, s \in \mathbb{R}$, we define
\begin{align*}
\langle y \rangle & := \big (1 + y^2 \big )^{1/2}, \\
I^s & := \mathcal{F}^{-1} \langle \xi \rangle^s \mathcal{F}.
\end{align*}
We define the \textit{Bessel potential space}
\begin{equation*}
H^s_p(\mathbb{R}) := \{ f : f \in S'(\mathbb{R}), \|f | H^s_p(\mathbb{R})\| := \| I^s f \|_{L_p} < \infty \}.
\end{equation*}
We note that for any $s \in \mathbb{R}$ and $ 1 < p < \infty, \, H^s_p(\mathbb{R})$ is a Banach space. Moreover, both the Schwartz space and the space of infinitely smooth functions with compact support are dense in $H^s_p(\mathbb{R})$. \\

We are interested in the half-line $\mathbb{R}_+$ and, accordingly, we define
\begin{equation*}
\widetilde{H}^s_p(\overline{\mathbb{R}_+}) := \{ u : u \in H^s_p(\mathbb{R}), \, \text{ supp }u \subseteq \overline{\mathbb{R}_+} \}.
\end{equation*}
Of course, by definition, we have $\widetilde{H}^s_p(\overline{\mathbb{R}_+}) \subset H^s_p(\mathbb{R})$. \\

Let $C^\infty_0(\mathbb{R}_+)$ denote the space of infinitely differentiable functions on $\mathbb{R}_+$ with compact support in $\mathbb{R}_+$, and let $\chi_{\mathbb{R}_\pm}$ denote the characteristic function of the sets $\mathbb{R}_\pm$ respectively.\\

It will also be useful to let $r_+$ denote the restriction operator from $\mathbb{R}$ to ${\mathbb{R}_+}$. In addition, we let $l_+$ denote an arbitrary extension operator  from  ${\mathbb{R}_+}$ to $\mathbb{R}$ and  $e_+$ be the particular extension by zero.  
We now define
\begin{equation*}
H^s_p(\overline{\mathbb{R}_+}) := \{ r_+ u : u \in H^s_p(\mathbb{R}) \},
\end{equation*}
with norm
\begin{equation} \label{Hsp+norm}
\| u | H^s_p(\overline{\mathbb{R}_+}) \| := \inf \{ \| u_0 | H^s_p(\mathbb{R})\| : u_0 \in H^s_p(\mathbb{R}), \, r_+ u_0 = u \}. \\
\end{equation}
Of course, if $u \in {H}^s_p(\overline{\mathbb{R}_+})$ then 
\begin{equation*}
r_+l_+ u = u,
\end{equation*}
so that $r_+ l_+$ acts as the identity on the space ${H}^s_p(\overline{\mathbb{R}_+})$. \\

Assuming $s > 1 + 1/p$, it will also convenient to define
\begin{equation} \label{Hsp0definition}
H^s_{p,0}(\overline{\mathbb{R}_+}) := \{ u \in H^s_p(\overline{\mathbb{R}_+}) \, | \, u'(0)=0 \}. \\
\end{equation}

When working with the Fourier transform, we define
\begin{equation*}
D := i \, \dfrac{\partial}{\partial x}, 
\end{equation*}
so that, for example, for a given $u \in S(\mathbb{R})$ we have $\mathcal{F}_{x \to \xi} (D^k u) = \xi^k \mathcal{F} {u}$,  for all $k \in \mathbb{N}$. \\

Suppose $1 < p < \infty$. We say that $a \in L_\infty(\mathbb{R})$ is a \textit{Fourier $L_p$ multiplier} if, for all $u \in L_2(\mathbb{R}) \cap L_p(\mathbb{R})$, we have $\mathcal{F}^{-1} a \mathcal{F} u \in L_p(\mathbb{R})$ and
\begin{equation*}
\| \mathcal{F}^{-1} a \mathcal{F}u \|_{p} \leq C_{p} \| u \|_{p},
\end{equation*}
where the constant $C_{p}$ is independent of $u$. The set of all such Fourier multipliers is denoted by $\mathfrak{M}_{p}$. \\

If $a \in \mathfrak{M}_{p}$, then the operator $\mathcal{F}^{-1} a \mathcal{F}:L_2(\mathbb{R}) \cap L_p(\mathbb{R}) \to L_p(\mathbb{R})$ extends continuously to a bounded operator on $L_p(\mathbb{R})$. This extension is called a \textit{Fourier convolution operator} with \textit{symbol} $a$, and is denoted by $W^0(a)$. Moreover, we let $W(a) = r_+ W^0(a) e_+$ denote the corresponding \textit{Wiener-Hopf} operator. \\

We now define the operator $Z_{p}:L_p(\mathbb{R}_+) \to L_p(\mathbb{R})$ by
\begin{equation*}
(Z_{p}u)(y) := \sqrt{2 \pi} \, e^{-y/p}\,  u(e^{-y}), \quad y \in \mathbb{R},
\end{equation*}
for any $u \in L_p(\mathbb{R}_+)$.  It is easy to show that 
\begin{equation*}
\| Z_p u \|_p = \sqrt{2 \pi} \, \|u \|_{L_p(\mathbb{R}_+)},
\end{equation*}
so that $ Z_p$ is bounded. In addition, $Z_p$ is invertible and its inverse $Z^{-1}_{p}:L_p(\mathbb{R}) \to L_p(\mathbb{R}_+)$ is given by
\begin{equation*}
(Z^{-1}_{p}v)(x) :=  \frac{1}{\sqrt{2 \pi}} \, x^{-1/p} \, v(- \log(x)), \quad x \in \mathbb{R}_+,
\end{equation*}
for any $v \in L_p(\mathbb{R})$. \\

The operator $\mathcal{M}_{p} := \mathcal{F} Z_{p}$ is called the \textit{Mellin transform}, and it is given explicitly by
\begin{equation} \label{MTexplicit}
(\mathcal{M}_{p}u)(\eta) = \int^\infty_0  x^{1/p - 1 - i\eta} \,  u(x) \, dx, \quad \eta \in \mathbb{R},
\end{equation}
and moreover, its inverse, $\mathcal{M}^{-1}_{p} = Z^{-1}_{p} {\mathcal{F}}^{-1} $, can be written as
\begin{equation*} 
(\mathcal{M}^{-1}_{p}v)(x) = \dfrac{1}{2 \pi} \int^\infty_{-\infty}  x^{-1/p+ i\eta}  \, v(\eta) \, d\eta, \quad x \in \mathbb{R}_+.
\end{equation*}

Let $b\in \mathfrak{M}_{p}$. Then the operator
\begin{equation} \label{MellinConvOp}
M^0(b) := \mathcal{M}^{-1}_{p} b \mathcal{M}_{p}
\end{equation}
is bounded on $L_p(\mathbb{R}_+)$, and is called the \textit{Mellin convolution operator} with {symbol} $b$. \\

We will be particularly interested in integral operators of the form
\begin{equation} \label{MellinIntOp}
(M \varphi)(x) = \int^\infty_0 K \bigg ( \dfrac{x}{y} \bigg ) \varphi (y) \, \dfrac{dy}{y}, \quad \varphi \in L_p(\mathbb{R}_+),
\end{equation}
where the \textit{kernel} $K$ satisfies the integrability condition
\begin{equation} \label{kernelinteg}
\int^\infty_0 |K(t)| t^{1/p-1} \, dt < \infty.
\end{equation}

It is easy to show (see, for example, p.\,174, \cite{Es}, for the case $p=2$) that
\begin{equation*} 
\mathcal{M}_{p} (M \varphi) = \mathcal{M}_{p} (K) \, \cdot \, \mathcal{M}_{p} (\varphi).
\end{equation*}

If the kernel $K$ satisfies the integrability condition \eqref{kernelinteg}, then the associated integral operator $M$ is a Mellin convolution operator, say $M^0(b)$, where the symbol $b$ is given by the formula
\begin{equation} \label{MellinSymbol}
b = \mathcal{M}_{p} (K).
\end{equation}
That is, the symbol of the integral operator $M$ is the Mellin transform of its kernel $K$. \\

We will need to consider fractional powers of complex numbers. To make the complex argument, $\arg (\cdot)$, single valued we insist that
\begin{equation} \label{argsingleval}
-\pi < \arg z \leq \pi, \quad z \in \mathbb{C} \setminus \{ 0 \}.
\end{equation}
That is, we assume that the cut in the complex plane is along the negative horizontal axis. 
\\

Thus, given $z \in \mathbb{C} \setminus \{ 0 \}$, we can write
\begin{equation*}
z = |z| \exp ({i \arg z}), \quad -\pi < \arg z \leq \pi,
\end{equation*}
and define
\begin{equation*} \label{logdef}
\log z := \log |z| + i \arg z. 
\end{equation*}

Finally, for any $z \in \mathbb{C} \setminus \{ 0 \}$ and $\gamma \in \mathbb{C}$ we define
\begin{equation} \label{exponentdef}
z^\gamma := \exp (\gamma \log z ).
\end{equation} 

It is immediately clear from definition \eqref{exponentdef}, that for any $z \in \mathbb{C} \setminus \{ 0 \}$ and $\beta, \gamma \in \mathbb{R}$, 
\begin{equation*}
z^\beta z^\gamma = z^{\beta + \gamma}.
\end{equation*}

However, suppose that $z_1, z_2 \in \mathbb{C}$ and $\gamma \in \mathbb{R}$. Then the relationship
\begin{equation*} 
(z_1 z_2)^\gamma = z_1^\gamma z_2^\gamma
\end{equation*}
does $\mathbf{not}$ hold universally. \\

\begin{remark} \label{complexexponent}
Suppose $\nu \in \mathbb{R}$. Then, for $z_1,z_2 \in \mathbb{C}$,
\begin{equation} \label{exponentrule}
(z_1 \, z_2)^\nu = z_1^\nu \, z_2^\nu,
\end{equation}
\textit{provided} 
\begin{equation} \label{exponentcondition}
-\pi < \arg z_1 + \arg z_2 \leq \pi.
\end{equation}
Of course, condition \eqref{exponentcondition} is automatically satisfied if $z_1$ and $z_2$ have (non-trivial) imaginary parts of opposite sign. \\
\end{remark}

Suppose $\nu, \xi \in \mathbb{R}$. The following useful results are immediate consequences of Remark \ref{complexexponent}:
\begin{align*}
(1+ \xi^2)^\nu & = (1 - i \xi)^\nu \, (1+ i \xi)^\nu \\
(1+ \xi^2)^\nu & = (\xi - i )^\nu \, (\xi + i)^\nu \\
(1 + i \xi)^\nu & = (i \xi)^\nu (1 - i/\xi)^\nu \qquad  (\xi \not = 0).
\end{align*}

%% file: KCL_Thesis_ProbDef_v5.tex
\section{The problem}
\subsection{Introduction}
Markov processes play a central role in a wide range
of applications in physical sciences and engineering, biology and medicine,
industry, finance and business, and in other fields. 
They are especially well understood in the case
where the the process takes place in the complete Euclidean space, $\mathbb{R}^n$.
However, the situation becomes considerably more complicated if a Markov process is
restricted to a domain with a boundary. For, in that case, a satisfactory theory exists only
for processes, such as Brownian motion, with continuous trajectories. Further significant complications arise for Markov processes with jumps. Informally, at least, 
these difficulties are best understood by the observation that a process with
continuous paths can leave a domain \textbf{only} by passing through its boundary,
whilst a process with discontinuous trajectories can jump into the complement of the domain
without ever hitting the boundary. \\

The wider context for this work is a large class of
Markov processes -- namely the so-called Feller and (hence) L\'evy processes. (For more details on Feller and L\'evy processes, see Appendix \ref{FellerLevy}.)
It follows from a well known result of Ph. Courr\`ege that the
generator of a Feller process in $\mathbb{R}^n$ is a pseudodifferential operator. Moreover, this pseudodifferential operator
can be represented as a
sum of a second order partial differential operator and an integro-differential
operator with a L\'evy kernel (see \cite{Ja}). It turns out that each term in this
representation has a simple probabilistic interpretation. The second
order differential operator describes the diffusion
part, the first order terms are related to the drift, the zero
order term is responsible for the killing part and, finally, the
L\'evy kernel describes the jumps of the paths.  \\

An important example of a linear partial differential operator with constant coefficients is $-\Delta$, where $\Delta$ is the Laplacian. It is straightforward to show that
\begin{equation*} 
- \Delta = \mathcal{F}^{-1} |\xi|^2 \mathcal{F},
\end{equation*}
and thus the symbol of $-\Delta$ is $|\xi|^2$. \\

This research is concerned with non-classical boundary-value problems for elliptic pseudodifferential operators. Such problems can often be associated with Feller and L\'{e}vy processes \cite{Ap}, \cite{Ja}. For example, the generator of the symmetric $2 \alpha$-stable L\'{e}vy process in $\mathbb{R}^n$ is the \textit{fractional} Laplacian \cite{DPV}, \cite{Va} which can be defined by
\begin{equation} \label{fraclaplacian}
(- \Delta)^\alpha := \mathcal{F}^{-1} |\xi|^{2\alpha} \mathcal{F},  \quad 0 < \alpha < 1.
\end{equation}
Unlike the Laplacian, the fractional Laplacian is a \textit{non-local} operator. In a discrete setting, say the lattice $\mathbb{Z}^n$, we can think of the fractional Laplacian as random walk in which a particle may experience arbitrarily long jumps, albeit with a small probability, \cite{Va}. There are at least ten equivalent definitions of the fractional Laplacian, see \cite{Kwas}, and it can be very useful to interchange them. Indeed, suppose $u \in H^\alpha_2 (\mathbb{R}^n)$, then we can also write
\begin{equation} \label{fractlap}
(-\Delta)^\alpha u(x) = \lim_{\epsilon \searrow 0} \, c_{n,\alpha} \int_{y \in \mathbb{R}^n, |x-y| > \epsilon} \dfrac{u(x)-u(y)}{|x-y|^{n+2\alpha} }\,dy,
\end{equation}
where the constant $c_{n,\alpha}$ depends only on $n$ and $\alpha$. \\

Finally, we make special note of an excellent survey paper, \cite{DPV}, that details applications as disparate as crystal dislocation, finance and water waves, where the fractional Laplacian is playing an important role. \\

Let $G \subset \mathbb{R}^n$. In 1938, M. Riesz \cite{Ri} showed that the correct boundary condition for the Dirichlet problem for the fractional Laplacian is not the customary $u|_{\partial G} = g$ but rather the \textit{balayage} condition:
\begin{equation*}
(-\Delta)^\alpha u = 0 \text{ in }G, \quad u|_{\mathbb{R}^n \setminus G} = g.
\end{equation*}
Heuristically, the difference between the Laplacian and the fractional Laplacian can be described very simply in terms of stochastic processes. For the Laplacian, inside a smooth boundary, the process paths follow Brownian motion and are therefore continuous (almost surely). On the other hand, the paths for the fractional Laplacian are right continuous with left limits, or \textit{c\`{a}dl\`{a}g}. In this case, the first hitting of $G^c$ will not occur (almost surely) at the boundary $\partial G$, but rather somewhere in $\overline{G}^c$. That is to say, the process will jump over the boundary (almost surely).  \\

\subsection{Truncation} \label{sectiontruncation}
Let us consider a general pseudodifferential operator $B$. Of course, the definition of $B$ depends on the Fourier transform, which acts on the whole of $\mathbb{R}^n$.  Now suppose that our domain of interest is some set $G \subset \mathbb{R}^n$. \\

Starting from $G$, an obvious way forward is to extend any functions defined on $G$ by zero to the whole of $\mathbb{R}^n$, apply the pseudodifferential operator $B$ and then restrict back to $G$.  To make this more precise, suppose that $Y(\mathbb{R}^n), \, Z(\mathbb{R}^n)$ are functions spaces defined on $\mathbb{R}^n$. Let $r_G$ denote the operator of restriction from $\mathbb{R}^n$ \textit{to} $G$. Then for any function space $Y(\mathbb{R}^n)$, we define
\begin{equation*}
Y(\overline{G}) := r_G Y(\mathbb{R}^n).
\end{equation*}
Similarly, we define $e_G$ to be the operator of extension, by zero, \textit{from} $G$ to $\mathbb{R}^n$. Hence, we can write
\begin{equation*} 
B_G u := r_G B e_G u, \quad u \in Y(\overline{G}).
\end{equation*}

Unfortunately, this simple approach can have some difficult consequences. For example, it may be the case that the extended function $e_G u$ is discontinuous on the boundary of $G$, and this may cause singularities in $B \, e_Gu$. Indeed, even when $B:Y(\mathbb{R}^n) \rightarrow Z(\mathbb{R}^n)$ and $u \in Y(\overline{G})$, it may happen that $B_Gu \not \in Z(\overline{G})$.  An example of this is given in Appendix \ref{AppendixYZmapping}.\\

\subsection{Transmission conditions} \label{sectiontransmission}
One way to mitigate the singularities at the boundary of the domain, $G$, caused by truncation, is to consider  pseudodifferential operators that satisfy the \textit{transmission condition}, see \cite{GH}, on $G$. Indeed, we say that $B$ satisfies the transmission condition if
\begin{equation*}
B_G: C^\infty(\overline{G}) \rightarrow C^\infty(\overline{G}).
\end{equation*}
There is an equivalent condition on the symbol of the operator $B$ described in local coordinate systems in a neighbourhood of $\partial G$. The transmission condition is widely applicable and there is a well developed theory, notably \cite{Bo},  for elliptic boundary value problems. (See also \cite{Grubb2008}.) \\

Unfortunately, the fractional Laplacian does \textit{not} satisfy the transmission condition. Moreover, this is a characteristic of Feller processes. The generator of a subordinate diffusion does not have the transmission condition \cite{Is}. So, in the context of Feller processes, the transmission condition is too restrictive.  Fortunately, there is a theory of boundary value problems for elliptic pseudodifferential operators in the absence of the transmission condition. This work was pioneered by M.I. Vishik and G. Eskin, see \cite{Es}, in the 1960's. The main tool is the Wiener-Hopf factorization. \\

In a series of recent papers, see for example \cite{Grubb2014,Grubb2015}, Grubb has adopted a more general transmission condition. Indeed, we say that $B$ satisfies the $\mu-$transmission condition if
\begin{equation*}
B_G: x^\mu_n C^\infty(\overline{G}) \rightarrow C^\infty(\overline{G}).
\end{equation*}
In Appendix \ref{GGrubb}, we consider the fractional Laplacian, $(-\Delta)^\alpha$, acting on the domain $G = {\mathbb{R}^n_+}, \, n \geq 2$. It can be shown that the fractional Laplacian satisfies the $\mu-$transmission condition with $\mu = \alpha$. To accommodate the singularities that arise at the boundary, this approach seeks solutions to the Dirichlet problem in the so-called H\"ormander space, $H^{\alpha (t + 2 \alpha)}(\overline{\mathbb{R}^n_+})$, defined for $t \geq 0$, as

\begin{equation*}  
H^{\alpha (t + 2 \alpha)}(\overline{\mathbb{R}^n_+}) := \mathcal{F}^{-1} (\langle \xi' \rangle - i \xi_n )^{-\alpha} \mathcal{F} \,   (e_+ H^{t+\alpha}(\overline{\mathbb{R}^n_+})) ,
\end{equation*}
where $\xi=(\xi', \xi_n)$.\\

If $ -\tfrac{1}{2} <  t + \alpha < \tfrac{1}{2}$ then, see Section 2.8.7, p.158, \cite{Tr83}, we can identify $e_+ H^{t+\alpha}(\overline{\mathbb{R}^n_+})$ with $\widetilde{H}^{t+\alpha}(\overline{\mathbb{R}^n_+})$. Hence,
\begin{equation*}
H^{\alpha (t + 2 \alpha)}(\overline{\mathbb{R}^n_+}) = \widetilde{H}^{t+ 2 \alpha}(\overline{\mathbb{R}^n_+}) \quad \text{if} \quad -\tfrac{1}{2} <  t + \alpha < \tfrac{1}{2}. 
\end{equation*}

On the other hand, if $t + \alpha > \tfrac{1}{2}$ then functions from $e_+ H^{t+\alpha}(\overline{\mathbb{R}^n_+})$ may have a jump at $x_n=0$. This gives rise to a singularity at $x_n=0$ when the operator $\mathcal{F}^{-1} (\langle \xi' \rangle - i \xi_n )^{-\alpha} \mathcal{F}$ is applied.\\

In summary, Grubb presents a useful approach, combining the power of Wiener-Hopf factorisation with a more general transmission condition, that works well for fractional powers of elliptic operators. In this research, we take a parallel, but rather different, path where we perturb the equation, by adding a carefully chosen potential, \textit{before} solving the Dirichlet problem. It turns out that this will allow us to seek solutions in (the more conventional) Bessel potential spaces.\\

\subsection{Adding a potential term} \label{sectionpotential}
The Dirichlet form, see \cite{MR}, associated with $(-\Delta)^\alpha$ is given by
\begin{equation*}
\mathcal{E}^{(\alpha)} (u,v) = \int_{\mathbb{R}^n} |\xi|^{2 \alpha} \hat{u}(\xi) \overline{ \hat{v}(\xi) } d\xi,
\end{equation*}
or, equivalently,
\begin{equation*}
\mathcal{E}^{(\alpha)} (u,v) = \dfrac{c_{n,\alpha}}{2} \int_{\mathbb{R}^n} \int_{\mathbb{R}^n} \dfrac{ (u(x)-u(y))(\overline{v(x)}-\overline{v(y)}) } {|x-y|^{n+2\alpha} }\,dxdy,
\end{equation*}
where, see for example \cite{Kwas}, the positive constant $c_{n, \alpha}$ is given by
\begin{equation*}
c_{n,\alpha} = - \dfrac{2^{2\alpha} \Gamma(\alpha + \frac{n}{2})}{\pi^\frac{n}{2} \Gamma(- \alpha)}.
\end{equation*} 

If $u|_{\mathbb{R}^n \setminus G} = 0$ and $v|_{\mathbb{R}^n \setminus G} = 0$, we can write
\begin{equation} \label{DirichletForm}
\mathcal{E}^{(\alpha)} (u,v) = \dfrac{c_{n,\alpha}}{2} \int_{G} \int_{G} \dfrac{ (u(x)-u(y))(\overline{v(x)}-\overline{v(y)}) } {|x-y|^{n+2\alpha} }\,dxdy.
\end{equation}
Dirichlet forms relating to the symmetric stable L\'{e}vy process on a domain have been studied by a number of authors. See, for example, Bogdan et al. \cite{BBC}, Chen and Kim \cite{CK}, Chen and Song \cite{CS}, Guan and Ma \cite{GM1, GM2}. \\

The generator of the Dirichlet form \eqref{DirichletForm} is known as the \textit{regional} fractional Laplacian, and can be written as
\begin{equation} \label{regFLdefn}
\Lambda^\alpha_G u(x) := \lim_{\epsilon \searrow 0} \, c_{n,\alpha} \int_{y \in G, |x-y| > \epsilon} \dfrac{u(x)-u(y)}{|x-y|^{n+2\alpha} }\,dy, \quad x \in G.
\end{equation}
For more details see, for example, \cite{GM1, GM2}.  Of course, if $G$ is the whole of $\mathbb{R}^n$, then $\Lambda^\alpha_{\mathbb{R}^n} = (-\Delta)^\alpha$, and we simply recover equation \eqref{fractlap}. \\

Using a weak representation of \eqref{regFLdefn}, in their Corollary 7.7, Guan and Ma \cite{GM1} consider a Dirichlet boundary value problem in $G$, where the boundary value defined on $\partial G$ is continuous. They show that, subject to several technical qualifications, there exists a unique (low regularity) solution in $H^\alpha(\overline{G}) \cap C_b(\overline{G})$ for all $\alpha$ in the range $0 < \alpha < 1$, where $C_b(\overline{G})$ denotes the space of continuous bounded functions on $\overline{G}$.\\

In preparation for this research, we now express the regional fractional Laplacian directly in terms of the operator $(-\Delta)^\alpha$.  \\

Let us define
\begin{equation} \label{kappaaG}
\kappa^\alpha_G :=  \Lambda^\alpha_G -  r_G (-\Delta)^\alpha e_G,
\end{equation}
and then, for $x \in G$, we have
\begin{align*}
(\kappa^\alpha_G u)(x) & = c_{n,\alpha}\lim_{\epsilon \searrow 0}  \int_{y \in G, |x-y| > \epsilon} \dfrac{ u(x) -  u(y)}{|x-y|^{n+2\alpha}} \, dy - c_{n,\alpha} \int_{y \in \mathbb{R}^n} \dfrac{e_G u(x) - e_G u(y)}{|x-y|^{n+2\alpha}} \, dy  \\
& = - c_{n,\alpha} \int_{y \in \mathbb{R}^n \setminus G} \dfrac{e_G u(x) - e_G u(y)}{|x-y|^{n+2\alpha}} \, dy \\
& = - c_{n,\alpha} u(x) \int_{y \in \mathbb{R}^n \setminus G} \dfrac{1}{|x-y|^{n+2\alpha}} \, dy,
\end{align*}
since if $y \in G$ then $e_G u(y) = u(y)$, and $e_G u(y)=0$ if $y \in \mathbb{R}^n \setminus G$. \\

In other words, 
\begin{equation} \label{killingpot}
\kappa^\alpha_G(x) = - c_{n,\alpha} \int_{y \in \mathbb{R}^n \setminus G} \dfrac{1}{|x-y|^{n+2\alpha}} \, dy, \quad x \in G.
\end{equation}

So, if $x \in G$, then, from equation \eqref{killingpot},
\begin{align*}
\kappa^\alpha_G(x) &= - c_{n,\alpha} \int_{\mathbb{R}^n} \dfrac{\chi_{\mathbb{R}^n \setminus G}(y)}{|x-y|^{n+2\alpha}} \, dy \\
&= c_{n,\alpha} \int_{\mathbb{R}^n} \dfrac{ ( \chi_{\mathbb{R}^n \setminus G }(x) - \chi_{\mathbb{R}^n \setminus G   }(y)  )}{|x-y|^{n+2\alpha}} \, dy \\
&= (-\Delta)^\alpha (\chi_{\mathbb{R}^n \setminus G})(x),
\end{align*}
where $\chi_{\mathbb{R}^n \setminus G}$ denotes the characteristic function of the set ${\mathbb{R}^n \setminus G}$. \\

Hence, we can write the regional fractional Laplacian, purely in terms of $(-\Delta)^\alpha$, as
\begin{equation} \label{regFLalt}
\Lambda^\alpha_G  = r_G (-\Delta)^\alpha e_G + r_G ((-\Delta)^\alpha \chi_{\mathbb{R}^n \setminus G}) \, I.
\end{equation} \\
It will be useful to think of the term $r_G ((-\Delta)^\alpha \chi_{\mathbb{R}^n \setminus G})$ as an \textit{added potential} to the (truncated) operator $r_G (-\Delta)^\alpha e_G$. \\

\subsection{Transcendental equation}
As an example, let us now consider equation \eqref{regFLalt}, in the case $n=1$ and $G= \mathbb{R}_+$. We denote $r_G, \, e_G, \, \Lambda^\alpha_G$ by $r_+, \, e_+$ and $\Lambda^\alpha_+$ respectively. Let $\theta$ denote the Heaviside step function. \\

Then, see equation (2.34), p. 23, \cite{Es}, 
\begin{equation*}
\mathcal{F}_{x \to \xi} e_+ x^\lambda = e^{i \pi (\lambda+1)/2} \cdot \dfrac{\Gamma(\lambda+1)}{\sqrt{2 \pi}} \cdot (\xi + i 0)^{-\lambda -1}, \quad \lambda > -1,
\end{equation*}
where $(\xi + i 0)^{\mu} := |\xi|^\mu \, e^{i \mu \pi \theta(-\xi)}$, for any $\mu \in \mathbb{R}$. \\

Noting that $\operatorname{sgn}(\xi) = 2 \, \theta(\xi) - 1$, we can write
\begin{equation*}
\mathcal{F} e_+ x^\lambda = \dfrac{\Gamma(\lambda+1)}{\sqrt{2 \pi}} \bigg \{ -\sin \dfrac{\lambda \pi}{2} + i \operatorname{sgn} (\xi) \,  \cos \dfrac{\lambda \pi}{2}\bigg \}\cdot |\xi|^{-\lambda-1},
\end{equation*}
and hence deduce 
\begin{equation*}
\mathcal{F} |x|^\lambda =  - \dfrac{2 \, \, \Gamma(\lambda+1)}{\sqrt{2 \pi}} \sin \dfrac{\lambda \pi}{2} \cdot |\xi|^{-\lambda-1}.
\end{equation*}

Thus, after some calculation, we have the simple result that
\begin{equation} \label{FLxlambda}
r_+ (-\Delta)^\alpha e_+ \, x^\lambda = \bigg (\dfrac{ \Gamma(2 \alpha - \lambda) \Gamma(1+ \lambda) \sin \pi( \alpha - \lambda)}{\pi} \bigg ) \cdot x^{\lambda - 2 \alpha},
\end{equation}
and
\begin{equation*}
r_+ (-\Delta)^\alpha \, \chi_{\mathbb{R}_\pm} =  \bigg (\pm \dfrac{ \Gamma(2 \alpha)  \sin \pi \alpha}{\pi} \bigg ) \cdot x^{- 2 \alpha}.
\end{equation*}
Hence, from equation \eqref{regFLalt},
\begin{equation*}
\Lambda^\alpha_+ \, x^\lambda = r_+ (-\Delta)^\alpha e_+ \, x^\lambda + x^\lambda r_+ \big ( (-\Delta)^\alpha \, \chi_{\mathbb{R}_-} \big )= 0,
\end{equation*}
if and only if
\begin{equation} \label{TEatzeroFL}
\Gamma(2 \alpha - \lambda) \Gamma(1+ \lambda) \sin \pi( \alpha - \lambda) = \Gamma(2 \alpha)  \sin \pi \alpha.
\end{equation}

Of course, as expected, given the homogeneity of $|\xi|^{2\alpha}$, we have the immediate solution $\lambda =0$. \\

Later, in this research, we consider an \enquote{inhomogeneous} variant of the fractional Laplacian. We shall see that calculations involving the Fourier transform are considerably more complex. However, the transcendental equation \eqref{TEatzeroFL} will reappear and, moreover, will play a critical role. (See Chapter \ref{ChapterTE} and, particularly, equation \eqref{TEatzero}.)\\

\newpage
 \subsection{Problem statement} \label{ProblemStatement}
Equation \eqref{regFLalt} provides the central motivation for this current research. Indeed, given a pseudo-differential operator $B$, acting on $\mathbb{R}^n$, we will, in general, consider the operator
\begin{equation} \label{regFLaltgen}
\mathcal{B}_G  := r_G B e_G + r_G (B \chi_{\mathbb{R}^n \setminus G}) I.
\end{equation}
 
We shall demonstrate now that the operator $B_G + (r_G \, B \chi_{\mathbb{R}^n \setminus G}) I$ is always  \enquote{less singular} than $B_G$ itself. Let $U \in Y(\mathbb{R}^n)$ be any extension of $u \in Y(\overline{G})$. Then, by definition, $r_G U = u$ and 
\begin{align*}
B_Gu + (r_G B\chi_{\mathbb{R}^n \setminus G})u & = r_G Be_Gu + (r_G B\chi_{\mathbb{R}^n \setminus G}) r_G U \\ 
& = r_G Be_Gu + r_G (U B\chi_{\mathbb{R}^n \setminus G}) \\ 
& = r_G (Be_Gu + U B\chi_{\mathbb{R}^n \setminus G}) \\
& = r_G (B (U\chi_G) + B(U\chi_{\mathbb{R}^n \setminus G}) + [U I,B]\chi_{\mathbb{R}^n \setminus G}) \\
& = r_G (BU + [U I,B]\chi_{\mathbb{R}^n \setminus G}), 
\end{align*}
where $[U I,B]$ is the commutator given by 
\begin{equation*}
[U I,B] = UB - (BU)I.
\end{equation*}
Of course, $[UI ,B]$ is of \textit{lower order} than $B$, if $U$ has a degree of smoothness. \\

In summary, the perturbed operator has the simple form
\begin{equation} \label{rBurComm}
B_Gu + (r_G B\chi_{\mathbb{R}^n \setminus G})u = r_G (BU + [U I,B]\chi_{\mathbb{R}^n \setminus G}).
\end{equation}
Moreover, it can be determined purely in terms of a restriction, to the domain $G$, of an appropriate interaction between $B$ and $U$ over $\mathbb{R}^n$.  \\

\begin{remark}
It is easy to see that the representation on the right-hand side of \eqref{rBurComm} is independent of the extension $U$. Indeed, suppose that $U_1, U_2 \in Y(\mathbb{R}^n)$ are two extensions of $u \in Y(\overline{G})$. Then, by definition, $r_GU_1= u = r_GU_2$. Moreover,
\begin{align*}
r_G (BU_1 & + [U_1 I,B]\chi_{\mathbb{R}^n \setminus G}) - r_G (BU_2 + [U_2 I,B]\chi_{\mathbb{R}^n \setminus G}) \\
& = r_GB(U_1-U_2) + r_G (U_1-U_2)B \chi_{\mathbb{R}^n \setminus G} -r_GB(U_1-U_2)\chi_{\mathbb{R}^n \setminus G} \\
& = r_GB(U_1-U_2)  -r_GB(U_1-U_2)\chi_{\mathbb{R}^n \setminus G} \\
& = r_GB(U_1-U_2)\chi_G \\
& = 0. \\
\end{align*}
\end{remark} 

The genesis of this research is the simple realisation that adding the potential term, 
$r_G (B \chi_{\mathbb{R}^n \setminus G})$, actually does improve the situation. This is because the leading singular terms of $r_G B e_G$ and $r_G (B \chi_{\mathbb{R}^n \setminus G}) I$ cancel each other at $\partial G$. Appendix \ref{AppendixYZmapping} details a simple illustrative example of this phenomenon. \\

So, the resulting (perturbed) operator $\mathcal{B}_G$ is \enquote{less singular}  than either the truncated operator $r_G B e_G$ or the multiplier $r_G (B \chi_{\mathbb{R}^n \setminus G})$ when viewed separately. (In passing, we note that the cancellation of singularities will not occur if the coefficient of  $r_G (B \chi_{\mathbb{R}^n \setminus G}) u$ is a constant different from 1.) \\

Given the above analysis, the underlying problem, in abstract terms, is to develop a theory of boundary value problems for operators which are sums of pseudodifferential operators and ``fine-tuned potentials", which when taken together are less singular
than the corresponding pseudodifferential operators considered on their own.  In this research, to make the problem more tractable, we shall consider a particular elliptic pseudodifferential operator chosen because it possesses two important characteristics, namely:
\begin{enumerate}[\hspace{1cm}(a)]
\item All the richness intrinsic to the fractional Laplacian;
\item Representative of a large class of operators. 
\end{enumerate}

For additional simplicity, we shall restrict our attention to one spatial dimension. However, it worth remarking that even one-dimensional problems
of this kind can have important applications in various fields including financial mathematics
(non-Gaussian market models).  \\

Finally, suppose $0 < \alpha < 1$. Let $A$ denote the pseudodifferential operator of order $2 \alpha$, with symbol 
\begin{equation} \label{symboldef}
A(\xi) = (1+ \xi^2)^\alpha. 
\end{equation}
Our (simplified) problem is to investigate the solvability of the equation
\begin{equation} \label{rAef}
\mathcal{A} u := r_+ \, A \, e_+ u + u \, r_+ A( \chi_{\mathbb{R}_{-}}) = f, 
\end{equation}
where 
 $u \in H^s_p(\overline{\mathbb{R}_+})$ for a given $f \in H^{s-2\alpha}_p(\overline{\mathbb{R}_+})$.  We assume that
 \begin{equation} \label{constraintonp}
 1 < p < \infty,
 \end{equation}
 since, in this case, $H^s_p(\mathbb{R})$ is a Banach space and we can also apply the techniques and results of harmonic analysis.\\

\newpage
\subsection{Outline of this research}
From equations \eqref{symboldef} and \eqref{rAef}, we note our operator of interest, $\mathcal{A}$, is defined via the Fourier transform. In Chapter \ref{ChapterSingInt}, we formulate $\mathcal{A}$, acting on a restriction of the Schwartz space $S(\mathbb{R})$, as a singular integral. Using this representation, in Chapter \ref{ChapterTrivKer}, we establish conditions under which the operator $\mathcal{A}$, acting on Bessel potential spaces, is bounded. Moreover, at least for the case $p=2$, we also determine sufficient conditions for $\mathcal{A}$ to have a trivial kernel. Later, in Chapters \ref{ChapterIndInv} and \ref{ChapterHigherReg}, this latter result is generalised to any $p$ satisfying the constraint $1 < p < \infty$. \\

We begin with the case of lower regularity, where $1/p < s < 1 + 1/p$. In Chapters \ref{ChapterOpAlgI} and \ref{OpAlgLp}, we reformulate the problem in $L_p(\mathbb{R}_+)$, in terms of an operator algebra containing multiplication, Mellin and Wiener-Hopf operators. Our goal is to establish precise conditions under which $\mathcal{A}$ is Fredholm, and then calculate its index. A significant part of the index calculation, see Chapter \ref{GenSymbol}, involves a certain transcendental equation, which includes terms containing gamma functions with complex arguments. \\

The analysis for $1/p < s < 1 + 1/p$ is completed in Chapter \ref{ChapterIndInv}, where we determine the range of parameters for which the Fredholm index is zero. Given this, and the trivial kernel results, we are able to establish the conditions under which $\mathcal{A}$ is invertible. Of course, Fredholm theory in the multi-dimensional case requires invertibility of the one-dimensional operator.\\

In Chapter \ref{ChapterHigherReg}, the case of higher regularity, namely $1 + 1/p < s < 2 + 1/p$, is examined using the methods established previously. To improve readability, and also because of its significant technical complexity, we delay detailed examination of the transcendental equation until Chapter \ref{ChapterTE}. Finally, areas of possible future research are discussed in Chapter \ref{ChapterFutureResearch}.

%% file: KCL_Thesis_KeyResults_v5.tex
\section{Key results} \label{ChapterKeyResults}
Suppose $\epsilon > 0$, and let $\chi_\epsilon$ denote the characteristic function of the interval $(\epsilon, \infty)$. Let the space $H^s_{p,0}(\overline{\mathbb{R}_+})$ be as defined in \eqref{Hsp0definition}. \\

The operator $\mathcal{A}$ can be represented as a (hyper-)singular integral operator.
\begin{theorem} 
Suppose $0 < \alpha < 1$. Let $v \in S(\mathbb{R})$ and define $u := r_+ v$. Then, for $x > 0$,
\begin{equation*}
(\mathcal{A}u)(x) = u(x) +\lim_{\epsilon \searrow 0} \int^\infty_0 (u(x)-u(y)) \, \chi_\epsilon( |x-y| ) \, m(|x-y|) \, dy.
\end{equation*}
Moreover,
\begin{equation*}
(\mathcal{A}u, u) =  \int^\infty_0 |u|^2 \, dx + \tfrac{1}{2} \int^\infty_0 \int^\infty_0 |u(x)-u(y)|^2 m(|x-y|) \, dy dx,
\end{equation*}
for a certain function $m(w)$, which is $O(|w|^{-1-2 \alpha})$ for small $|w|$ and $O(e^{-|w|})$ as $|w| \to \infty$. \\
\end{theorem}

\begin{remark}
We have the following explicit representation
\begin{equation*}
m(y) = \dfrac{\alpha}{\Gamma(1-\alpha)} \, \dfrac{2^{\frac{1}{2}+\alpha}}{\sqrt{\pi}} |y|^{-\frac{1}{2}-\alpha} K_{\frac{1}{2}+\alpha}(|y|),
\end{equation*}
where $K_\nu$ is a modified Bessel function of the second kind of order $\nu$. See, for example, Chapter 10, \cite{NIST}.\\
\end{remark}

Under certain conditions $\mathcal{A}$ is bounded and has a trivial kernel.
\begin{theorem} 
Suppose $1 < p < \infty$ and $0 < \alpha <1$. If either
\begin{enumerate} [\hspace{18pt}(a)]
\item $2\alpha -1 + 1/p  < s < 1+1/p \text{ then } \mathcal{A}:  H^s_p(\overline{\mathbb{R}_+}) \to H^{s-2\alpha}_p(\overline{\mathbb{R}_+}) \text{ is bounded},  \text{ or }$ 
\item $1+1/p < s < 2+1/p \text{ then } \mathcal{A}:  H^s_{p,0}(\overline{\mathbb{R}_+}) \to H^{s-2\alpha}_p(\overline{\mathbb{R}_+}) \text{ is bounded}$.
\end{enumerate} 

Moreover, if $p=2$ and either
\begin{enumerate} [\hspace{18pt}(i)]
\item $0 < \alpha < \tfrac{1}{2}, \,\,  \tfrac{1}{2}  < s < 1+\tfrac{1}{2},   \text{ or }$ 
\item $0 < \alpha < 1, \,\,  1 +\tfrac{1}{2}  < s < 2+\tfrac{1}{2}$, 
\end{enumerate} 
then $\mathcal{A}$ has a trivial kernel. \\
\end{theorem}

\begin{remark}
The condition that $p=2$ for $\mathcal{A}$ to have a trivial kernel is not as restrictive as it might appear. Under appropriate conditions, we will be able to determine sufficent conditions for $\mathcal{A}$ to have a trivial kernel for any $p$ in the range $1 < p < \infty$, using the result (above) for $p=2$. \\
\end{remark}

Let $\tau := s -1 /p$. Then, it turns out that the following transcendental equation, see \eqref{TEatzeroFL}, will play a pivotal role in our analysis:
\begin{equation*} 
\Gamma(2 \alpha - \tau) \Gamma(\tau+1) \sin \pi(\alpha-\tau) = \Gamma(2 \alpha) \sin \pi \alpha.
\end{equation*} 

Indeed, if $0 < \alpha < \tfrac{1}{2}$ and $0 < \tau < 1$, we prove that equation \eqref{TEatzeroFL} has no solution. On the other hand, if $0 < \alpha < 1$ and $1 < \tau <2$, we prove that equation \eqref{TEatzeroFL} has a unique solution of the form $\tau = 1 + \alpha_c$, where $\alpha_c$ only depends on $\alpha$ and satisfies $0 < \alpha_c < \alpha$. \\

Finally,  via a calculation of the Fredholm index, we establish conditions for the invertibility of $\mathcal{A}$.
\begin{theorem}
Suppose $0 < \alpha < \tfrac{1}{2}, \,\, 1 < p < \infty$ and $1/p  < s < 1+1/p$. Then the operator $\mathcal{A}: H^s_p(\overline{\mathbb{R}_+}) \to H^{s-2\alpha}_p(\overline{\mathbb{R}_+})$ is invertible. \\
\end{theorem} 

\begin{theorem}
Suppose $0 < \alpha < 1, \,\, 1 < p < \infty$ and $1 + 1/p < s < 1+1/p + \alpha_c$. Then the operator $\mathcal{A}: H^s_{p,0}(\overline{\mathbb{R}_+}) \to H^{s-2\alpha}_p(\overline{\mathbb{R}_+})$ is invertible. \\

On the other hand, if $0 < \alpha <1, \,\, 1 < p < \infty$ and $1 + 1/p + \alpha_c < s < 2 +1/p$, then $\mathcal{A}$ has a trivial kernel and is Fredholm with index equal to $-1$. \\
\end{theorem}


%% file: KCL_Thesis_Chapter2_v5.tex
\chapter{Singular integral representation} \label{ChapterSingInt}
In this chapter we examine the operator $\mathcal{A}$ given in equation \eqref{rAef} in more detail. We consider its action on the restriction of the Schwartz space $S(\mathbb{R})$ to the positive half-line, and formulate a singular integral representation. \\

Suppose $\epsilon > 0$, and let $\chi_\epsilon$ denote the characteristic function of the interval $(\epsilon, \infty)$.

\section{Main result}

\begin{theorem} \label{Theorem:SingularIntegralRep}
Suppose $0 < \alpha < 1$. Let $v \in S(\mathbb{R})$ and define $u := r_+ v$. Then, for $x > 0$,
\begin{equation} \label{Avsingularint}
(\mathcal{A}u)(x) = u(x) +\lim_{\epsilon \searrow 0} \int^\infty_0 (u(x)-u(y)) \, \chi_\epsilon( |x-y|) \, m(|x-y|) \, dy.
\end{equation}
Moreover,
\begin{equation} \label{Avvinnerproduct}
(\mathcal{A}u, u) =  \int^\infty_0 |u|^2 \, dx + \tfrac{1}{2} \int^\infty_0 \int^\infty_0 |u(x)-u(y)|^2 m(|x-y|) \, dy dx,
\end{equation}
for a certain function $m(w)$, which is $O(|w|^{-1-2 \alpha})$ for small $|w|$ and $O(e^{-|w|})$ as $|w| \to \infty$. \\
\end{theorem}

\begin{remark}
We have the following explicit representation
\begin{equation*}
m(y) = \dfrac{\alpha}{\Gamma(1-\alpha)} \, \dfrac{2^{\frac{1}{2}+\alpha}}{\sqrt{\pi}} |y|^{-\frac{1}{2}-\alpha} K_{\frac{1}{2}+\alpha}(|y|),
\end{equation*}
where $K_\nu$ is a modified Bessel function of the second kind of order $\nu$. See, for example, Chapter 10, \cite{NIST}.\\
\end{remark}

\begin{remark}
Consider the integral operator representation for $\mathcal{A}$ given by equation \eqref{Avsingularint} in Theorem \ref{Theorem:SingularIntegralRep}. Suppose $x > 0$ is fixed and let $\epsilon \searrow 0$. Then, in a small neighbourhood of $x$ the integrand has a singularity which is (typically) of order $-2\alpha$. In particular, if $\tfrac{1}{2} \leq \alpha < 1$ then the integral is hypersingular. Nonetheless, even in this case, the limit as $\epsilon \searrow 0$ does exist and is finite. It turns out that this is due to a cancellation, arising from the fact that the weight $m$ is symmetric about $x$. \\

On the other hand, the double integral in equation \eqref{Avvinnerproduct} in the inner product $(\mathcal{A}u, u)$ has a weaker singularity of order $-2\alpha +1$. We will show that, for all $0 < \alpha < 1$, the double integral exists, in the conventional sense, and is finite. \\
\end{remark}

\begin{definition} \label{mepsdefinition}
Suppose $0 < \epsilon < 1$. Then we define
\begin{equation*}
m_\epsilon (y) := \chi_\epsilon(y) \, m(y), \quad y \geq 0.
\end{equation*}
In particular, given $\epsilon$, the function $m_\epsilon$ is bounded. Given $m_\epsilon$, we further define
\begin{equation*}
(\mathcal{A}_\epsilon u)(x) := u(x) + \int^\infty_0 (u(x)-u(y)) m_\epsilon(|x-y|) \, dy \quad (x>0).
\end{equation*}
\end{definition}

\section{Proof of main result}
\begin{lemma} \label{LemuplusInt}
Suppose $0 < \alpha < 1$. Let $v \in S(\mathbb{R})$ and define $u := r_+ v$. Then, for $x > 0$,
\begin{equation*}
(\mathcal{A}u)(x) = u(x) + \lim_{\epsilon \searrow 0} \int^\infty_0 (u(x)-u(y)) \, m_\epsilon (|x-y|) \, dy.
\end{equation*}
\end{lemma}

\begin{proof}
We have $v \in S(\mathbb{R})$ and $u := r_+ v$. Moreover, from equation \eqref{rBurComm},
\begin{equation*}
r_+ \, A \, e_+ u + u \, r_+ A(\chi_{\mathbb{R}_-})= r_+ \, A \, v +  r_+ \, [vI,A]\chi_{\mathbb{R}_-}.
\end{equation*}

From Lemma \ref{LemAChim}, 
\begin{equation*}
r_+ \, (v A \chi_{\mathbb{R}_-})(x) = - r_+ \,\int^0_{-\infty} v(x) m(|x-y|) \, dy \quad (x > 0).
\end{equation*}

On the other hand, by Lemma \ref{LemAvchi}
\begin{equation*}
r_+ \, A(v\chi_{\mathbb{R}_-})  = -r_+ \int^0_{-\infty}  v(y)m(|x-y|) \, dy. 
\end{equation*}

Thus, combining these results
\begin{align} \label{r+commutatorchim}
r_+ \, [vI,A]\chi_{\mathbb{R}_-}(x) & := r_+ \, \big\{ v(x) (A \chi_{\mathbb{R}_-})(x) - A(v \,\chi_{\mathbb{R}_-})(x) \big \} \nonumber\\
& = -r_+ \int^0_{-\infty}  (v(x)-v(y)) m(|x-y|) \, dy. 
\end{align}

Finally, from Lemma \ref{LemAvInt}, 
\begin{align*} 
r_+ \, A \, e_+ u + u \, r_+ A(\chi_{\mathbb{R}_-}) &= r_+ \, A \, v +  r_+ \, [vI,A] \chi_{\mathbb{R}_-}\\
& = r_+ v + r_+ \, \lim_{\epsilon \searrow 0} \int^\infty_0 (v(x)-v(y)) m_\epsilon (|x-y|) \, dy \\
&= u + \lim_{\epsilon \searrow 0} \int^\infty_0 (u(x)-u(y)) \, m_\epsilon (|x-y|) \, dy \quad (x>0).
\end{align*}
This completes the proof of the lemma. \\
\end{proof}

\begin{lemma} \label{AvvIv2}
Suppose $0 < \alpha < 1$. Let $v \in S(\mathbb{R})$ and define $u := r_+ v$. Then, \begin{equation*}
(\mathcal{A}u, u) =  \int^\infty_0 |u|^2 \, dx + \tfrac{1}{2} \int^\infty_0 \int^\infty_0 |u(x)-u(y)|^2 m(|x-y|) \, dy dx.
\end{equation*}
\end{lemma} 

\begin{proof}
Suppose $0 < \epsilon < 1$. From Remark \ref{Remmy}, $m(|w|)$ is $O(|w|^{-1-2\alpha})$ for small $|w|$, and $O(e^{-|w|})$ for large $|w|$. Moreover, $m(w) > 0$ for all finite $w \geq 0$.\\

From Definition \ref{mepsdefinition},
\begin{align*}
m_\epsilon(w) & = \chi_\epsilon(w) \, m(w) \quad w \geq 0;\\
(\mathcal{A}_\epsilon u)(x) & = u(x) + \int^\infty_0 (u(x)-u(y)) m_\epsilon(|x-y|) \, dy \quad (x>0).
\end{align*}

Hence
\begin{equation} \label{eqnux2}
(\mathcal{A}_\epsilon u, u)  = \int^\infty_0 |u|^2 \, dx + \int^\infty_0 \int^\infty_0 \overline{u(x)} (u(x)-u(y)) m_\epsilon(|x-y|) \, dy dx.
\end{equation}

Interchanging the roles of $x$ and $y$
\begin{equation} \label{eqnuy2}
(\mathcal{A}_\epsilon u, u)  = \int^\infty_0 |u|^2 \, dx + \int^\infty_0 \int^\infty_0 (-1) \overline{u(y)} (u(x)-u(y)) m_\epsilon(|x-y|) \, dx dy.
\end{equation}

Using Fubini's theorem, and adding equations \eqref{eqnux2} and \eqref{eqnuy2},
\begin{equation} \label{Aepsilonvv}
(\mathcal{A}_\epsilon u, u) =  \int^\infty_0 |u|^2 \, dx + \tfrac{1}{2} \int^\infty_0 \int^\infty_0 |u(x)-u(y)|^2 m_\epsilon (|x-y|) \, dy dx.
\end{equation} \\

Our method of proof is to take the limit in equation \eqref{Aepsilonvv} as $\epsilon \searrow 0$. For the left-hand side we use the Dominated Convergence Theorem, and for the right-hand side we use the Monotone Convergence Theorem. \\

Firstly, consider the left-hand side.
\begin{align*}
(i) & \quad \lim_{\epsilon \searrow 0} (\mathcal{A}_\epsilon u)(x) \to (\mathcal{A} u)(x) \quad (x > 0) \quad (\text{Lemma } \ref{LemAvInt});\\
(ii) & \quad |(\mathcal{A}_\epsilon u)(x)| \leq \chi_{[0,1]}(x) g(x) + C \quad (C>0: \text{ Lemma }\ref{Aepsvestimate}),
\end{align*}
where $g(x)$ is $O(1)$ as $x \searrow 0$ for $0 < \alpha < \tfrac{1}{2}$, and is $O(x^{1-2 \alpha})$ as $x \searrow 0$ for $\tfrac{1}{2} \leq \alpha < 1$. \\

Hence,  $u(x) \big [\chi_{[0,1]}(x) g(x) + C \big ] \in L_1[0, \infty)$. \\

Therefore, by the Dominated Convergence Theorem,
\begin{equation*}
\int^\infty_0 u(x) (\mathcal{A}_\epsilon u)(x) \, dx \to \int^\infty_0 u(x) (\mathcal{A} u )(x) \, dx \quad \text{ as } \epsilon \searrow 0.
\end{equation*} \\

On the other hand, for the right-hand side as $ \epsilon \searrow 0$,
\begin{equation*}
\int^\infty_0 \int^\infty_0 |u(x)-u(y)|^2 m_\epsilon (|x-y|) \, dy dx \to \int^\infty_0 \int^\infty_0 |u(x)-u(y)|^2 m (|x-y|) \, dy dx,
\end{equation*}
by a routine application of the Monotone Convergence Theorem. This completes the proof of the lemma. \\
\end{proof}

\section{Supporting lemmas}
An infinitely differentiable function $f:(0, \infty) \to \mathbb{R}$ is said to be a \textit{Bernstein function} if
\begin{equation*}
f \geq 0; \qquad (-1)^k \dfrac{d^k f}{dx^k} \leq 0, \quad k=1,2,3,\dots
\end{equation*}
It is easy to verify directly from this definition that $f(x) = (1+x)^\alpha$ is a Bernstein function if $0 < \alpha <1$.  \\

Any Bernstein function $g:(0, \infty) \to \mathbb{R}$ can be written in the standard L\'evy-Khinchine representation, see equation (12), p.\,6, \cite{JS}, 
\begin{equation*}
g(x) = a + bx + \int^\infty_0 (1- e^{-xs}) \, \tau(ds),
\end{equation*}
where $\tau$ is a Radon measure on $(0,\infty)$ such that $\int^\infty_{0+} \min \{ s,1 \} \, \tau(ds) < \infty$. \\

From 3.434, p.\,361, \cite{GR},  we have
\begin{equation*}
\int^\infty_0 \dfrac{e^{-\nu s} - e^{-\mu s}}{s^{\rho+1}} \, ds = \dfrac{\mu^\rho - \nu^\rho}{\rho} \, \Gamma(1-\rho),
\end{equation*}
provided $\mu > 0, \, \nu >0$ and $\rho <1$. Taking $\nu =1, \, \mu = 1+x$ and $\rho= \alpha$ gives
\begin{equation*}
\int^\infty_0 \dfrac{e^{-s} - e^{-(1+x) s}}{s^{\alpha+1}} \, ds = \dfrac{(1+x)^\alpha - 1}{\alpha} \, \Gamma(1-\alpha).
\end{equation*}
Rearranging
\begin{equation} \label{oneplusxalpha}
(1+x)^\alpha = 1 + \dfrac{\alpha}{\Gamma(1-\alpha)}\int^\infty_0 (1- e^{-xs}) \, e^{-s}s^{-\alpha-1}  \, ds. \\
\end{equation}
In other words, in the standard form representation of the Bernstein function $(1+x)^\alpha$, for $0< \alpha <1$, we take $a=1, \, b=0$ and
\begin{equation*}
\tau(ds) = \dfrac{\alpha}{\Gamma(1-\alpha)}e^{-s}s^{-\alpha-1}  \, ds. 
\end{equation*}\\

\begin{remark} \label{relativiststabsubs}
If we take $x = \lambda / m^{\frac{1}{\alpha}}$ in equation \eqref{oneplusxalpha}, and make the change of variable $s \to s m^{\frac{1}{\alpha}}$ in the right-hand side, we obtain the following result:
\begin{equation*}
(\lambda + m^{1/\alpha})^\alpha - m = \dfrac{\alpha}{\Gamma(1-\alpha)} \int^\infty_0 (1- e^{-\lambda s}) \, e^{-m^{1/\alpha}s} s^{-\alpha-1}  \, ds. 
\end{equation*}
as given in Example 5.9, p.\,97, \cite{GS}, for relativistic stable subordinators. \\
\end{remark}

We say that a function $\psi: \mathbb{R} \to \mathbb{R}$ is \textit{negative definite} if for all $N \in \mathbb{N}$ and $\xi_1, \xi_2, \dots, \xi_N \in \mathbb{R}$ we have
\begin{equation*}
\psi(0) \geq 0; \qquad \sum^N_{j,k=1} \psi(\xi_j-\xi_k)\lambda_j \overline{\lambda_k} \leq 0, \quad \forall \lambda_j \in \mathbb{C} \text{ such that } \sum^N_{j=1} \lambda_j =0.
\end{equation*}
In particular, we note that $\psi(\xi) = \xi^2$ is negative definite. \\

Now since $(1+x)^\alpha$ is a Bernstein function,  we see immediately that $(1+\xi^2)^\alpha$, for $0< \alpha <1$, is also a continuous negative definite function. See, for example, \cite{JS}. Moreover, from equation \eqref{oneplusxalpha} and Lemma 2.1,  p.\,7, \cite{JS}, we have the general representation
\begin{equation*}
\displaystyle (1+ \xi^2)^\alpha = 1 + \int_{\mathbb{R}} (1- \cos y\xi) \,m(y) \, dy.
\end{equation*}

\begin{lemma}  \label{LemDefm}
Suppose $0 < \alpha < 1$. Then
\begin{equation*}
\displaystyle (1+ \xi^2)^\alpha = 1 + \int_{\mathbb{R}} (1- \cos y\xi) \,m(y) \, dy,
\end{equation*}
where $m(y) := \displaystyle \dfrac{\alpha}{\Gamma(1-\alpha)}\,  \int^\infty_0 \dfrac{1}
{\sqrt{4 \pi s}} \exp \bigg ( -\dfrac{y^2}{4s} \bigg ) \, e^{-s} s^{-1-\alpha} \, ds$.
\end{lemma}
\begin{proof}
From equation \eqref{oneplusxalpha}
\begin{equation*} 
(1+x)^\alpha = 1 + \dfrac{\alpha}{\Gamma(1-\alpha)}\int^\infty_0 (1- e^{-xs}) \, e^{-s}s^{-\alpha-1}  \, ds. \\
\end{equation*}

Now define
\begin{equation*}
\tau(s) := \dfrac{\alpha}{\Gamma(1-\alpha)} e^{-s}s^{-\alpha-1},
\end{equation*}
so that
\begin{equation*} 
(1+\xi^2)^\alpha = 1 + \int^\infty_0 (1- e^{-s \xi^2}) \tau(s) \, ds. \\
\end{equation*}

From Lemma \ref{Lem1minuse}, we have
\begin{equation*}
1- e^{-s \xi^2} = \int_{\mathbb{R}} (1- \cos \xi y) \dfrac{1}{\sqrt{4 \pi s}} \exp \bigg ( - \dfrac{y^2}{4 s} \bigg )\, dy.
\end{equation*}

Hence
\begin{align*}
(1+\xi^2)^\alpha & = 1 + \int^\infty_0 \int_{\mathbb{R}} (1- \cos \xi y) \dfrac{1}{\sqrt{4 \pi s}} \exp \bigg ( - \dfrac{y^2}{4 s} \bigg ) \tau(s) \, dy ds \\
& = 1 + \int_{\mathbb{R}} (1- \cos \xi y) \bigg \{ \int^\infty_0  \dfrac{1}{\sqrt{4 \pi s}} \exp \bigg ( - \dfrac{y^2}{4 s} \bigg ) \tau(s) \, ds \bigg \} \, dy \\
& = 1 + \int_{\mathbb{R}} (1- \cos y\xi) \,m(y) \, dy.
\end{align*}
\end{proof}

\begin{lemma} \label{Lem1minuse}
Suppose $s>0$, then
\begin{equation*}
1- e^{-s \xi^2} = \int_{\mathbb{R}} (1- \cos \xi y) \dfrac{1}{\sqrt{4 \pi s}} \exp \bigg ( - \dfrac{y^2}{4 s} \bigg )\, dy.
\end{equation*}
\end{lemma}
\begin{proof}
By a change of variable from the standard formula $\int^\infty_0 e^{-x^2} \, dx = \tfrac{\sqrt{\pi}}{2}$, we have
\begin{equation*}
\int^\infty_0 \exp(-q^2 y^2) \, dy = \dfrac{\sqrt{\pi}}{2q}, \quad q>0.
\end{equation*}

Hence 
\begin{align*}
\int_{\mathbb{R}} \dfrac{1}{\sqrt{4 \pi s}} \exp \bigg ( - \dfrac{y^2}{4 s} \bigg )\, dy &= 2 \int^\infty_0 \dfrac{1}{\sqrt{4 \pi s}} \exp \bigg ( - \dfrac{y^2}{4 s} \bigg )\, dy =1.
\end{align*}

Therefore, it remains to show that
\begin{equation*}
e^{-s \xi^2} = \int_{\mathbb{R}} \cos \xi y \cdot \dfrac{1}{\sqrt{4 \pi s}} \exp \bigg ( - \dfrac{y^2}{4 s} \bigg )\, dy.
\end{equation*}

But from 3.896 2, p.\,488, \cite{GR}, we have
\begin{equation*}
\int_{\mathbb{R}} e^{-q^2 y^2} \cos[p(y+\lambda)] \, dy = \dfrac{\sqrt{\pi}}{q} e^{-\frac{p^2}{4 q^2}} \cos p\lambda.
\end{equation*}
So, taking $q= 1/(2 \sqrt{s}), \, p=\xi$ and $\lambda =0$ 
\begin{align*}
\int_{\mathbb{R}} \cos \xi y \cdot \exp \bigg ( - \dfrac{y^2}{4 s} \bigg )\, dy & = \sqrt{\pi} \cdot 2 \sqrt{s} \exp \bigg ( - \dfrac{\xi^2}{4} \, 4s\bigg ) \\
&= \sqrt{4 \pi s} \, \exp(-s \xi^2), \quad \text{as required.}
\end{align*}
\end{proof}

Finally, we recall that the function $m(y)$ defined in Lemma \ref{LemDefm} is given by\begin{equation*}
m(y) = \displaystyle \dfrac{\alpha}{\Gamma(1-\alpha)}\,  \int^\infty_0 \dfrac{1}
{\sqrt{4 \pi s}} \exp \bigg ( -\dfrac{y^2}{4s} \bigg ) \, e^{-s} s^{-1-\alpha} \, ds. \\
\end{equation*}
We now derive a simple closed-form expression for $m(y)$. \\

\begin{lemma} \label{Lemmexplicit}
Suppose $m(y)$ is as defined in Lemma \ref{LemDefm}. Then
\begin{equation*}
m(y) = \dfrac{\alpha}{\Gamma(1-\alpha)} \, \dfrac{2^{\frac{1}{2}+\alpha}}{\sqrt{\pi}} |y|^{-\frac{1}{2}-\alpha} K_{\frac{1}{2}+\alpha}(|y|),
\end{equation*}
where $K_\nu$ is a modified Bessel function of the second kind of order $\nu$. See, for example, Chapter 10, \cite{NIST}.
\end{lemma}
\begin{proof}
From 3.478 4, p.\,372, \cite{GR}, we have
\begin{equation*}
\int^\infty_0 x^{\nu-1} \exp(-\beta x^p - \gamma x^{-p}) \, dx = \dfrac{2}{p} \bigg (\dfrac{\gamma}{\beta} \bigg)^{\frac{\nu}{2p}} K_{\frac{\nu}{p}}(2 \sqrt{\beta\gamma}),
\end{equation*}
provided $\beta, \gamma > 0$ and $p \not = 0$. We take $p=1, \, \beta =1, \, \gamma = y^2/4$ and $\nu = -(\alpha + \frac{1}{2})$. Hence
\begin{equation*}
\int^\infty_0 s^{-(\alpha + \frac{1}{2})-1} \exp \bigg (- s - \dfrac{y^2}{4s} \bigg ) \, ds = 2 \bigg (\dfrac{y^2}{4} \bigg)^{\frac{-(\alpha + \frac{1}{2})}{2}} K_{-(\alpha + \frac{1}{2})}(|y|).
\end{equation*}

So, finally
\begin{align*}
m(y) &= \dfrac{\alpha}{\Gamma(1-\alpha)} \, \dfrac{1}{\sqrt{4 \pi}} \, 2 \, 2^{\alpha + \frac{1}{2}} \, |y|^{- \frac{1}{2}-\alpha} \, K_{-(\alpha + \frac{1}{2})}(|y|) \\
&= \dfrac{\alpha}{\Gamma(1-\alpha)} \, \dfrac{1}{\sqrt{\pi}}  \, 2^{\frac{1}{2} + \alpha} \, |y|^{ - \frac{1}{2} -\alpha} \, K_{\frac{1}{2} + \alpha}(|y|),
\end{align*}
noting that $K_\nu(x) = K_{-\nu}(x)$ for $x>0, \, \nu \in \mathbb{R}$. (See  10.27.3, \cite{NIST}.) \\ \\
\end{proof}

We now give a detailed consideration of the function $m$. \\

By Lemma \ref{Lemmexplicit}, for $y >0$
\begin{equation*}
m(y) = c_\alpha \, y^{-\frac{1}{2} - \alpha} K_{\frac{1}{2}+\alpha}(y),
\end{equation*}
where the constant $c_\alpha$ only depends on $\alpha$. From 10.25.2, 10.27.4 and 10.31.1 \cite{NIST}, $m(y) \in C^\infty([1,\infty))$. Moreover, from 10.40.2, \cite{NIST}, for $y \geq \tfrac{1}{2}$, the function $m(y)$, together with its derivatives, is bounded and $O(e^{-y})$ as $y \to \infty$. \\

On the other hand, if $\alpha \not = \tfrac{1}{2}$ then, from 10.25.2 and 10.27.4, \cite{NIST}
\begin{align*}
m(y) & = c_\alpha \, y^{-\frac{1}{2} - \alpha} \big ( y^{-\frac{1}{2} - \alpha} \phi_{\alpha} (y) + y^{\frac{1}{2} + \alpha} \psi_{\alpha} (y)\big )  \\
& = c_\alpha \, \big ( y^{-1 - 2\alpha} \phi_{\alpha} (y) + \psi_{\alpha} (y) \big ),
\end{align*} 
where $\phi_\alpha, \psi_\alpha \in C^\infty([0,2])$. \\

Similarly, for $\alpha = \tfrac{1}{2}$, from 10.31.1, \cite{NIST}
\begin{align*}
m(y) & = c_{\frac{1}{2}} \, y^{-1} K_{1}(y) \\
& = c_{\frac{1}{2}}  \, y^{-1} \big (  y^{-1} \phi_\frac{1}{2}(y) + \vartheta(y) \log y + \psi_\frac{1}{2}(y) \big ) \\
& = c_{\frac{1}{2}} \, y^{-2} \big (  \phi(y) + y \, \vartheta(y) \log y \big ),
\end{align*}
where $\phi, \vartheta \in C^\infty([0,2])$. \\

Let $\varphi \in C^\infty_0(\mathbb{R})$ be such that
$$
\varphi(x) = 
\begin{cases}
1 \quad \text{if  } |x| \leq 1 \\
0 \quad \text{if  } |x| > 2.
\end{cases}
$$

If $\alpha \not = \tfrac{1}{2}$ then
\begin{align*}
m(y) & = \varphi(y) m(y) + (1-\varphi(y))m(y) \\
& = \varphi(y) \big ( c_\alpha \, y^{-1 - 2\alpha} \phi_{\alpha} (y) + \psi_{\alpha} (y) \big ) + (1-\varphi(y))m(y) \\
& = y^{-1 - 2\alpha} [ c_\alpha \varphi(y) \phi_{\alpha} (y)] + [ c_\alpha \varphi(y) \psi_{\alpha} (y) + (1-\varphi(y))m(y) ],
\end{align*}
with a similar result for $\alpha = \tfrac{1}{2}$. \\

Given the above analysis, the following remark details the essential characteristics of the function $m$. \\
\begin{remark} \label{Remmy}
From Lemma \ref{Lem1minuse}, it is easy to see that $m(y) = m(|y|)$ and $m(y) > 0$ for all finite $y$. Moreover, for $y>0$
$$
m(y) = 
\begin{cases}
y^{-1 - 2\alpha} \phi_1 (y) + \phi_2 (y), \quad  & \alpha \not = \tfrac{1}{2} \\
y^{-2} \big (  \phi_3(y) + y \, \phi_4(y) \log y \big ), \quad  & \alpha   = \tfrac{1}{2},
\end{cases} 
$$
where $\phi_1, \phi_2, \phi_3, \phi_4 \in C^\infty(\mathbb{R})$ and, together with their derivatives, are bounded and $O(e^{-y})$ as $y\to \infty$. \\

Finally, for $0 < \alpha < 1$, we have $m(y) = O(|y|^{-1-2\alpha})$ for small $|y|$ and $m(y) = O(e^{-|y|})$ as $|y| \to \infty$. \\
\end{remark}


We will refer to Remark \ref{Remmy} several times both in this chapter and Chapter \ref{ChapterTrivKer}, where it will play a central role on the discussion on boundedness of the operator $\mathcal{A}$. \\

\begin{lemma} \label{Lemvdiff}
Suppose $v \in S(\mathbb{R})$ and (fixed) $x \in \mathbb{R}$. Then
\begin{equation*}
\displaystyle  \int_{\mathbb{R}} (v(x)-v(y)) \, m_\epsilon (|x-y|) \,dy = - \dfrac{1}{2}\int_{\mathbb{R}} (v(x+y) + v(x-y) - 2v(x)) \, m_\epsilon (|y|) \,dy,
\end{equation*}
where $m_\epsilon (y)$ is given in Definition \ref{mepsdefinition}.
\end{lemma}
\begin{proof}
Suppose $v \in S(\mathbb{R})$ and (fixed) $x \in \mathbb{R}$. Then
\begin{align*}
\displaystyle  & \int_{\mathbb{R}} (v(x)-v(y)) \, m_\epsilon (|x-y|) \,dy \\
& = - \int_{\mathbb{R}} (v(y)-v(x)) \, m_\epsilon (|x-y|) \,dy \\
& = - \int_{\mathbb{R}} (v(x+z)-v(x)) \, m_\epsilon (|z|) \,dz \quad (z=y-x).
\end{align*}

But 
\begin{align*} 
& \int_{\mathbb{R}} (v(x+z)-v(x)) \, m_\epsilon (|z|) \,dz \\
& = \int_{\mathbb{R}} (v(x-w)-v(x)) \, m_\epsilon (|w|) \,dw \quad (w=-z) \\
& = \int_{\mathbb{R}} (v(x-z)-v(x)) \, m_\epsilon (|z|) \,dz \quad (z=w). 
\end{align*}

Hence, 
\begin{align*}
& \int_{\mathbb{R}} (v(x)-v(y)) \, m_\epsilon (|x-y|) \,dy \\
&= -\dfrac{1}{2} \bigg \{ \int_{\mathbb{R}} (v(x+z)-v(x)) \, m_\epsilon (|z|) \,dz + \int_{\mathbb{R}} (v(x-z)-v(x)) \, m_\epsilon (|z|) \,dz \bigg \} \\
&= -\dfrac{1}{2} \bigg \{ \int_{\mathbb{R}} (v(x+z) + v(x-z) - 2 v(x)) \, m_\epsilon (|z|) \,dz \bigg \} \\
&= -\dfrac{1}{2} \bigg \{  \int_{\mathbb{R}} (v(x+y) + v(x-y) - 2 v(x)) \, m_\epsilon (|y|) \,dy\bigg \}.
\end{align*}
\end{proof}

\begin{remark}
It turns out that we could take $\epsilon =0$ in the right-hand side of the equation in Lemma \ref{Lemvdiff} as there is, in fact, only a weak singularity at the origin.  Indeed, by Remark \ref{Remmy}, $m(y) = O(|y|^{-1-2\alpha})$ for small $|y|$ and hence the integrand is $O(|y|^{2-1-2\alpha}) = O(|y|^{1-2 \alpha})$. Provided $0 < \alpha < 1$ then $1- 2 \alpha > -1$, and hence the integral exists in the conventional sense. \\
\end{remark}

\begin{lemma} \label{uxuymeps}
Suppose $0 < \epsilon < 1$ and $v \in S(\mathbb{R})$. Let $u = r_+v$. Then, for $x > 0$,
\begin{align*}
& \int^\infty_0(u(x) - u(y)) \, m_\epsilon(|x - y|) dy \\
& = \int^x_0(2u(x) - u(x+w) - u(x-w)) \, m_\epsilon(|w|) dw + \int^\infty_x(u(x) - u(x+w)) \, m_\epsilon(|w|) dw.
\end{align*}
\end{lemma}
\begin{proof}
Let $w:= y - x$. Then
\begin{align*}
& \int^\infty_0(u(x) - u(y)) \, m_\epsilon(|x - y|) dy \\
& = \int^\infty_{-x} (u(x) - u(w+x)) \, m_\epsilon(|w|) dw \\
& = \int^x_{-x} (u(x) - u(w+x)) \, m_\epsilon(|w|) dw + \int^\infty_x (u(x) - u(w+x)) \, m_\epsilon(|w|) dw.
\end{align*}
But
\begin{equation*}
\int^0_{-x} (u(x) - u(w+x)) \, m_\epsilon(|w|) dw = \int^x_0 (u(x) - u(x-z)) \, m_\epsilon(|z|) dz,
\end{equation*}
and the required result follows immediately. \\
\end{proof}

\begin{lemma} \label{Aepsvestimate}
Suppose $0 < \epsilon < 1$ and $0 < \alpha < 1$. Let $v \in S(\mathbb{R})$ and define $u := r_+ v$. Then there exists a strictly positive constant $C$ and a function $g(x)$, both independent of $\epsilon$, such that
\begin{equation*}
|(\mathcal{A}_\epsilon u)(x)| \leq \chi_{[0,1]}(x) g(x) + C \quad (x>0),
\end{equation*}
where $g(x)$ is $O(1)$ as $x \searrow 0$ for $0 < \alpha < \tfrac{1}{2}$, and is $O(x^{1-2 \alpha})$ as $x \searrow 0$ for $\tfrac{1}{2} \leq \alpha < 1$. 
\end{lemma}

\begin{proof}
We now define:
\begin{align*}
M_1 & := \int^\infty_1 w \, m(w) dw; \\
h(x) & := \int^1_x w \, m(w) dw \quad (0 < x < 1); \\
M_2 & := \int^\infty_0 w^2 \, m(w) dw.
\end{align*}

Then $M_1, M_2 < \infty$ and
$$h(x) = 
\begin{cases} 
O(1) & \mbox{if } 0 < \alpha < \tfrac{1}{2}  \\ 
O(x^{1-2\alpha}) & \mbox{if } \tfrac{1}{2} \leq \alpha < 1 
\end{cases} \text{ as } x \searrow 0.
$$

Moreover, noting that $u:= r_+ v, \, (v \in S(\mathbb{R}))$, we define
\begin{align*}
V_0 & := \sup_{x \geq 0} |v(x)|; \\
V_1 & := \sup_{x \geq 0} |v'(x)|; \\
V_2 & := \sup_{x \geq 0, \, 0 < w \leq x} |(2 v(x) - v(x+w) - v(x-w))/w^2|.
\end{align*}
Clearly, $V_0, V_1, V_2 < \infty$. \\

Our goal now is to determine point-wise estimates for $(\mathcal{A}_\epsilon u)(x)$.
Suppose, initially that $0 < x < 1$. Then, from Definition \ref{mepsdefinition} and Lemma \ref{uxuymeps},
\begin{equation*}
|(\mathcal{A}_\epsilon u)(x)| \leq |u(x)| + I_1 + I_2,
\end{equation*}
where
\begin{align*}
I_1 & := \bigg | \int^x_0(2u(x) - u(x+w) - u(x-w)) \, m_\epsilon(|w|) dw \bigg |, \\
I_2 & := \bigg | \int^\infty_x(u(x) - u(x+w)) \, m_\epsilon(|w|) dw \bigg |.
\end{align*}

But
\begin{equation*}
I_1  \leq \int^x_0 w^2 V_2 \, m_\epsilon(w) dw \leq V_2 M_2.
\end{equation*}

On the other hand,
\begin{equation*}
I_2  \leq \int^1_x w V_1 \, m_\epsilon(w) dw + \int^\infty_1 w V_1 \, m_\epsilon(w) dw \leq V_1 h(x) + V_1 M_1.
\end{equation*}

In summary, for $0 < x < 1$,
\begin{equation} \label{AbsAepsvsmallx}
|(\mathcal{A}_\epsilon u)(x)| \leq g(x) + V_0 + V_1M_1 + V_2 M_2,
\end{equation} 
where $g(x) := V_1 \, h(x)$.\\

Now suppose that $x \geq 1$. Then, from Lemma \ref{uxuymeps},
\begin{equation*}
|(\mathcal{A}_\epsilon u)(x)| \leq |u(x)| + I_1 + I_2,
\end{equation*}
where $I_1$ and $I_2$ are as defined previously but now, of course, the value of $x$ is in a different range. \\

Now
\begin{equation*}
I_1  \leq \int^x_0 w^2 V_2 \, m_\epsilon(w) dw \leq V_2 M_2 \quad \text{(as previously)}.
\end{equation*}

On the other hand,
\begin{equation*}
I_2  \leq \int^\infty_1 w V_1 \, m_\epsilon(w) dw \leq V_1 M_1.
\end{equation*}

In summary, for $x \geq 1$,
\begin{equation} \label{AbsAepsvbigx}
|(\mathcal{A}_\epsilon u)(x)| \leq V_0 + V_1M_1 + V_2 M_2.
\end{equation} \\
\begin{remark}
Estimates \eqref{AbsAepsvsmallx} and \eqref{AbsAepsvbigx} are independent of $\epsilon$. \\
\end{remark}
\end{proof}

\begin{lemma} \label{LemAvInt}
Suppose $m(y)$ is as defined in Lemma \ref{LemDefm}, and the pseudo-differential operator $A$ has symbol $(1+\xi^2)^\alpha$, where $0 < \alpha < 1$. Then, for all $v \in S(\mathbb{R})$,
\begin{equation*}
(Av)(x) = v(x) +  \lim_{\epsilon \searrow 0} \int_{\mathbb{R}} (v(x)-v(y)) \, m_\epsilon (|x-y|) \, dy.
\end{equation*}
\end{lemma}
\begin{proof}
Let $\mathcal{F}$ denote $\mathcal{F}_{x \to \xi}$. Then
\begin{align*}
\mathcal{F}(Av)(\xi) &= (1+\xi^2)^\alpha (\mathcal{F} v)(\xi) \\
&= (\mathcal{F} v)(\xi) +\bigg[ \int_{\mathbb{R}} (1- \cos y\xi) m(y) \, dy \bigg] (\mathcal{F} v)(\xi) \quad \text{by Lemma \ref{LemDefm}}\\
&= (\mathcal{F} v)(\xi) - \frac{1}{2} \int_{\mathbb{R}} (\mathcal{F} v)(\xi) \big \{ e^{i \xi y} + e^{-i \xi y} - 2 \big \} m(y) \, dy \\
&= (\mathcal{F} v)(\xi) - \frac{1}{2}\int_{\mathbb{R}}   \mathcal{F} \big ( ( v(\cdot+y)+v(\cdot-y)-2v(\cdot))\big ) (\xi)  m(y) \, dy \\
&= (\mathcal{F}v)(\xi) - \frac{1}{2} \mathcal{F} \bigg ( \int_{\mathbb{R}}  ( v(\cdot+y)+v(\cdot-y)-2v(\cdot))  m(y) \, dy \bigg ) (\xi)\\
&= (\mathcal{F} v)(\xi) + \mathcal{F} \bigg (  \lim_{\epsilon \searrow 0} \int_{\mathbb{R}}  ( v(x)-v(y)) m_\epsilon (|x-y|) \, dy \bigg ) (\xi) \quad \text{by Lemma \ref{Lemvdiff}},
\end{align*}
where we have used Lemma \ref{FubiniTest} to justify the change in order of $\mathcal{F}$ and integration with respect to $y$. \\

So now applying the inverse transform $\mathcal{F}^{-1}_{\xi \to x}$ to both sides
\begin{equation} \label{AFAIS}
(Av)(x) = v(x) + \lim_{\epsilon \searrow 0} \int_{\mathbb{R}} (v(x)-v(y)) \, m_\epsilon (|x-y|) \, dy.
\end{equation}
This completes the proof of the lemma. \\
\end{proof}
\begin{lemma} \label{FubiniTest}
Suppose $v \in S(\mathbb{R})$ and $0 < \alpha < 1$. Then 
\begin{equation*}
\big | v(x+y) + v(x-y) - 2 v(x) \big | \, m(|y|)
\end{equation*}
is integrable over $\mathbb{R} \times \mathbb{R}$.
\end{lemma}
\begin{proof}
Firstly, we integrate with respect to $x$ and define
\begin{equation*}
I(y):= \int_\mathbb{R} \, \big | v(x+y) + v(x-y) - 2 v(x) \big | \, m(|y|) \, dx.
\end{equation*}
Now if $|y| \leq 1$, then $m(|y|) = O(|y|^{-1-2\alpha})$. On the other hand, if $|y| > 1$ then 
$m(|y|) = O(e^{-|y|})$. Hence, for certain positive constants $C_1$ and $C_2$,
\begin{align*}
I(y) & \leq \int_\mathbb{R} C_1 \, \chi_{[-1,1]}(y) \, \dfrac{| v(x+y) + v(x-y) - 2 v(x) |}{|y|^{1+2 \alpha}} \, dx \\
& \quad + \int_\mathbb{R} C_2 \, \chi_{\mathbb{R} \setminus [-1,1]}(y) \,  \dfrac{| v(x+y) + v(x-y) - 2 v(x) |}{e^{|y|}} \, dx \\
& \leq C_1 \, \chi_{[-1,1]}(y) \, |y|^{2-1-2\alpha} \, \int_\mathbb{R} \,\, \sup_{z \in [x-1, x+1]} |v''(z)| \, dx \\
& \quad + C_2 \, \chi_{\mathbb{R} \setminus [-1,1]}(y) \, e^{-|y|} \, \int_\mathbb{R} \,  4 |v(x)| \, dx.
\end{align*}
But 
\begin{equation*}
\int_\mathbb{R} \,\, \sup_{z \in [x-1, x+1]} |v''(z)| \, dx \leq C_3 \int_\mathbb{R} \dfrac{1}{1+x^2} \, dx = \pi \, C_3,
\end{equation*}
and
\begin{equation*}
\int_\mathbb{R} |v(x)| \, dx \leq C_4,
\end{equation*}
for certain positive constants $C_3$ and $C_4$.
Hence
\begin{align*}
I(y) & \leq C \bigg ( \chi_{[-1,1]}(y) \, |y|^{1-2\alpha} + \chi_{\mathbb{R} \setminus [-1,1]}(y) \, e^{-|y|} \bigg) \\
& \in L_1(\mathbb{R}),
\end{align*}
and the required result now follows directly from Tonelli's theorem. 
\end{proof}

\begin{lemma} \label{LemAvchi}
Suppose $m(y)$ is as defined in Lemma \ref{LemDefm} and the pseudodifferential operator $A$ has symbol $(1+ \xi^2)^\alpha$, where $0 < \alpha < 1$. Then for all $v \in S(\mathbb{R})$ and $x>0$, 
\begin{equation*}
\big( r_+ \, A(\chi_{\mathbb{R}_-} v) \big)(x) = - \int^0_{-\infty}  v(y) m(|x-y|) \, dy.
\end{equation*}
\end{lemma}
\begin{proof}
Choose any $w \in C^\infty_0(\mathbb{R}_+)$, and define $\chi_n \in C^\infty_0(\mathbb{R})$ such that \\
\[
\chi_n(x):=
\begin{cases}
1 & x \in [-n, - \tfrac{1}{n}] \\
0 & x \not \in [-(n+1), 0].
\end{cases}
\]

Then, from Lemma \ref{LemAvInt}
\begin{equation} \label{r+Achinu}
(r_+ \, A(\chi_n v), w) = - \int_{\mathbb{R}_+} w(x) \, \bigg ( \int^0_{-\infty} \chi_n (y) v(y) m(|x-y|) \, dy \bigg ) \, dx.
\end{equation}
But $\chi_n v \to \chi_{\mathbb{R}_-} v$ in $L_1(\mathbb{R}) \hookrightarrow S'(\mathbb{R})$, and hence  $A(\chi_n v) \to A(\chi_{\mathbb{R}_-} v)$ in $S'(\mathbb{R})$. \\

On the other hand, since $x > 0$, the right-hand side of equation \eqref{r+Achinu} converges to
\begin{equation*}
- \int_{\mathbb{R}_+} w(x) \, \bigg ( \int^0_{-\infty}  v(y) m(|x-y|) \, dy \bigg ) \, dx.
\end{equation*}
Thus, letting $n \to \infty$, we obtain
\begin{equation*}
(r_+ \, A(\chi_{\mathbb{R}_-} v), w) = - \int_{\mathbb{R}_+} w(x) \, \bigg ( \int^0_{-\infty} v(y) m(|x-y|) \, dy \bigg ) \, dx.
\end{equation*}
But since $w \in C^\infty_0(\mathbb{R}_+)$ was arbitrary
\begin{equation*}
\big ( r_+ \, A(\chi_{\mathbb{R}_-} v) \big )(x)  = -  \int^0_{-\infty} v(y) m(|x-y|) \, dy.
\end{equation*}
This completes the proof of the lemma. \\
\end{proof}

\begin{lemma} \label{LemAChim}
Suppose $m(y)$ is as defined in Lemma \ref{LemDefm} and the pseudodifferential operator $A$ has symbol $(1+ \xi^2)^\alpha$, where $0 < \alpha < 1$. Then for $x>0$, 
\begin{equation*}
\big( r_+ \, A(\chi_{\mathbb{R}_-}) \big)(x) = - \int^0_{-\infty} m(|x-y|) \, dy.
\end{equation*}
\end{lemma}
\begin{proof}
Choose any $w \in C^\infty_0(\mathbb{R}_+)$, and let $\chi_n \in C^\infty_0(\mathbb{R})$ 
be as defined in the proof of Lemma \ref{LemAvchi}. Then, from Lemma \ref{LemAvchi}
\begin{equation} \label{r+Achinu2}
(r_+ \, A(\chi_n \chi_{\mathbb{R}_-}), w) = - \int_{\mathbb{R}_+} w(x) \, \bigg ( \int^0_{-\infty} \chi_n (y) \chi_{\mathbb{R}_-}(y) m(|x-y|) \, dy \bigg ) \, dx.
\end{equation}
But $\chi_n \chi_{\mathbb{R}_-} \to \chi_{\mathbb{R}_-}$ in $L_1(\mathbb{R}, \frac{dx}{1+x^2}) \hookrightarrow S'(\mathbb{R})$, and hence  $A(\chi_n \chi_{\mathbb{R}_-}) \to A(\chi_{\mathbb{R}_-})$ in $S'(\mathbb{R})$. \\

On the other hand, the right-hand side of equation \eqref{r+Achinu2} converges to
\begin{equation*}
- \int_{\mathbb{R}_+} w(x) \, \bigg ( \int^0_{-\infty}  m(|x-y|) \, dy \bigg ) \, dx.
\end{equation*}
Thus, letting $n \to \infty$, we obtain
\begin{equation*}
(r_+ \, A(\chi_{\mathbb{R}_-} ), w) = - \int_{\mathbb{R}_+} w(x) \, \bigg ( \int^0_{-\infty} m(|x-y|) \, dy \bigg ) \, dx.
\end{equation*}
But since $w \in C^\infty_0(\mathbb{R}_+)$ was arbitrary
\begin{equation*}
\big ( r_+ \, A(\chi_{\mathbb{R}_-}) \big ) (x) = -  \int^0_{-\infty} m(|x-y|) \, dy.
\end{equation*}
This completes the proof of the lemma.
\end{proof}


%% file: KCL_Thesis_Chapter3_v5.tex
\chapter{Trivial kernel} \label{ChapterTrivKer}
\section{Main result}
\begin{theorem} \label{thmabddtrivker}
Suppose $1 < p < \infty$ and $0 < \alpha <1$. If either
\begin{enumerate} [\hspace{18pt}(a)]
\item $2\alpha -1 + 1/p  < s < 1+1/p \text{ then } \mathcal{A}:  H^s_p(\overline{\mathbb{R}_+}) \to H^{s-2\alpha}_p(\overline{\mathbb{R}_+}) \text{ is bounded},  \text{ or }$ 
\item $1+1/p < s < 2+1/p \text{ then } \mathcal{A}:  H^s_{p,0}(\overline{\mathbb{R}_+}) \to H^{s-2\alpha}_p(\overline{\mathbb{R}_+}) \text{ is bounded}$.
\end{enumerate} 

Moreover, if $p=2$ and either
\begin{enumerate} [\hspace{18pt}(i)]
\item $0 < \alpha < \tfrac{1}{2}, \,\,  \tfrac{1}{2}  < s < 1+\tfrac{1}{2},   \text{ or }$ 
\item $0 < \alpha < 1, \,\,  1 +\tfrac{1}{2}  < s < 2+\tfrac{1}{2}$, 
\end{enumerate} 
then $\mathcal{A}$ has a trivial kernel. \\
\end{theorem}

\begin{remark}
The condition that $p=2$ for $\mathcal{A}$ to have a trivial kernel is not as restrictive as it might appear. Under appropriate conditions, we will be able to determine sufficent conditions for $\mathcal{A}$ to have a trivial kernel for any $p$ in the range $1 < p < \infty$, using the result (above) for $p=2$. \\
\end{remark}

\section{Proof of main result}
The proof of the boundedness of the operator $\mathcal{A}$ is given in Lemma \ref{AuHsp}. For $\tfrac{1}{2} < s < 1+ \tfrac{1}{2}$ and $1+ \tfrac{1}{2} < s < 2 + \tfrac{1}{2}$ respectively, Lemmas \ref{BevIvsmall} and \ref{Bevlvlargep2} establish sufficient conditions for $\mathcal{A}$ to have a trivial kernel. \\

\begin{lemma} \label{AuHsp} 
Suppose $1 < p < \infty$ and $0 < \alpha < 1$. If 
\begin{enumerate} [\hspace{18pt}(a)]
\item $2\alpha -1 + 1/p  < s < 1+1/p \text{ then } \mathcal{A}:  H^s_p(\overline{\mathbb{R}_+}) \to H^{s-2\alpha}_p(\overline{\mathbb{R}_+}) \text{ is bounded}$;
\item $1+1/p < s < 2+1/p \text{ then } \mathcal{A}:  H^s_{p,0}(\overline{\mathbb{R}_+}) \to H^{s-2\alpha}_p(\overline{\mathbb{R}_+}) \text{ is bounded}$, 
\end{enumerate} 
where the space $H^s_{p,0}(\overline{\mathbb{R}_+})$ is as defined in \eqref{Hsp0definition}, Section \ref{preamble}.
\end{lemma}

\begin{proof}
Suppose initially that $1/p < s < 1+1/p$ or $1 + 1/p < s < 2+1/p$. Our first step is to show that $\mathcal{A}:  r_+\widetilde{H}^s_p(\overline{\mathbb{R}_+}) \to H^{s-2\alpha}_p(\overline{\mathbb{R}_+}) \text{ is bounded}$, and to do this we use the representation for $\mathcal{A}$ given in \eqref{rAef}. \\

Firstly, suppose that $\alpha \not = \tfrac{1}{2}$. Then, from Lemma \ref{LemAChim} and Remark \ref{Remmy}, 
\begin{equation*}
(r_+ A(\chi_{\mathbb{R}_-}))(x) = x^{-2 \alpha} \phi(x) + \psi(x),
\end{equation*}
where $\phi, \psi \in C^\infty(\mathbb{R})$, and their derivatives, are bounded and $O(e^{-x})$ as $x \to \infty$. Hence, from Lemmas \ref{x-alphaI} and \ref{multxgammaa}, $\mathcal{A}$ is bounded from $r_+ \, \widetilde{H}^s_p(\overline{\mathbb{R}_+})$ to $H^{s-2\alpha}_p(\overline{\mathbb{R}_+})$, provided $ s > 2\alpha -1 + 1/p$. \\

On the other hand, if $\alpha = \tfrac{1}{2}$ then again, from Lemma \ref{LemAChim} and Remark \ref{Remmy}, 
\begin{align*}
(r_+ A(\chi_{\mathbb{R}_-}))(x) & = x^{-1} \big \{ \phi(x) + x \, \vartheta(x) \log x \big \} \\
& = x^{-1} \big \{ \phi(x) + \vartheta(x)e^{-x/2} \cdot e^{-x/2} x \, \log x \big \} , 
\end{align*}
where $\phi, \vartheta \in C^\infty(\mathbb{R})$, and their derivatives, are bounded and $O(e^{-x})$ as $x \to \infty$. With the additional use of Lemma \ref{eepslogbounded}, the boundedness of $\mathcal{A}$ from $r_+ \, \widetilde{H}^s_p(\overline{\mathbb{R}_+})$ to $H^{s-1}_p(\overline{\mathbb{R}_+})$ now follows as in the case $\alpha \not = \tfrac{1}{2}$. \\

The case $s < 1/p$ follows similarly, providing that, in our use of Lemma \ref{x-alphaI}, we note the constraint that $s > 2\alpha -1 + 1/p$. Also, if $s<0$ and hence $\alpha < \tfrac{1}{2}$, we use Theorem 4.2.2(ii), p. 203, \cite{Tr92} in place of Lemma \ref{multxgammaa}. But for $-1 + 1/p < s < 1/p$, we can identify $e_+ \, {H}^s_p(\overline{\mathbb{R}_+})$ with $\widetilde{H}^s_p(\overline{\mathbb{R}_+})$, see Section 2.8.7, p. 158, \cite{Tr83}, and the proof for $2\alpha -1 + 1/p  < s < 1/p$ is thus complete. \\

It remains to consider the case $s \geq 1/p$. \\

We let $\eta(x) \in C^\infty_0(\mathbb{R})$ be such that
\begin{equation*}
\eta(x) = 
\begin{cases} 
	1 &\mbox{if } |x| \leq 1 \\ 
	0 & \mbox{if } |x|>2.
\end{cases} 
\end{equation*}
Suppose $u \in H^s_p(\overline{\mathbb{R}_+})$ or $u \in H^s_{p,0}(\overline{\mathbb{R}_+})$, as $1/p < s < 1+1/p$ or $1 + 1/p < s < 2+1/p$ respectively. Then we can define
\begin{equation*}
u_0(x) := u(x) - u(0) r_+ \eta(x),
\end{equation*}
and hence write
\begin{equation*}
u(x) = u_0(x) + u(0) r_+ \eta(x).
\end{equation*}

Then, by construction, $u_0(0)=0$. Moreover, if we assume $u'(0)=0$, then $u'_0(0)=0$. Therefore, see, for example, Lemma 1.15, p.\,55, \cite{Shar}, we have  $u_0 \in r_+ \widetilde{H}^s_p(\overline{\mathbb{R}_+})$. Hence, it remains to consider $\mathcal{A}$ acting on $r_+ \eta$. \\

But, from equation \eqref{rBurComm}, it is therefore enough to show that 
$r_+ [\eta I,A]\chi_{\mathbb{R}_-}$ is bounded from the one-dimensional subspace of $H^s_p(\overline{\mathbb{R}_+})$, or $H^s_{p,0}(\overline{\mathbb{R}_+})$ if $1+1/p < s < 2 +1/p$, spanned by $\eta$ to $H^{s-2\alpha}_p(\overline{\mathbb{R}_+})$.\\

Let $\psi_1$ be any smooth function defined on $\mathbb{R}$ such that $\psi_1(x) = 0$ if $x \leq \tfrac{1}{2}$, and $\psi_1(x) = 1$ if $x \geq 1$. From equation \eqref{r+commutatorchim}, for $x >0$ we have
\begin{align*}
r_+ [\eta I,A]\chi_{\mathbb{R}_-}(x) &= -r_+ \int^0_{-\infty} \big(\eta(x)-\eta(y) \big) m(x-y) \,dy \\
&= -r_+ \int^0_{-\infty} \big(\eta(x)-\eta(y) \big) \psi_1(x-y) \, m(x-y) \,dy, 
\end{align*}
since $\eta(x) - \eta(y) = 1-1=0$ if $x-y <1$. (Indeed,  $x >0, y < 0$ and $x-y < 1$ implies that $0 < x <1$ and $-1 < y < 0$.) \\

We note that $\psi_1(x) m(x)$ is smooth on $\mathbb{R}_+$ and decays exponentially as $x \to \infty$. Hence, $r_+ [\eta I,A]\chi_{\mathbb{R}_-} \in H^{s-2\alpha}_p(\overline{\mathbb{R}_+})$ as required. Finally, boundedness follows immediately since the linear operator $r_+ [\eta I,A]\chi_{\mathbb{R}_-}$ is defined on a one-dimensional space. This completes the proof for the ranges $1/p < s < 1+1/p$ and $1/p < s < 2 +1/p$.\\

Finally, to complete the proof of the lemma, we note that boundedness for the exceptional value $s = 1/p$ follows directly by interpolation. See, for example, Chapter 1, \cite{Tr83}.\\
\end{proof}

\begin{remark} \label{zeroderivapprox}
Suppose $1 < p < \infty$ and $1 + 1/p < s < 2 + 1/p$. Then, from the proof of Lemma \ref{AuHsp}, given any $u \in  H^s_{p,0}(\overline{\mathbb{R}_+})$  we can write
\begin{equation*}
u = u_0 + u(0) r_+ \eta,
\end{equation*}
where $e_+ u_0 \in \widetilde{H}^s_{p}(\overline{\mathbb{R}_+})$ and $\eta \in C^\infty_0(\mathbb{R})$, with $\eta'(0)=0$. \\

Since $e_+ C^\infty_0(\mathbb{R}_+)$ is dense in $\widetilde{H}^s_{p}(\overline{\mathbb{R}_+})$, see Section 2.10.3, p.\,231, \cite{Tr}, this allows us to approximate $u$ arbitrarily closely by a sequence $\{ u_n \}^\infty_{n=1} \subset r_+  C^\infty_0(\mathbb{R})$ with $u_n(0) = u(0)$ and, importantly, ${u}_n'(0)=0$ for each $n$. \\
\end{remark}

\begin{remark} \label{FunctionalIequivalent}
From Lemma \ref{Lemmexplicit}, we have the following explicit representation
\begin{equation*}
m(y) = \dfrac{\alpha}{\Gamma(1-\alpha)} \, \dfrac{2^{\frac{1}{2}+\alpha}}{\sqrt{\pi}} |y|^{-\frac{1}{2}-\alpha} K_{\frac{1}{2}+\alpha}(|y|).
\end{equation*}

Now for any $u \in H^\alpha_2 (\overline{\mathbb{R}_+})$, we define the functional
\begin{equation} \label{Ivdefn}
I(u) := \bigg \{ \int^\infty_0 |u(x)|^2 \, dx + \tfrac{1}{2} \int^\infty_0 \int^\infty_0 |u(x)-u(y)|^2 m(|x-y|) \, dydx \bigg \}^{\tfrac{1}{2}}.
\end{equation} 

From Remark 4.2, p.\,62, \cite{Es}, 
\begin{equation*}
\| u \|^{+}_{\alpha,2} := \bigg \{ \int^\infty_0 |u(x)|^2 \, dx + \int^\infty_0 \int^\infty_0 \dfrac{|u(x)-u(y)|^2}{|x-y|^{1+2\alpha}} \, dydx \bigg \}^{\tfrac{1}{2}},
\end{equation*}
is an equivalent norm on $H^\alpha_2 (\overline{\mathbb{R}_+})$. Moreover, see Remark \ref{Remmy}, we have
\begin{equation*}
I(u) \leq \text{ const } \| u \|^+_{\alpha,2}.
\end{equation*}

It is easy to show that $I(\cdot)$ is, in fact, a norm on $H^\alpha_2 (\overline{\mathbb{R}_+})$ for $0 < \alpha < 1$. \\
\end{remark}

\begin{lemma} \label{Iruconv}
Suppose $0 < \alpha < 1$ and $u \in {H}^\alpha_2(\overline{\mathbb{R}_+})$. Further let the sequence $\{ u_n \}_{n \geq 1}$ in $r_+ S(\mathbb{R})$ be such that
\begin{equation*} 
\| u_n - u \|^+_{\alpha,2} \to 0 \quad \text{as } n \to \infty.
\end{equation*}
Then
\begin{equation*} 
\lim_{n \to \infty} I(u_n) = I(u). 
\end{equation*}
\end{lemma}

\begin{proof}
From Remark \ref{FunctionalIequivalent}, 
\begin{equation*}
|  I(u_n) - I(u) | \leq I( u_n - u) \leq \text{ const } \|   u_n - u \|^+_{\alpha,2} \to 0 \quad \text{as} \,\, n \to \infty.
\end{equation*}
Therefore,
\begin{equation*}
\lim_{n \to \infty} I( u_n) = I(u). 
\end{equation*} 
\end{proof}

\begin{lemma} \label{BevIvsmall}
Suppose $0 < \alpha < \tfrac{1}{2}, \, \tfrac{1}{2}  < s < 1 +\tfrac{1}{2}$ and $u \in H^s_2 (\overline{\mathbb{R}_+})$.  Then
\begin{equation*}
(\mathcal{A}u,  u) = (I(u))^2.
\end{equation*}
In particular, if $u \in \operatorname{Ker} \mathcal{A}$ then $u=0$.
\end{lemma}

\begin{proof}
Since $s > \tfrac{1}{2} > \alpha $, we have the continuous embedding
\begin{equation*}
H^s_2(\mathbb{R}) \hookrightarrow H^\alpha_2(\mathbb{R}).
\end{equation*} 

Moreover, as $0 < \alpha < \tfrac{1}{2}$ we have $\alpha > 2\alpha - 1 + \tfrac{1}{2}$ and thus from Lemma \ref{AuHsp}, 
\begin{equation*}
\mathcal{A}: {H}^{\alpha}_2 (\overline{\mathbb{R}_+}) \to {H}^{-\alpha}_2 (\overline{\mathbb{R}_+})
\end{equation*}
is bounded. \\

Let $f \in H^{-\alpha}_2(\overline{\mathbb{R}_+}), \, g \in H^{\alpha}_2(\overline{\mathbb{R}_+})$ then, from Plancherel and Cauchy-Schwartz, we have the estimate
\begin{equation*}
|(f, g)_{\mathbb{R}_+}| = |(e_+ f, e_+ g)_\mathbb{R}| \leq \| e_+ f \|_{-\alpha, 2} \| e_+ g \|_{\alpha, 2} \leq \text{ const } \| f \|^+_{-\alpha, 2} \|  g \|^+_{\alpha, 2}. \\
\end{equation*}

Since $0< \alpha < \tfrac{1}{2}, \,\, C^\infty_0(\mathbb{R}_+)$ is dense in ${H}^\alpha_2(\overline{\mathbb{R}_+})$, see Section 2.9.3, p.\,220, \cite{Tr}, and there exists a sequence $\{ u_n \}_{n\geq 1}$ in $C^\infty_0(\mathbb{R}_+)$ such that
\begin{equation*} 
\| u_n - u \|^+_{\alpha,2} \to 0 \quad \text{as } n \to \infty.
\end{equation*}

Hence 
\begin{align*}
(\mathcal{A}u,u) -(\mathcal{A} u_n, u_n) & = (\mathcal{A}(u- u_n),u) +(\mathcal{A} u_n, u-  u_n)\\
& \to 0 \quad \text{as } n \to \infty. 
\end{align*}

That is, 
\begin{equation*} 
\lim_{n \to \infty} (\mathcal{A}  u_n,   u_n ) = (\mathcal{A}u,  u),
\end{equation*}
and from Lemma \ref{Iruconv},
\begin{equation*} 
\lim_{n \to \infty} I( u_n) = I(u). 
\end{equation*} 

Hence, from Lemma \ref{AvvIv2},
\begin{equation*}
(\mathcal{A}u,u)= (I(u))^2.
\end{equation*} 
Finally, if $\mathcal{A} u =  0$ then, see Remark \ref{FunctionalIequivalent}, we have $u=0$. \\
\end{proof}

\begin{lemma} \label{Bevlvlargep2}
Suppose $0<  \alpha < 1, \,  1 + \tfrac{1}{2} <  s < 2 +\tfrac{1}{2}$ and $u \in H^s_{2,0} (\overline{\mathbb{R}_+})$. Then
\begin{equation*}
(\mathcal{A}u,  u) = (I(u))^2.
\end{equation*}
In particular, if $u \in \operatorname{Ker} \mathcal{A}$ then $u=0$.
\end{lemma}

\begin{proof}
For $0 < \alpha <1$, we define
\begin{equation*}
\beta := 
\begin{cases} 
	\alpha &\mbox{if } 0 < \alpha < \tfrac{1}{2} \\ 
	\alpha - \tfrac{1}{2} & \mbox{if } \tfrac{1}{2} \leq \alpha < 1
\end{cases} 
\end{equation*}
so that $0 \leq \beta < \tfrac{1}{2}$.  As previously, if $f \in H^{-\beta}_2(\overline{\mathbb{R}_+}), \, g \in H^{\beta}_2(\overline{\mathbb{R}_+})$ then, from Plancherel and Cauchy-Schwartz, we have the estimate
\begin{align*}
|(f, g)_{\mathbb{R}_+}| &= |(e_+ f, e_+ g)_\mathbb{R}| \leq \| e_+ f \|_{-\beta, 2} \| e_+ g \|_{\beta, 2} \leq \text{ const } \|  f \|^+_{-\beta, 2} \| g \|^+_{\beta, 2}. 
\end{align*} 

Moreover, since $0 < \alpha < 1$ and $1 +\tfrac{1}{2} < s < 2 +\tfrac{1}{2}$ from Lemma \ref{AuHsp}, the operator
\begin{equation*}
\mathcal{A}: {H}^{s}_{2,0} (\overline{\mathbb{R}_+}) \to {H}^{s - 2\alpha}_2 (\overline{\mathbb{R}_+})
\end{equation*}
is bounded. In addition, ${H}^{s}_{2,0} (\overline{\mathbb{R}_+}) \hookrightarrow {H}^\beta_2 (\overline{\mathbb{R}_+})$ and ${H}^{s-2\alpha}_2 (\overline{\mathbb{R}_+}) \hookrightarrow {H}^{-\beta}_2 (\overline{\mathbb{R}_+})$.\\

From Remark \ref{zeroderivapprox}, there exists a sequence $\{ u_n : u_n(0) = u(0), \,\, u'_n(0) = 0 \}_{n\geq 1}$ in $r_+  C^\infty_0(\mathbb{R})$ such that
\begin{equation*} 
\| u_n - u \|^+_{s,2} \to 0 \quad \text{as } n \to \infty.
\end{equation*}
Therefore, as ${H}^{s}_2 (\overline{\mathbb{R}_+}) \hookrightarrow {H}^\beta_2 (\overline{\mathbb{R}_+})$,
\begin{equation*} 
\| u_n - u \|^+_{\beta, 2}\to 0 \quad \text{as } n \to \infty.
\end{equation*}

Hence 
\begin{align*}
(\mathcal{A}u,u) -(\mathcal{A} u_n,  u_n) & = (\mathcal{A}(u- u_n),u) +(\mathcal{A} u_n, u-  u_n)\\
& \to 0 \quad \text{as } n \to \infty. 
\end{align*}

That is, 
\begin{equation*} 
\lim_{n \to \infty} (\mathcal{A} u_n,  u_n ) = (\mathcal{A}u,  u),
\end{equation*}
and, since ${H}^{s}_2 (\overline{\mathbb{R}_+}) \hookrightarrow {H}^{\alpha}_2 (\overline{\mathbb{R}_+})$, from Lemma \ref{Iruconv},
\begin{equation*} 
\lim_{n \to \infty} I(u_n) = I(u). 
\end{equation*} 

Hence, from Lemma \ref{AvvIv2},
\begin{equation*}
(\mathcal{A}u,u)= (I(u))^2.
\end{equation*} 
Finally, if $\mathcal{A} u =  0$ then, see Remark \ref{FunctionalIequivalent}, we have $u=0$. \\
\end{proof}

\section{Supporting lemmas}
\begin{lemma} \label{x-alphaI}
Let $\gamma > 0$ and $s > \gamma + 1/p -1$. Then the multiplication operator
\begin{equation*}
x^{-\gamma}I : \widetilde{H}^s_p(\overline{\mathbb{R}_+}) \to \widetilde{H}^{s-\gamma}_p(\overline{\mathbb{R}_+})
 \quad \text{is bounded}.
\end{equation*}
\end{lemma}

\begin{proof}
The first case we consider is $s=\gamma$. Then $\widetilde{H}^{s-\gamma}_p(\overline{\mathbb{R}_+}) = L_p(\mathbb{R}_+)$ and the required result follows directly from the proof of   Proposition 1, Section 2.8.6, \cite{Tr83}. In other words, if $u \in \widetilde{H}^\gamma_p(\overline{\mathbb{R}_+})$ then
\begin{equation*}
\| x^{-\gamma} u \|_p \leq C_{\gamma,p} \|u \|_{\gamma,p}.
\end{equation*}

Secondly,  suppose that $s=\gamma+1$, and let $u \in \widetilde{H}^{\gamma+1}_p(\overline{\mathbb{R}_+})$. Then we will show that $x^{-\gamma}I : \widetilde{H}^{\gamma+1}_p(\overline{\mathbb{R}_+}) \to \widetilde{H}^{1}_p(\overline{\mathbb{R}_+})$ is bounded. Note that $\widetilde{H}^{\gamma+1}_p(\overline{\mathbb{R}_+}) \hookrightarrow \widetilde{H}^{\gamma}_p(\overline{\mathbb{R}_+})$ and $\partial u \in \widetilde{H}^{\gamma}_p(\overline{\mathbb{R}_+})$. Therefore, using an equivalent norm on $\widetilde{H}^{\gamma+1}_p(\overline{\mathbb{R}_+})$, see Chapter 1, p.\,6, \cite{Tr92},
\begin{align*}
\| x^{-\gamma} u \|_{1,p} & \leq \text{ const } \big \{ \| x^{-\gamma} u \|_p + \| \partial (x^{-\gamma} u) \|_p \big \} \\
& \leq \text{ const } \big \{ \| x^{-\gamma} u \|_p + \| x^{-\gamma} \partial u \|_p + \gamma \|x^{-(\gamma+1)} u \|_p   \big \} \\
& \leq \text{ const } \big \{ \| u \|_{\gamma, p} + C_{\gamma,p}\| \partial u \|_{\gamma, p} + C_{\gamma+1,p} \gamma \| u \|_{\gamma+1,p}  \big \} \\
& \leq \text{ const } \big \{  \| u \|_{\gamma+1,p}  \big \}.
\end{align*}

Hence, $x^{-\gamma}I : \widetilde{H}^{\gamma+1}_p(\overline{\mathbb{R}_+}) \to \widetilde{H}^{1}_p(\overline{\mathbb{R}_+})$ is bounded. In the same way, the result for $s=\gamma+m$, for any $m \in \mathbb{N}$ follows by induction. \\

Moreover, the proof of the lemma for any $s \geq \gamma$ follows by interpolation. See, for example, Chapter 1, \cite{Tr83}.\\

Finally, we consider the remaining case $-1+\gamma + 1/p < s < \gamma$. From the first case, it is clear that $x^{-s} I: \widetilde{H}^s_p(\overline{\mathbb{R}_+}) \to L_p(\mathbb{R}_+)$ is bounded. Hence, it is sufficient to show that $x^{-(\gamma -s)}I: L_p(\mathbb{R}_+) \to \widetilde{H}^{s-\gamma}_p(\overline{\mathbb{R}_+})$ is bounded. Since $-1+1/p< s- \gamma < 0$, this operator is adjoint to $x^{-(\gamma -s)} I: \widetilde{H}^{\gamma-s}_{p'} (\overline{\mathbb{R}_+}) \to L_{p'}(\mathbb{R}_+)$ which is bounded by the first case. This completes the proof of the lemma. \\
\end{proof}

\begin{lemma} \label{multxgammaa}
Let $\gamma \geq 0, \, \phi \in r_+ S(\mathbb{R})$ and $s \geq 0$. Then
\begin{equation*}
x^\gamma \phi  I :\widetilde{H}^s_p(\mathbb{\overline{R_+}}) \to \widetilde{H}^s_p(\mathbb{\overline{R_+}})
\end{equation*}
is bounded.
\end{lemma}
\begin{proof}
The result is clearly true for $s=0$. \\

We now use proof by induction on $s$. Suppose result is true for $s = m \in \mathbb{N} \cup \{ 0 \}$, for all $\phi \in r_+ S(\mathbb{R})$ and all $\gamma \geq 0$. Then, we shall prove it is also true for $s=m+1$. \\

Let $u \in \widetilde{H}^{m+1}_p(\mathbb{\overline{R_+}})$. Then, using the inductive hypothesis,
\begin{align*}
\| x^\gamma \phi u \|_{m+1,p} & \leq \text{const } \big \{ \| x^\gamma \phi u \|_{m,p} + \| \tfrac{d}{dx} (x^\gamma \phi u ) \|_{m,p}  \big \} \\
& \leq \text{const } \big \{ \| u \|_{m,p} + \| x^\gamma \phi u' \|_{m,p} + \| x^\gamma \phi' u \|_{m,p} + \| x^{\gamma-1} \phi u \|_{m,p}\big \} \\
& \leq \text{const } \big \{ \| u \|_{m,p} +  \| u' \|_{m,p} + \| x^{\gamma-1} \phi u \|_{m,p}\big \}.
\end{align*}

It remains to consider the term $\| x^{\gamma-1} \phi u \|_{m,p}$. \\

If $\gamma - 1 \geq 0$, then by the inductive hypothesis,
\begin{equation*}
\| x^{\gamma-1} \phi u \|_{m,p} \leq \text{const } \| u \|_{m,p}.
\end{equation*}

 
Finally, if $\gamma - 1 < 0$, then by Lemma \ref{x-alphaI},
\begin{equation*}
\| x^{\gamma-1} \phi u \|_{m,p} \leq \text{const } \| \phi u \|_{m+1 -\gamma,p} \leq \text{ const } \| u \|_{m+1,p}.
\end{equation*}

In summary, for $s= m+1$ and all $ \phi \in r_+ S(\mathbb{R}), \gamma >0$ and $u \in \widetilde{H}^{m+1}_p(\mathbb{\overline{R_+}})$, we have
\begin{equation*}
\| x^{\gamma-1} \phi u \|_{m,p} \leq \text{const } \| u \|_{m+1,p}.
\end{equation*}
This completes the proof by induction for $s \in \mathbb{N}$. \\

Hence, by interpolation, the required result follows for all $s \geq 0$. \\
\end{proof}

\begin{lemma} \label{eepslogbounded}
Suppose $1 < p < \infty, \, k \in \mathbb{N} \cup \{ 0 \}$ and $s \geq 0$. Then for any $\beta >0$ and $\epsilon >0$, the map
\begin{equation*}
u \mapsto e^{-\beta x} x^\epsilon \log^k x \cdot  u,
\end{equation*}
from $\widetilde{H}^s_p(\overline{\mathbb{R}_+}) \to \widetilde{H}^{s}_p(\overline{\mathbb{R}_+})$ is bounded.
\end{lemma}
\begin{proof}
We proceed by induction on $s$. \\

Suppose $k \in \mathbb{N} \cup \{ 0 \}$ and let $s=0$. Since the function $e^{-\beta x} x^\epsilon \log^k x$ is bounded for $x \geq 0$, the map $u(x) \mapsto  e^{-\beta x} x^\epsilon \log^k x \cdot  u(x)$,
from $L_p({\mathbb{R}_+}) \to L_p({\mathbb{R}_+})$ is bounded. \\

Now suppose the result is true for some $s =m \in \mathbb{N} \cup \{0\}$. We shall prove it is also true for $s = m+1$. Suppose $u \in \widetilde{H}^{m+1}_p(\overline{\mathbb{R}_+})$ and let us define
\begin{equation*}
F(x) := (e^{-\beta x} x^\epsilon \log^k x) \cdot  u(x).
\end{equation*}

Then, a routine calculation gives
\begin{equation*}
F'(x) = T_1 + T_2 +T_3 + T_4,
\end{equation*}
where
\begin{align*}
T_1 & := - \beta \, e^{-\beta x} x^{ \epsilon} \log^k x \cdot  u(x) = (- \beta \, e^{-\beta x} x^\epsilon \log^k x) \cdot u; \\
T_2 & := e^{-\beta x} \epsilon \, x^{-1 + \epsilon} \log^k x \cdot u(x) = (\epsilon \, e^{-\beta x} x^\epsilon \log^k  x) \cdot x^{-1} u;\\
T_3 & := e^{-\beta x} x^{-1 + \epsilon} k \log^{k-1} x \cdot u(x) = (k e^{-\beta x} x^\epsilon \log^{k-1} x) \cdot x^{-1} u \quad (k \geq 1);\\
T_4 & := e^{-\beta x} x^{\epsilon} \log^k x \cdot u'(x) = (e^{-\beta x} x^\epsilon \log^k x) \cdot u'.
\end{align*}
For each of $T_1,T_2, T_3$ and $T_4$, the function in parentheses on the right-hand side is of the correct form for the inductive hypothesis. Moreover, taking $s=m+1$ and $\gamma =1$ in Lemma \ref{x-alphaI}, 
\begin{equation*}
\| x^{-1} u \|_{m,p} \leq \text{ const } \| u \|_{m+1,p}.
\end{equation*}

Then, using the inductive hypothesis, but only including the term $T_3$ if $k \geq 1$, 
\begin{align*}
 \| (e^{-\beta x} & x^\epsilon \log^k x) \cdot  u \|_{m+1,p} & \\
& \leq \text {const } \big \{ \|F\|_{m,p} +  \|F'\|_{m,p} \big \} \\
& \leq \text {const } \big \{ \|F\|_{m,p} +  \|T_1\|_{m,p}  + \|T_2\|_{m,p} + \|T_3\|_{m,p} + \|T_4\|_{m,p} \big \} \\
& \leq \text {const } \big \{ \|u\|_{m,p} +  \|u\|_{m,p}  + \|x^{-1}u\|_{m,p} + \|x^{-1}u\|_{m,p} + \|u'\|_{m,p} \big \} \\
& \leq \text {const } \big \{ \|u\|_{m,p} + \|u\|_{m+1,p} + \|u'\|_{m,p} \big \} \\
& \leq \text {const }  \|u\|_{m+1,p}.
\end{align*}

This completes the proof by induction for $s=0,1,2,3, \dots$. Hence, by interpolation, the required result holds for all $s \geq 0$. \\
\end{proof}

\begin{corollary} \label{LogMultiplier}
Suppose $1 < p < \infty, \, k \in \mathbb{N} \cup \{ 0 \}$ and $s \geq 0$. Then for any $\beta >0$ and $0< \epsilon < s+1 -1/p$, the map
\begin{equation*}
u \mapsto e^{-\beta x} \log^k x \cdot  u,
\end{equation*}
from $\widetilde{H}^s_p(\overline{\mathbb{R}_+}) \to \widetilde{H}^{s-\epsilon}_p(\overline{\mathbb{R}_+})$ is bounded.
\end{corollary}
\begin{proof}
Suppose $u \in \widetilde{H}^s_p(\overline{\mathbb{R}_+})$. We write
\begin{equation*}
e^{-\beta x} \log^k x \cdot  u = x^{-\epsilon} \big \{ e^{-\beta x} x^\epsilon \log^k x \cdot  u \big \},
\end{equation*}
and the required result now follows directly from Lemmas \ref{eepslogbounded} and \ref{x-alphaI}. \\
\end{proof}

%% file: KCL_Thesis_Chapter4_v5.tex
\chapter{Operator algebra - Part I} \label{ChapterOpAlgI}
\section{Introduction}
This chapter details the first step in describing our problem in the context of an operator algebra of multiplication, Mellin and Wiener-Hopf operators acting on $L_p(\mathbb{R}_+)$. The results calculated here act as the starting point for the second, and final, step given in Chapter \ref{OpAlgLp}. \\

Throughout this chapter we assume the problem constraints $0 < \alpha < \tfrac{1}{2}, \, 1 < p < \infty$ and $1/p < s < 1+1/p$. Moreover, we suppose that $u \in H^s_p(\overline{\mathbb{R}_+})$. (However, where appropriate, we shall also prove variants of certain results that apply in the case of higher regularity, namely $1+ 1/p < s < 2+1/p$.)\\

The discontinuity of the function $e_+u$, at $x=0$, gives rise to a delta function on the boundary. For, if $\delta$ denotes the \textit{Dirac delta} function then, see Lemma \ref{lemma:Deu}, we have: 

\begin{equation*}
(D-i) e_+ u = e_+ (D-i)u + i u(0) \, \delta. 
\end{equation*}

Terms, as above, involving the trace value $u(0)$ pose a significant difficulty. However, it will be seen that we can combine such terms with the \enquote{added} potential to form expressions including the factor $(u(x) - u(0))$. These differences can then be reformulated as a composition of certain multiplication, Mellin and Wiener-Hopf operators. Such conversions are a significant part of the analysis of this present chapter. \\

Since our ultimate objective is a reformulation of our problem in $L_p(\mathbb{R}_+)$, we introduce 
\begin{equation*}
u_s := (D+i)^{s-1}e_+ (D-i)u. 
\end{equation*}\\

From Lemma \ref{lemma:us}, $u_s \in L_p(\mathbb{R})$ with supp $u_s \subseteq \overline{\mathbb{R}_+}$. Therefore, we have
\begin{equation*}
e_+ r_+ u_s = u_s.
\end{equation*}
This relationship will prove essential in dealing with the Wiener-Hopf operators. For example, we show in Lemma \ref{ustou} that
\begin{equation*}
u = W(c) (r_+ u_s),
\end{equation*}
where the Wiener-Hopf operator $W(c)$ has symbol $c(\xi) = (\xi - i)^{-1}(\xi + i)^{1-s}$. \\

Our goal in this chapter is to reformulate equation \eqref{rAef} in the form
\begin{equation} \label{aMCtilde}
\tilde{a}_0(x) u(0) + \sum^N_{j=1}  \tilde{a}_j(x) \, M^0(\tilde{b}_j) \, (r_+ \tilde{C}_j e_+)(r_+ u_s) + \tilde{K} u = f,
\end{equation}
where the operator $\tilde{K}:H^s_p(\overline{\mathbb{R}_+}) \to H^{s-2\alpha}_p(\overline{\mathbb{R}_+})$ is compact. \\

In doing so, we provide precise determinations of the multiplication symbols $\{ \tilde{a}_k \}^N_{k=0}$, the Mellin symbols $\{ \tilde{b}_j \}^N_{j=1}$ and the symbols $\{ \tilde{c}_j \}^N_{j=1}$ of the pseudodifferential operators $\{ \tilde{C}_j \}^N_{j=1}$. Since, our ultimate goal is to a calculate the Fredholm index of the corresponding operator, along the way we will effectively discard any compact operators - as the Fredholm index is invariant under compact perturbations. \\

Finally, we note that, by hypothesis,  $f \in H^{s-2\alpha}_p(\overline{\mathbb{R}_+})$. In Chapter \ref{OpAlgLp}, we apply the operator $r_+ (D-i)^{s-2\alpha} l_+$ to each side of equation \eqref{aMCtilde}, to obtain our required formulation in $L_p(\mathbb{R}_+)$. (See, in particular, Lemma \ref{lemma:rLambdae}.) In this sense, the value of the results from the current chapter will only be apparent later. Accordingly, they are described here as \textit{interim} results.

\section{Problem reformulation}
As an initial step in reformulating equation \eqref{rAef}, we define
\begin{equation} \label{Aminus1}
A^{-}(D) := A(D) (D-i)^{-1}.
\end{equation}
Since $A$ has order $2 \alpha$, $A^-$ is a pseudodifferential operator of order $2 \alpha -1$. Of course, as $0 < \alpha < \tfrac{1}{2}$, $A^-$ has negative order. We now recast equation \eqref{rAef} in terms of the operator $A^-$. \\

In passing, and looking ahead to the case of higher regularity, we also define
\begin{equation} \label{Aminus2temp}
A^{=}(D) := A(D) (D-i)^{-2}.
\end{equation}

From Lemma \ref{lemma:Deu},
\begin{equation*}
r_+ A e_+ u =  r_+ A^{-}(D-i) e_+ u  = r_+ A^{-} e_+ (D-i) u + i u(0) r_+ A^{-} \delta.
\end{equation*}

Moreover,
\begin{equation*}
r_+ \, A( \chi_{\mathbb{R}_{-}}) =  r_+ \, A^{-} (D - i)  \chi_{\mathbb{R}_{-}} \\
= r_+ \, A^{-} ( - i \, \delta  - i \chi_{\mathbb{R}_{-}} ),
\end{equation*}
since $D(\chi_{\mathbb{R}_{-}}) = - D(\chi_{\mathbb{R}_{+}}) = -i \, \delta$. (See  Example 1.3, p. 10, \cite{Es}.) \\

Hence, with these substitutions, equation \eqref{rAef} becomes
\begin{equation} \label{rA(s-1)eg}
r_+ A^{-} e_+ (D-i) u  - i (u(x)-u(0)) \, r_+ A^{-} \delta - iu(x) r_+ \, A^{-} ( \chi_{\mathbb{R}_{-}} ) = f.  
\end{equation} 

Let us now define
\begin{equation} \label{usdefn}
u_s := (D+i)^{s-1}e_+ (D-i)u. 
\end{equation}
Then, we can write
\begin{align*}
r_+ A^{-} e_+ (D-i) u & = r_+ A^{-} (D+i)^{1-s}(D+i)^{s-1}e_+ (D-i) u \\
& = r_+ A_{s} u_s, 
\end{align*}
where 
\begin{equation} \label{Aminuss}
A_{s}(D) := A^{-}(D+i)^{1-s}. \\
\end{equation}

Hence, equation \eqref{rAef} becomes
\begin{equation} \label{rA-s}
r_+ \, A_{s} \, u_s - i (u(x)-u(0)) \, r_+ A^{-} \delta - iu(x) r_+ \, A^{-} ( \chi_{\mathbb{R}_{-}} ) = f.  \\
\end{equation} 
We will see subsequently that the function $u$ appearing in the potential term, $- iu(x) r_+ \, A^{-} ( \chi_{\mathbb{R}_{-}} )$ in equation \eqref{rA-s}, can also be expressed appropriately in terms of $u_s$. Moreover, it turns out that the difference $u(x)-u(0)$ can be described in terms of the composition of a multiplication, Mellin and Wiener-Hopf operators. Finally, we are able to calculate both $r_+ A^{-} \delta$ and $r_+ \, A^{-} ( \chi_{\mathbb{R}_{-}} )$ explicitly, using special functions. \\

\section{Interim results} \label{OpAlgInit}
It will now be convenient to introduce certain functions. We have
\begin{equation*}
M(a,b,z) := 1+ \sum^\infty_{k=1} \dfrac{(a)_k}{(b)_k} \dfrac{z^k}{k!},
\end{equation*}
as defined in 13.1.2, \cite{AandS} or 9.210, \cite{GR}. (We use the notation $(a)_k = a(a+1) \cdots (a+k-1)$ for any $k \in \mathbb{N}$.) \\

Following 13.1.3, \cite{AandS} and 9.210 2, \cite{GR}, we also introduce the \textit{confluent hypergeometric} function
\begin{equation*}
U(a,b,z) := \dfrac{\Gamma(1-b)}{\Gamma(a-b+1)} M(a,b,z) + \dfrac{\Gamma(b-1)}{\Gamma(a)} z^{1-b} M(a-b+1, 2 -b, z),
\end{equation*}
for $a>0$ and $b > 0$, \textit{ provided } $b \not \in \mathbb{N}$. In the exceptional case that $b \in \mathbb{N}$, the corresponding expression for $U(a,b,z)$ includes a logarithmic term. (See, for example, 13.1.6, \cite{AandS}). \\

It turns out that it will be sufficient for our purposes to assume $a>1$ and $0 < b < 3$. Then, see Lemma \ref{eUab2x}, for $x > 0$,
\[
e^{-x} \, U(a,b,2x) =
\begin{cases}
x^{1-b} \psi(a,b,x) + \phi(x)   & \text{if   } b \not= 1,2 \\
x^{1-b} \psi(a,b,x) + \vartheta(x) \log x + \phi(x) & \text{if   } b = 2 \\
\vartheta(x) \log x +\phi(x)   & \text{if   } b = 1,
\end{cases}
\]
where $\vartheta, \phi \in C^\infty(\mathbb{R})$, and together with their derivatives, are bounded and $O(e^{-x})$ as $x \to +\infty$. Moreover, $\psi \in C^\infty_0(\mathbb{R})$ with $\psi(a,b,x)=0$ for $x > 2$.\\

Finally, we let $\epsilon > 0$ be a small parameter. \\

With these preparations complete, we are now ready to examine the individual summands in the left-hand side of equation \eqref{rA-s}. \\

\subsection{First term}
Consider the term $r_+ \, A_{s} \, u_s$. From equations \eqref{Aminus1} and \eqref{Aminuss}, $A_{s}(D) := A(D) (D-i)^{-1} (D+i)^{1-s}$ and hence, we can write
\begin{equation*}
r_+ \, A_{s} \, u_s = (r_+ A(D) (D-i)^{-1} (D+i)^{1-s} e_+)(r_+ u_s),
\end{equation*}
since, by Lemma \ref{lemma:us}, $u_s \in L_p(\mathbb{R})$ and supp $u_s \subseteq \overline{\mathbb{R}_+}$. \\

Thus, in the notation of equation \eqref{aMCtilde}, 
\begin{align} \label{tildea1}
\tilde{a}_1(x) &= 1;  \nonumber \\
\tilde{b}_1(\xi) &= 1;  \\
\tilde{c}_1(\xi) &= (1+\xi^2)^\alpha (\xi-i)^{-1}(\xi+i)^{1-s}. \nonumber  
\end{align} 

\subsection{Middle term}
Now consider the middle term, $-i(u(x)-u(0)) r_+ A^- \delta$. From Lemma \ref{lemma:rA(s-1)delta}
\begin{equation*}
(r_+ \, A^{-} \delta) (x) = C_{\alpha}  \, e^{-x} \, U(\alpha+1, 2 \alpha+1, 2x),
\end{equation*}
where the constant $C_\alpha$ only depends on $\alpha$, and is given by equation \eqref{defCalpha} in the statement of Lemma \ref{lemma:rA(s-1)delta} as 
\begin{equation*}
C_{\alpha} = -i \, \dfrac{\alpha \, 2^{2\alpha}}{\Gamma (1- \alpha)}. \\
\end{equation*}

From Lemma \ref{eUab2x},
\begin{equation*}
(r_+A^- \delta)(x) = C_\alpha \big ( \phi(x) + x^{-2 \alpha}  \psi(\alpha + 1, 2\alpha+1, x) \big ),
\end{equation*}
where $ \phi \in C^\infty(\mathbb{R})$ and, together with its derivatives, is bounded and $O(e^{-x})$ as $x \to +\infty$. Moreover, $ \psi \in C^\infty_0(\mathbb{R})$ with $ \psi(\alpha + 1, 2\alpha+1, x) =0$ for $x >2$. \\

Hence, we can write
\begin{equation*}
-i (u(x)-u(0)) r_+ A^-(\delta) = -i C_\alpha \, (T_{11}  + T_{12} ) u
\end{equation*}
where
\begin{align*}
T_{11} u(x) &:= \phi(x) (u(x)- u(0)) \\
T_{12} u(x) &:= x^{-2\alpha} \psi(\alpha + 1, 2\alpha+1, x) (u(x) - u(0)).
\end{align*}
Firstly, we will show that $T_{11} : H^s_p(\overline{\mathbb{R}_+}) \to H^{s-\epsilon}_p(\overline{\mathbb{R}_+})$ is compact. Now
\begin{equation*}
\phi (x) (u(x) - u(0)) = \phi (x)  e^{x/2} \cdot e^{-x/2}(u(x)-u(0)).
\end{equation*}
By Lemma \ref{phiuu0small}, $u \mapsto e^{-x/2}(u(x)-u(0))$ defines a bounded operator from $H^s_p(\overline{\mathbb{R}_+})$ to $r_+ \widetilde{H}^s_p(\overline{\mathbb{R}_+})$. Moreover, $\phi(x) e^{x/2} \in H^s_p(\overline{\mathbb{R}_+})$, since it and its derivatives are bounded, smooth and $O(e^{-x/2})$ as $ x \to + \infty$. Finally, the compactness of $T_{11}: H^s_p(\overline{\mathbb{R}_+}) \to H^{s-\epsilon}_p(\overline{\mathbb{R}_+})$ follows directly from Lemma \ref{multiplierbHcompact}. \\

It remains to consider 
\begin{equation*}
-i C_\alpha \,T_{12}u(x)  = -i C_\alpha x^{-2\alpha} \psi(\alpha + 1, 2\alpha+1, x) (u(x) - u(0)), 
\end{equation*}
and it is convenient to write
\begin{align*}
-i C_\alpha & x^{-2\alpha} \psi(\alpha + 1, 2\alpha+1, x) (u(x) - u(0)) \\
&= -i C_\alpha \, \psi(\alpha + 1, 2\alpha+1, x) \cdot \big \{ x^{-2\alpha} (u(x) - u(0)) \big \} ,
\end{align*}
noting that $ \psi \in C^\infty_0(\mathbb{R})$ with $\psi(\alpha + 1, 2\alpha+1, x) =0$ for $x >2$. \\

On the other hand, from Lemma \ref{lemma:mellinop1} and Appendix \ref{Appendix FC},
\begin{equation*}
x^{-2 \alpha} (u(x) - u(0)) = \int^\infty_0 K_{2\alpha} \bigg( \dfrac{x}{y} \bigg ) h(y) \, \dfrac{dy}{y} \,\, := {M}_{2\alpha} h
\end{equation*}
where $h(x) = (C^{2\alpha}_{0^+}u) (x)$. Moreover, from Lemma \ref{htous}  
\begin{equation*}
h = (r_+ C(D) e_+) (r_+ u_s) + i \dfrac{u(0)}{\sqrt{2\pi}} \, r_+ \mathcal{F}^{-1} (-i \xi)^{2\alpha-1}(\xi-i)^{-1},
\end{equation*}
where $C(D)$ has the symbol $c(\xi) = (-i \xi)^{2\alpha} (\xi + i)^{1-s} (\xi-i)^{-1}$.
From Lemma \ref{lemma:mellinop2}, ${M}_{2\alpha}$ is a Mellin convolution operator with symbol $b(\xi) = B(1/p' + i \xi, 2 \alpha) / \Gamma(2\alpha) $.\\

Thus, in the notation of equation \eqref{aMCtilde}, we have
\begin{align} \label{tildea2}
\tilde{a}_2(x) &= -i C_\alpha \, \psi(\alpha + 1, 2\alpha+1, x) \quad (\in C^\infty_0(\mathbb{R}));\nonumber \\
\tilde{b}_2(\xi) &= B(1/p' + i \xi, 2 \alpha) / \Gamma(2\alpha) ;\\
\tilde{c}_2(\xi) &= (-i \xi)^{2\alpha} (\xi + i)^{1-s} (\xi-i)^{-1}.\nonumber  
\end{align}
and
\begin{equation} \label{tildea0}
\tilde{a}_0(x) = \tilde{a}_2(x) M^0(\tilde{b}_2)  \dfrac{i}{\sqrt{2\pi}} \, r_+ \mathcal{F}^{-1} (-i \xi)^{2\alpha-1}(\xi-i)^{-1}. 
\end{equation}

\subsection{Final term} \label{Smallalphafinalterm}
It remains to consider the last term, $-i u r_+ A^-(\chi_{\mathbb{R}_-})$. From Lemma \ref{rAminuschiminus},
\begin{equation*}
(r_+A^- \chi_{\mathbb{R}_-})(x) =  C_\alpha \big ( \phi_1(x) + x^{1-2 \alpha}  \phi_2(x) \big ),
\end{equation*}
where $ \phi_1, \phi_2 \in C^\infty(\mathbb{R})$ and, together with their derivatives, are bounded and $O(e^{-x})$ as $x \to +\infty$.  \\

Hence, we can write
\begin{equation*}
-i u r_+ A^-(\chi_{\mathbb{R}_-}) = -i C_\alpha \, (T_{21} + T_{22} + T_{23})u
\end{equation*}
where
\begin{align*}
T_{21} u(x) &:= \phi_1(x) u(x) \\
T_{22} u(x) &:= x^{1-2\alpha} \phi_2(x) (u(x) - u(0))\\
T_{23} u(x) &:= x^{1-2\alpha} \phi_2(x) u(0).
\end{align*}
We will now show that $T_{21},T_{22},T_{23} : H^s_p(\overline{\mathbb{R}_+}) \to H^{s-2\alpha}_p(\overline{\mathbb{R}_+})$ are compact operators. \\

Firstly, consider $T_{21}$. We note that the compactness of $u \mapsto \phi_1(x) u(x)$ from $H^s_p(\overline{\mathbb{R}_+}) \to H^{s-\epsilon}_p(\overline{\mathbb{R}_+})$ follows immediately from Lemma \ref{multiplierbHcompact}. \\

Secondly, we will show $T_{22}$ is compact. We can write
\begin{equation*}
 x^{1-2 \alpha} \phi_2(x) (u(x) - u(0)) = \phi_2(x)  e^{x/2} \cdot x^{1-2\alpha} e^{-x/4} \cdot e^{-x/4}(u(x)-u(0)).
\end{equation*}
By Lemma \ref{phiuu0small}. $u \mapsto e^{-x/4}(u(x)-u(0))$ defines a bounded operator from $H^s_p(\overline{\mathbb{R}_+})$ to $r_+ \widetilde{H}^s_p(\overline{\mathbb{R}_+})$. Since $1-2\alpha >0$, from Lemma \ref{multxgammaa}, the operator $x^{1-2\alpha} e^{-x/4}I$ is bounded on $\widetilde{H}^s_p(\overline{\mathbb{R}_+})$. Moreover, $\phi_2(x) e^{x/2} \in H^s_p(\overline{\mathbb{R}_+})$, since it and its derivatives are bounded, smooth and $O(e^{-x/2})$ as $ x \to + \infty$. Finally, the compactness of $T_{22}: H^s_p(\overline{\mathbb{R}_+}) \to H^{s-\epsilon}_p(\overline{\mathbb{R}_+})$ follows directly from Lemma \ref{multiplierbHcompact}. \\

Thirdly, we will show $T_{23}$ is compact. We can write
\begin{equation*}
x^{1-2 \alpha} \phi_2(x) u(0) = \phi_2(x)  e^{x/2} \cdot x^{-2\alpha} \cdot x e^{-x/2} u(0).
\end{equation*}
Let $s' =\max \{s,1\}$. We note that $x e^{-x/2} \in \widetilde{H}^{s'}_p(\overline{\mathbb{R}_+})$, since it is smooth,  assumes the value zero at $x=0$ and decays exponentially. Therefore, $u \mapsto x e^{-x/2} u(0)$ defines a bounded operator from $H^s_p(\overline{\mathbb{R}_+})$ to $r_+ \widetilde{H}^{s'}_p(\overline{\mathbb{R}_+})$. Since $-2\alpha <0$, from Lemma \ref{x-alphaI}, the operator $x^{-2\alpha}I : \widetilde{H}^{s'}_p(\overline{\mathbb{R}_+}) \to \widetilde{H}^{s'-2\alpha}_p(\overline{\mathbb{R}_+})$ is bounded. As $0 < \alpha < \tfrac{1}{2}, \,\, s'- 2\alpha >0$. Moreover, $\phi_2(x) e^{x/2}$ and its derivatives are bounded, smooth and $O(e^{-x/2})$ as $ x \to + \infty$, and thus the operator $\phi_2(x) e^{x/2}I$ is bounded on $\widetilde{H}^{s'-2\alpha}_p(\overline{\mathbb{R}_+})$ by Lemma \ref{multxgammaa}. Finally, $T_{23}:H^s_p(\overline{\mathbb{R}_+}) \to H^{s-2\alpha}_p(\overline{\mathbb{R}_+})$ is bounded and rank one, and is therefore compact. \\

\subsection{Summary} \label{OpAlg1Summary}
So, in summary, taking $N=2$, we have the required representation
\begin{equation*} 
\tilde{a}_0(x) u(0) + \sum^2_{j=1}  \tilde{a}_j(x) \, M^0(\tilde{b}_j) \, (r_+ \tilde{C}_j e_+)(r_+ u_s) + \tilde{K}u = f,
\end{equation*}
where the operator $\tilde{K}:H^s_p(\overline{\mathbb{R}_+}) \to H^{s-2\alpha}_p(\overline{\mathbb{R}_+})$ is compact. The symbols $\tilde{a}_0$ and $(\tilde{a}_j, \tilde{b}_j, \tilde{c}_j)$ for $j=1,2$, are given by equations \eqref{tildea0}, \eqref{tildea1} and \eqref{tildea2} respectively. \\

Purely for convenience, these results are also repeated here:
\begin{align*} 
\tilde{a}_0(x) & = \tilde{a}_2(x) M^0(\tilde{b}_2)  \dfrac{i}{\sqrt{2\pi}} \, r_+ \mathcal{F}^{-1} (-i \xi)^{2\alpha-1}(\xi-i)^{-1}. \\
\tilde{a}_1(x) &= 1;   \\
\tilde{b}_1(\xi) &= 1;   \\
\tilde{c}_1(\xi) &= (1+\xi^2)^\alpha (\xi-i)^{-1}(\xi+i)^{1-s}.  \\  \\
\tilde{a}_2(x) &= -iC_\alpha \, \psi(\alpha+1,2\alpha+1,x) \quad (\in C^\infty_0(\mathbb{R})); \\
\tilde{b}_2(\xi) &= B(1/p' + i \xi, 2 \alpha) / \Gamma(2\alpha) ;\\
\tilde{c}_2(\xi) &= (-i \xi)^{2\alpha} (\xi + i)^{1-s} (\xi-i)^{-1}. \\ 
\end{align*}

\newpage
\section{Supporting lemmas}
\begin{lemma}  \label{lemma:Deu}
Suppose $1< p < \infty$ and $1/p < s < 1 + 1/p$. If  $u \in H^s_p(\overline{\mathbb{R}_+})$ then
\begin{equation*}
(D-i) e_+ u = e_+ (D-i)u + i u(0) \, \delta.
\end{equation*} 
\end{lemma}

\begin{proof}
We first show that if $v \in r_+ C^\infty_0(\mathbb{R})$, then $(e_+ v)' = v(0) \delta + e_+v'$. Take any $\varphi \in S(\mathbb{R})$. Then
\begin{align*}
\langle (e_+v)', \varphi \rangle &= - \langle e_+v, \varphi' \rangle \\
&= - \int^\infty_{-\infty} (e_+v)(t) { \varphi'(t)} \, dt \\
&= - \int^\infty_0 v(t){ \varphi'(t)} \, dt \\
&= -\big [v(t) {\varphi(t)} \big ]^\infty_0 + \int^\infty_0 v'(t) {\varphi(t)} \, dt \\
& = v(0) {\varphi(0)} + \int^\infty_{-\infty} (e_+ v')(t) {\varphi(t)} \, dt \\
&= \langle v(0) \delta + e_+ v', \varphi \rangle,
\end{align*}
which gives the required result, since $\varphi \in S(\mathbb{R})$ was arbitrary. \\

Since $1/p < s < 1+ 1/p$, the value $u(0)$ is well-defined. (See Section 2.9, \cite{Tr}.) Therefore, by continuity 
\begin{align*}
(D - i ) e_+ u &= D e_+ u - i e_+ u \\
&= i \, (e_+ u)' - i e_+ u \\
&= i \, u(0) \delta + e_+(D u) - i e_+ u \\
&=  e_+(D  - i)  u + i \, u(0) \delta. 
\end{align*}
This completes the proof of the lemma. \\
\end{proof}

If $1+1/p < s < 2+1/p$, we have the following equivalent of Lemma \ref{lemma:Deu}.
\begin{lemma}  \label{lemma:D2eu}
Suppose $1< p < \infty$ and $1+1/p < s < 2 + 1/p$. If  $u \in H^s_p(\overline{\mathbb{R}_+})$ then
\begin{equation*}
(D-i)^2 e_+ u = e_+ (D-i)^2 u -u(0) \, \delta' + 2u(0) \, \delta - u'(0)\, \delta.
\end{equation*} 
\end{lemma}

\begin{proof}
We first show that if $v \in r_+ C^\infty_0(\mathbb{R})$, then $(e_+ v)'' = v(0) \delta' +v'(0)\delta  + e_+v''$. From the proof of Lemma \ref{lemma:Deu},
\begin{align*}
(e_+ v)'' & = (v(0) \delta + e_+v')' \\
& = v(0) \delta' +v'(0)\delta  + e_+v''.
\end{align*}

Since $1+1/p < s < 2+ 1/p$, the values $u(0)$ and $u'(0)$ are well-defined. (See, for example, Section 2.9, \cite{Tr}). Therefore, by continuity 
\begin{align*}
(D - i )^2 e_+ u &= D^2 e_+ u - 2i D e_+ u - e_+u\\
&= - \, (e_+ u)'' +2 (e_+u)' - e_+ u \\
&= [-e_+u'' - u(0)\delta'- u'(0)\delta]+2[e_+u'+u(0)\delta]-e_+u\\
&= -e_+(u''-2u'+u)-u(0)\delta'+2u(0)\delta-u'(0)\delta \\
&=e_+ (D-i)^2 u -u(0) \, \delta' + 2u(0) \, \delta - u'(0)\, \delta,
\end{align*}
as required. \\
\end{proof}

\begin{lemma} \label{lemma:us}
Suppose $1< p < \infty, \, 1/p < s < 1 + 1/p$ and $u \in H^s_p(\overline{\mathbb{R}_+})$. Let 
\begin{equation*} 
u_s := (D+i)^{s-1}e_+ (D-i)u. 
\end{equation*}
Then $u_s \in L_p(\mathbb{R})$ with supp $u_s \subseteq \overline{\mathbb{R}_+}$. \\
\end{lemma}
\begin{proof}
Let $u_{(-1)} = (D  - i ) u$. Since $u \in H^s_p(\overline{\mathbb{R}_+})$, we have $u=r_+ u_0$ for some $u_0 \in H^s_p(\mathbb{R})$. Hence
\begin{equation*}
u_{(-1)} = (D - i)  u = (D  - i )  r_+ u_0 = r_+ (D  - i )   u_0.
\end{equation*}
But $(D  - i )   u_0 \in H^{s-1}_p(\mathbb{R})$ and, therefore, $
u_{(-1)} \in H^{s-1}_p(\overline{\mathbb{R}_+})$. \\

Since, by hypothesis, $1/p - 1 < s-1 < 1/p$, from Section 2.10.3, p.\,232, \cite{Tr}, we have 
$e_+ u_{(-1)} \in H^{s-1}_p(\mathbb{R}) \text{ and, of course, supp } e_+ u_{(-1)} \subseteq \overline{\mathbb{R}_+}$. 
Now since
\begin{equation*} 
u_s = (D + i )^{s-1} e_+ u_{(-1)},
\end{equation*}
then, see Theorem 1.9, p.\,52, \cite{Shar}, we have $u_s \in L_p(\mathbb{R})$ and $\text{supp } u_s \subseteq \overline{\mathbb{R}_+}$.Ê

This completes the proof of the lemma. \\
\end{proof}

The following counterpart of Lemma \ref{lemma:us} applies in the case of higher regularity, namely $1+1/p < s < 2+1/p$.
\begin{lemma} \label{lemma:u2s}
Suppose $1< p < \infty, \, 1+1/p < s < 2 + 1/p$ and $u \in H^s_p(\overline{\mathbb{R}_+})$. Let $u_s := (D+i)^{s-2}e_+ (D-i)^2u$. Then $u_s \in L_p(\mathbb{R})$ with supp $u_s \subseteq \overline{\mathbb{R}_+}$.
\end{lemma}
\begin{proof}
The proof follows the method used in the proof of Lemma \ref{lemma:us}. \\
\end{proof}

\begin{lemma} \label{lemma:rLambdae}
Suppose $1< p < \infty$ and $\sigma, \nu \in \mathbb{R}$. Let $l_+ :H^{\sigma}_p(\overline{\mathbb{R}_+}) \to H^{\sigma}_p(\mathbb{R})$ be an arbitrary extension operator. Then  $\Lambda^{\nu}_- = r_+ (D - i )^{\nu} l_+$ is bounded from $H^\sigma_p(\overline{\mathbb{R}_+})$ to $H^{\sigma-\nu}_p(\overline{\mathbb{R}_+})$, and does not depend on the choice of extension $l_+$. Moreover,
\begin{equation*} 
(r_+ (D - i )^{\nu} l_+)r_+ =  r_+ (D - i )^{\nu}. \\
\end{equation*}
\end{lemma}
\begin{proof}
From Theorem 1.12, p. 54, \cite{Shar}, the pseudodifferential operator $(D-i)^\nu$ is bounded from $H^\sigma_p(\mathbb{R})$ to $H^{\sigma-\nu}_p(\mathbb{R})$. In addition, its symbol $(\xi - i)^{\nu}$ admits an analytic continuation with respect to $\xi$ to the lower complex half-plane such that
\begin{equation*}
| (\xi + i\tau - i)^\nu | \leq ( |\xi| + |\tau| + 1 )^{\max \{0, \nu\}}, \quad \tau \leq 0. \\
\end{equation*}

Therefore, from Theorem 1.10, p.\,53, \cite{Shar}, $\Lambda^{\nu}_- = r_+ (D - i )^{\nu} l_+$ is continuous from $H^{\sigma}_p(\overline{\mathbb{R}_+})$ to $H^{\sigma-\nu}_p(\overline{\mathbb{R}_+})$, and does \textit{not} depend on the choice of the extension $l_+$.\\ 

Finally, by Remark 1.11, p.\,53 \cite{Shar}, we also have 
\begin{equation*}
(r_+ (D - i )^{\nu} l_+ ) r_+ = r_+ (D - i )^{\nu}.
\end{equation*} 
This completes the proof of the lemma. \\
\end{proof}

\begin{lemma} \label{phiuu0small}
Suppose $1 < p < \infty$ and $1/p < s < 1 + 1/p$. If $\varphi \in r_+ S(\mathbb{R})$ then the map  $T_\varphi: H^s_p(\overline{\mathbb{R}_+}) \to r_+ \widetilde{H}^s_p(\overline{\mathbb{R}_+})$ given by
\begin{equation*}
(T_\varphi u)(x) = \varphi(x)  (u(x)-u(0)) \quad (x > 0),
\end{equation*}
is bounded.
\end{lemma}
\begin{proof}
Firstly, we note that $\varphi \in H^s_p(\overline{\mathbb{R}_+})$. Moreover, since $s > 1/p, \, H^s_p(\overline{\mathbb{R}_+})$ is a Banach algebra. (See Section 2.8.3, Remark 3, p. 146, \cite{Tr83}.) Hence, $\varphi u \in H^s_p(\overline{\mathbb{R}_+})$ and
\begin{equation*}
\varphi(x)  (u(x)-u(0)) \in H^s_p(\overline{\mathbb{R}_+}).
\end{equation*}
Bur $\varphi(x)  (u(x)-u(0)) \big |_{x=0} = 0$, and hence by Corollary 3.4.3, p. 210, \cite{Tr83}, 
\begin{equation*}
T_\varphi u \in r_+ \widetilde{H}^s_p(\overline{\mathbb{R}_+}).
\end{equation*}
Finally, we note that
\begin{align*}
\| T_\varphi u | r_+ \widetilde{H}^s_p(\overline{\mathbb{R}_+}) \| & = \| \varphi(x)  (u(x)-u(0)) | r_+ \widetilde{H}^s_p(\overline{\mathbb{R}_+}) \| \\
& \leq \text{ const } \| \varphi(x)  (u(x)-u(0)) | {H}^s_p(\overline{\mathbb{R}_+}) \| \\
& \leq \text{ const } \| \varphi  u \|_{s,p} + \| \varphi u(0) \|_{s,p} \\
& \leq \text{ const } \|u \|_{s,p},
\end{align*}
since $s > 1/p$,  and thus $|u(0)| \leq  \|u \|_{s,p}$ by the Sobolev embedding theorem. \\
\end{proof}

\begin{remark} \label{phiuu0big}
Using the same method of proof as Lemma \ref{phiuu0small}, it is easy to show that if $1 + 1/p < s < 2 +1/p$, then the map  $T_\varphi: H^s_{p,0}(\overline{\mathbb{R}_+}) \to r_+ \widetilde{H}^s_p(\overline{\mathbb{R}_+})$, as defined above, is also bounded. \\
\end{remark}

\begin{lemma} \label{FinverseKU}
Suppose $\alpha < 1$. Then, for $x>0$,
\begin{align*}
\mathcal{F}^{-1} (1 + \xi^2)^{\alpha-1} &= \dfrac{2^\alpha}{\Gamma(1-\alpha)} \, x^{-\alpha + \frac{1}{2}} \, K_{\alpha - \frac{1}{2}}(x) \\
&= \sqrt{\dfrac{\pi}{2}}  \, \dfrac{2^{2 \alpha}}{\Gamma(1-\alpha)} \, e^{-x} \, U(\alpha, 2\alpha, 2x),
\end{align*}
where $K_\nu(x)$ and $U(a,b,x)$ denote the modified Bessel function and confluent hypergeometric function respectively.
\end{lemma}
\begin{proof} By definition, for $x>0$, 
\begin{align*}
\mathcal{F}^{-1} (1 + \xi^2)^{\alpha-1} &= \dfrac{1}{\sqrt{2\pi}} \int^\infty_{-\infty} (1+\xi^2)^{\alpha-1} e^{-i \xi x} \, d\xi\\
&= \dfrac{1}{\sqrt{2\pi}} \int^\infty_{-\infty} (1+\xi^2)^{\alpha-1} \cos \xi x \, d\xi\\
&= \dfrac{2}{\sqrt{2\pi}} \int^\infty_{0} (1+\xi^2)^{\alpha-1} \cos \xi x \, d\xi.\\
\end{align*}
From  3.771 2, p.\,445, \cite{GR}, we have
\begin{equation*}
\int^\infty_0 (1+\xi^2)^{\nu - \frac{1}{2}} \cos \xi x \, d\xi = \dfrac{1}{\sqrt{\pi}} \, \bigg ( \dfrac{2}{x} \bigg )^\nu \cos (\pi \nu) \, \Gamma(\nu + \frac{1}{2}) \, K_{-\nu}(x),
\end{equation*}
provided $x>0$ and $\nu < \frac{1}{2}$. Hence, taking $\nu = \alpha - \frac{1}{2}$ we have
\begin{align*}
\mathcal{F}^{-1} (1 + \xi^2)^{\alpha-1} &= \dfrac{2}{\sqrt{2\pi}} \cdot \dfrac{1}{\sqrt{\pi}} \, \bigg ( \dfrac{2}{x} \bigg )^{\alpha - \frac{1}{2}} \textstyle \cos \pi (\alpha - \frac{1}{2}) \, \Gamma(\alpha) \, K_{\frac{1}{2}-\alpha}(x) \\
&= \dfrac{\sqrt{2}}{\pi} \cdot 2^{\alpha-\frac{1}{2}} \, \sin (\pi \alpha)  \, \Gamma(\alpha) \, x^{-\alpha + \frac{1}{2}} \, K_{\alpha - \frac{1}{2}}(x) \qquad (K_{-\nu}(x) =  K_{\nu}(x))\\
&= \dfrac{2^\alpha}{\pi} \cdot \dfrac{\pi}{\Gamma(1-\alpha)} \, x^{-\alpha + \frac{1}{2}} \, K_{\alpha - \frac{1}{2}}(x) \qquad (\text{see }  \,  5.5.3, \, \cite{NIST})\\
&= \dfrac{2^\alpha}{\Gamma(1-\alpha)} \, x^{-\alpha + \frac{1}{2}} \, K_{\alpha - \frac{1}{2}}(x), \qquad \text{as required.}\\
\end{align*}

Finally, since $K_\nu(z) = \sqrt{\pi} \, (2z)^\nu \, e^{-z} \, U(\nu +\frac{1}{2}, 2 \nu + 1, 2z)$, see 10.39.6, \cite{NIST}, we have
\begin{align*}
\mathcal{F}^{-1} (1 + \xi^2)^{\alpha-1} &= \dfrac{2^\alpha}{\Gamma(1-\alpha)} \, x^{-\alpha + \frac{1}{2}} \cdot  \sqrt{\pi} \, (2x)^{\alpha - \frac{1}{2}} \, e^{-x} \, U(\alpha, 2 \alpha, 2x)\\
&= \sqrt{\dfrac{\pi}{2}} \, \dfrac{2^{2\alpha}}{\Gamma(1-\alpha)} \, e^{-x} \, U(\alpha, 2 \alpha, 2x).
\end{align*} 
This completes the proof of the lemma. \\
\end{proof}

\begin{lemma} \label{lemma:rA(s-1)delta}
Suppose $\alpha<1$. Then
\begin{equation*}
(r_+ \, A^{-} \delta) (x) = C_{\alpha}  \, e^{-x} \, U(\alpha+1, 2 \alpha+1, 2x),
\end{equation*}
where the constant $C_{\alpha}$ depends only on $\alpha$, and is given by
\begin{equation} \label{defCalpha}
C_{\alpha} = -i \, \dfrac{\alpha \, 2^{2\alpha}}{\Gamma (1- \alpha)}.
\end{equation}
\end{lemma}
\begin{proof}
Now, by definition, we have 
\begin{align} \label{Aminusdeltafinal}
A^{-}(D) \delta & = \mathcal{F}^{-1} A^{-}(\xi) \mathcal{F} \delta \nonumber \\
& = \mathcal{F}^{-1} (1+\xi^2)^\alpha (\xi - i)^{-1} \cdot (1/\sqrt{2 \pi}) \nonumber \\
& = \dfrac{i}{\sqrt{2 \pi}} \mathcal{F}^{-1} (1+\xi^2)^\alpha (1+ i \xi)^{-1} \nonumber \\
&= \dfrac{i}{\sqrt{2 \pi}} \bigg \{ \bigg ( 1 + \dfrac{d}{dx} \bigg )  \mathcal{F}^{-1} (1 - i \xi)^{-1} \mathcal{F} \bigg \} \mathcal{F}^{-1} (1+\xi^2)^\alpha (1+ i \xi)^{-1}  \nonumber \\
&= \dfrac{i}{\sqrt{2 \pi}} \bigg ( 1 + \dfrac{d}{dx} \bigg )  \mathcal{F}^{-1}  (1 + \xi^2)^{\alpha-1}. 
\end{align}

Using Lemma \ref{FinverseKU}, we can write
\begin{align*}
r_+ A^{-}(D) \delta & =  \dfrac{i}{\sqrt{2 \pi}} \cdot \sqrt{\dfrac{\pi}{2}}  \, \dfrac{2^{2 \alpha}}{\Gamma(1-\alpha)} \,  \bigg ( 1 + \dfrac{d}{dx} \bigg ) \, e^{-x} \, U(\alpha, 2\alpha, 2x) \\
& =  \dfrac{i}{2}  \, \dfrac{2^{2 \alpha}}{\Gamma(1-\alpha)} \,  \bigg ( 1 + \dfrac{d}{dx} \bigg ) \, e^{-x} \, U(\alpha, 2\alpha, 2x) \\
&=  \dfrac{i}{2}  \, \dfrac{2^{2 \alpha}}{\Gamma(1-\alpha)}  e^{-x} \, \dfrac{d}{dx}U(\alpha, 2\alpha, 2x)\\
&=  \dfrac{i}{2}  \, \dfrac{2^{2 \alpha}}{\Gamma(1-\alpha)}  e^{-x} \, (-2 \alpha) \,  U(\alpha+1, 2\alpha+1, 2x) \qquad \text{(see 13.3.22, \cite{NIST})}\\
&= -i \, \dfrac{\alpha \, 2^{2 \alpha}}{\Gamma(1-\alpha)}  \, e^{-x} \, U(\alpha+1, 2\alpha+1, 2x) \\
&= C_{\alpha} \, e^{-x} \, U(\alpha+1, 2\alpha+1, 2x). 
\end{align*}
This completes the proof of the lemma. \\
\end{proof}

%
%
Lemmas \ref{lemma:rA(s-2)delta} and \ref{lemma:rA(s-2)deltadashdelta} are the counterparts of Lemma \ref{lemma:rA(s-1)delta} for the operator $A^=$.
\begin{lemma} \label{lemma:rA(s-2)delta}
Suppose $\alpha < 1$. Then
\begin{equation*}
(r_+ \, A^{=} \delta) (x) = \tfrac{1}{2} \, i \, C_{\alpha}  \, e^{-x} \, U(\alpha+1, 2 \alpha, 2x),
\end{equation*}
where the constant $C_{\alpha}$,  defined in equation \eqref{defCalpha}, depends only on $\alpha$.
\end{lemma}
\begin{proof}
Now, by definition, we have 
\begin{align*}
A^{=}(D) \delta & = \mathcal{F}^{-1} A^{=}(\xi) \mathcal{F} \delta \\
& = \mathcal{F}^{-1} (1+\xi^2)^\alpha (\xi - i)^{-2} \cdot (1/\sqrt{2 \pi}) \\
& = -\dfrac{1}{\sqrt{2 \pi}} \mathcal{F}^{-1} (1+\xi^2)^\alpha (1 + i\xi)^{-2}  \\
&= -\dfrac{1}{\sqrt{2 \pi}} \bigg \{ \bigg ( 1 + \dfrac{d}{dx} \bigg )^2  \mathcal{F}^{-1} (1 - i \xi)^{-2} \mathcal{F} \bigg \} \mathcal{F}^{-1} (1+\xi^2)^\alpha (1 + i\xi)^{-2}   \\
&= -\dfrac{1}{\sqrt{2 \pi}} \bigg ( 1 + \dfrac{d}{dx} \bigg )^2  \mathcal{F}^{-1}  (1 + \xi^2)^{\alpha-2}. \\
\end{align*}

Using Lemma \ref{FinverseKU}, and noting that $\alpha-2 = (\alpha-1)-1$, we can write
\begin{align*}
r_+ A^{=}(D) \delta & =  -\dfrac{1}{\sqrt{2 \pi}} \cdot \sqrt{\dfrac{\pi}{2}}  \, \dfrac{2^{2 \alpha-2}}{\Gamma(1-(\alpha-1))} \,  \bigg ( 1 + \dfrac{d}{dx} \bigg )^2 \, e^{-x} \, U(\alpha-1, 2\alpha-2, 2x) \\
& =  -\dfrac{1}{2}  \, \dfrac{2^{2 \alpha-2}}{\Gamma(2-\alpha)} \,   e^{-x} \, \dfrac{d^2}{d^2x}U(\alpha-1, 2\alpha-2, 2x) \\
&=  -\dfrac{1}{2}  \, \dfrac{2^{2 \alpha-2}}{\Gamma(2-\alpha)}  e^{-x} \, (\alpha-1) \alpha \,2^2  \,  U(\alpha+1, 2\alpha, 2x) \qquad \text{(See 13.3.23, \cite{NIST})}\\
&=  \dfrac{\alpha \, 2^{2 \alpha-1}}{\Gamma(1-\alpha)}  \, e^{-x} \, U(\alpha+1, 2\alpha, 2x) \qquad (\text{since }\Gamma(2-\alpha)=(1-\alpha) \Gamma(1-\alpha)) \\
&= \tfrac{1}{2} \, i C_{\alpha} \, e^{-x} \, U(\alpha+1, 2\alpha, 2x). 
\end{align*}
This completes the proof of the lemma. \\
\end{proof}

\begin{lemma} \label{lemma:rA(s-2)deltadashdelta}
Suppose $\alpha < 1$. Then
\begin{equation*}
(r_+ \, A^{=} (\delta' - \delta)) (x) = -i \, C_{\alpha}  \, e^{-x} \, U(\alpha+1, 2 \alpha+1, 2x),
\end{equation*}
where the constant $C_{\alpha}$,  defined in equation \eqref{defCalpha}, depends only on $\alpha$.
\end{lemma}
\begin{proof}
Firstly, we note that \begin{align*}
\mathcal{F}(\delta'-\delta) &= \mathcal{F}(\delta') -\mathcal{F}(\delta) \\
&= - i^2\mathcal{F}(\delta') -\mathcal{F}(\delta) \\
&= -i  \mathcal{F}(D \delta) -\mathcal{F}(\delta) \\
&= - (1+ i \xi) \mathcal{F} \delta.
\end{align*}
Now, by definition, we have 
\begin{align*}
A^{=}(D) (\delta'-\delta) & = \mathcal{F}^{-1} A^{=}(\xi) \mathcal{F} (\delta'-\delta) \\
& = \mathcal{F}^{-1} (1+\xi^2)^\alpha (\xi - i)^{-2} \cdot (-1) (1 + i \xi) (1/\sqrt{2 \pi}) \\
& = \dfrac{1}{\sqrt{2 \pi}} \mathcal{F}^{-1} (1+\xi^2)^\alpha (1+ i \xi)^{-2} (1+ i \xi)\\
& = \dfrac{1}{\sqrt{2 \pi}} \mathcal{F}^{-1} (1+\xi^2)^\alpha (1+ i \xi)^{-1} \\
&= \dfrac{1}{\sqrt{2 \pi}} \bigg \{ \bigg ( 1 + \dfrac{d}{dx} \bigg )  \mathcal{F}^{-1} (1 - i \xi)^{-1} \mathcal{F} \bigg \} \mathcal{F}^{-1} (1+\xi^2)^\alpha (1+ i \xi)^{-1}  \\
&= \dfrac{1}{\sqrt{2 \pi}} \bigg ( 1 + \dfrac{d}{dx} \bigg )  \mathcal{F}^{-1}  (1 + \xi^2)^{\alpha-1}. 
\end{align*}

Hence, from equation \eqref{Aminusdeltafinal},
\begin{align*}
r_+ A^{=}(D)(\delta'-\delta) & =  -i \, C_{\alpha} \, e^{-x} \, U(\alpha+1, 2\alpha+1, 2x). 
\end{align*}
This completes the proof of the lemma. \\
\end{proof}

\begin{lemma} \label{eUab2x}
Suppose $a > 0$ and $0 < b < 3$. Then, for $x > 0$,
\[
e^{-x} \, U(a,b,2x) =
\begin{cases}
x^{1-b} \psi(a,b,x) + \phi(x)   & \text{if   } b \not= 1,2 \\
x^{1-b} \psi(a,b,x) + \vartheta(x) \log x + \phi(x) & \text{if   } b = 2 \\
\vartheta(x) \log x +\phi(x)   & \text{if   } b = 1,
\end{cases}
\]
where $\vartheta, \phi \in C^\infty(\mathbb{R})$, and their derivatives, are bounded and $O(e^{-x})$ as $x \to +\infty$. Moreover, $\psi \in C^\infty_0(\mathbb{R})$ with $\psi(a,b,x)=0$ for $x > 2$. Finally,
\begin{equation*}
\psi(a,b,0) = 2^{1-b} \, \dfrac{\Gamma(b-1)}{\Gamma(a)} \quad (b \not = 1). 
\end{equation*} \\
\end{lemma} 

\begin{proof}
Suppose $0 < b < 3$ with $b \not = 1,2$. From 13.1.3, \cite{AandS}, $U(a, b, 2 x) \in C^\infty([1, \infty))$. Moreover, from 13.5.2, \cite{AandS}, for $x \geq \tfrac{1}{2}$ the function $U(a, b, 2 x)$, together with its derivatives, is bounded and $O(x^{-a})$ as $x \to +\infty$. \\

On the other hand, we can write (see 13.1.3, \cite{AandS}),
\begin{equation*}
U(a, b, 2 x) = F(a,b, x) + x^{1-b}G(a,b, x),
\end{equation*}
where $F, G \in C^\infty([0,2])$. Let $\varphi \in C^\infty_0(\mathbb{R})$ be such that 
\begin{align*}
&\varphi(x) =
\begin{cases} 
	1 &\mbox{if } |x| \leq 1 \\
	0 & \mbox{if } |x| >2. 
\end{cases}
\end{align*} 
Then, for $x >0$, we have
\begin{align*}
& e^{-x} U(a, b, 2 x) \\
&= \varphi(x) e^{-x}U(a, b, 2 x) + (1- \varphi(x)) e^{-x}U(a, b, 2 x) \\
&= \varphi(x) e^{-x}(F(a,b, x) + x^{1-b}G(a,b, x)) +(1-\varphi(x)) e^{-x}U(a, b, 2 x)\\
&= \big \{ \varphi(x) e^{-x}F(a,b, x) + (1 - \varphi(x))e^{-x} U(a,b, 2 x) \big \} + x^{1-b} \big \{ \varphi(x) e^{-x}G(a,b, x) \big \} \\
&:=  \phi(x) + x^{1-b} \psi(a,b,x), 
\end{align*}
where $ \phi \in C^\infty(\mathbb{R})$ and, together with its derivatives, is bounded and $O(e^{-x})$ as $x \to +\infty$. Moreover, $ \psi \in C^\infty_0(\mathbb{R})$ with $ \psi(a,b,x) =0$ for $x >2$. \\

Now, see 13.5.6 and 13.5.8, \cite{AandS},
\begin{equation*}
\psi(a,b,0) = G(a,b,0) = 2^{1-b} \, \dfrac{\Gamma(b-1)}{\Gamma(a)}. \\
\end{equation*}

Finally, the proof for each of the remaining cases, $b=1,2$, follows in a similar manner, but using the logarithmic solution described in 13.1.6, \cite{AandS}. \\
\end{proof}

In the following two lemmas, we make use of a Mellin integral operator with kernel $K_{2 \alpha}$. See Section \ref{preamble} for more details. In addition, the operator $C^{2 \alpha}_{0^+}$ is discussed in Appendix \ref{Appendix FC}. \\

\begin{lemma} \label{lemma:mellinop1}
Suppose $0 < \alpha < \frac{1}{2}$ and $u \in r_+ C^\infty_0(\mathbb{R})$. Then
\begin{equation*}
x^{-2 \alpha} (u(x) - u(0)) = \int^\infty_0 K_{2 \alpha} \bigg( \dfrac{x}{y} \bigg ) h(y) \, \dfrac{dy}{y},
\end{equation*}
where $h(x) = (C^{2 \alpha}_{0^+}u) (x)$ and 
\begin{equation*} 
K_{2 \alpha}(t) = \dfrac{\chi_{[1, \infty)}(t)}{\Gamma(2 \alpha) \, t^{2 \alpha} \, (t-1)^{1-2 \alpha}}.
\end{equation*}
\end{lemma}

\begin{proof}
From Appendix \ref{Appendix FC} equation \eqref{TaylorCase1}, taking $a=0$, \begin{equation*} 
u(x)-u(0) = I^{2 \alpha}_{0^+}C^{2 \alpha}_{0^+}u(x).
\end{equation*}
Now consider the operator $(P_{2 \alpha} u)(x) = x^{-{2 \alpha}} [u(x) - u(0)]$. We have
\begin{align*}
(P_{2 \alpha} u)(x) & = x^{-{2 \alpha}} (I^{2 \alpha}_{0^+} C^{2 \alpha}_{0^+} u)(x) \\
& = x^{-{2 \alpha}} (I^{2 \alpha}_{0^+} h)(x) \qquad \text{ (where } h(x) = (C^{2 \alpha}_{0^+} u)(x)) \\
& = \dfrac{1}{\Gamma({2 \alpha})} \int^x_0 \dfrac{h(y)}{x^{2 \alpha} (x-y)^{1 - {2 \alpha}}} \, dy \\
& =  \dfrac{1}{\Gamma({2 \alpha})} \int^\infty_0 \chi_{[0,x]}(y) \dfrac{h(y)}{x^{2 \alpha} (x-y)^{1 - {2 \alpha}}} \, dy \\
& =  \dfrac{1}{\Gamma({2 \alpha})} \int^\infty_0 \chi_{[1,\infty)}\bigg( \dfrac{x}{y} \bigg ) \dfrac{h(y)}{x^{2 \alpha} (x-y)^{1 - {2 \alpha}}} \, dy \\
& = \int^\infty_0 K_{2 \alpha} \bigg( \dfrac{x}{y} \bigg ) h(y) \, \dfrac{dy}{y},
\end{align*}
where
\begin{equation*} 
K_{2 \alpha}(t) = \dfrac{\chi_{[1, \infty)}(t)}{\Gamma(2 \alpha) \, t^{2 \alpha} \, (t-1)^{1-2 \alpha}}.
\end{equation*}
\end{proof}

\begin{remark} \label{lowregnobigalpha}
Lemma \ref{lemma:mellinop12} is the counterpart of Lemma \ref{lemma:mellinop1} in the case that $\tfrac{1}{2} \leq \alpha < 1$. We note, in particular, that the required boundary condition, $u'(0) =0$, means we will not consider the case $\tfrac{1}{2} \leq \alpha < 1$ and $1/p < s < 1+1/p$. \\
\end{remark}

\begin{lemma} \label{lemma:mellinop12}
Suppose $\frac{1}{2} \leq \alpha < 1$ and $u \in r_+C^\infty_0(\mathbb{R})$ with $u'(0)=0$. Then
\begin{equation*}
x^{-2 \alpha} (u(x) - u(0)) = \int^\infty_0 K_{2 \alpha} \bigg( \dfrac{x}{y} \bigg ) h(y) \, \dfrac{dy}{y},
\end{equation*}
where $h(x) = (C^{2 \alpha}_{0^+}u) (x)$ and 
\begin{equation*} 
K_{2 \alpha}(t) = \dfrac{\chi_{[1, \infty)}(t)}{\Gamma(2 \alpha) \, t^{2 \alpha} \, (t-1)^{1-2 \alpha}}.
\end{equation*}
\end{lemma}

\begin{proof}
From Appendix \ref{Appendix FC} equation \eqref{TaylorCase2}, taking $a=0$ and $u'(0)=0$,
\begin{equation*} 
u(x)-u(0) = I^{2 \alpha}_{0^+}C^{2 \alpha}_{0^+}u(x), 
\end{equation*}
and the proof now follows as Lemma \ref{lemma:mellinop1}. \\

Finally, we note, in passing, that if $\alpha = \tfrac{1}{2}$, then
\begin{equation} \label{alphahalfudiff}
(I^{1}_{0^+}u)(x) = \int^x_0 u(y) \, dy; \quad  (C^{1}_{0^+}u)(x) = u'(x) - u'(0) = u'(x);
\end{equation}
and we simply have
\begin{equation*}
u(x)-u(0) = \int^x_0 u'(y) \,dy. \\
\end{equation*}

\end{proof}

\begin{lemma} \label{MboundedLp}
Suppose $1 < p < \infty$. Let $M$ denote the Mellin integral operator with kernel $K$, as defined in \eqref{MellinIntOp}. Then  $M$ is bounded on $L_p(\mathbb{R}_+)$ if the function $K(t) t^{-1/p^\prime}$  belongs to $L_1(\mathbb{R}_+)$, where $1/p + 1/p^\prime = 1$. \\
\end{lemma}
\begin{proof}
By definition, see \eqref{MellinIntOp}, the action of the operator $M$ on $u \in L_p(\mathbb{R}_+)$ is given by
\begin{equation*}
(Mu)(t) = \int^\infty_0 K \bigg (\dfrac{t}{\tau} \bigg ) \dfrac{u(\tau)}{\tau} \, d\tau. \\
\end{equation*}
We now define $t/\tau = x$,  and hence can write
\begin{equation} \label{Murep}
(Mu)(t) = \int^\infty_0 K(x) u(t/x) x^{-1} \, dx. \\
\end{equation}
By definition,
\begin{align*}
\| u(\cdot /x) \|_p & := \bigg ( \int^\infty_0 |u(t/x)|^p dt \bigg )^{1/p} \\
& = \bigg ( \int^\infty_0 |u(s)|^p x \, ds \bigg )^{1/p} \quad (\text{ where } s = t/x) \\
& = x^{1/p} \| u \|_p.\\
\end{align*}
Applying the $L_p(\mathbb{R}_+)$ norm to equation \eqref{Murep} we have
\begin{align*}
\| Mu \|_p & = \bigg \| \int^\infty_0 K(x) u(\cdot / x) x^{-1} dx \bigg \|_p \\
& \leq  \int^\infty_0 \| K(x) u(\cdot /x) x^{-1}  \|_p \, dx\\
& =  \int^\infty_0  | K(x) x^{-1} | \cdot  \| u(\cdot / x) \|_p \, dx   \\
& = \bigg ( \int^\infty_0  | K(x) x^{-1/p^\prime} | \, dx \bigg ) \cdot \| u \|_p, 
\end{align*}
which completes the proof of the lemma. \\
\end{proof}

\begin{lemma} \label{lemma:mellinop2}
Suppose $1 < p < \infty, \, \rho > 1/p-1$ and $\gamma >0$. Then the Mellin integral operator $M_{\gamma, \rho}$ with kernel
\begin{equation*}
K_{\gamma, \rho}(t) = \dfrac{\chi_{[1, \infty)}(t)}{t^\rho \, \Gamma( \gamma) \, t^\gamma \, (t-1)^{1-\gamma}}.
\end{equation*}
is bounded on $L_p(\mathbb{R}_+)$. Moreover, see \eqref{MellinSymbol}, $M_{\gamma, \rho}$ has symbol
\begin{equation*}
b(y) := (\mathcal{M}_{p} K_{\gamma, \rho})(y) = \dfrac{B(\rho+1/p' +iy, \gamma)}{\Gamma(\gamma)},
\end{equation*}
where $\mathcal{M}_{p}$ denotes the Mellin transform.
\end{lemma}

\begin{proof}
From Lemma \ref{MboundedLp}, to prove boundedness on $L_p(\mathbb{R}_+)$ it is enough to show that 
\begin{equation*}
\Gamma(\gamma) \int^\infty_0 |K_{\gamma, \rho}(t)| t^{-1/p' } \, dt < \infty.
\end{equation*}
We will make use of the following result, see 5.12.3, \cite{NIST},
\begin{equation*}
\int^\infty_0 \dfrac{t^{a-1} \, dt }{(1+t)^{a+b}} = B(a,b), \quad \operatorname{Re} {a}>0, \, \operatorname{Re} {b}>0,
\end{equation*}
where $B(\cdot, \cdot)$ denotes the \textit{beta function}.  Now

\begin{align*}
\Gamma(\gamma) \int^\infty_0 |K_{\gamma, \rho}(t)| t^{-1/p' } \, dt &= \int^\infty_1 t^{-\rho} \, t^{-\gamma} \, (t-1)^{\gamma -1} \, t^{-1/p' } \, dt \\
&= \int^\infty_1 t^{-(\rho + \gamma + 1/p')} \, (t-1)^{\gamma -1} \, dt \\
&= \int^\infty_0 (1+w)^{-(\rho + \gamma + 1/p')} \, w^{\gamma -1} \, dt  \qquad (w=t-1)\\
&= \int^\infty_0 \dfrac{w^{\gamma-1} \, dw }{(1+w)^{\gamma + \rho + 1 /p'}} \\
&= B( \gamma, \rho + 1 /p') \\
&= B( \rho + 1 /p', \gamma) \\
&< \infty. \\
\end{align*}

From \eqref{MellinSymbol}, to calculate the symbol, we take the Mellin transform of the kernel:
\begin{align*}
(\mathcal{M}_{p } K_{\gamma, \rho})(y) & = \int^\infty_0 t^{-1/p'- i y} \, K_{\gamma, \rho}(t)  \, dt \\
&= \int^\infty_0 \dfrac{w^{\gamma-1} \, dw }{\Gamma(\gamma) \, (1+w)^{\gamma + \rho + 1 /p' + i y}} \\
&= \dfrac{B( \rho + 1 /p' + i y, \gamma)}{\Gamma(\gamma)}, 
\end{align*}
as required. This completes the proof of the lemma. \\
\end{proof}

Suppose a function $f: \overline{\mathbb{R}} \to \mathbb{C}$. Then we define the total variation, $V(f)$ as
\begin{equation*}
V(f) := \sup \bigg ( \sum^N_{k=1} |f(t_k)- f(t_{k-1}) |\bigg),
\end{equation*}
where the supremum is taken over all partitions $-\infty \leq t_0 < t_1 < \dots < t_N \leq + \infty$ of $\overline{\mathbb{R}}$. We denote the set of all bounded functions on $\overline{\mathbb{R}}$ with finite total variation by $BV(\overline{\mathbb{R}})$.  See \cite{Dudu, Roch}. We note, in passing, that this set is a Banach space under the norm
\begin{equation*}
\| f \|_{BV} := \| f \|_\infty + V(f).
\end{equation*}
\\
One important motivation for the study of functions of bounded variation, see, for example, Proposition 4.2.2, p. 200, \cite{Roch}, is the inclusion
\begin{equation} \label{BVinclusionMp}
BV(\overline{\mathbb{R}}) \subset \mathfrak{M}_p, \quad 1 < p < \infty,
\end{equation}

The following remark describes a useful way to demonstrate that certain functions have bounded variation on $\overline{\mathbb{R}}$. \\

\begin{remark} \label{BVtestmethod}
Suppose $f: \overline{\mathbb{R}} \to \mathbb{C}$ is bounded, and differentiable almost everywhere with $f' \in L_1(\mathbb{R})$. Since we can write
\begin{equation*}
f(t_k)- f(t_{k-1}) = \int^{t_k}_{t_{k-1}} f'(t) \, dt \quad \text{for} \quad k=1, \dots, N, 
\end{equation*}
it is easy to see that
\begin{equation*}
V(f) \leq \| f' \|_1.
\end{equation*}
Therefore, $f \in BV(\overline{\mathbb{R}})$. \\
\end{remark}

\begin{lemma} \label{BVBeta}
Suppose $d > c > 0$, and define the function
\begin{equation*}
g(y) := \dfrac{\Gamma( c + i y)}{\Gamma (d + i y)}, \quad y \in \mathbb{R}.
\end{equation*}
Then $g$ is continuous and bounded on $\dot{\mathbb{R}}$, and has bounded variation.
\end{lemma}
\begin{proof}
Since $c,d > 0$, by 5.2.1, \cite{NIST}, the functions $\Gamma( c + i y), \,\Gamma (d + i y)$ are continuous for $y \in \mathbb{R}$, and have no zeroes. Hence, $g$ is continuous on ${\mathbb{R}}$. \\

Moreover, from 5.11.12, \cite{NIST}, we have the following asymptotic
\begin{equation*}
g(y) \sim (i y)^{c-d}, \quad |y| \to \infty.
\end{equation*}
Thus, as $c - d < 0$, the function $g$ is continuous and bounded on $\dot{\mathbb{R}}$. \\

In terms of the digamma function, $\psi(z)$, see 5.2.2, \cite{NIST},
\begin{equation*}
g'(y) = i g(y) \big ( \psi(c+i y) - \psi(d + i y) \big ).
\end{equation*}
From 5.11.2, \cite{NIST}, $\psi(z) \sim \log z$ as $|z| \to \infty$, and we have
\begin{equation*}
g'(y) \sim (i y)^{c-d} \,  \log \bigg ( \dfrac{c + i y}{d + i y}\bigg) \sim (i y)^{c-d} \, \dfrac{(c-d)}{d+i y}.
\end{equation*}
Since $c - d < 0$, it clear that $g' \in L_1(\mathbb{R})$. \\

Finally, from Remark \ref{BVtestmethod}, $g$ has bounded variation on $\dot{\mathbb{R}}$. \\
\end{proof}

\begin{remark} \label{remark:mellinop}
Suppose $\gamma > 0$. We note from Lemma \ref{lemma:mellinop2}, that the kernel $K_{\gamma, 0}$, of the Mellin integral operator $M_{\gamma, 0}$, satisfies the conditions
\begin{equation} \label{Kgamma0integ}
\text{supp } K_{\gamma, 0} \subseteq [1, \infty) \quad \text{and} \quad \int^\infty_0 | K_{\gamma, 0}(t) | t^{-\epsilon} \, dt < \infty, \quad \text{for all} \,\, \epsilon > 0.
\end{equation}

If, in addition, $1 < p < \infty$ and $\rho > 1/p-1$ then, from Lemma \ref{BVBeta}, and its proof, the symbol $b_{\gamma, \rho} (y) = B(\rho + 1/p' +i y, \gamma) / \Gamma(\gamma)$ is continuous, with bounded variation, as $y$ varies over $\dot{\mathbb{R}}$. Moreover, $b_{\gamma, \rho}(\pm \infty) =0$. \\

From inclusion \eqref{BVinclusionMp},  $b_{\gamma, \rho}$ is a Fourier $L_p$-multiplier. Hence, see equation \eqref{MellinConvOp}, $M_{\gamma, \rho} = M^0(b_{\gamma, \rho})$ is a Mellin convolution operator. \\
\end{remark}

We will make extensive use of Remark \ref{remark:mellinop} in subsequent chapters. \\

\begin{lemma} \label{ustou}
Suppose $1/p < s < 1+1/p$ and $u \in H^s_p(\overline{\mathbb{R}_+})$. Let $u_s = (D+i)^{s-1} e_+ (D-i)u$. Then 
\begin{equation*}
u = (r_+ C_s(D) e_+) (r_+ u_s),
\end{equation*}
where $C_s(D)$ has the symbol $c_s(\xi) = (\xi-i)^{-1} (\xi+i)^{1-s}$. 
\end{lemma}
\begin{proof} We have $u_s = (D+i)^{s-1} e_+ (D-i)u$. Hence
\begin{align*} 
(D+i)^{1-s} u_s &= e_+ (D-i)u\\
r_+ (D+i)^{1-s} u_s &= (D-i)u\\
r_+ (D+i)^{1-s} u_s &= (D-i)(r_+ u_0) \quad \text{where } u=r_+ u_0, \, u_0 \in H^s_p(\mathbb{R})\\
r_+ (D+i)^{1-s} u_s &= r_+ (D-i) u_0 \\
(r_+ (D-i)^{-1} l_+)r_+ (D+i)^{1-s} u_s &= (r_+ (D-i)^{-1} l_+)r_+ (D-i) u_0 \\
r_+ (D-i)^{-1} (D+i)^{1-s} u_s &= r_+ (D-i)^{-1} (D-i) u_0 \quad \text {by Lemma } \ref{lemma:rLambdae}\\
r_+ (D-i)^{-1} (D+i)^{1-s} e_+ r_+ u_s &= r_+ u_0 \quad \text{since  supp} \, u_s \subseteq \overline{\mathbb{R}_+} \text{ by Lemma } \ref{lemma:us}\\
(r_+ C_s(D) e_+) (r_+ u_s) &= u,  \qquad \text{as required.}
\end{align*}
\end{proof}

If $1+1/p < s < 2+1/p$, we have the following counterpart to Lemma \ref{ustou}.
\begin{lemma} \label{ustou2}
Suppose $1+1/p < s < 2+1/p$ and $u \in H^s_p(\overline{\mathbb{R}_+})$. Let $u_s = (D+i)^{s-2} e_+ (D-i)^2u$. Then 
\begin{equation*}
u = ( r_+ C_s(D) e_+ ) (r_+ u_s),
\end{equation*}
where $C_s(D)$ has the symbol $c_s(\xi) = (\xi-i)^{-2}(\xi+i)^{2-s}$.
\end{lemma}
\begin{proof} 
The proof follows as Lemma \ref{ustou}, but using Lemma \ref{lemma:u2s} instead of Lemma \ref{lemma:us}. \\
\end{proof}

\begin{lemma} \label{htous}
Suppose $0 < \alpha < \frac{1}{2}, \,\, 1 < p < \infty,  \,\, 1/p < s < 1 +1/p$ and $u \in H^s_p(\overline{\mathbb{R}_+})$. Let $h = C^{2\alpha}_{0^+} u$ and $u_s=(D+i)^{s-1}e_+(D-i)u$.  Then 
\begin{equation*}
h = (r_+ C(D) e_+) (r_+ u_s) + i \dfrac{u(0)}{\sqrt{2\pi}} \, r_+ \mathcal{F}^{-1} (-i \xi)^{2\alpha-1}(\xi-i)^{-1},
\end{equation*}
where $C(D)$ has the symbol $c(\xi) = (-i \xi)^{2\alpha} (\xi + i)^{1-s} (\xi-i)^{-1}$.
\end{lemma}
\begin{proof} From Lemma \ref{lemma:Deu} and the definition of $u_s$, we have
\begin{align*} 
\mathcal{F}( e_+ u) &= \mathcal{F}( (D-i)^{-1} (D-i) e_+ u) \\
&= \mathcal{F}( (D-i)^{-1} e_+ (D-i)  u) + \mathcal{F}((D-i)^{-1} \, i \, u(0) \, \delta) \\
&= \mathcal{F}( (D-i)^{-1} (D+i)^{1-s} u_s) + i \, u(0) \, (\xi-i)^{-1} \cdot \tfrac{1}{\sqrt{2\pi}} \\
&= (\xi-i)^{-1} (\xi+i)^{1-s} \mathcal{F}(u_s) + i \, u(0) \, (\xi-i)^{-1} \cdot \tfrac{1}{\sqrt{2\pi}}.  
\end{align*}
Moreover,
\begin{align*}
i \, e_+ h(x) &= i \, e_+ (C^{2\alpha}_{0^+} u)(x) \\
&= e_+ I^{1-2\alpha}_{0^+} D u \qquad \text{(see Appendix \ref{Appendix FC})} \\
&= I^{1-2\alpha}_{+} (e_+ D u) \\
&= I^{1-2\alpha}_{+} (e_+ (D-i) u + i \, e_+ u)\\
&= I^{1-2\alpha}_{+} ((D+i)^{1-s}u_s + i \, e_+ u). 
\end{align*}

Applying the Fourier transform, see Appendix \ref{Appendix FC}, 
\begin{align*}
i \mathcal{F}(e_+ h) & = (-i\xi)^{2\alpha-1} \big \{ \mathcal{F}((D+i)^{1-s} u_s) + i \mathcal{F}(e_+ u) \big \} \\
&= (-i\xi)^{2\alpha-1} \bigg \{  (\xi+i)^{1-s} \mathcal{F}(u_s) + i\, (\xi-i)^{-1} (\xi+i)^{1-s} \mathcal{F}(u_s) + i^2 \dfrac{u(0)}{\sqrt{2\pi}}(\xi-i)^{-1} \bigg \} \\
&= (-i\xi)^{2\alpha-1} (\xi+i)^{1-s} \mathcal{F}(u_s) \{ 1 + i(\xi-i)^{-1} \} -  \dfrac{u(0)}{\sqrt{2\pi}} (-i\xi)^{2\alpha-1} (\xi-i)^{-1}.
\end{align*}

Noting that $1 + i(\xi-i)^{-1} = \xi (\xi-i)^{-1}$ we have
\begin{equation*}
\mathcal{F}(e_+ h) = (-i\xi)^{2\alpha} (\xi+i)^{1-s} (\xi-i)^{-1} \mathcal{F}(u_s) + i \, \dfrac{u(0)}{\sqrt{2\pi}} (-i\xi)^{2\alpha-1} (\xi-i)^{-1}.
\end{equation*}
But since supp $u_s \subseteq \overline{\mathbb{R}_+}$,
\begin{equation*}
h = (r_+ C(D) e_+) (r_+ u_s) + i \dfrac{u(0)}{\sqrt{2\pi}} \, r_+ \mathcal{F}^{-1} (-i \xi)^{2\alpha-1}(\xi-i)^{-1},
\end{equation*}
which completes the proof of the lemma. \\
\end{proof}

Lemmas \ref{htous2} and \ref{htous2supplement} are the counterparts of Lemma \ref{htous} for the case of higher regularity, namely $1+1/p < s < 2+1/p$. \\

\begin{remark} \label{InvFTalphahalf}
Let $\theta$ denote the Heaviside step function. We note, in preparation for Lemma \ref{htous2}, that if $\alpha = \tfrac{1}{2}$, then
\begin{align*}
r_+ \, \mathcal{F}^{-1} & (-i \xi)^{2 \alpha-1} (\xi - i)^{-2} \\
& =  -r_+ \, \mathcal{F}^{-1} (\xi - i)^{-2} \\
& = r_+ \, \mathcal{F}^{-1} (1 + i\xi)^{-2} \\
& = -r_+ \, \dfrac{1}{\sqrt{2 \pi}} \int^\infty_{-\infty} (1+ i \xi)^{-2} e^{-i \xi x} \, d\xi \\
& = r_+ \, \big ( \sqrt{2 \pi} \, x \, e^x  \, \theta (-x) \big ) \quad \text{(3.382 6,  p. 349 \cite{GR})} \\
& = 0. \\
\end{align*} 
\end{remark}

\begin{lemma} \label{htous2}
Suppose $\frac{1}{2} \leq \alpha < 1, \,\, 1 < p < \infty,  \,\, 1+1/p < s < 2 +1/p$ and $u \in H^s_p(\overline{\mathbb{R}_+})$ with $u'(0)=0$. Let $h(x) = (C^{2\alpha}_{0^+} u)(x)$ and $u_s=(D+i)^{s-2}e_+(D-i)^2 u$.  Then 
\begin{equation*}
h=
(r_+ C(D) e_+) (r_+ u_s) + \dfrac{u(0)}{\sqrt{2\pi}} \, r_+ \mathcal{F}^{-1} (-i \xi)^{2\alpha-1}(\xi-i)^{-2},
\end{equation*}
where $C(D)$ has the symbol $c(\xi) = (-i \xi)^{2\alpha} (\xi + i)^{2-s} (\xi-i)^{-2}$. \\
\end{lemma}
\begin{proof}  Firstly, suppose that $\alpha = \tfrac{1}{2}$. Then, $2 \alpha =1$, and we simply have
\begin{align*}
h(x) & = (C^1_{0^+} u)(x) \\
& = u'(x) \quad \text{from equation} \, \, \eqref{alphahalfudiff} \\
& = \dfrac{d}{dx} \, (r_+ (D-i)^{-2}(D+i)^{2-s} e_+) (r_+ u_s) \quad \text{from Lemma }\ref{ustou2} \\
& = (r_+ C(D) e_+) (r_+ u_s), 
\end{align*}
since $\dfrac{d}{dx} = - i D$ and $Dr_+ = r_+D$. Noting the result in Remark \ref{InvFTalphahalf}, this completes the proof for $\alpha = \tfrac{1}{2}$. \\

We now consider the case $\tfrac{1}{2} < \alpha <1$. From Lemma \ref{lemma:D2eu}, with $u'(0)=0$, and the definition of $u_s$ we have
\begin{align*} 
\mathcal{F}( e_+ u) &= \mathcal{F}( (D-i)^{-2} (D-i)^2 e_+ u) \\
&= \mathcal{F}( (D-i)^{-2} e_+ (D-i)^2  u) - \mathcal{F}((D-i)^{-2} \, u(0) \, (\delta' - 2\delta)) \\
&= \mathcal{F}( (D-i)^{-2} (D+i)^{2-s} u_s) + i \, u(0) \, (\xi-i)^{-2} (\xi - 2i) \cdot \tfrac{1}{\sqrt{2\pi}} \\
&= (\xi-i)^{-2} (\xi+i)^{2-s} \mathcal{F}(u_s) + i \, u(0) \, (\xi-i)^{-2} (\xi - 2i) \cdot \tfrac{1}{\sqrt{2\pi}}, 
\end{align*}
since $-\mathcal{F}(\delta'-2 \delta) = i \mathcal{F} (D \delta) + 2 \mathcal{F} \delta = i(\xi- 2i) \mathcal{F} \delta$. \\

Now $(e_+u)'=e_+ u' + u(0) \delta$, and hence $D(e_+ u) = e_+Du + i u(0)\delta$. Therefore,
\begin{align*}
\mathcal{F}(e_+Du) &= \mathcal{F}(D(e_+u)) -\mathcal{F}(i u(0) \delta)\\
&= \xi \mathcal{F}(e_+u) - \dfrac{i u(0)}{\sqrt{2 \pi}}. 
\end{align*}

Moreover,
\begin{align*}
-e_+ h(x) &= -e_+ (C^{2\alpha}_{0^+} u)(x) \\
&= -e_+ I^{2-2\alpha}_{0^+}  u'' \qquad \text{(See Appendix \ref{Appendix FC})} \\
&= e_+ I^{2-2\alpha}_{0^+}  D^2 u \\
&=  I^{2-2\alpha}_{+} (e_+ D^2 u) \\
&= I^{2-2\alpha}_{+} \big ( e_+ (D-i)^2 u + 2 i \, e_+Du +  e_+u \big ) \\
&= I^{2-2\alpha}_{+} \big ( (D+i)^{2-s}u_s + 2 i \, e_+Du +  e_+u \big ). \\
\end{align*}

Applying the Fourier transform, see Appendix \ref{Appendix FC},  and using the expressions recently established for $\mathcal{F}(e_+ u)$ and $\mathcal{F}(e_+ Du)$,
\begin{align*}
\mathcal{F}(e_+ h) & = - (-i\xi)^{2\alpha-2} \bigg\{ \mathcal{F}((D+i)^{2-s} u_s) + 2i \mathcal{F}(e_+ D u) + \mathcal{F}(e_+u) \bigg \}\\
& = -(-i\xi)^{2\alpha-2} \bigg\{ (\xi+i)^{2-s}\mathcal{F}(u_s) + (1+2i \xi) \mathcal{F}(e_+ u) + \dfrac{2u(0)}{\sqrt{2\pi}}\bigg \} \\
& = -(-i\xi)^{2\alpha-2} \bigg\{ (\xi+i)^{2-s}\mathcal{F}(u_s) + (1+2i \xi) (\xi-i)^{-2} (\xi+i)^{2-s} \mathcal{F}(u_s) \bigg\} \\
& \quad -(-i\xi)^{2\alpha-2} \bigg ( \dfrac{u(0)}{\sqrt{2\pi}(\xi-i)^2}\bigg )\cdot \bigg [ (1+ 2i \xi) i  (\xi - 2i) + 2(\xi-i)^2 \bigg ] \\
\end{align*}

But $(\xi+i)^{2-s}\mathcal{F}(u_s) + (1+2i \xi) (\xi-i)^{-2} (\xi+i)^{2-s} \mathcal{F}(u_s)$
\begin{align*}
&= (\xi-i)^{-2} (\xi+i)^{2-s} \mathcal{F}(u_s) \bigg \{  (\xi-i)^2 + 1+ 2 i \xi \bigg \} \\
&= \xi^2 (\xi-i)^{-2} (\xi+i)^{2-s} \mathcal{F}(u_s),
\end{align*}
and $((1+ 2i \xi) i  (\xi - 2i) + 2(\xi-i)^2 = i \xi$. Thus, we have
\begin{equation*}
\mathcal{F}(e_+ h) = (-i\xi)^{2\alpha} (\xi+i)^{2-s} (\xi-i)^{-2} \mathcal{F}(u_s) + \dfrac{u(0)}{\sqrt{2\pi}} (-i\xi)^{2\alpha-1} (\xi-i)^{-2}.
\end{equation*}
But since supp $u_s \subseteq \overline{\mathbb{R}_+}$,
\begin{equation*}
h = (r_+ C(D) e_+) (r_+ u_s) + \dfrac{u(0)}{\sqrt{2\pi}} \, r_+ \mathcal{F}^{-1} (-i \xi)^{2\alpha-1}(\xi-i)^{-2},
\end{equation*}
which completes the proof of the lemma. \\
\end{proof}

\begin{lemma} \label{htous2supplement}
Suppose $0 <  \alpha < \frac{1}{2}, \,\, 1 < p < \infty,  \,\, 1+1/p < s < 2 +1/p$ and $u \in H^s_p(\overline{\mathbb{R}_+})$, with $u'(0)=0$.  Let $h(x) = (C^{2\alpha}_{0^+} u)(x)$ and $u_s=(D+i)^{s-2}e_+(D-i)^2 u$.  Then 
\begin{equation*}
h= (r_+ C(D) e_+) (r_+ u_s) + \dfrac{u(0)}{\sqrt{2\pi}} \, r_+ \mathcal{F}^{-1} (-i \xi)^{2\alpha-1}(\xi-i)^{-2}, 
\end{equation*}
where $C(D)$ has the symbol $c(\xi) = (-i \xi)^{2\alpha} (\xi + i)^{2-s} (\xi-i)^{-2}$. \\
\end{lemma}
\begin{proof}  As in Lemma \ref{htous2}, we have
\begin{equation*} 
\mathcal{F}( e_+ u) = (\xi-i)^{-2} (\xi+i)^{2-s} \mathcal{F}(u_s) + \dfrac{i u(0)}{\sqrt{2 \pi}} \, (\xi-i)^{-2} (\xi - 2i),
\end{equation*}
and
\begin{equation*}
\mathcal{F}(e_+Du) = \xi \mathcal{F}(e_+u) - \dfrac{i u(0)}{\sqrt{2 \pi}}.
\end{equation*}

Moreover,
\begin{align*}
i\, e_+ h(x) &= i \, e_+ (C^{2\alpha}_{0^+} u)(x) \\
&= i \, e_+ I^{1-2\alpha}_{0^+}  u' \qquad \text{(See Appendix \ref{Appendix FC})} \\
&= e_+ I^{1-2\alpha}_{0^+}  D u \\
&=  I^{1-2\alpha}_{+} (e_+ D u).
\end{align*}

Applying the Fourier transform, see Appendix \ref{Appendix FC}, and using the expression above for $\mathcal{F}(e_+ Du)$,
\begin{align*}
\mathcal{F}(e_+ h) & = -i \, (-i\xi)^{2\alpha-1} \mathcal{F}(e_+ Du) \\
&= -i \, (-i\xi)^{2\alpha-1} \bigg \{  \xi \mathcal{F}(e_+u) - \dfrac{i u(0)}{\sqrt{2 \pi}}\bigg \} \\
&= (-i\xi)^{2\alpha} \, \mathcal{F}(e_+u) - (-i \xi)^{2 \alpha-1}\dfrac{u(0)}{\sqrt{2 \pi}}.
\end{align*}
But, using the expression above for $\mathcal{F}( e_+ u)$, and collecting the terms containing $u(0)$,
\begin{align*}
(-i \xi)^{2 \alpha} & \bigg \{ i \, \dfrac{u(0)}{\sqrt{2\pi}} \, (\xi-i)^{-2} (\xi - 2i)  \bigg \}  - (-i \xi)^{2 \alpha-1}\dfrac{u(0)}{\sqrt{2 \pi}} \\
&= \dfrac{u(0)}{\sqrt{2 \pi}} \cdot \dfrac{(-i \xi)^{2 \alpha-1}}{(\xi-i)^2} \cdot \big \{ (-i\xi) i (\xi- 2i)-(\xi-i)^2 \big \} \\
&= \dfrac{u(0)}{\sqrt{2 \pi}} \cdot \dfrac{(-i \xi)^{2 \alpha-1}}{(\xi-i)^2} \cdot \big \{ \xi^2 - 2 i \xi - \xi^2 + 2 i \xi +1 \big \} \\
&= \dfrac{u(0)}{\sqrt{2 \pi}} \cdot \dfrac{(-i \xi)^{2 \alpha-1}}{(\xi-i)^2}. 
\end{align*}

Thus, we have
\begin{equation*}
\mathcal{F}(e_+ h) = (-i\xi)^{2\alpha} (\xi+i)^{2-s} (\xi-i)^{-2} \mathcal{F}(u_s) + \dfrac{u(0)}{\sqrt{2\pi}} (-i\xi)^{2\alpha-1} (\xi-i)^{-2}.
\end{equation*}
But since supp $u_s \subseteq \overline{\mathbb{R}_+}$,
\begin{equation*}
h = (r_+ C(D) e_+) (r_+ u_s) + \dfrac{u(0)}{\sqrt{2\pi}} \, r_+ \mathcal{F}^{-1} (-i \xi)^{2\alpha-1}(\xi-i)^{-2},
\end{equation*}
which completes the proof of the lemma. \\
\end{proof}

\begin{lemma} \label{a3(x)}
Suppose $0 < \alpha < \frac{1}{2}$. Then
\begin{equation*}
a(x) := (r_+ \, A^- \, \chi_{\mathbb{R}_{-}}) (x) = C_{\alpha} \int^\infty_x e^{-t} U(\alpha+1, 2\alpha+1, 2t) \, dt,
\end{equation*}
where the constant $C_\alpha$ only depends on $\alpha$, and is given by equation \eqref{defCalpha} in the statement of Lemma \ref{lemma:rA(s-1)delta}. \\
\end{lemma}
\begin{proof}
We being by noting the following standard results:
\begin{equation*}
\mathcal{F} (\delta) = 1 / \sqrt{2 \pi} \quad \text{and} \quad \mathcal{F}(f*g) = \sqrt{2 \pi} \,\mathcal{F}(f) \, \mathcal{F}(g).
\end{equation*}
See, for example, Chapter I, Section 2, \cite{Es}, with an appropriate correction for the different constant used in the Fourier transform definitions. \\

Since
\begin{align*}
\mathcal{F}^{-1}  (1+\xi^2)^\alpha(\xi-i)^{-1}
& = \mathcal{F}^{-1} (1+\xi^2)^\alpha(\xi-i)^{-1} \cdot 1 \\
& = \mathcal{F}^{-1} (1+\xi^2)^\alpha(\xi-i)^{-1}\sqrt{2 \pi} \mathcal{F} (\delta) \\
& = \sqrt{2 \pi} \, A^{-} \delta, 
\end{align*}
we have
\begin{equation*}
(1+\xi^2)^\alpha(\xi-i)^{-1} = \sqrt{2 \pi} \, \mathcal{F} (A^{-} \delta).
\end{equation*}
Hence, for $x>0$,
\begin{align*}
(r_+ \, A^- \, \chi_{\mathbb{R}_{-}}) (x) &= r_+ \mathcal{F}^{-1} (1+\xi^2)^\alpha(\xi-i)^{-1} \mathcal{F}(\chi_{\mathbb{R}_{-}})\\
&= r_+ \mathcal{F}^{-1} \sqrt{2 \pi} \, \mathcal{F} (A^{-} \delta) \mathcal{F}(\chi_{\mathbb{R}_{-}})\\
&= \big( (A^-\delta)*\chi_{\mathbb{R}_{-}}\big) (x) \\
&= C_{\alpha} \int_{\mathbb{R}} e^{-t} U(\alpha+1, 2\alpha+1, 2t) \chi_{\mathbb{R}_{-}} (x-t)\, dt \quad \text{by Lemma \ref{lemma:rA(s-1)delta}}\\
&= C_{\alpha} \int^\infty_x e^{-t} U(\alpha+1, 2\alpha+1, 2t) \, dt.
\end{align*}
\end{proof}

The following result is the counterpart of Lemma \ref{a3(x)} for the operator $A^=$.
\begin{lemma} \label{a3(x)2}
Suppose $0 < \alpha < 1$. Then
\begin{equation*}
a(x) := (r_+ \, A^= \, \chi_{\mathbb{R}_{-}}) (x) = \tfrac{1}{2} \, i \, C_{\alpha} \int^\infty_x e^{-t} U(\alpha+1, 2\alpha, 2t) \, dt,
\end{equation*}
where the constant $C_\alpha$ only depends on $\alpha$, and is given by equation \eqref{defCalpha} in the statement of Lemma \ref{lemma:rA(s-1)delta}. 
\end{lemma}
\begin{proof}
Following the method of proof of Lemma \ref{a3(x)}, it is easy to show that 
\begin{equation*}
(1+\xi^2)^\alpha(\xi-i)^{-2} = \sqrt{2 \pi} \, \mathcal{F} (A^{=} \delta).
\end{equation*}
The required result now follows, as Lemma  \ref{a3(x)}, but now using Lemma \ref{lemma:rA(s-2)delta}. \\
\end{proof}

\begin{lemma} \label{rAminuschiminus}
Suppose $0 < \alpha < \tfrac{1}{2}$. Then, for $x > 0$, we can write
\begin{equation*}
(r_+A^- \chi_{\mathbb{R}_-})(x) =  C_\alpha \, \big ( \phi_1(x) + x^{1-2 \alpha}  \phi_2(x) \big ),
\end{equation*}
where $ \phi_1, \phi_2 \in C^\infty(\mathbb{R})$, and together with their derivatives, are bounded and $O(e^{-x})$ as $x \to +\infty$. \\
\end{lemma}
\begin{proof}
From Lemma \ref{a3(x)}, for $x > 0$, we have
\begin{equation*}
(r_+A^- \chi_{\mathbb{R}_-})(x) = C_\alpha \int^\infty_x e^{-t} U(\alpha+1, 2\alpha+1, 2 t) \, dt
\end{equation*}
for some constant $C_\alpha$. Noting that
\begin{equation*}
\int^\infty_x t^{-2\alpha} \psi(\alpha+1, 2\alpha+1, t) \, dt
\end{equation*}
contributes to both $\phi_1(x)$ and $x^{1-2\alpha} \phi_2(x)$, the required result now follows directly from Lemma \ref{eUab2x}.\\
\end{proof}

\begin{remark} \label{rAminus2chiminus}
Similarly, if $0 < \alpha < 1$ and $\alpha \not = \tfrac{1}{2}$, then for $x > 0$, we can write
\begin{equation*}
(r_+A^= \chi_{\mathbb{R}_-})(x) =  \tfrac{1}{2} \, i \, C_\alpha \, \big ( \phi_1(x) + x^{2-2 \alpha}  \phi_2(x) \big ),
\end{equation*}
where $ \phi_1, \phi_2 \in C^\infty(\mathbb{R})$, and together with their derivatives, are bounded and $O(e^{-x})$ as $x \to +\infty$. \\

On the other hand, if $\alpha = \tfrac{1}{2}$, then for $x > 0$, we can write
\begin{equation*}
(r_+A^= \chi_{\mathbb{R}_-})(x) =  \tfrac{1}{2} \, i \, C_\alpha \, \big ( \vartheta_1(x) x \log x+ \phi_3(x) \big ),
\end{equation*}
where $ \vartheta_1, \phi_3 \in C^\infty(\mathbb{R})$, and together with their derivatives, are bounded and $O(e^{-x})$ as $x \to +\infty$. \\

\end{remark}

\begin{lemma} \label{multiplierbcompact}
Suppose $1 < p < \infty$ and $\phi \in C^\infty_0(\mathbb{R})$. Then
\begin{equation*}
\phi I: H^t_p(\overline{\mathbb{R}_+}) \to H^{t-\epsilon}_p(\overline{\mathbb{R}_+})
\end{equation*}
is compact for all $t \in \mathbb{R}$ and all $\epsilon > 0$.
\end{lemma}
\begin{proof}
Since $\phi I =  \phi r_+ l_ += r_+ (\phi l_+)$, it is enough to prove that 
\begin{equation*}
\phi I: H^t_p({\mathbb{R}}) \to H^{t-\epsilon}_p({\mathbb{R}})
\end{equation*}
is compact for all $t \in \mathbb{R}$ and all $\epsilon$. From Section 3.3.1, p. 195, \cite{Tr83}, the multiplication operator
\begin{equation*}
\phi I : H^t_p({\mathbb{R}}) \to H^{t}_p({\mathbb{R}})  \quad (\hookrightarrow H^{t-\epsilon}_p({\mathbb{R}}))
\end{equation*}
is bounded. \\

Suppose supp $\phi \subset \Omega$, where $\Omega \subset \mathbb{R}$ is a bounded open set.  Let $r_\Omega : H^t_p({\mathbb{R}}) \to H^{t}_p(\Omega)$ and $e_\Omega: H^{t-\epsilon}_p(\Omega) \to \widetilde{H}^{t-\epsilon}_p(\Omega) $ denote the operations of restriction to $\Omega$, and extension by zero from $\Omega$ respectively. Let $i_{\Omega, t,p, \epsilon}$ denote the inclusion map
\begin{equation*}
i_{\Omega, t,p, \epsilon}: H^t_p(\Omega) \to H^{t-\epsilon}_p(\Omega).
\end{equation*}
Then, we have the operator identity
\begin{equation*}
\phi I = e_\Omega \, i_{\Omega, t,p, \epsilon} \, r_\Omega \, \phi I,
\end{equation*}
where on the left-hand side we note that $\phi I : H^t_p({\mathbb{R}}) \to H^{t-\epsilon}_p({\mathbb{R}})$ and, on the right-hand side, we simply assume $\phi I : H^t_p({\mathbb{R}}) \to H^{t}_p({\mathbb{R}})$. \\

But from Section 2.9.1, p. 166, \cite{Tr83}, the restriction $r_\Omega : H^t_p({\mathbb{R}}) \to H^{t}_p(\Omega)$ is bounded and from Section 3.4.3, Remark 2, p. 211, \cite{Tr83}, the extension operator $e_\Omega: H^{t-\epsilon}_p(\Omega) \to \widetilde{H}^{t-\epsilon}_p(\Omega) \hookrightarrow H^{t-\epsilon}_p(\mathbb{R})$ is also bounded. Moreover, from Section 4.3.2, Remark 1, p. 233, \cite{Tr83}, the inclusion map $i_{\Omega, t,p, \epsilon}$ is compact. Hence, $\phi I: H^t_p({\mathbb{R}}) \to H^{t-\epsilon}_p({\mathbb{R}})$
is compact, as required. \\ 
\end{proof}

\begin{lemma} \label{multiplierbHcompact}
Suppose $1 < p < \infty$ and $t > 1/p$. Let $a \in H^t_p(\overline{\mathbb{R}_+})$. Then
\begin{equation*}
aI: H^t_p(\overline{\mathbb{R}_+}) \to H^{t-\epsilon}_p(\overline{\mathbb{R}_+})
\end{equation*}
is compact for any $\epsilon >0$.
\end{lemma}
\begin{proof}
Since, by hypothesis, $t  > 1/p, \, H^t_p(\overline{\mathbb{R}_+})$ is a Banach algebra. (See Section 2.8.3, Remark 3, p. 146, \cite{Tr83}.) Thus
\begin{equation*}
aI: H^t_p(\overline{\mathbb{R}_+}) \to H^t_p(\overline{\mathbb{R}_+}) \hookrightarrow H^{t-\epsilon}_p(\overline{\mathbb{R}_+})
\end{equation*}
is a bounded operator for all $\epsilon >0$. In particular, it has operator norm
\begin{equation*}
\|a I\|_{Op} \leq \text{ const } \|a\|_{t,p}.
\end{equation*} 
Since $C^\infty_0(\mathbb{R})$ is dense in $H^t_p({\mathbb{R}})$, we can approximate $a$ arbitrarily closely by a sequence $\{ \phi_n \}^\infty_{n=1} \subset C^\infty_0(\mathbb{R})$. Finally, from Lemma \ref{multiplierbcompact}, the operator $\phi_nI$ is compact for each $n \in \mathbb{N}$, and hence $a I$ is compact, as required. \\
\end{proof}

%% file: KCL_Thesis_Chapter5_v5.tex
\chapter{Operator algebra - Part II} \label{OpAlgLp}
\section{Main result}
From Section \ref{ProblemStatement}, our (initial) problem is to investigate the solvability of the equation
\begin{equation*}
\mathcal{A} u = f,
\end{equation*}
where $u \in H^s_p(\overline{\mathbb{R}_+)}$, for a given $f \in H^{s-2\alpha}_p(\overline{\mathbb{R}_+)}$, under the assumptions $1 < p < \infty$ and lower regularity,  namely $1/p < s < 1+1/p$. Moreover, see Remark \ref{lowregnobigalpha}, we further assume that $0 < \alpha < \tfrac{1}{2}$. \\

In Lemma \ref{lemma:us} we defined
\begin{equation*}
u_s:= (D + i)^{s-1} \, e_+ (D - i) u,
\end{equation*}
so that $u_s \in L_p(\mathbb{R})$ with $\text{supp } u_s \subseteq \overline{\mathbb{R}_+}$. \\

\begin{theorem} \label{TheoremaWMplusT}
Suppose $1 < p < \infty, \, 1/p < s < 1+1/p$ and $0 < \alpha < \tfrac{1}{2}$. Then we can recast the equation $\mathcal{A} u = f$ in the form
\begin{equation*}
\big ( W(c_1)  + a_2M^0(b_2)W(c_2) +T \big ) (r_+  u_s) = g,
\end{equation*}
where
\begin{align} 
g:&= r_+ (D-i)^{s-2\alpha} l_+ f\quad (\in L_p(\mathbb{R}_+)) ; \nonumber \\
{c}_1(\xi) &= (1+\xi^2)^\alpha (\xi-i)^{s- 2\alpha-1}(\xi+i)^{1-s} ; \nonumber  \\
{a}_2(x) &=  -iC_\alpha \, \psi(\alpha+1, 2\alpha+1, x); \nonumber \\
{b}_2(\xi) &= B(s- 2\alpha + 1- 1/p + i \xi, 2 \alpha) / \Gamma(2\alpha) ; \nonumber \\
{c}_2(\xi) &=(-i \xi)^{2\alpha} (\xi-i)^{s- 2\alpha-1}(\xi+i)^{1-s},\nonumber  
\end{align} 
and the operator $T$, acting on $L_p(\mathbb{R}_+)$, is compact. \\

The constant $C_\alpha$ is given in equation \eqref{defCalpha}, and the smooth compactly supported function $\psi$ is discussed in Lemma \ref{eUab2x}.\\
\end{theorem}

\section{Introduction}
We have seen in Section \ref{OpAlgInit} that equation \eqref{rAef} can be written as (see equation \eqref{aMCtilde} with $N=2$)
\begin{equation} \label{aMCtilde3}
\tilde{a}_0(x) u(0) + \sum^2_{j=1}  \tilde{a}_j(x) \, M^0(\tilde{b}_j) \, (r_+ \tilde{C}_j e_+)(r_+ u_s) + \tilde{K}u = f,
\end{equation}
where the given function $f \in H^{s-2\alpha}_p(\overline{\mathbb{R}_+})$ and the operator $\tilde{K}:H^s_p(\overline{\mathbb{R}_+}) \to H^{s-2\alpha}_p(\overline{\mathbb{R}_+})$ is compact. \\

In this section, we present a formulation in $L_p(\mathbb{R}_+)$ of the form
\begin{equation} \label{aMW}
\big ( {a}_1 (x) \, M^0({b}_1) \, W(c_1) + {a}_2 (x) \, M^0({b}_2) \, W(c_2) + T \big ) \, (r_+ u_s) = g,
\end{equation}
where the operator $T$, acting on $L_p(\mathbb{R}_+)$, is compact. The function $g \in L_p(\mathbb{R}_+)$ is defined by
\begin{equation} \label{defng}
g:= r_+ (D-i)^{s-2\alpha} l_+ f,
\end{equation}
where by Lemma \ref{lemma:rLambdae}, $g$ does not depend on the choice of the extension $l_+$. \\

The subsequent analysis will show that, after the application of the operator $r_+ (D-i)^{s-2\alpha} l_+$, some of the terms in equation \eqref{aMCtilde3} represent compact operators on $L_p(\mathbb{R}_+)$. \\

We now consider the action of the operator $r_+ (D-i)^{s-2\alpha} l_+$ on the individual summands on the left-hand side of equation \eqref{aMCtilde3} in turn. \\

\begin{remark}
In this chapter we will make repeated use of Proposition 5.3.4, p. 267, \cite{Roch}, concerning the compactness on $L_p(\mathbb{R}_+)$ of the operator $M^0(b) \, W(c)$ and the commutator $[M^0(b), W(c)]$. In all cases where we use this result the symbols $b$ and $c$ will be continuous on $\mathbb{R}$ and have bounded variation, thus ensuring the applicability of Proposition 5.3.4, ibid. For more details, see Lemma \ref{BVBeta} and \ref{BVm}. \\
\end{remark}

\section{Term by term analysis}
\subsection{First term}
For the first term, from equation \eqref{tildea0}, we have 
\begin{align*} 
\tilde{a}_0(x) &= \tilde{a}_2(x) M^0(\tilde{b}_2)  \dfrac{i}{\sqrt{2\pi}} \, r_+ \mathcal{F}^{-1} (-i \xi)^{2\alpha-1}(\xi-i)^{-1} \\
&= \tilde{a}_2(x) M^0(\tilde{b}_2) \, r_+ h_1(x), 
\end{align*}
where
\begin{equation*}
h_1(x) :=  \dfrac{i}{\sqrt{2\pi}}  \mathcal{F}^{-1} (-i \xi)^{2\alpha-1}(\xi-i)^{-1}. 
\end{equation*} 

Note that, from equation \eqref{tildea2},
\begin{align*} 
\tilde{a}_2(x) &= -iC_\alpha \,\psi(\alpha+1, 2\alpha+1, x) ;\nonumber \\
\tilde{b}_2(\xi) &= B(1/p' + i \xi, 2 \alpha) / \Gamma(2\alpha).
\end{align*}
Our goal is to show that
\begin{equation*}
\Lambda^{s-2\alpha}_- \tilde{a}_0(x) = \Lambda^{s-2\alpha}_- \, \tilde{a}_2(x) M^0(\tilde{b}_2) r_+ h_1(x) \in L_p(\mathbb{R}_+),
\end{equation*}
because then the operator
\begin{equation*}
r_+ u_s(x) \longmapsto \Lambda^{s-2\alpha}_- \tilde{a}_0(x) u(0)
\end{equation*}
is bounded on $L_p(\mathbb{R}_+)$ and has rank one,  and is therefore compact.\\

We note that $\tilde{a}_2 \in r_+ C^\infty_0(\mathbb{R})$ and $\tilde{a}_2(x) =0$ for $x \geq 2$. Let $\chi \in C^\infty_0(\mathbb{R})$ be such that
\begin{align*}
&\chi(t) :=
\begin{cases} 
	1 &\mbox{if }  |t| \leq 2\\
	0 & \mbox{if } |t| > 3. \\
\end{cases}
\end{align*} 

Then, see Lemmas \ref{lemma:mellinop2} and \ref{M0bchi},
\begin{equation*}
\Lambda^{s-2\alpha}_- \tilde{a}_2 M^0(\tilde{b}_2) (r_+ h_1) = \Lambda^{s-2\alpha}_- \tilde{a}_2(x) M^0(\tilde{b}_2) (r_+ \chi  h_1).
\end{equation*} 

Since $h_1$ is the inverse Fourier transform of an integrable function it is continuous and vanishes at infinity. Hence $r_+ \chi h_1 \in L_p(\mathbb{R}_+)$. Thus, if $s-2\alpha <0$, using Lemma \ref{MellinOpBdd}, the required result follows immediately. \\

It remains to consider the case $s-2\alpha \geq 0$.  Set $\mu = 1 - 2 \alpha, \,\, r=s-2\alpha$, so that $0 < \mu < 1, \,\,  r < 2 - 2\alpha = 1 + \mu$. Then, from Lemma \ref{Hrcondition}, for any $\chi_1 \in C^\infty_0(\mathbb{R})$,  
\begin{equation*}
\chi_1 (D - i)^{s-2\alpha} h_1  \in L_p(\mathbb{R}),
\end{equation*}
subject \textit{only} to the condition
\begin{equation} \label{Condpsalpha}
 s < 1 + \dfrac{1}{p}.
\end{equation}

Hence, see Lemma \ref{lemmapushchi}, $(D - i)^{s-2\alpha} \chi h_1 \in L_p(\mathbb{R})$, and so, after applying the operator $r_+ (D-i)^{2\alpha-s}$, we have
\begin{equation*}
r_+ \chi h_1 \in H^{s-2\alpha}_p(\overline{\mathbb{R}_+}).
\end{equation*}
Therefore, as $s - 2\alpha \geq 0$, again from Lemma \ref{MellinOpBdd},
\begin{equation*}
M^0(\tilde{b}_2) \, r_+ \chi h_1 \in H^{s-2\alpha}_p(\overline{\mathbb{R}_+}),
\end{equation*}
and hence, 
\begin{equation*}
\Lambda^{s-2\alpha}_- \, \tilde{a}_2(x)  \,M^0(\tilde{b}_2) r_+ h_1(x) \in L_p(\mathbb{R}_+), 
\end{equation*}
as required, since $\tilde{a}_2   \in r_+C^\infty_0(\mathbb{R})$. \\

\subsection{Second term}
Using \eqref{tildea1}, we have
\begin{equation*}
\tilde{a}_1(x) \, M^0(\tilde{b}_1) \, (r_+ \tilde{C}_1 e_+) = r_+ \tilde{C}_1 e_+
\end{equation*}
where the pseudodifferential operator $ \tilde{C}_1$ has symbol $(1+\xi^2)^\alpha (\xi-i)^{-1}(\xi+i)^{1-s}$. Hence, by Lemma \ref{lemma:rLambdae}
\begin{equation*}
r_+ (D-i)^{s-2\alpha} l_+ \, r_+ \tilde{C}_1 e_+ = r_+ (D-i)^{s-2\alpha} \tilde{C}_1 e_+.
\end{equation*}
Now $(D-i)^{s-2\alpha} \tilde{C}_1$ has symbol $(1+\xi^2)^\alpha (\xi-i)^{s- 2\alpha-1}(\xi+i)^{1-s}$, which is clearly a Fourier $L_p$ multiplier. (See Lemma \ref{mxiFm}.) Therefore, in the notation of equation \eqref{aMW}, 
\begin{align} \label{a1}
{a}_1(x) &= 1;  \nonumber \\
{b}_1(\xi) &= 1;  \\
{c}_1(\xi) &= (1+\xi^2)^\alpha (\xi-i)^{s- 2\alpha-1}(\xi+i)^{1-s}. \nonumber  
\end{align} 

\subsection{Third term}
Now from \eqref{tildea2} we have
\begin{align*}
\tilde{a}_2(x) &= -i C_\alpha \, \psi(\alpha+1, 2\alpha+1, x);\nonumber \\
\tilde{b}_2(\xi) &= B(1/p' + i \xi, 2 \alpha) / \Gamma(2\alpha) ;\\
\tilde{c}_2(\xi) &= (-i \xi)^{2\alpha} (\xi + i)^{1-s} (\xi-i)^{-1}.\nonumber  
\end{align*}

Let us define
\begin{equation*}
r:= s - 2\alpha,
\end{equation*}
so that we need to consider
\begin{equation*}
- 1 + 1/p < r < 1+1/p,
\end{equation*}
since $0 < \alpha < \tfrac{1}{2}$ and $1/p < s < 1 +1/p$. We note that the pseudodifferential operator $\tilde{C}_2$ has order $-r$. \\

If $r \geq 0$, then from Lemma \ref{MellinOpBdd}, the operator $M^0(\tilde{b}_2) \, (r_+ \tilde{C}_2 e_+) : L_p(\mathbb{R}_+) \to H^r_p(\overline{\mathbb{R}_+})$ is bounded. \\

On the other hand, if $-1 +1/p < r < 0$, we can write
\begin{equation*}
\tilde{c}_2(\xi) =  (i \xi)^{-r} \cdot \tilde{c}_0(\xi) 
\end{equation*}
where
\begin{equation*}
\tilde{c}_0(\xi) :=  (i \xi)^r (-i \xi)^{2 \alpha} (\xi + i)^{1-s} (\xi-i)^{-1}.
\end{equation*}
Since $r + 2\alpha = s > 0, \,\, \tilde{c}_0(0)=0$.  Moreover, as $r=s-2\alpha$, the operator $\tilde{C}_0$, with symbol $\tilde{c}_0$, has order $0$. \\

From Lemma \ref{AhomoMellin}, $M_{2\alpha, 0} r_+ (iD)^{-r} l_+ = r_+ (iD)^{-r}l_+ M_{2\alpha,r}$, and from Lemma \ref{lemma:rHomore}, $r_+ (iD)^{-r} l_+: L_p(\mathbb{R}_+) \to H^r_p(\overline{\mathbb{R}_+})$ is bounded. Therefore, the operator 
\begin{align*}
M^0(\tilde{b}_2) \, (r_+ \tilde{C}_2 e_+) & = M_{2\alpha, 0} \, (r_+ \tilde{C}_2 e_+) \quad \text{(see Lemma } \ref{lemma:mellinop2} \text{)} \\
& = M_{2\alpha, 0} \, (r_+ (iD)^{-r} l_+ r_+ \tilde{C}_0 e_+) \quad \text{(see Lemma } \ref{lemma:rHomore} \text{)} \\
& = r_+ (iD)^{-r}l_+ M_{2\alpha,r} (r_+ \tilde{C}_0 e_+)
\end{align*}
is bounded from $L_p(\mathbb{R}_+) \to H^r_p(\overline{\mathbb{R}_+})$. \\

So now, using Lemma \ref{multxgammaa}, each of the three operators in the identity
\begin{align*}
& \Lambda^r_- \, \tilde{a}_2(x) \, M^0(\tilde{b}_2) \, (r_+ \tilde{C}_2 e_+) \\
&= [\Lambda^r_- , \tilde{a}_2(x)] \, M^0(\tilde{b}_2) \, (r_+ \tilde{C}_2 e_+) +\tilde{a}_2(x) \, \Lambda^r_- \, M^0(\tilde{b}_2) \, (r_+ \tilde{C}_2 e_+),
\end{align*}
is bounded on $L_p(\mathbb{R}_+)$. \\

Moreover, the compactness of the operator involving the commutator term follows directly from Lemma \ref{CommutatorLambphi}. Thus, it remains to consider $\tilde{a}_2(x) \, \Lambda^r_- \, M^0(\tilde{b}_2) \, (r_+ \tilde{C}_2 e_+)$. \\

Firstly, suppose that $0 < r < 1$. Then, using Lemma \ref{LambdaMreverse},
\begin{align*}
\Lambda^r_-  & \, M^0(\tilde{b}_2) \, (r_+ \tilde{C}_2 e_+) \\
& = \Lambda^r_- \, M_{2\alpha,0} \, (r_+ \tilde{C}_2 e_+) \\
& = \big ( M_{2\alpha,r} \,  \Lambda^r_- +(-i)^r (M_{2\alpha,0} - M_{2\alpha,r}) +T \big ) \, (r_+ \tilde{C}_2 e_+) \\
& = M_{2\alpha,r} \, (r_+ C_2 e_+) + (-i)^r (M_{2\alpha,0} - M_{2\alpha,r}) \, (r_+ \tilde{C}_2 e_+)
+T \, (r_+ \tilde{C}_2 e_+).
\end{align*}

From Lemma \ref{LambdaMreverse}, $T: H^r_p(\overline{\mathbb{R}_+}) \to L_p(\mathbb{R}_+)$ is compact. Moreover, the pseudodifferential operator $\tilde{C}_2$ has order $-r$, and hence $T  \, (r_+ \tilde{C}_2 e_+)$ is compact on $L_p(\mathbb{R}_+)$. \\

By Remark \ref{remark:mellinop}, the symbols of both $M_{2\alpha,0}$ and $M_{2\alpha,r}$ take the value zero at $\pm \infty$. Hence, $(M_{2\alpha,0} - M_{2\alpha,r}) \, (r_+ \tilde{C}_2 e_+)$ is compact on $L_p(\mathbb{R}_+)$, from Proposition 5.3.4 (i), p. 267, \cite{Roch}. \\

So, in summary, if $0 < r < 1$ then
\begin{equation*}
\Lambda^r_- \, M^0(\tilde{b}_2) \, (r_+ \tilde{C}_2 e_+) = M_{2\alpha,r} \, (r_+ C_2 e_+) + K_1,
\end{equation*}
where $C_2$ has symbol
\begin{equation*}
{c}_2(\xi) = (-i \xi)^{2\alpha}  (\xi-i)^{s - 2 \alpha -1} (\xi + i)^{1-s},
\end{equation*}
and the operator $K_1$, acting on $L_p(\mathbb{R}_+)$, is compact. \\

Similarly, in the case that $-1 + 1/p < r < 0$, then we can again apply Lemma \ref{LambdaMreverse}, noting that the operator $\Lambda^{2 r}_- \, r_+ \tilde{C}_2 e_+$ has order $r < 0$. 

\begin{align*}
\Lambda^r_-  & \, M^0(\tilde{b}_2) \, (r_+ \tilde{C}_2 e_+) \\
& = \Lambda^r_- \, M_{2\alpha,0} \, (r_+ \tilde{C}_2 e_+) \\
& = \big ( M_{2\alpha,r} \,  \Lambda^r_- -(-i)^r (M_{2\alpha,0} - M_{2\alpha,r}) \, \Lambda^{2r}_-+T \big ) \, (r_+ \tilde{C}_2 e_+) \\
& = M_{2\alpha,r} \, (r_+ C_2 e_+) - (-i)^r (M_{2\alpha,0} - M_{2\alpha,r}) \,  \Lambda^{2r}_-\, (r_+ \tilde{C}_2 e_+)
+T \, (r_+ \tilde{C}_2 e_+).
\end{align*}

The compactness of $T  \, (r_+ \tilde{C}_2 e_+)$ on $L_p(\mathbb{R}_+)$ follows exactly as in the case $0 < r < 1$. Moreover, 
\begin{equation*}
(M_{2 \alpha,0} - M_{2 \alpha, r}) \Lambda^{2 r}_- \, r_+ \tilde{C}_2 e_+
\end{equation*}
is compact on $L_p(\mathbb{R}_+)$, from Proposition 5.3.4 (i), p. 267, \cite{Roch}.\\

So, in summary, if $-1 + 1/p < r < 0$ then
\begin{equation*}
\Lambda^r_- \, M^0(\tilde{b}_2) \, (r_+ \tilde{C}_2 e_+) = M_{2\alpha,r} \, (r_+ C_2 e_+) + K_2
\end{equation*}
where the operator $K_2$, acting on $L_p(\mathbb{R}_+)$, is compact. \\

The case $1 < r < 1 + 1/p$ follows similarly, except that we now apply Lemma \ref{Mlambdaextension}. In particular, we note that the operator $S_1 \, r_+ \tilde{C}_2 e_+$ has order $1 - r < 0$. Hence, as in the case $0 < r < 1$ discussed above, 
\begin{equation*}
S_1 \, r_+ \tilde{C}_2 e_+
\end{equation*}
is compact on $L_p(\mathbb{R}_+)$, from Proposition 5.3.4 (i), p. 267, \cite{Roch}. \\

Finally, the case $r=1$ follows in the same way, and for the case $r=0$, there is nothing to prove.\\

Hence, using Lemma \ref{lemma:mellinop2}, in the notation of equation \eqref{aMW} we have, 
\begin{align} \label{a2}
{a}_2(x) &=-i C_\alpha \, \psi(\alpha+1, 2\alpha+1, x);\nonumber \\
{b}_2(\xi) &= B(s- 2\alpha +1/p' + i \xi, 2 \alpha) / \Gamma(2\alpha) ;\\
{c}_2(\xi) &= (-i \xi)^{2\alpha}  (\xi-i)^{s - 2 \alpha -1} (\xi + i)^{1-s}.\nonumber  
\end{align} 
Note that a routine application of Lemma \ref{mxiFm} confirms that $c_2$ is a Fourier $L_p$ multiplier.

\subsection{Summary} \label{OpAlg2Summary}
Our base assumptions are that 

\begin{equation} \label{Finalcondalphaps}
1 < p < \infty, \,\, 1/p < s <  1+ 1/p \,\, \text{  and  } \,\, 0 < \alpha < \tfrac{1}{2}. \\
\end{equation}

So, finally, subject to condition \eqref{Finalcondalphaps},  the formulation given by equation \eqref{aMW} becomes 
\begin{equation} \label{WamW}
\big ( W(c_1) + a_2 M^0(b_2) W(c_2) + T \big ) (r_+ u_s) = g,
\end{equation}
where the operator $T$, acting on $L_p(\mathbb{R}_+)$, is compact and
\begin{align} \label{Finalabc}
g:&= r_+ (D-i)^{s-2\alpha} l_+ f; \nonumber \\
{c}_1(\xi) &= (1+\xi^2)^\alpha (\xi-i)^{s- 2\alpha-1}(\xi+i)^{1-s}. \nonumber  \\
{a}_2(x) &= -i C_\alpha \, \psi(\alpha+1, 2\alpha+1, x) \quad \text{(see Lemma } \ref{eUab2x}); \nonumber \\
{b}_2(\xi) &= B(s- 2\alpha +1/p' + i \xi, 2 \alpha) / \Gamma(2\alpha) ; \\
{c}_2(\xi) &= (-i \xi)^{2\alpha} (\xi-i)^{s- 2\alpha-1}(\xi+i)^{1-s},\nonumber  
\end{align} 
and the constant $C_\alpha$ is given by
\begin{equation*} 
C_{\alpha} = -i \, \dfrac{\alpha \, 2^{2\alpha}}{\Gamma (1- \alpha)}.
\end{equation*}

\section{Mellin operator boundedness}
We have noted previously, see Lemma \ref{lemma:mellinop2}, that the Mellin integral operator $M_{\gamma, \rho}$ with kernel 
\begin{equation*}
K_{\gamma, \rho}(t) = \dfrac{\chi_{[1, \infty)}(t)}{t^\rho \, \Gamma( \gamma) \, t^\gamma \, (t-1)^{1-\gamma}}.
\end{equation*}
is bounded on $L_p(\mathbb{R}_+)$. 

\begin{lemma} \label{MellinOpBdd}
Suppose $1 < p < \infty, \, \rho > 1/p-1$ and $\gamma >0$. If $t \geq 0$,  then
\begin{equation*}
M_{\gamma, \rho}: H^t_p(\overline{\mathbb{R}_+}) \to H^t_p(\overline{\mathbb{R}_+})
\end{equation*}
is bounded. \\
\end{lemma}
\begin{proof}
The special case $t=0$ follows directly from Lemma \ref{lemma:mellinop2}. \\

Now suppose that $t = m \in \mathbb{N}$.  Let $\varphi \in H^m_p(\overline{\mathbb{R}_+})$. Then, using an equivalent norm on $H^m_p(\overline{\mathbb{R}_+})$,
\begin{align*}
\| M_{\gamma, \rho} \varphi \|_{m,p} & \leq \text{const } \sum^m_{k=0} \| (M_{\gamma, \rho} \varphi)^{(k)}\|_p \\
& = \text{const } \sum^m_{k=0} \| M_{\gamma, \rho + k} \varphi^{(k)}\|_p  \quad \text{by Lemma \ref{DMellin}}\\
& \leq  \sum^m_{k=0}  c_{k} \| M_{\gamma, \rho + k} \| \|\varphi^{(k)}\|_p \quad \text{by Lemma \ref{lemma:mellinop2}}\\
& \leq C_{\gamma, \rho, m,p} \sum^m_{k=0} \| \varphi^{(k)}\|_p \quad \text{for some positive constant  } C_{\gamma, \rho, m,p} \\
& = C_{\gamma, \rho, m,p}  \|  \varphi\|_{m,p}. 
\end{align*}
In other words, the operator $M_{\gamma, \rho}: H^m_p(\overline{\mathbb{R}_+}) \to H^m_p(\overline{\mathbb{R}_+})$ is bounded for any $m \in \mathbb{N}$. \\

Let $t > 0$. Choose any $m \in \mathbb{N}$ such that $m >  t$. Then we have boundedness on $H^t_p(\overline{\mathbb{R}_+})$ by interpolation between $H^m_p(\overline{\mathbb{R}_+})$ and $H^0_p(\overline{\mathbb{R}_+}) = L_p(\mathbb{R}_+)$. \\
\end{proof}

\section{Multiplication operator commutator}
Suppose $1< p < \infty$ and $\sigma, \nu \in \mathbb{R}$. Then, see Lemma \ref{lemma:rLambdae},
\begin{equation*}
\Lambda^{\nu}_- := r_+ (D - i )^{\nu} l_+
\end{equation*}
is bounded from $H^\sigma_p(\overline{\mathbb{R}_+})$ to $H^{\sigma-\nu}_p(\overline{\mathbb{R}_+})$, and does not depend on the choice of extension $l_+$. \\

\begin{lemma} \label{CommutatorLambphi}
Let $1< p < \infty$ and $\phi \in C^\infty_0(\mathbb{R})$. Then
\begin{equation*}
[\Lambda^r_-, \phi I] : H^s_p(\overline{\mathbb{R}_+}) \to H^{s-r}_p(\overline{\mathbb{R}_+})
\end{equation*}
is compact for all $r,s \in \mathbb{R}$.
\end{lemma}
\begin{proof}
Since $\phi \in C^\infty_0(\mathbb{R})$, the commutator $[\Lambda^r_-, \phi I]$ is a pseudodifferential operator of order $r-1$, and thus
\begin{equation*}
[\Lambda^r_-, \phi I] : H^s_p(\overline{\mathbb{R}_+}) \to H^{s-r+1}_p(\overline{\mathbb{R}_+}).
\end{equation*}
From Lemma \ref{multiplierbcompact},
\begin{equation*}
\phi I: H^t_p(\overline{\mathbb{R}_+}) \to H^{t-\epsilon}_p(\overline{\mathbb{R}_+})
\end{equation*}
is compact for $- \infty < t < +\infty$ and all $\epsilon > 0$. Therefore, with $t=s$, 
\begin{equation*}
\Lambda^r_- \phi I: H^s_p(\overline{\mathbb{R}_+}) \to H^{s-r- \epsilon}_p(\overline{\mathbb{R}_+})
\end{equation*}
and then taking $t=s-r$,
\begin{equation*}
\phi \Lambda^r_- : H^s_p(\overline{\mathbb{R}_+}) \to H^{s-r-\epsilon}_p(\overline{\mathbb{R}_+})
\end{equation*}
are both compact. \\

In summary, 
\begin{align*}
[\Lambda^r_-, \phi I] & : H^s_p(\overline{\mathbb{R}_+}) \to H^{s-r+1}_p(\overline{\mathbb{R}_+}) \quad \text{is bounded, and} \\
& : H^s_p(\overline{\mathbb{R}_+}) \to H^{s-r- \epsilon}_p(\overline{\mathbb{R}_+}) \quad \text{is compact.} \\
\end{align*}
Therefore, by interpolation, see \cite{Cw},
\begin{equation*}
[\Lambda^r_-, \phi I] : H^s_p(\overline{\mathbb{R}_+}) \to H^{s-r}_p(\overline{\mathbb{R}_+})
\end{equation*}
is compact for all $r,s \in \mathbb{R}$.
\end{proof}

\newpage


\section{Pseudodifferential and Mellin operators}
\begin{remark}
Lemma \ref{LambdaMreverse} and \ref{Mlambdaextension} describe the action of the operator $\Lambda^r_-$ on the Mellin integral operator $M_{2 \alpha, 0}$, since this is sufficient for our purposes. However, it is clear that these results also hold for a wider class of Mellin operators. Indeed, we can replace $M_{2 \alpha, 0}$ by a general Mellin convolution operator, with symbol $b$, such that $b(\pm \infty)=0$, and whose kernel, $K$, satisfies the two conditions in \eqref{Ksuppinteg}. \\
\end{remark}

\begin{lemma} \label{LambdaMreverse}
Suppose $0 < r < 1$. Then
\begin{equation*}
\Lambda^r_-  \, M_{2 \alpha, 0} = M_{2 \alpha, r} \, \Lambda^r_- + (-i)^r (M_{2\alpha,0} - M_{2\alpha,r}) + T, 
\end{equation*}
where $T: H^r_p(\overline{\mathbb{R}_+}) \to L_p(\mathbb{R}_+)$ is compact. \\

On the other hand, if $-1+1/p < r < 0$ then
\begin{equation*}
\Lambda^r_-  \, M_{2 \alpha, 0} = M_{2 \alpha, r} \, \Lambda^r_- - (-i)^r  (M_{2\alpha,0} - M_{2\alpha,r}) \Lambda^{2r}_- + T, 
\end{equation*}
where $T: H^r_p(\overline{\mathbb{R}_+}) \to L_p(\mathbb{R}_+)$ is compact. \\
\end{lemma}

\begin{proof}
\textbf{We first consider the case where $r >0$.} \\

Suppose $0< r < 1$. Then, from Lemma \ref{fxi0bddlim}
\begin{equation} \label{HomogRepresentation}
(\xi - i )^r = (-i)^r (i \xi)^r + (-i)^r + c_r(\xi),
\end{equation}
where $c_r$ is bounded for $\xi \in \mathbb{R}$ and $c_r(0)=0, c_r(\pm \infty) = -(-i)^r$. \\

Moreover, from Lemma \ref{AhomoMellin}, for $\nu >0 $
\begin{equation} \label{iDMcommute}
r_+ (iD)^\nu l_+ \, M_{2\alpha, 0} = M_{2\alpha, \nu} \, r_+ (iD)^\nu l_+. 
\end{equation}

Hence 
\begin{align*}
\Lambda^r_- \, M_{2 \alpha, 0} &=  r_+ \big \{ (-i)^r (i D)^r + (-i)^r + c_r(D) \big \} l_+ M_{2 \alpha, 0} \\
&=  M_{2 \alpha, r} (-i)^r r_+ (i D)^r l_+ + M_{2 \alpha, 0} (-i)^r + r_+ c_r(D) l_+ M_{2 \alpha, 0}  \\
&= M_{2 \alpha, r} \, \Lambda^r_- +(-i)^r \big \{ M_{2\alpha, 0} - M_{2\alpha, r}\big\} \\
& \quad + \big \{r_+ c_r(D) l_+ M_{2 \alpha, 0} - M_{2 \alpha, r} r_+ c_r(D) l_+\big \}.
\end{align*}

Now, for ease of notation, define
\begin{equation*}
C_r := r_+ c_r(D)  l_+.
\end{equation*} 
Hence, we can write
\begin{equation} \label{LamM-MLam}
\Lambda^r_- \, M_{2 \alpha, 0} - M_{2 \alpha, r} \, \Lambda^r_-  - (-i)^r (M_{2\alpha,0} - M_{2\alpha,r})=  C_r M_{2\alpha, 0} - M_{2\alpha, r}C_r ,
\end{equation}
where $C_r$ has order $0$. For the case $0 < r < 1$, it now remains to prove the compactness of the two operators on the right-hand side of equation \eqref{LamM-MLam}.\\

From Proposition 5.3.4(i), p. 267, \cite{Roch},
\begin{equation*}
\Lambda^{-r}_-C_r M_{2\alpha, 0} : L_p({\mathbb{R}_+}) \to L_p({\mathbb{R}_+})
\end{equation*}
 is compact. Therefore,
\begin{equation*}
C_r M_{2\alpha, 0} : L_p({\mathbb{R}_+}) \to H^{-r}_p(\overline{\mathbb{R}_+})
\end{equation*}
is compact. But 
\begin{equation*}
C_r M_{2\alpha, 0} : H^t_p(\overline{\mathbb{R}_+}) \to H^{t-r}_p(\overline{\mathbb{R}_+})
\end{equation*}
is bounded for all $t \geq 0$. So, taking $t > r$, we obtain by interpolation that
\begin{equation*}
C_r M_{2\alpha, 0} : H^r_p(\overline{\mathbb{R}_+}) \to L_p(\overline{\mathbb{R}_+})
\end{equation*}
is compact. \\ 

Similarly, from Proposition 5.3.4(i), p. 267, \cite{Roch}, $M_{2\alpha, r} C_r  \Lambda^{-r}_-: L_p({\mathbb{R}_+}) \to L_p({\mathbb{R}_+})$ is compact. Therefore,
\begin{equation*}
 M_{2\alpha, r} C_r: H^{r}_p(\overline{\mathbb{R}_+}) \to L_p(\overline{\mathbb{R}_+})
\end{equation*}
is compact.  This completes the proof for the case $0 < r < 1$. \\ \\

\textbf{Now suppose that $r < 0$.} \\

Suppose $-1 + 1/p < r < 0$ and write $r = -s$ where $s > 0$. Our starting point is equation \eqref{LamM-MLam}, where for the pseudodifferential terms we replace r by s, and for the Mellin operators we replace the pair $M_{2\alpha, 0}$ and $M_{2\alpha, r}$ by $M_{2\alpha, r}$ and $M_{2\alpha, 0}$ respectively. Hence
\begin{equation} \label{LamM-MLamNeg}
\Lambda^s_- \, M_{2 \alpha, r} - M_{2 \alpha, 0} \, \Lambda^s_-  - (-i)^s (M_{2\alpha,r} - M_{2\alpha,0})=  C_s M_{2\alpha, r} - M_{2\alpha, 0}C_s.
\end{equation}

From Proposition 5.3.4(i), p. 267, \cite{Roch},
\begin{equation*}
\Lambda^{r}_-C_s M_{2\alpha, r} : L_p({\mathbb{R}_+}) \to L_p({\mathbb{R}_+})
\end{equation*}
 is compact. Therefore,
\begin{equation*}
C_s M_{2\alpha, r} : L_p({\mathbb{R}_+}) \to H^{r}_p(\overline{\mathbb{R}_+})
\end{equation*}
is compact. \\

Similarly, 
\begin{equation*}
C_s M_{2\alpha, 0} : L_p({\mathbb{R}_+}) \to H^{r}_p(\overline{\mathbb{R}_+})
\end{equation*}
is compact. But, in addition, from Proposition 5.3.4 (ii)1, p. 267, \cite{Roch},
\begin{equation*}
[M_{2\alpha,0},C_s] : L_p({\mathbb{R}_+}) \to L_p({\mathbb{R}_+})
\end{equation*}
is compact and so, finally,
\begin{equation*}
M_{2\alpha, 0} C_s  : L_p({\mathbb{R}_+}) \to H^{r}_p(\overline{\mathbb{R}_+})
\end{equation*}
is compact. In summary,
\begin{equation*}
T_s:= C_s M_{2\alpha,r} - M_{2\alpha,0}C_s \quad : L_p({\mathbb{R}_+}) \to H^{r}_p(\overline{\mathbb{R}_+}) 
\end{equation*}
is compact.
From equation \eqref{LamM-MLamNeg}
\begin{equation*}
\Lambda^r_- \Lambda^s_-  \, M_{2 \alpha, r} \Lambda^r_-= \Lambda^r_- M_{2 \alpha, 0} \, \Lambda^s_-\Lambda^r_-  + \Lambda^r_- (-i)^r ( M_{2\alpha,0} - M_{2\alpha,r} ) \Lambda^r_- + \Lambda^r_- T_s \Lambda^r_-. \\
\end{equation*}
But since $r + s = 0$, we have 
\begin{equation*}
M_{2 \alpha, r} \Lambda^r_-  = \Lambda^r_- M_{2 \alpha, 0}    + \Lambda^r_- (-i)^r ( M_{2\alpha,0} - M_{2\alpha,r} ) \Lambda^r_- + \Lambda^r_- T_s \Lambda^r_-.
\end{equation*}
Rearranging
\begin{equation*}
\Lambda^r_- M_{2 \alpha, 0} = M_{2 \alpha, r} \Lambda^r_-  - \Lambda^r_- (-i)^r ( M_{2\alpha,0} - M_{2\alpha,r} ) \Lambda^r_-+ T_1,
\end{equation*}
where $T_1 = - \Lambda^r_- T_s \Lambda^r_-$. Since $T_s: L_p(\mathbb{R}_+) \to H^{r}_p(\overline{\mathbb{R}_+})$ is compact, it follows that $T_1: H^r_p(\overline{\mathbb{R}_+}) \to L_p(\mathbb{R}_+)$ is compact. \\

Finally, we note that
\begin{equation*}
 \Lambda^r_- ( M_{2\alpha,0} - M_{2\alpha,r} ) \Lambda^r_- = [ \Lambda^r_-, ( M_{2\alpha,0} - M_{2\alpha,r} )] \Lambda^r_- + ( M_{2\alpha,0} - M_{2\alpha,r} )\Lambda^{2r}_-. \\
\end{equation*}
 
But, since $r < 0$, the commutator $[ \Lambda^r_-, ( M_{2\alpha,0} - M_{2\alpha,r} )]$ is compact on $L_p(\mathbb{R}_+)$ by Proposition 5.3.4 (ii)1, p. 267, \cite{Roch}. Hence $[ \Lambda^r_-, ( M_{2\alpha,0} - M_{2\alpha,r} )] \Lambda^r_-:H^r_p(\overline{\mathbb{R}_+}) \to L_p(\mathbb{R}_+)$ is compact. \\

This completes the proof of the lemma. \\
\end{proof}

\begin{remark}
In passing, we note that there is a minor inaccuracy in the proof of Proposition 5.3.4 (i) p. 267, \cite{Roch}. \\

The sum in the display formula 9 lines below (5.7) might not, in fact, be identically zero. However, by Proposition 4.2.10, p. 204, \cite{Roch}, it can be made arbitrarily small, so that for any $\epsilon >0$, we can choose $f$ such that 
\begin{equation*}
\| W(a) M^0(f_b) - W(a) M^0(f) \| < \epsilon / 2.
\end{equation*}
On the other hand, $W(a) M^0(f)$ can be represented as the sum of a compact operator and an operator of norm $\epsilon /2$. Therefore, $W(a)M^0(f_b)$ can be represented as the sum of a compact operator and an operator of arbitrarily small norm. That is, $W(a) M^0(f_b)$ is also compact. \footnote{This correction to the proof was confirmed in a personal communication from Prof. Steffen Roch on 1st November 2016.} \\
\end{remark}

We note that for $0 < \alpha < 1$ and any $r>0$ the Mellin integral operator $M_{2 \alpha, r}$ has a symbol that vanishes at $\pm \infty$. For a fixed $\alpha$, it will be convenient to define a \textit{composite} Mellin operator to be any linear combination of operators $M_{2 \alpha, r}$, as $r >0$ varies. Clearly, by its construction, any composite Mellin operator also has a symbol that vanishes at $\pm \infty$. \\

\begin{lemma} \label{Mlambdaextension}
Suppose $r > 0$. If
\begin{enumerate}[\hspace{0.5cm}(a)]
\item $r = k \in \mathbb{N}$, then $\Lambda^k_-  \, M_{2 \alpha, 0} = M_{2 \alpha, k} \, \Lambda^k_- + S_{k-1}$;
\item $r \not \in \mathbb{N}$, then $\Lambda^r_-  \, M_{2 \alpha, 0} = M_{2 \alpha, r} \, \Lambda^r_- + S_{[r]} + T$,
\end{enumerate}
where $T:H^r_p(\overline{\mathbb{R}_+}) \to L_p(\mathbb{R}_+)$ is compact, and
\begin{equation*}
S_{\sigma} := \sum_{0 \leq \mu \leq \sigma} M_\mu \, \Lambda^\mu_- \quad (\sigma \geq 0),
\end{equation*}
for certain composite Mellin operators $M_\mu$. (The sum in the expression for $S_{\sigma}$ always has finitely many terms.)
\end{lemma}

\begin{proof}
Firstly, suppose $r = k \in \mathbb{N}$. Since 
\begin{equation*}
(D-i)^k = \sum^k_{j=0} \binom{k}{j} (-i)^{k-j} D^j,
\end{equation*}
part $(a)$ follows directly from Lemma \ref{DMellin}. \\

Secondly, suppose $r>0$ and $r \not \in \mathbb{N}$. Then, we write
\begin{equation*}
\Lambda^r_- = \Lambda^{\{r\}}_- \, \Lambda^{[r]}_-.
\end{equation*}
Hence, from part $(a)$ and Lemma \ref{LambdaMreverse},
\begin{align*}
\Lambda^r_- \, M_{2 \alpha,0} & = \Lambda^{\{r\}}_- \, \Lambda^{[r]}_- \, M_{2 \alpha,0} \\
& = \Lambda^{\{r\}}_- \, \big \{ \, M_{2 \alpha,[r]} \,  \Lambda^{[r]}_- + S_{[r]-1} \big \} \\
& = \big( M_{2 \alpha, r} \, \Lambda^{\{r\}}_-  + S_0 + T\big) \Lambda^{[r]}_- + S_{r-1} \\
& = M_{2 \alpha, r} \, \Lambda^r_- + S_{[r]} + T \Lambda^{[r]}_-,
\end{align*}
which completes the proof of part $(b)$. \\
\end{proof}

\begin{lemma} \label{lemma:rHomore}
Suppose $1< p < \infty, \, \sigma \in \mathbb{R}$ and $\nu > 0$. Let $l_+ :H^{\sigma}_p(\overline{\mathbb{R}_+}) \to H^{\sigma}_p(\mathbb{R}_+)$ be an arbitrary extension operator. Then  $r_+ (i D)^{\nu} l_+$ is bounded from $H^\sigma_p(\overline{\mathbb{R}_+})$ to $H^{\sigma-\nu}_p(\overline{\mathbb{R}_+})$, and does not depend on the choice of the extension $l_+$. Moreover,
\begin{equation*} 
(r_+ ( i D)^{\nu} l_+)r_+ =  r_+ (i D)^{\nu}. \\
\end{equation*}
\end{lemma}
\begin{proof}
We follow the approach taken in Lemma \ref{lemma:rLambdae}. Let us define the symbol
\begin{equation*}
A_v(\xi) := \dfrac{(i \xi)^\nu}{(1+ \xi^2)^{\nu/2}}.
\end{equation*}
Then, see Lemma \ref{fxi0bddlim}, with $a=\nu, b=c=-\nu/2$, it is a routine calculation to show that $A_v$ is a Fourier $H^\sigma_p$ multiplier. But, directly from the definition of Bessel potential spaces, we have
\begin{equation*}
\mathcal{F}^{-1} ( 1 + \xi^2 )^{\nu/2} \mathcal{F}: H^{\sigma}_p(\mathbb{R}) \to H^{\sigma-\nu}_p(\mathbb{R})
\end{equation*}
 is bounded. Thus, the pseudodifferential operator $(i D)^\nu$ is bounded from $H^\sigma_p(\mathbb{R})$ to $H^{\sigma-\nu}_p(\mathbb{R})$. \\
 
 In addition, its symbol $( i \xi)^{\nu}$ admits an analytic continuation with respect to $\xi$ to the lower complex half-plane ($\tau <0$) such that
\begin{equation*}
| (i \xi - \tau)^\nu | \leq ( |\xi| + |\tau| + 1 )^{\nu}, \quad \tau \leq 0. \\
\end{equation*}

Therefore, from Theorem 1.10, p. 53, \cite{Shar}, $r_+ (i D)^{\nu} l_+$ is continuous from $H^{\sigma}_p(\overline{\mathbb{R}_+})$ to $H^{\sigma-\nu}_p(\overline{\mathbb{R}_+})$, and does \textit{not} depend on the choice of extension $l_+$.\\ 

Finally, by Remark 1.11, p. 53, \cite{Shar}, we also have 
\begin{equation*}
(r_+ (i D) ^{\nu} l_+ ) r_+ = r_+ ( i D)^{\nu}.
\end{equation*} 
This completes the proof of the lemma. \\
\end{proof}

\begin{lemma} \label{ADttau}
Let $B$ be a pseudodifferential operator whose symbol satisfies the condition $|B(\xi)| \leq C(1+ |\xi|)^\nu$, for certain constants $C$ and $\nu$. Suppose $\varphi \in S(\mathbb{R})$. Then
\begin{equation*}
B(D_t) \, \varphi \bigg ( \dfrac{t}{\tau} \bigg) = \bigg [ B \bigg(  \dfrac{D_t}{\tau} \bigg )\varphi \bigg] \bigg( \dfrac{t}{\tau} \bigg).
\end{equation*}
\end{lemma}
\begin{proof}
\begin{align*}
B(D_t) \, \varphi \bigg ( \dfrac{t}{\tau} \bigg) &= \mathcal{F}^{-1} B(\xi) \mathcal{F} \varphi \bigg ( \dfrac{t}{\tau} \bigg) \\
&= \mathcal{F}^{-1} B(\xi) \tau \, (\mathcal{F}{\varphi})(\tau \xi) \quad \text{Proposition 2.2.11 (8), p. 100, \cite{Graf}} \\
&=  \dfrac{1}{\sqrt{2\pi}}\int^\infty_{-\infty} e^{-i \xi t} B(\xi) \tau \, (\mathcal{F}{\varphi})(\tau \xi) \, d\xi \\
&=  \dfrac{1}{\sqrt{2\pi}}\int^\infty_{-\infty} e^{-i \eta t / \tau} B(\eta / \tau) \, (\mathcal{F}{\varphi})(\eta) \, d\eta \quad \text{(where} \, \, \eta = \tau \xi) \\
&=  \dfrac{1}{\sqrt{2\pi}}\int^\infty_{-\infty} e^{-i \eta (t / \tau)} \mathcal{F} \bigg (B \bigg (\dfrac{D_t}{\tau} \bigg ) \varphi  \bigg ) \, d\eta \\
&= \bigg [ B \bigg(  \dfrac{D_t}{\tau} \bigg )\varphi \bigg] \bigg( \dfrac{t}{\tau} \bigg), \quad \text{as required.} \\
\end{align*}
\end{proof}

We have noted previously, see Remark \ref{remark:mellinop}, that the Mellin integral operator $M_{\gamma, \rho}$ with kernel 
\begin{equation*}
K_{\gamma, \rho}(t) = \dfrac{\chi_{[1, \infty)}(t)}{t^\rho \, \Gamma( \gamma) \, t^\gamma \, (t-1)^{1-\gamma}}.
\end{equation*}
is bounded on $L_p(\mathbb{R}_+)$. Moreover, the function $K_{\gamma, \rho}(t) t^{-\epsilon}$  belongs to $L_1(\mathbb{R}_+)$ for all $\epsilon > 0$. \\

\begin{lemma} \label{DMellin}
Let $\varphi \in S(\mathbb{R})$. Suppose $K:\mathbb{R}_+ \to \mathbb{R}$ satisfies the two conditions
\begin{equation} \label{Ksuppinteg}
\operatorname{supp} K \subset [1, \infty) \quad \text{and} \quad \int^\infty_0 |K(\tau)| \tau^{-\epsilon} \, d\tau < \infty, \quad \text{for all} \,\, \epsilon >0.
\end{equation}
Then, for all $t > 0$, 
\begin{equation*}
D_t \int^\infty_0 K \bigg(\dfrac{t}{\tau} \bigg ) r_+ \varphi (\tau) \, \dfrac{d \tau}{\tau} = \int^\infty_0 \dfrac{K(\tau)}{\tau} \, r_+ (D_t \varphi) \bigg( \dfrac{t}{\tau}\bigg) \, \dfrac{d \tau}{\tau}.
\end{equation*}
In other words, on applying the operator $D_t$, we have $K(s) \mapsto s^{-1} K(s)$ and $r_+ \varphi(t) \mapsto r_+ D_t \varphi(t)$.
\end{lemma}
\begin{proof} 
Firstly, we note that
\begin{equation*}
\int^\infty_0 K \bigg(\dfrac{t}{\tau} \bigg ) r_+ \varphi (\tau) \, \dfrac{d \tau}{\tau} \\
= \int^\infty_0 K(\tau) \, r_+ \varphi \bigg( \dfrac{t}{\tau}\bigg) \, \dfrac{d \tau}{\tau}. \\
\end{equation*}
Now let us define
\begin{equation*}
F(t, \tau) := \dfrac{K(\tau)}{\tau} \, r_+ \varphi \bigg( \dfrac{t}{\tau}\bigg), 
\end{equation*}
and thus
\begin{equation*}
\int^\infty_0 | F(t, \tau) |  \, d\tau =  \int^\infty_0  \bigg | \dfrac{K(\tau)}{\tau} \,  r_+ \varphi \bigg( \dfrac{t}{\tau}\bigg) \bigg | \, d\tau \leq C_{1,\varphi} \int^\infty_0 |K(\tau)| \tau^{-1} \, d\tau < \infty,
\end{equation*}
where the constant $C_{1,\varphi}$ only depend on $\varphi$. \\

Since supp $K \subset [1,\infty)$, from Theorem 16.11, p. 213, \cite{Jost}, to prove that we can differentiate through the integral sign, it remains to show that 
\begin{equation*}
\dfrac{\partial}{\partial t} F(t, \tau )
\end{equation*}
is dominated, uniformly for all $t>0$, by an integrable function over the range $1 \leq \tau < \infty$. But, clearly
\begin{equation*}
\bigg | \dfrac{\partial}{\partial t} F(t, \tau) \bigg |  = \bigg | \dfrac{K(\tau)}{\tau^2} r_+ (D_t\varphi)\bigg( \dfrac{t}{\tau} \bigg ) \bigg | \leq C_{2,\varphi} |K(\tau)|\tau^{-2},
\end{equation*}
where the constant $C_{2,\varphi}$ only depends on $\varphi$. \\

Hence, for $t>0$, 
\begin{align*}
D_t \int^\infty_0 K(\tau) \, r_+ \varphi \bigg( \dfrac{t}{\tau}\bigg) \, \dfrac{d \tau}{\tau}
&= \int^\infty_0 K(\tau) D_t \,  r_+ \varphi \bigg( \dfrac{t}{\tau}\bigg) \, \dfrac{d \tau}{\tau} \\
&= \int^\infty_0 \dfrac{K(\tau)}{\tau} \, r_+ (D_t \varphi) \bigg( \dfrac{t}{\tau}\bigg) \, \dfrac{d \tau}{\tau}. 
\end{align*}
This completes the proof of the lemma. \\
\end{proof}


Lemma \ref{DMellin} allows us to change the order of (repeated) differentiation and integration, within a certain class of Mellin integral operators. It will be useful to consider an extension of this result to include \enquote{fractional} differentiation. \\

Suppose $\varphi \in S(\mathbb{R})$ and $\nu > 0$. We write
\begin{equation*}
\nu = [\nu] + \{ \nu \}.
\end{equation*}
Then, as (5.8), \cite{Samko}, we define
\begin{align*}
\mathcal{D}^\nu_- \varphi
& := \bigg ( - \dfrac{d}{dt} \bigg )^{[\nu]+1} I^{1 - \{\nu\}}_- \varphi \\
& = (i D_t)^{[\nu]+1} I^{1 - \{\nu\}}_- \varphi,
\end{align*}
where from (5.3) ibid., 
\begin{equation*}
\big ( I^{1 - \{\nu\}}_- \varphi \big )(t) := \dfrac{1}{\Gamma(1-\{\nu\})} \int^\infty_t \dfrac{\varphi(s)}{(s-t)^{\{\nu\}}} \, ds.
\end{equation*}

But from (7.4) ibid.,
\begin{equation*}
\mathcal{F} (\mathcal{D}^\nu_- \varphi) = (i \xi)^\nu \mathcal{F} \varphi \quad (\nu \geq 0),
\end{equation*}
and thus
\begin{align} \label{DnuDnI}
(iD_t)^\nu & 
 =(i D_t)^{[\nu]+1} I^{1 - \{\nu\}}_-.  
\end{align}

In other words, we can consider the fractional operator $(iD_t)^\nu$ to be the composition of a certain Riemann-Liouville integral of order $1 - \{\nu\}$ with a (conventional) differential operator of order $[\nu]+1$. \\

Hence, we would now like to show that we can change the order of integration in the following iterated integral:
\begin{equation*}
\big( I^{1 - \{\nu\}}_- M \varphi \big ) (t) = \int^\infty_t \dfrac{1}{(s-t)^{\{\nu\}}} \, \bigg ( \int^\infty_1 K(\tau) \varphi \bigg ( \dfrac{s}{\tau}  \bigg ) \dfrac{d\tau}{\tau} \bigg ) \,ds.
\end{equation*} \\


\begin{lemma} \label{MellinFubTon}
Suppose the kernel, $K$, of a Mellin integral operator satisfies the two conditions in \eqref{Ksuppinteg}. Let $\varphi \in S(\mathbb{R})$ and $0 < \gamma < 1$. Then, for $t>0$, 
\begin{equation*}
I(t) := \int^\infty_1 \dfrac{|K(\tau)|}{\tau}  \bigg ( \int^\infty_t \dfrac{1}{(s-t)^\gamma} \bigg | r_+ \varphi  \bigg ( \frac{s}{\tau} \bigg )  \bigg | \, ds \bigg ) \, d\tau < \infty. 
\end{equation*} \\
\end{lemma} 
\begin{proof}
It is convenient to define
\begin{equation*}
J(\tau,t) := \int^\infty_t \dfrac{1}{(s-t)^\gamma} \bigg | r_+ \varphi  \bigg ( \frac{s}{\tau} \bigg )  \bigg | \, ds.
\end{equation*}
Let $w = (s-t)/\tau$, and thus
\begin{align*}
J(\tau,t) & = \int^\infty_0 \dfrac{1}{(\tau w)^\gamma} | r_+ \varphi ( w + t/\tau ) | \, \tau dw \\
& = \dfrac{\tau}{\tau^\gamma} \int^\infty_0 w^{-\gamma} | r_+ \varphi ( w + t/\tau ) | \, dw \\
& \leq C_{\varphi, \gamma} \, \tau^{1-\gamma},
\end{align*}
since $0 < \gamma < 1$. (The constant $C_{\varphi, \gamma}$ only depends on $\varphi$ and $\gamma$.)  \\

Hence
\begin{equation*}
I(t) \leq \int^\infty_1 \dfrac{|K(\tau)|}{\tau} \cdot C_{\varphi, \gamma} \, \tau^{1-\gamma} \, d \tau 
 = C_{\varphi, \gamma} \int^\infty_1 |K(\tau)| \, \tau^{-\gamma} \, d\tau 
< \infty.
\end{equation*}
\end{proof}

We have the following immediate Corollary of Lemma \ref{MellinFubTon}. \\

\begin{corollary} \label{DfracMellin}
Suppose the kernel, $K$, of a Mellin integral operator, $M$, satisfies the two conditions in \eqref{Ksuppinteg}. Let $\varphi \in S(\mathbb{R})$ and $0 < \gamma < 1$. Then, from Lemma \ref{MellinFubTon} and the Fubini-Tonelli Theorem, 
we can change the order of integration in the following iterated integral:
\begin{equation*}
\big( I^{1 - \{\nu\}}_- M \varphi \big ) (t) = \int^\infty_t \dfrac{1}{(s-t)^{\{\nu\}}} \, \bigg ( \int^\infty_1 K(\tau) \varphi \bigg ( \dfrac{s}{\tau}  \bigg ) \dfrac{d\tau}{\tau} \bigg ) \,ds.
\end{equation*} \\
\end{corollary}

\begin{remark} \label{DfracMellinAll}
Combining Corollary \ref{DfracMellin} with Lemma \ref{DMellin}, we see that when applying the operator $r_+(iD_t)^\nu l_+$ to Mellin integral operators whose kernels satisfy the two conditions in \eqref{Ksuppinteg}, we can reverse the order of $r_+(iD_t)^\nu l_+$ and (Mellin) integration for all $\nu > 0$.
\end{remark}

With these preparations complete, we can now compute the action of $r_+(iD_t)^\nu l_+$ on our class of Mellin integral operators. \\

\begin{lemma} \label{AhomoMellin}
 Suppose the kernel, $K$, of a Mellin integral operator satisfies the two conditions in \eqref{Ksuppinteg}. Let $\varphi \in S(\mathbb{R})$. Then, for $\nu>0$ and $t>0$,
\begin{equation*}
 r_+ (i D_t)^\nu l_+ \int^\infty_0 K \bigg(\dfrac{t}{\tau} \bigg ) r_+ \varphi(\tau) \, \dfrac{d \tau}{\tau} = \int^\infty_0 \dfrac{K(\tau)}{\tau^\nu} \, r_+ ((iD_t)^\nu \varphi) \bigg( \dfrac{t}{\tau}\bigg) \, \dfrac{d \tau}{\tau}.
\end{equation*}
In other words, on applying the operator $r_+ (i D_t)^\nu l_+$, we have $K(s) \mapsto s^{-\nu} K(s)$ and $r_+ \varphi(t) \mapsto r_+ (iD_t)^\nu \varphi(t)$.
\end{lemma}
\begin{proof}
For $t>0$, \\

$\displaystyle \quad r_+ (i D_t)^\nu l_+ \int^\infty_0 K \bigg(\dfrac{t}{\tau} \bigg ) r_+ \varphi(\tau) \, \dfrac{d \tau}{\tau}$ 
\begin{align*}
&=
r_+ (i D_t)^\nu l_+  \int^\infty_0 K(\tau) \, r_+ \varphi \bigg( \dfrac{t}{\tau}\bigg) \, \dfrac{d \tau}{\tau} \\
&=
  \int^\infty_0 K(\tau) \, r_+ (i D_t)^\nu l_+r_+ \varphi \bigg( \dfrac{t}{\tau}\bigg) \, \dfrac{d \tau}{\tau} \quad \text{by Remark \ref{DfracMellinAll}}\\
&=
  \int^\infty_0 K(\tau) \, r_+ (i D_t)^\nu \varphi  \bigg( \dfrac{t}{\tau}\bigg) \, \dfrac{d \tau}{\tau} \quad \text{by Lemma \ref{lemma:rHomore}}\\
&= \int^\infty_0 K(\tau) r_+ \bigg [  \bigg(  \dfrac{i D_t}{\tau} \bigg )^\nu \varphi \bigg] \bigg( \dfrac{t}{\tau} \bigg)  \, \dfrac{d \tau}{\tau} \quad \text{by Lemma \ref{ADttau}} \\
&= \int^\infty_0 \dfrac{K(\tau)}{\tau^\nu} \, r_+ ((i D_t)^\nu \varphi) \bigg( \dfrac{t}{\tau}\bigg) \, \dfrac{d \tau}{\tau}.
\end{align*}
This completes the proof of the lemma. \\
\end{proof}


\begin{lemma} \label{fxi0bddlim}
Suppose $0 <  r < 1$, and $c_r: \mathbb{R} \to \mathbb{C}$ is given by
\begin{equation*}
c_r(\xi) = (\xi-i)^r - (-i)^r (i \xi)^r - (-i)^r, \quad \xi \in \mathbb{R}.
\end{equation*}
Then $c_r(\xi)$ is bounded  for all $\xi \in \mathbb{R}$. Moreover,
\begin{equation*}
c_r(0) = 0 \quad \text{and} \quad \lim_{\xi \to \pm \infty} c_r(\xi) = -(-i)^r,
\end{equation*} 
and $c_r$ is a Fourier $L_p$-multiplier.
\end{lemma}

\begin{proof}
The boundedness of $c_r(\xi)$ will follow immediately once the limits of $c_r(\xi)$ at $0$ and $\pm \infty$ are established. Of course, it is elementary to verify that $c_r(0) = 0$.\\

Now suppose $|\xi| > 1$. Then, for $0  < r <1$,
\begin{align*}
c_r(\xi) &= (-i)^r \, \big \{ (1+ i \xi)^r - (i\xi)^r -  1 \big \}  \\
&= (-i)^r \, \big \{ (i \xi)^r (1- i /\xi)^r - (i\xi)^r -  1 \big \}  \\
&= (-i)^r \, \big \{ (i \xi)^r (-i r)/\xi  -  1 +O (|\xi|^{-2+r}) \big \}  \\
& \to -(-i)^r, \quad \text{ as } |\xi| \to  \infty. 
\end{align*}

From the Mikhlin multiplier theorem, to show that $c_r$ is a Fourier $L_p$-multiplier, it remains to show that $\xi c'_r(\xi)$ is bounded. \\

From the definition of $c_r(\xi)$, we have
\begin{equation*}
c'_r(\xi) = r (\xi -i)^{r-1} - (-i)^r (ir)(i \xi)^{r-1},
\end{equation*}
and thus
\begin{equation*}
\xi c'_r(\xi) = r \big \{ \xi (\xi -i)^{r-1} - (-i)^r (i \xi)^{r} \big \}.
\end{equation*}

Hence $\xi c'_r(\xi) \big |_{\xi=0} = 0$.  \\

Now suppose that $|\xi| > 1$. Then, writing
\begin{equation*}
c_r(\xi) = (-i)^r \, \big \{ (1+ i \xi)^r - (i\xi)^r -  1 \big \}
\end{equation*}
we have
\begin{equation*}
c'_r(\xi) = (-i)^r r \, \big \{ i(1+ i \xi)^{r-1} - i(i\xi)^{r-1}\big \}
\end{equation*}

Thus, 
\begin{align*}
\xi c'_r(\xi) &= (-i)^r r \, \big \{ (i \xi) (1+ i \xi)^{r-1} - (i\xi)^{r} \big \}  \\
&= (-i)^r \, \big \{ (i \xi)^r (1- i /\xi)^{r-1} - (i\xi)^r \big \}  \\
&= (-i)^r \, \big \{ (i \xi)^r (-i (r-1))/\xi  +O (|\xi|^{-2+r}) \big \}  \\
& \to 0, \quad \text{ as } |\xi| \to  \infty. 
\end{align*}
This completes the proof of the lemma. \\
\end{proof}

\newpage
\section{Supporting lemmas}

\begin{lemma} \label{M0bchi}
Suppose that $\psi_0 \in r_+C^\infty_0(\mathbb{R})$ with $\psi_0(t) = 0$ for $t \geq 2$. Let $\chi \in C^\infty_0(\mathbb{R})$ be such that $\chi(t) = 1$ if $|t| \leq 2$. Then
\begin{equation*}
\psi_0 M_{2\alpha,0} r_+ h = \psi_0 M_{2\alpha,0} (r_+ \chi h)
\end{equation*}
for all $h \in S(\mathbb{R})$.  \\
\end{lemma}
\begin{proof}
Let $K_{2 \alpha,0}$ denote the kernel of the Mellin integral operator $M_{2\alpha, 0}$. \\

Firstly, suppose $0 \leq t \leq 2$. Then
\begin{align*}
\psi_0(t) \, (M_{2\alpha,0} r_+ h)(t) & := \psi_0(t) \int^\infty_0 K_{2 \alpha,0}(\tau) r_+ h \bigg ( \dfrac{t}{\tau} \bigg ) \, \dfrac{d\tau}{\tau} \\
& = \psi_0(t) \chi(t) \int^\infty_0 K_{2 \alpha,0}(\tau) r_+ h \bigg ( \dfrac{t}{\tau} \bigg ) \, \dfrac{d\tau}{\tau} \quad (0 \leq t \leq 2)\\
& = \psi_0(t) \chi(t) \int^\infty_1 K_{2 \alpha,0}(\tau) r_+ h \bigg ( \dfrac{t}{\tau} \bigg ) \, \dfrac{d\tau}{\tau}  \quad ( \text{ supp } K_{2 \alpha,0} \subset [1,\infty) )\\
& = \psi_0(t)  \int^\infty_1 K_{2 \alpha,0}(\tau) r_+ \chi(t) h \bigg ( \dfrac{t}{\tau} \bigg ) \, \dfrac{d\tau}{\tau}  \\
& = \psi_0(t)  \int^\infty_1 K_{2 \alpha,0}(\tau) r_+ \chi\bigg ( \dfrac{t}{\tau} \bigg ) h \bigg ( \dfrac{t}{\tau} \bigg ) \, \dfrac{d\tau}{\tau}  \quad ( 0 \leq t/\tau \leq 2 )\\
& = \psi_0(t)  \int^\infty_0 K_{2 \alpha,0}(\tau) r_+ \chi\bigg ( \dfrac{t}{\tau} \bigg ) h \bigg ( \dfrac{t}{\tau} \bigg ) \, \dfrac{d\tau}{\tau}  \\
&=  \psi_0(t) \, (M_{2\alpha,0} (r_+ \chi h))(t).
\end{align*}

On the other hand, if  $t >2$ then 
\begin{equation*}
\psi_0(t) \, (M_{2\alpha,0} r_+ h)(t) = 0 = \psi_0(t) \, (M_{2\alpha,0} (r_+ \chi h))(t).
\end{equation*}

This completes the proof of the lemma. \\
\end{proof}

If $1 + 1/p < s < 2+1/p$, for any $\mu, r \in \mathbb{R}$, it will be convenient to define
\begin{equation} \label{defwpmrxbig}
v_{\pm, \mu, r}(x) := \frac1{\sqrt{2\pi}}\int_{-\infty}^\infty e^{-ix\xi} (-i\xi)^{\mu}\, (\xi \pm i)^{r-2} \, d\xi , \quad  x \in \mathbb{R}.
\end{equation} \\

\begin{lemma} \label{InverFTmubig}
Let $-1 < \mu < 1$ and $r < 2-\mu$. Then $v_{\pm, \mu, r}(x)$ is bounded away from $x=0$ for all finite $x$. Moreover, as $x \to 0$
\begin{align*}
v_{\pm, \mu, r} &=
\begin{cases}
O(1)        & \text{if  r} < 1- \mu \\
O(|x|^{1- \mu-r})        & \text{if  r} > 1- \mu.
\end{cases} 
\end{align*}
Finally, $v_{+, \mu, 1-\mu} = O(1), \,\, v_{-,0,1} = O(1)$ and, for $\mu \not = 0$, we have $v_{-,\mu, 1-\mu}=O(\log |x|)$. 
\end{lemma} 

\begin{proof}
If $r < 1 - \mu$, then $v_{\pm, \mu, r}$ is the inverse Fourier transform of an integrable function, and hence is continuous and vanishes at infinity. \\

Now suppose that $1- \mu \leq r < 2 -\mu$. Initially, we assume that $\epsilon < x< N$, for some $\epsilon >0$ and $N < \infty$, noting that the case $-N < x < -\epsilon$ follows in a similar manner. Changing the variable of integration, we obtain
\begin{equation} \label{wpmrxbig}
v_{\pm, \mu, r}(x) = x^{1- \mu-r} \, \frac1{\sqrt{2\pi}}\int_{-\infty}^\infty e^{-i\eta} (-i\eta)^{\mu}\,(\eta \pm i x)^{r-2} \, d\eta .
\end{equation} 
Now,
\begin{align*}
& \int_1^\infty e^{-i\eta} (-i\eta)^{\mu}\,(\eta \pm i x)^{r-2}\, d\eta = 
i\int_1^\infty \left(\frac{d}{d\eta}e^{-i\eta}\right) (-i\eta)^{\mu}\,(\eta \pm i x)^{r-2}\, d\eta \\
& \stackrel{\mbox{\footnotesize integrating by parts}}=  -i\, e^{-i}(-i)^{\mu} (1 \pm ix)^{r-2} - 
\mu\, \int_1^\infty e^{-i\eta} (-i\eta)^{\mu - 1}\,(\eta \pm i x)^{r-2}\, d\eta\\
& + i (2-r)\, \int_1^\infty e^{-i\eta} (-i\eta)^{\mu}\,(\eta \pm i x)^{r-3}\, d\eta = O\left( 1 \right), \ \mbox{ for }  \epsilon < x < N.
\end{align*}
Moreover, it is easy to see that the above estimate also applies in the limit as $x \searrow 0$. The case $\int_{-\infty}^{-1} \cdots$ follows similarly. \\

On the other hand
\begin{equation*}
\int_{-1}^1 e^{-i\eta} (-i\eta)^{\mu}\,(\eta \pm i x)^{r-2}\, d\eta = O\left( 1 \right), \ \mbox{ for }  \epsilon < x < N.
\end{equation*}
In addition, the above estimate also holds in the limit as $x \searrow 0$, \textit{provided} we exclude the case $r=1-\mu$. See Lemma \ref{InverFTmulog}. \\

Hence, $v_{\pm, \mu, r}(x)$ is bounded away from $x=0$ for all finite $x$. Moreover, as $x \to 0$
\begin{align*}
v_{\pm, \mu, r} &=
\begin{cases}
O(1)        & \text{if  r} < 1- \mu \\
O(|x|^{1- \mu-r})        & \text{if  r} > 1- \mu. \\
\end{cases} 
\end{align*}
Finally, by Lemma \ref{InverFTmulog}, $v_{+, \mu, 1-\mu} = O(1), \,\, v_{-,0,1} = O(1)$ and, for $\mu \not = 0$, we have $v_{-,\mu,1-\mu}=O(\log |x|)$. \\
\end{proof}

\begin{lemma} \label{Hrconditionbig}
Let $-1 < \mu < 1, \, r < 2-\mu$ and $1 < p < \infty$. Define  
\begin{equation*}
v_{\pm, \mu}(x) := \frac1{\sqrt{2\pi}}\int_{-\infty}^\infty e^{-ix\xi} (-i\xi)^{\mu} \,\frac1{(\xi \pm i)^2}\, d\xi , \quad  x \in \mathbb{R} .
\end{equation*} 
Then, for any $\chi_1 \in C^\infty_0(\mathbb{R})$,
\begin{equation*}
\chi_1 (D \pm i)^r v_{\pm, \mu} \in L_p(\mathbb{R}) \quad  \text{if} \quad r < 1- \mu + 1/p. \\
\end{equation*}

\end{lemma}
\begin{proof}
Firstly, we note from equation \eqref{defwpmrxbig} and the definition of $v_{\pm, \mu}$ that
\begin{align*}
v_{\pm, \mu, r}(x) &:= \frac1{\sqrt{2\pi}}\int_{-\infty}^\infty e^{-ix\xi} (-i\xi)^{\mu}\, (\xi \pm i)^{r-2} \, d\xi \\
&= \mathcal{F}^{-1} \bigg ( (\xi \pm i)^r \dfrac{(-i\xi)^{\mu}}{(\xi \pm i)^2} \bigg ) \\
&= \mathcal{F}^{-1} \big ( (\xi \pm i)^r \widehat{v_{\pm, \mu}} \big ) \\
&= (D \pm i)^r v_{\pm, \mu}.
\end{align*}

But from Lemma \ref{InverFTmubig}, $v_{\pm, \mu, r}$ is bounded away from $x=0$ for all finite $x$, and as $x \to 0$
\begin{align*}
v_{\pm, \mu, r} &=
\begin{cases}
O(1)        & \text{if  r} < 1- \mu \\
O(|x|^{1-\mu-r})        & \text{if  r} > 1- \mu.
\end{cases} 
\end{align*}
Moreover, 
$v_{+, 1-\mu}$ is $O(1)$, and
\begin{align*}
v_{-, \mu, 1-\mu} &=
\begin{cases}
O(1)        & \text{if  }  \mu =0 \\
O(\log |x|)        & \text{if  }  \mu \not = 0,
\end{cases} 
\end{align*}
as $x \to 0$. \\

Hence, $\chi_1 v_{\pm, \mu, r} \in L_p(\mathbb{R})$ if
\begin{equation*}
p(1-\mu-r) > -1.
\end{equation*}
So, finally
\begin{equation*}
\chi_1 (D \pm i)^r v_{\pm, \mu} \in L_p(\mathbb{R}) \quad  \text{if} \quad r < 1- \mu + 1/p. \\
\end{equation*}
This completes the proof of the lemma. \\ \\
\end{proof}


We now the consider the case of lower regularity. If $1/p < s < 1+1/p$, for any $\mu, r \in \mathbb{R}$, we let
\begin{equation} \label{defwpmrx}
w_{\pm, \mu, r}(x) := \frac1{\sqrt{2\pi}}\int_{-\infty}^\infty e^{-ix\xi} (-i\xi)^{-\mu}\, (\xi \pm i)^{r-1} \, d\xi , \quad  x \in \mathbb{R}.
\end{equation} \\

Given definition \eqref{defwpmrx}, the following corollary and lemma are the counterparts, for lower regularity,  of Lemmas \ref{InverFTmubig} and \ref{Hrconditionbig} respectively. \\

\begin{corollary} \label{InverFTmu}
Let $0 < \mu < 1$ and $r < 1 + \mu$. Then $w_{\pm, \mu, r}(x)$ is bounded away from $x=0$ for all finite $x$. Moreover, as $x \to 0$
\begin{align*}
w_{\pm, \mu, r} &=
\begin{cases}
O(1)        & \text{if  r} < \mu \\
O(|x|^{\mu-r})        & \text{if  r} > \mu,
\end{cases} 
\end{align*}
and $w_{+,\mu, \mu} = O(1)$ and $w_{-,\mu, \mu} = O( \log |x|)$.
\end{corollary} 

\begin{proof}
Let us define $\tilde{\mu} := -\mu$ and $\tilde{r} := r-1$. Then, if we recast Lemma \ref{InverFTmubig} in terms of $\tilde{\mu}$ and $\tilde{r}$, we obtain Corollary \ref{InverFTmu}, with $\tilde{\mu}$ and $\tilde{r}$ in place of $\mu$ and $r$ respectively. \\
\end{proof}

\begin{lemma} \label{Hrcondition}
Let $0< \mu < 1, \, r < 1 + \mu$ and $1 < p < \infty$. Define  
\begin{equation*}
w_{\pm, \mu}(x) := \frac1{\sqrt{2\pi}}\int_{-\infty}^\infty e^{-ix\xi} (-i\xi)^{-\mu} \,\frac1{\xi \pm i}\, d\xi , \quad  x \in \mathbb{R} .
\end{equation*} 
Then, for any $\chi_1 \in C^\infty_0(\mathbb{R})$
\begin{equation*}
\chi_1 (D \pm i)^r  w_{\pm, \mu} \in L_p(\mathbb{R}) \quad  \text{if} \quad r < \mu + 1/p.  \\
\end{equation*}

\end{lemma}
\begin{proof}
Firstly, we note that
\begin{equation*}
(D \pm i)v_{\pm, \mu} = \mathcal{F}^{-1} (\xi \pm i)(-i \xi)^\mu (\xi \pm i)^{-2} = \mathcal{F}^{-1} (-i \xi)^\mu (\xi \pm i)^{-1} = w_{\pm, -\mu}.
\end{equation*}
Now let us define $\tilde{\mu} := -\mu$ and $\tilde{r} := r-1$. Then, if we recast Lemma \ref{Hrconditionbig} in terms of $\tilde{\mu}$ and $\tilde{r}$, we obtain Lemma \ref{Hrcondition}, with $\tilde{\mu}$ and $\tilde{r}$ in place of $\mu$ and $r$ respectively. \\
\end{proof}

\begin{lemma} \label{mxiFm}
Suppose $a, b,c \in \mathbb{R}$ are such that $a \geq 0$ and $a+b+c \leq 0$. Let $m_p$ be given by
\begin{equation*}
m_p(\xi) := (-i \xi)^a \, (\xi+i)^b \, (\xi - i)^c, \quad \xi \in \mathbb{R}.
\end{equation*}
\end{lemma}
Then $m_p(\xi)$ is a Fourier multiplier.
\begin{proof}
The conditions $a \geq 0$ and $a+b+c \leq 0$ ensure that $|m_p(\xi)|$ is bounded. \\

We now assume that each of $a,b$ and $c$ is non-zero, noting that in the special cases where at least one of these exponents is zero, the same method of proof applies. A routine calculation gives
\begin{align*}
m'_p(\xi) &= (-i \xi)^{a-1} \, (\xi +i)^{b-1} \, (\xi -i)^{c-1} \, \big \{- i a - (b-c)\xi -i (a+b+c) \xi^2 \big \}. 
\end{align*}
For $|\xi| \geq 1$, it is easy to see that
\begin{align*}
|\xi m'_p(\xi)| & \leq C \, \big ( \sqrt{{1+\xi^2}} \big )^{a+(b-1)+(c-1) +2} \\
& = C \, \big (\sqrt{{1+\xi^2}} \big )^{a+b+c} \\
& \leq C, \quad \text{since }  a+b+c \leq 0. 
\end{align*}
On the other hand, suppose $|\xi| <1$. Then
\begin{equation*}
|\xi m'_p(\xi)|  \leq  C |\xi|^a \leq  C,  \quad \text{since we are assuming }  a > 0. 
\end{equation*}

Hence, $m_p(\xi)$ is a Fourier multiplier, by the Mikhlin multiplier theorem. \\
\end{proof}

Following \cite{Dudu87, Roch}, we let $C_0$ denote the algebra of all continuous and piecewise linear functions on $\dot{\mathbb{R}}$, and $PC_0$ the algebra of all piecewise constant functions on $\mathbb{R}$, with only finitely many discontinuities. Further, for $1 < p < \infty$, let $C_p$ and $PC_p$ represent the closure of $C_0$ and $PC_0$ in $\mathfrak{M}_p$ respectively. \\

\begin{remark} \label{BVCCpinclusions}
We note that a number of the results that we use from \cite{Dudu87, Roch}, require that a given Wiener-Hopf symbol belongs to either $C_p$ or $PC_p$. Fortunately, we have the following inclusions:
\begin{enumerate}[\hspace{0.5cm}(a)]
\item The set of all continuous functions on  $\dot{\mathbb{R}}$ with bounded variation is contained in $C_p$. See, for example, Proposition 5.1.2 (iii), p. 261, \cite{Roch}. 
\item The set of all (piecewise) continuous functions on $\overline{\mathbb{R}}$ with bounded variation is contained in $PC_p$. See, for example, Proposition 5.1.4 (ii), p. 261, \cite{Roch}. \\
\end{enumerate}
\end{remark}

\begin{lemma} \label{BVm}
Let the function $m_p$ be as defined in Lemma \ref{mxiFm}. Then $m_p \in BV(\overline{\mathbb{R}})$ and is continuous. 
\end{lemma}
\begin{proof}
From Lemma \ref{mxiFm}, the function $m_p$ is continuous and bounded on $\overline{\mathbb{R}}$. \\

As in Lemma \ref{mxiFm}, we assume that each of $a,b,c$ is non-zero, noting that in the remaining special cases, the same method of proof applies. From the proof of Lemma \ref{mxiFm}, we have
\begin{equation*}
m'_p(\xi) = (-i \xi)^{a-1} \, (\xi +i)^{b-1} \, (\xi -i)^{c-1} \, \big \{- i a - (b-c)\xi -i (a+b+c) \xi^2 \big \}. 
\end{equation*} \\
As $ | \xi | \to 0$, we have $m'_p(\xi) = O(|\xi|^{a-1})$. Since we are assuming that $a > 0$,  $m'_p$ is integrable in a neighbourhood of $\xi = 0$. \\

As $| \xi | \to \infty$, we have $m'_p(\xi) = O(|\xi|^{-2})$ or $m'_p(\xi) = O(|\xi|^{a+b+c-1})$, as either $a+b+c=0$ or $a+b+c < 0$. Hence, $m'_p(\xi)$ is integrable near $\xi = \pm \infty$, and thus $m'_p \in L_1(\mathbb{R})$. \\

Fiinally, from Remark \ref{BVtestmethod}, we have $m_p \in BV(\overline{\mathbb{R}})$, as required. \\
\end{proof}

%% file: KCL_Thesis_Chapter6_v5.tex
\chapter{Fredholm analysis} \label{GenSymbol}
\section{Introduction}
Suppose $N \in \mathbb{N}$. Our goal is to establish the conditions under which the operator
\begin{equation*} 
\tilde{A} := \sum^N_{j=1} a_j M^0(b_j) W(c_j)
\end{equation*}
acting on $L_p(\mathbb{R}_+)$ is Fredholm. It will be sufficient, for our purposes to assume that, for $j = 1 , \dots, N$, we have $a_j$ continuous on $\overline{\mathbb{R}_+}$, and $ b_j, c_j$ continuous on $\overline{\mathbb{R}}$. In addition, see Lemmas \ref{BVBeta}, \ref{BVm} and Remark \ref{BVCCpinclusions}, we may suppose that the symbols $\{ b_j, c_j \}^N_{j=1}$ have bounded variation.  \\ 

Later in this chapter, we will apply these general results to the specific set of symbols $\{ c_1, a_2,b_2, c_2 \}$, derived earlier in Chapter \ref{OpAlgLp}, for the case of lower regularity $1/p < s < 1+1/p$. \\

We detail the method originally developed by Duduchava \cite{Dudu, Dudu87}, and later reviewed in \cite{Roch}. Indeed, our starting point is Corollary 5.5.10, p. 290 \cite{Roch}. To use this important result most effectively, we will combine it with Theorems 5.5.3, 5.5.4 and 5.5.7 as given in pp. 279 to 290, ibid. In preparation for this, the following remark addresses some important points on notation. \\

\begin{remark}
We adopt the convention that Mellin and Wiener-Hopf operators are given by $M^0(b)$ and $W(c)$, with symbols $b$ and $c$ respectively. However, in \cite{Roch}, this convention is reversed. (So that, for example, in Theorem 5.5.3, p. 279, ibid, the symbol $b$ is used to describe a Wiener-Hopf operator.) \\

Moreover, in \cite{Roch}, the Fourier transform is defined using the opposite sign convention. (C.f. Equation \eqref{FTdefinition} and  equation (4.9), p. 199, \cite{Roch}.) So, if we denote this alternative Fourier transform by $\mathcal{F}_-$, then, by a routine calculation
\begin{equation} \label{FTdefsymbolsign}
{\mathcal{F}}^{-1}_- b(\xi) {\mathcal{F}}^{}_- = \mathcal{F}^{-1} b(-\xi) \mathcal{F}.
\end{equation} 
On the other hand, both here and in \cite{Roch}, the Mellin transform is defined identically.  (C.f. Equation \eqref{MTexplicit} and  equation (4.27), p. 203, \cite{Roch}.) \\

In our case, we note the multiplication symbol $a(x)$ and the Mellin symbol $b(\xi)$ are both continuous on $\overline{\mathbb{R}_+}$ and $\dot{\mathbb{R}}$ respectively. On the other hand, whilst the Wiener-Hopf symbols, $c(\xi)$, are continuous for all finite $\xi$, we do allow $c(\infty) \not = c(-\infty)$.\\
\end{remark}

We note that Theorem 5.5.4, p. 281, \cite{Roch}, makes use of the operator $S_\mathbb{R}$, see equation (4.18), p. 201, \cite{Roch}, with the important property that  $S_\mathbb{R}S_\mathbb{R}=I$. Therefore,
\begin{equation*}
\dfrac{(I \pm S_\mathbb{R})}{2}\dfrac{(I \pm S_\mathbb{R})}{2}= \dfrac{(I \pm S_\mathbb{R})}{2} \quad \text{and} \quad \dfrac{(I \pm S_\mathbb{R})}{2}\dfrac{(I \mp S_\mathbb{R})}{2} =0.
\end{equation*}
Hence, any operator of the form
\begin{equation*}
h_- \dfrac{(I - S_\mathbb{R})}{2} + h_+ \dfrac{(I + S_\mathbb{R})}{2}, \quad h_\pm \in \mathbb{C},
\end{equation*}
is invertible if and only if
\begin{equation*}
h_- \not = 0 \quad \text{and} \quad h_+ \not = 0.
\end{equation*}

Let $\chi_\pm$ denote the characteristic functions of the positive and negative half-lines respectively. Then trivially, any operator of the form
\begin{equation*}
h_- \chi_-I + h_+ \chi_+I,  \quad h_\pm \in \mathbb{C},
\end{equation*}
is invertible if and only if
\begin{equation*}
h_- \not = 0 \quad \text{and} \quad h_+ \not = 0.
\end{equation*}

\subsection{Loop functions}
Finally, in the light of  Theorem 5.5.7, p. 286, \cite{Roch} and observation \eqref{FTdefsymbolsign}, it will be convenient, in \textbf{our notation}, to define
\begin{equation} \label{gpinftydefinition}
g_p(\infty, \xi) := g(-\infty) \dfrac{1 + d(\xi)}{2} + g(+\infty) \dfrac{1 - d(\xi)}{2}
\end{equation}
where $d(\xi) := \coth \pi(i/p + \xi)$. \\

It is easy to verify that 
\begin{equation*}
\lim_{\xi \to \pm \infty} g_p(\infty, \xi) = g(\mp \infty). \\
\end{equation*}

The function $g_p(\infty, \xi)$, defined by equation \eqref{gpinftydefinition}, traces out an arc of a circle, dependent only on $p$ and the function values $g(\pm \infty)$, in the complex plane, as $\xi$ varies from $-\infty$ to $+\infty$. Indeed, if we assume, without loss of generality, that $g(-\infty) = -1$ and $g(\infty) = +1$, then a routine calculation shows that
\begin{equation*}
| g_p(\infty, \xi) - a | = r,
\end{equation*}
where the constants $a$ and $r$ only depend on $p$, and are given by
\begin{equation*}
a = i \, \cot (2 \pi /p) \quad \text{and} \quad r= \dfrac{1}{\sin (2 \pi /p)}.
\end{equation*}
Moreover, if $g(\mp \infty) = \mp 1$ then
\begin{equation*}
\operatorname{Im} \, g_p(\infty, \xi) = \dfrac{\sin (2 \pi /p)}{\cosh (2 \pi \xi) - \cos (2 \pi /p)},
\end{equation*}
so that the sign of the imaginary part of $g_p(\infty, \xi)$ is determined simply by the sign of $\sin (2 \pi /p)$. In other words, if $1 < p < 2$ then the circular arc is below the interval $[-1,1]$ and if $2<p <\infty$ it is above. Finally, in the special case that $p=2$ the arc degenerates precisely to the interval $[-1,1]$ in the complex plane. \\

\section{The contour $\Gamma_M$ and symbol $A_{\alpha, p, s}$} \label{contourgensymbol}
We now follow Duduchava, see p. 520, \cite{Dudu87},  and using his notation we define
\begin{equation*}
\omega := (x, \xi, \lambda) \quad \text{where} \quad 0 \leq x \leq \infty, \,\, -\infty \leq \xi,\lambda \leq \infty.
\end{equation*}
Then we consider the contour $\Gamma_M$, which can be described as
\begin{equation} \label{ContourGammaM}
\Gamma_M := \Gamma_1 \cup \Gamma^+_2 \cup \Gamma^+_3 \cup \Gamma_4 \cup \Gamma^-_3 \cup \Gamma^-_2, 
\end{equation}
where the order of the six segments indicates the direction to be taken. \\

On each of the six segments of $\Gamma_M$, two of the variables in the triple $(x, \xi, \lambda)$ are fixed, whilst the third varies over its permitted range. 
The precise definition, including orientation, of each segment of the contour $\Gamma_M$ is as follows: \\

$\Gamma_M = 
\begin{cases}
\Gamma_1 = \{ (0,\xi, \infty): -\infty \leq \xi \leq \infty \} \\ \\
\Gamma_2^+ = \{ (x,\infty, \infty): 0 \leq x \leq \infty \} \\ \\
\Gamma_3^+ = \{ (\infty,\infty, \lambda): \infty \geq \lambda \geq 0 \} \\ \\
\Gamma_4 = \{ (\infty,\xi, 0): \infty \geq \xi \geq -\infty \} \\ \\
\Gamma_3^- = \{ (\infty, -\infty, \lambda): 0 \leq \lambda \leq \infty \} \\ \\
\Gamma_2^- = \{ (x, -\infty, \infty): \infty \geq x \geq 0 \}. 
\end{cases} $ \\ \\

\begin{remark}
There is a typographical error in the statement of Theorem 5.5.7, pp. 286, 287, \cite{Roch}. \footnote{This was confirmed in a personal communication from Prof. Steffen Roch on 13th October 2016.} The right hand side of the display formula in the second line on p. 287 should read
\begin{equation*}
a(0^+) \chi_- I + a(+\infty) \chi_+I,
\end{equation*}
instead of 
\begin{equation*}
a(+\infty) \chi_- I + a(0^+) \chi_+I.
\end{equation*} \\
\end{remark}

With these preparations complete, we are now in a position to restate Corollary 5.5.10, p. 290, \cite{Roch}, in a more convenient form. \\

From Theorem 5.5.3, \cite{Roch}, we require the functions
\begin{align}
S_{\Gamma^+_3}(\lambda) :=  \sum^N_{j=1} &a_j(\infty) b_j(\infty) c_j(-\lambda) \,\, (\lambda >0),  \label{SGamma3plus}\\
S_{\Gamma^-_3}(\lambda) :=  \sum^N_{j=1}  & a_j(\infty) b_j(-\infty) c_j(\lambda)  \,\, (\lambda >0), \label{SGamma3minus}
\end{align}
to be non-zero. In addition, from Theorem 5.5.7, \cite{Roch}, we require the values
\begin{equation*}
\quad I_{\Gamma^+_3 \cap \Gamma_4} :=  \sum^N_{j=1} a_j(\infty)b_j(\infty)c_j(0), \,\, I_{\Gamma^-_3 \cap \Gamma^-_2} :=  \sum^N_{j=1} a_j(\infty)b_j(-\infty)c_j(\infty),
\end{equation*}
to be non-zero. \\

Similarly, from Theorem 5.5.4, \cite{Roch} we require the functions
\begin{align}
S_{\Gamma^+_2}(x) :=  \sum^N_{j=1} & a_j(x) b_j(\infty) c_j(-\infty) \,\, (x >0), \label{SGamma2plus} \\
S_{\Gamma^-_2}(x) :=  \sum^N_{j=1} & a_j(x) b_j(-\infty) c_j(\infty)  \,\, (x >0), \label{SGamma2minus}
\end{align}
to be non-zero. In addition, from Theorem 5.5.7, \cite{Roch}, we require the values
\begin{equation*}
I_{\Gamma^+_2 \cap \Gamma^+_3} :=  \sum^N_{j=1} a_j(\infty)b_j(\infty)c_j(-\infty), \,\, I_{\Gamma^-_2 \cap \Gamma_1} :=  \sum^N_{j=1} a_j(0)b_j(-\infty)c_j(\infty), 
\end{equation*}
to be non-zero. \\

Finally, from Theorem 5.5.7, \cite{Roch}, we require the functions
\begin{align}
S_{\Gamma_1}(\xi) :=   \sum^N_{j=1} & a_j(0) b_j(\xi) c_{jp}(\infty, \xi)  \quad (-\infty < \xi < \infty), \label{SGamma1}\\
S_{\Gamma_4}(\xi) := \sum^N_{j=1} & a_j(\infty) b_j(\xi) c_j(0)  \quad (-\infty < \xi < \infty), \label{SGamma4}
\end{align}
to be non-zero. In addition, again from Theorem 5.5.7, \cite{Roch}, we require the values
\begin{equation*}
I_{\Gamma_1 \cap \Gamma^+_2} :=  \sum^N_{j=1} a_j(0)b_j(\infty)c_j(-\infty), \,\, I_{\Gamma_4 \cap \Gamma^-_3} :=  \sum^N_{j=1} a_j(\infty)b_j(-\infty)c_j(0),
\end{equation*}
to be non-zero. \\

We now define the \textit{generalised symbol} $A_{\alpha,p,s}(\omega)$ by: \\

$\qquad A_{\alpha,p,s}(\omega):= 
\begin{cases} 
S_{\Gamma_1}(\xi) &\mbox{ on } \Gamma_1 \\ 
S_{\Gamma^\pm_2}(x) &\mbox{ on } \Gamma^\pm_2 \\ 
S_{\Gamma^\pm_3}(\lambda) &\mbox{ on } \Gamma^\pm_3 \\ 
S_{\Gamma_4}(\xi) &\mbox{ on } \Gamma_4. \\
\end{cases}$ \\

Then it is easy to see, from the above results, that as the triple $\omega = (x, \xi, \lambda)$ traverses the contour $\Gamma_M, \,\, A_{\alpha,p,s}(\omega)$ forms a closed loop in the complex plane. \\

Hence, we can re-write Corollary 5.5.10, p. 290, \cite{Roch} as:
\begin{theorem} \label{AtilldeFredholm}
The operator 
\begin{equation*}
\tilde{A} = \sum^N_{j=1} a_j M^0(b_j) W(c_j)
\end{equation*}
 is Fredholm on $L_p(\mathbb{R}_+)$ if and only if 
\begin{equation*}
\inf_{\omega \in \Gamma_M} | A_{\alpha,p,s}(\omega) | >0.
\end{equation*} \\
\end{theorem}

\begin{remark} \label{WindingFred}
In the case that $\widetilde{A}$ is a Fredholm operator on $L_p(\mathbb{R}_+)$, then, see Theorem 3.2, p. 521, \cite{Dudu87}, the index of $\widetilde{A}$ is given by
\begin{equation} \label{indwindingnumber}
\operatorname{ind} \widetilde{A} = - \, (\text{winding number of  } A_{\alpha,p,s}(\omega)).
\end{equation}
In particular, if the winding number of $A_{\alpha,p,s}(\omega)$ is zero, then $\widetilde{A}$ has Fredholm index equal to zero. \\
\end{remark}

We now verify Remark \ref{WindingFred} in a simple case.  Let us define the symbol
\begin{equation*}
c_{(n)} (\xi) := (\xi + i)^n (\xi - i)^{-n}, \quad n \in \mathbb{N}.
\end{equation*}
Then, see Chapter 1, Section 8, \cite{GohFeld}, the Wiener-Hopf operator $W(c_{(n)})$, acting on $L_p(\mathbb{R}_+)$, has the following properties:
\begin{enumerate}[\hspace{18pt}(a)]
\item $W(c_{(n)})$ is right-invertible;
\item $\operatorname{Ker} W(c_{(n)})$ is spanned by the set $\{ t^{k-1} e^{-t} \}^n_{k=1}$.
\end{enumerate}
Hence, $ \dim \operatorname{CoKer} W(c_{(n)}) = 0, \,\, \dim \operatorname{Ker} W(c_{(n)}) = n$ and $W(c_{(n)})$ is a Fredholm operator with index $n$. \\

On the other hand, if we now calculate the generalised symbol of $W(c_{(n)})$, using equations \eqref{SGamma1}, \eqref{SGamma2plus},  \eqref{SGamma2minus}, \eqref{SGamma3plus},  \eqref{SGamma3minus} and \eqref{SGamma4} respectively, we obtain: \\

$\qquad A_{\alpha,p,s}(\omega):= 
\begin{cases} 
c_{(n)p}(\infty, \xi) &\mbox{ on } \Gamma_1 \\ 
c_{(n)}(\mp\infty) &\mbox{ on } \Gamma^\pm_2 \\ 
c_{(n)}(\mp \lambda) &\mbox{ on } \Gamma^\pm_3 \\ 
c_{(n)}(0) &\mbox{ on } \Gamma_4. \\
\end{cases}$ \\

But since $c_{(n)}(+\infty) = c_{(n)}(-\infty)$,  it is easy to see that as the triple $\omega = (x, \xi, \lambda)$ traverses the contour $\Gamma_M$, the function $A_{\alpha,p,s}(\omega)$ can be represented by simply $c_{(n)}(\lambda)$. \\

Moreover, a routine calculation shows that
\begin{equation*}
\text{winding number of } c_{(n)}(\lambda) = -n,
\end{equation*}
and, thus, we have validated the formula given by \eqref{indwindingnumber}, in the special case that $\widetilde{A}= W(c_{(n)})$. \\

\section{Generalised symbol - lower regularity}
Suppose that $1 < p < \infty, \, 1/p < s <  1+ 1/p$ and $0 < \alpha < \tfrac{1}{2}$. \\

We are interested in the solvability of the equation \eqref{WamW}
\begin{equation*} 
\big ( W(c_1) + a_2 M^0(b_2) W(c_2) + T \big ) (r_+ u_s) = g,
\end{equation*}
where the operator $T$, acting on $L_p(\mathbb{R}_+)$, is compact and from \eqref{Finalabc}
\begin{align*} 
g:&= r_+ (D-i)^{s-2\alpha} l_+ f; \nonumber \\
{c}_1(\xi) &= (1+\xi^2)^\alpha (\xi-i)^{s- 2\alpha-1}(\xi+i)^{1-s}. \nonumber  \\
{a}_2(x) &= -i C_\alpha \, \psi(\alpha+1, 2\alpha+1, x) \quad \text{(see Lemma } \ref{eUab2x}); \nonumber \\
{b}_2(\xi) &= B(s- 2\alpha +1/p' + i \xi, 2 \alpha) / \Gamma(2\alpha) ; \\
{c}_2(\xi) &= (-i \xi)^{2\alpha} (\xi-i)^{s- 2\alpha-1}(\xi+i)^{1-s},\nonumber  
\end{align*} 
and the constant $C_\alpha$ is given by
\begin{equation*} 
C_{\alpha} = -i \, \dfrac{\alpha \, 2^{2\alpha}}{\Gamma (1- \alpha)}.
\end{equation*}

Our immediate goal is to show that the operator
\begin{equation*} 
\tilde{A} := W(c_1) + a_2 M^0(b_2) W(c_2)
\end{equation*}
acting on $L_p(\mathbb{R}_+)$ is Fredholm. \\

Finally, purely for notational convenience, we define
\begin{equation*}
a_1(x)  = 1 \quad \text{and} \quad b_1(\xi)=1. \\
\end{equation*}

\subsection{Segment $\Gamma_1$} \label{Gamma1}
Firstly, we note that
\begin{align*}
a_2(0) & = -i C_\alpha \, \psi(\alpha+ 1, 2 \alpha + 1, 0) \\
& = - i C_\alpha \, 2^{-2 \alpha} \dfrac{\Gamma(2 \alpha)}{\Gamma(\alpha +1)} \quad \text{(see Lemma \ref{eUab2x})}.
\end{align*}

Hence,
\begin{align*}
a_2(0) \, b_2(\xi) & = - i C_\alpha \, 2^{-2 \alpha} \dfrac{\Gamma(2 \alpha)}{\Gamma(\alpha +1)} \cdot \dfrac{B(s- 2\alpha +1/p' + i \xi, 2 \alpha)}{\Gamma(2\alpha)} \\
& = - \dfrac{\alpha \, 2^{2\alpha}}{\Gamma (1- \alpha)} \cdot \dfrac{2^{-2 \alpha}}{\Gamma(\alpha +1)} \cdot B(s- 2\alpha +1/p' + i \xi, 2 \alpha) \\
& = - \dfrac{1}{\Gamma (1- \alpha)\Gamma(\alpha)} \cdot B(s- 2\alpha +1/p' + i \xi, 2 \alpha) \quad (\Gamma(\alpha +1 ) = \alpha \Gamma(\alpha))\\
&= -\dfrac{\sin \pi \alpha}{\pi} \cdot B(s- 2\alpha +1/p' + i \xi, 2 \alpha) \quad \text{(5.5.3, \cite{NIST}).}
\end{align*}

Therefore, on the segment $\Gamma_1$, for $-\infty \leq \xi \leq \infty$, we have 
\begin{align*}
A_{\alpha, p, s}(\omega) & := a_1(0) \, b_1(\xi) \, c_{1p}(\infty, \xi) + a_2(0) \, b_2(\xi) \, c_{2p}(\infty, \xi) \\
& = c_{1p}(\infty, \xi) -  \dfrac{\sin \pi \alpha}{\pi} \cdot B(s- 2\alpha +1/p' + i \xi, 2 \alpha) \, c_{2p}(\infty, \xi). \\
\end{align*}

From Lemmas \ref{c1limits} and \ref{c1beta}, we have
\begin{equation*}
c_{1p}(\infty, \xi)  = e^{ i \pi \nu} \dfrac{\sin [\pi (1/p + \nu - i \xi) ]}{\sin \pi (1/p - i \xi)},  \qquad \nu = 1 - s + \alpha.
\end{equation*}

Similarly, from Lemmas \ref{c2limits} and \ref{c1beta}, we have
\begin{equation*}
c_{2p}(\infty, \xi)  = e^{- i \pi \alpha} e^{ i \pi \nu'} \dfrac{\sin [\pi (1/p + \nu' - i \xi) ]}{\sin \pi (1/p - i \xi)},  \qquad \nu' = 1 - s + 2 \alpha.
\end{equation*} \\

But $e^{- i \pi \alpha} e^{ i \pi \nu'} = e^{i \pi (1-s+\alpha)} = e^{i \pi \nu}$, and thus $c_{1p}(\infty, \xi)$ and $c_{2p}(\infty, \xi)$ have a common factor
\begin{equation*}
\dfrac{e^{i \pi \nu}}{\sin \pi (1/p - i \xi)}.
\end{equation*} \\

So, we are interested in establishing the precise conditions under which the quadruple $(\alpha, p, s, \xi)$ is \textbf{not} a solution of the following transcendental equation
\begin{equation} \label{transcendeqn}
\dfrac{\sin (\pi (1/p + \nu - i \xi))}{\sin (\pi (1/p + \nu' - i \xi))} - \dfrac{\sin \pi \alpha}{\pi} \, B(s -2\alpha +1/p' +i\xi, 2\alpha) =0.
\end{equation}

Let us now define
\begin{equation} \label{Tsdefinition}
T_s := \dfrac{\sin (\pi (1/p + \nu - i \xi))}{\sin (\pi (1/p + \nu' - i \xi))} \qquad \nu = 1 -s + \alpha; \quad v' = 1 - s +2 \alpha,
\end{equation}
and
\begin{equation} \label{TBdefinition}
T_B := \dfrac{\sin \pi \alpha}{\pi} \, B(s -2\alpha +1/p' +i\xi, 2\alpha).
\end{equation}
Then, the transcendental equation \eqref{transcendeqn} simply becomes
\begin{equation*}
T_s = T_B.
\end{equation*}

\subsection{Segment $\Gamma^\pm_2$}
Similarly, on $\Gamma^+_2$, for $0 \leq x \leq \infty$, we have
\begin{align*}
A_{\alpha, p, s}(\omega) & :=  a_1(x) \, b_1(\infty) \, c_1(-\infty) + a_2(x) \, b_2(\infty) \, c_2(-\infty) \\
& =  c_1(-\infty) + 0 \\
& = e^{2 \pi \nu i},  
\end{align*}
and on $\Gamma^-_2$, for $\infty \geq x \geq 0$,  
\begin{align*}
A_{\alpha, p, s}(\omega) & :=  a_1(x) \, b_1(-\infty) \, c_1(+\infty) + a_2(x) \, b_2(-\infty) \, c_2(+\infty) \\
& = c_1(+\infty) + 0 \\
& = 1.  
\end{align*}
Hence, 
\begin{equation*}
\inf_{\omega \in \Gamma^+_2 \cup \Gamma^-_2} | A_{\alpha, p, s}(\omega)| =1. \\
\end{equation*}

\subsection{Segment $\Gamma^\pm_3$} \label{Gamma3pm}
On $\Gamma^+_3$  for $\infty > \lambda \geq 0$, 
\begin{align*}
A_{\alpha, p, s}(\omega) & := a_1(\infty) \, b_1(\infty) \, c_{1}(-\lambda) + a_2(\infty) \, b_2(\infty) \, c_{2}(-\lambda) \\
& = c_{1}(-\lambda) + 0 \\
& = c_{1}(-\lambda),  
\end{align*}
and on $\Gamma^-_3$, for $0 \leq \lambda < \infty$,
\begin{align*}
A_{\alpha, p, s}(\omega) & :=  a_1(\infty) \, b_1(-\infty) \, c_{1}(\lambda) + a_2(\infty) \, b_2(-\infty) \, c_{2}(\lambda) \\
& = c_{1}(\lambda) + 0 \\
& = c_{1}(\lambda).  
\end{align*}
Note that on $\Gamma^+_3 $ and $\Gamma^-_3 $ the parameter $\lambda$ varies between $0$ and $\infty$ but, of course, in an opposite sense. So, in summary,  
\begin{equation*}
\inf_{\omega \in \Gamma^{\pm}_3} | A_{\alpha, p, s}(\omega)| =1.\\
\end{equation*}

\subsection{Segment $\Gamma_4$}
Finally, on $\Gamma_4$, for $-\infty \leq \xi \leq \infty$,
\begin{align*}
A_{\alpha, p, s}(\omega) & := a_1(\infty) \, b_1(\xi) \, c_{1}(0) + a_2(\infty) \, b_2(\xi) \, c_{2}(0) \\
&=  c_{1}(0) + 0 \\
&= c_1(0). 
\end{align*}
Hence, 
\begin{equation*}
\inf_{\omega \in \Gamma_4} | A_{\alpha, p, s}(\omega)| =1,
\end{equation*}
and this completes the review of the contour $\Gamma_M$.

\subsection{Summary}
Note that the preceding analysis of the segments of the contour has shown that $A_{\alpha, p,s}(\omega)$ is constant on the segments $\Gamma^{\pm}_2$ and $\Gamma_4$. Therefore, it remains to consider $A_{\alpha, p,s}(\omega)$ on $\Gamma_1 \cup \Gamma^+_3 \cup \Gamma^-_3$. But from subsection \ref{Gamma3pm}, we can combine $\Gamma^\pm_3$ to give a new segment $\Gamma_3$ (say), where now the parameter $\lambda$ varies from $-\infty$ to $\infty$. (Note that, as expected, the symbol $A_{\alpha, p,s}(\omega)$ is continuous at $\lambda =0$ on the new segment $\Gamma_3$.) \\

By construction, we observe that $A_{\alpha, p,s}(\omega)$ is continuous on $\Gamma_1 \cup \Gamma_3$. Indeed, from subsection \ref{Gamma1}, on the segment $\Gamma_1$
\begin{equation} \label{AGamma1}
A_{\alpha, p,s}(\omega)=c_{1p}(\infty, \xi) - \dfrac{\sin \pi \alpha}{\pi} \, B(s -2\alpha +1/p' +i\xi, 2\alpha) \, c_{2p}(\infty, \xi)
\end{equation}
with limits $c_{1p}(\infty, \pm \infty) = c_1(\mp \infty)$ at $\xi = \pm \infty$ respectively. \\

The condition that $A_{\alpha, p,s}(\omega)=0$ on $\Gamma_1$ gives rise to the transcendental equation
\begin{equation*}
\dfrac{\sin (\pi (1/p + \nu - i \xi))}{\sin (\pi (1/p + \nu' - i \xi))} = \dfrac{\sin \pi \alpha}{\pi} \, B(s -2\alpha +1/p' +i\xi, 2\alpha),
\end{equation*}
where $\nu = 1 - s + \alpha$ and $\nu'=1-s + 2 \alpha$. \\

Finally, on $\Gamma_3$ we have
\begin{equation} \label{AGamma3}
A_{\alpha, p,s}(\omega)= c_1(\lambda)
\end{equation}
with limits $c_1(\pm \infty)$ for $\lambda = \pm \infty$. \\

\begin{remark}
It turns out that we are in the subalgebra described by Duduchava, Section 3.2, p. 524, \cite{Dudu87}. In other words, the generators of the algebra are simply of the form $aM^0(b)$ and $W(c)$, rather than $a$, $M^0(b)$ and $W(c)$ individually. \\
\end{remark}

\section{Supporting lemmas}
\begin{lemma} \label{c1limits}
Suppose $\nu = m - s + \alpha$, and the symbol $c_1(\xi)$ is given by
\begin{equation*}
c_1(\xi) = (1 + \xi^2)^\alpha \,(\xi - i)^{s-2\alpha-m} \,  (\xi + i)^{m-s}, \quad m=1 \,\, \text{or} \,\, 2.
\end{equation*}
Then
\begin{equation*}
\lim_{\xi \to \infty} c_1(\xi) = 1; \qquad \lim_{\xi \to -\infty} c_1(\xi) = e^{2 \pi \nu i},
\end{equation*}
and
\begin{equation*}
\lim_{\xi \to 0^+ } c_1(\xi) = e^{ \pi \nu i}; \qquad \lim_{\xi \to 0^- } c_1(\xi) = e^{ \pi  \nu i}.\end{equation*}
Thus, the symbol $c_1(\xi)$ has a (single) discontinuity at $\xi =\infty$. 
\end{lemma}
\begin{proof}
It is easy to see that $| c_1(\xi) | =1$ for all $\xi \in \mathbb{R}$. Therefore
\begin{equation*}
c_1(\xi) = \exp \big[ 0+ i(s-2\alpha-m) \arg(\xi-i)  + i(m-s)\arg(\xi+i) \big].
\end{equation*}
As $\xi \to \infty$, we have $\arg (\xi \pm i) \to 0$. Hence $\lim_{\xi \to \infty} c_1(\xi) = 1$. \\

On the other hand, 
\begin{align*}
\lim_{\xi \to -\infty} c_1(\xi) &= \exp \big [  i(s-2\alpha-m)(-\pi) + i(m-s) \pi \big ] \\
&= \exp \big [  i \pi(m-s - s+2\alpha+m) \big ] \\
&= \exp [ 2 \pi i \nu], \qquad (\nu = m-s+ \alpha).
\end{align*}

Moreover
\begin{align*}
\lim_{\xi \to 0^+ } c_1(\xi) &=  \exp\big[  i(s-2\alpha-m) (-\pi/2) + i(m-s) (+\pi/2) \big] \\
&= \exp[ \pi i \nu] \\
&=  \lim_{\xi \to 0^- } c_1(\xi).
\end{align*}
This completes the proof of the lemma. \\
\end{proof}

\begin{lemma} \label{c2limits}
Suppose the symbol $c_2(\xi)$ is given by
\begin{equation*}
c_2(\xi) = (-i \xi)^{2\alpha} \, (\xi +i)^{m-s} \, (\xi -i)^{-m+s-2\alpha}, \quad m=1 \,\, \text{or} \,\, 2.
\end{equation*}
Then, if $\nu' := m- s + 2 \alpha$,
\begin{equation*}
\lim_{\xi \to \infty} c_2(\xi) = \exp[-i\pi \alpha] ; \qquad \lim_{\xi \to -\infty} c_2(\xi) =  \exp[-i\pi \alpha] \, \cdot \, \exp[i \pi 2 \nu'],
\end{equation*}
and
\begin{equation*}
\lim_{\xi \to 0^+ } c_2(\xi) = 0; \qquad \lim_{\xi \to 0^- } c_2(\xi) = 0
\end{equation*}
Thus, the symbol $c_2(\xi)$ has a (single) discontinuity at $\xi =\infty$. \\
\end{lemma}

\begin{proof}
We write $(-i \xi)^{2\alpha} \, (\xi +i)^{m-s} \, (\xi -i)^{-m+s-2\alpha}$
\begin{align*}
&=\dfrac{\exp[2\alpha \log |\xi| +i 2\alpha \arg (- i \xi) +(-m+s-2\alpha) \log |\xi-i| + i (-m+s-2 \alpha) \arg (\xi-i)]}{\exp [(s-m) \log |\xi +i| + i(s-m)\arg(\xi +i) ]} \\
&=\dfrac{ \exp[2 \alpha \log |\xi| +(s-m-2\alpha) \log |\xi-i | ]}{\exp[(s-m) \log |\xi+i|]} \\
& \quad \cdot \exp [ i2\alpha \arg (-i\xi) +i(m-s) \arg (\xi+i) + i (s-m-2\alpha) \arg (\xi - i)].
\end{align*}
Hence
\begin{equation*}
\lim_{\xi \to \infty} c_2(\xi) = \exp[ i2\alpha (-\pi/2) + 0+0] =  \exp[-i\pi \alpha],
\end{equation*}
and, if $\nu' = m- s + 2 \alpha$,
\begin{align*}
\lim_{\xi \to -\infty} c_2(\xi) & = \exp [ i 2\alpha (\pi/2) +i(m-s) \pi + i (s-m-2\alpha) (-\pi)]\\
& = \exp[i \pi(2m -2s + 3 \alpha ]\\
& = \exp[-i\pi \alpha] \, \cdot \, \exp[i \pi 2 \nu'].
\end{align*}

Finally, we consider the behaviour of $c_2(\xi)$ near $\xi =0$, and note that
\begin{equation*}
\lim_{\xi \to 0^+ } |c_2(\xi)| = 0 = \lim_{\xi \to 0^- } |c_2(\xi)|
\end{equation*}
and the required results follow immediately. \\
\end{proof}

\begin{lemma} \label{c1beta}
Suppose $1 < p < \infty$ and $\nu \in \mathbb{R}$. Let   
\begin{equation*}
d_{p}(\infty, \theta) := \dfrac{e^{2 \pi \nu i}}{2} \bigg [ 1 + \coth \pi \bigg (\frac{i}{p} + \theta \bigg ) \bigg  ] + \dfrac{1}{2} \bigg [ 1 - \coth \pi \bigg ( \frac{i}{p}+ \theta \bigg ) \bigg ].
\end{equation*}
Then
\begin{equation*}
d_{p}(\infty, \theta) = e^{ \pi \nu i} \, \dfrac{\sin(\pi(1/p + \nu - i \theta))}{\sin(\pi(1/p - i \theta))}. \\
\end{equation*} 
\end{lemma}

\begin{proof}
Let $z= \pi (i/p + \theta)$. Then
\begin{align*}
d_{p}(\infty, \theta) &= \dfrac{e^{ \pi \nu i}}{2}  \bigg \{ e^{ \pi \nu i} [1 + \coth z] +  e^{ -\pi \nu i} [1 - \coth z]  \bigg \} \\
&= \dfrac{e^{ \pi \nu i}}{2 (e^{z} -e^{-z})}  \bigg \{ e^{ \pi \nu i}(e^z - e^{-z} + e^z + e^{-z}) +  e^{ -\pi \nu i} (e^z - e^{-z} - e^z - e^{-z}) \bigg \} \\
& = \dfrac{e^{ \pi \nu i}}{e^{z} -e^{-z}}  \bigg \{ e^{ \pi \nu i+z} - e^{ -\pi \nu i -z} \bigg \} \\
& = e^{ \pi \nu i}  \, \dfrac{\sinh (\pi \nu i+z)}{\sinh z} \\
& = e^{ \pi \nu i}  \, \dfrac{\sin (-\pi \nu + iz)}{\sin i z} \quad \text{since } \sin(iw) = i \sinh(w) \quad \text{(4.28.8, \cite{NIST})} \\
& = e^{ \pi \nu i} \, \dfrac{\sin(\pi(1/p + \nu - i \theta))}{\sin(\pi(1/p - i \theta))}, \quad \text{as required}. 
\end{align*}
\end{proof}

%% file: KCL_Thesis_Chapter8_v5.tex
\chapter{Index and invertibility} \label{ChapterIndInv}
\section{Main results}
\begin{theorem} \label{TheoremIndexZero}
For all $\alpha, p, s$ satisfying the conditions $0 < \alpha < \tfrac{1}{2}, \, 1 < p < \infty$ and $1/p  < s < 1+1/p$, the winding number of the generalised symbol $\big ( A_{\alpha, p,s}, \, \Gamma_M \big )$ in the complex plane is $0$. Hence, the operator $W(c_1)  + a_2M^0(b_2)W(c_2)$, defined on $L_p(\mathbb{R}_+)$, has Fredholm index equal to zero. \\
\end{theorem}

\begin{theorem}
Suppose $0 < \alpha < \tfrac{1}{2}, \, 1 < p < \infty$ and $1/p  < s < 1+1/p$. Then the operator $\mathcal{A}: H^s_p(\overline{\mathbb{R}_+}) \to H^{s-2\alpha}_p(\overline{\mathbb{R}_+})$ is invertible.
\end{theorem}
\section{Proof of first main result}
The constraints are
\begin{equation} \label{alphapsconstraints}
0 < \alpha < \tfrac{1}{2}, \,\, 1 < p < \infty  \,\, \text{and} \,\, 1/p < s < 1+1/p.
\end{equation}

Let $\alpha, p ,s$ fall within their admissible ranges and be fixed. From Chapter \ref{GenSymbol}, we know that the generalised symbol $A_{\alpha,p,s}$ can be represented by a closed contour in the complex plane given by the union of the two curves, $S_1$ and $S_3$. \\

Indeed, from Section \ref{Gamma3pm} we have,
\begin{equation} \label{LegGamma3}
S_3(\xi):= (1+\xi^2)^\alpha \, (\xi -i)^{s-2\alpha-1} \, (\xi +i)^{1-s}, \quad -\infty \leq \xi \leq \infty.
\end{equation}

Now
\begin{align*}
S_3(\xi) & = (\xi + i)^\alpha (\xi - i)^\alpha \, (\xi -i)^{s-2\alpha-1} \, (\xi +i)^{1-s} \\
&= (\xi + i)^{1-s+\alpha} (\xi - i)^{-(1-s+\alpha)}.
\end{align*}


From Section \ref{Gamma1}, for $-\infty \leq \xi \leq \infty$,
\begin{equation} \label{LegGamma1}
S_1(\xi) := S_{11}(\xi) - \dfrac{\sin \pi \alpha}{\pi} \, B(s-2\alpha + 1-1/p+i \xi, 2\alpha) S_{12} (\xi), 
\end{equation}
where 
\begin{equation*}
S_{11}(\xi):=  e^{i \pi \nu} \dfrac{\sin [\pi(1/p + \nu - i \xi) ]}{\sin [\pi(1/p - i \xi) ]}, \quad S_{12}(\xi):= e^{i \pi \nu} \dfrac{\sin [\pi(1/p + \nu + \alpha - i \xi) ]}{\sin [\pi(1/p - i \xi) ]},
\end{equation*} 
and $\nu = 1-s+\alpha$. \\

Let us now choose the values
\begin{equation} \label{modelaps}
 \alpha = \tfrac{2}{5}, \quad p =2   \quad \text{and} \quad s=\tfrac{7}{5},
\end{equation}
as the set of parameters used to define the \textit{model} contour. \\

Note that the values for $\alpha, \, p$ and $s$ given in \eqref{modelaps} satisfy the constraints described in condition \eqref{alphapsconstraints}. Moreover, $\nu = 1 - s + \alpha =0$, and thus
\begin{equation} \label{S3S11one}
S_3(\xi) =1; \quad S_{11}(\xi) = 1.
\end{equation} 

With the chosen values for $\alpha, p$ and $s$, equation \eqref{LegGamma1} becomes
\begin{equation*} 
S_1(\xi) = 1 - \dfrac{\sin \tfrac{2 \pi}{5}}{\pi} \, B(\tfrac{11}{10} + i \xi, \tfrac{4}{5}) \, \dfrac{\sin[\pi(\tfrac{9}{10} - i \xi)]}{\sin[\pi(\tfrac{1}{2} - i \xi)]}. \\
\end{equation*}

But by Lemma \ref{lemTBalphasigmaxi}, with $\sigma = \tfrac{11}{10}$ and $\alpha= \tfrac{2}{5}$, we have
\begin{equation*}
| B(\tfrac{11}{10} + i \xi, \tfrac{4}{5}) | \leq B(\tfrac{11}{10}, \tfrac{4}{5})< 1.152 < 1.2.
\end{equation*}

Moreover, by Lemma \ref{lemSabsarg}, with $a = \tfrac{9}{10}$ and $b=\tfrac{1}{2}$,
\begin{equation*}
\bigg | \dfrac{\sin[\pi(\tfrac{9}{10} - i \xi)]}{\sin[\pi(\tfrac{1}{2} - i \xi)]} \bigg | \leq \bigg ( \dfrac{\cosh 2 \pi \xi}{\cosh 2 \pi \xi + 1} \bigg)^{\tfrac{1}{2}} \leq 1,
\end{equation*}
noting that $\cos 2 \pi a >0$ and $\cos 2 \pi b = -1$. \\

Hence, we have the estimate
\begin{equation} \label{BetaRadest}
\bigg |  \dfrac{\sin \tfrac{2\pi}{5}}{\pi} \, B(\tfrac{11}{10} + i \xi, \tfrac{4}{5}) \, \dfrac{\sin[\pi(\tfrac{9}{10} - i \xi)]}{\sin[\pi(\tfrac{1}{2} - i \xi)]}\bigg | < \bigg ( \dfrac{1}{\pi} \bigg) \, 1.2 < \tfrac{2}{5}. 
\end{equation} \\

So, from \eqref{S3S11one} and \eqref{BetaRadest}, the model contour, formed by the union of the sections $S_1$ and $S_3$, is wholly contained in the disc of radius $\tfrac{2}{5}$ centred on the point $1$ in the complex plane. Hence, the winding number of the model contour must be zero. \\

But, given any set of parameters $\alpha,p,s$ satisfying the constraints $0<\alpha< \tfrac{1}{2}, \,\, 1< p < \infty$ and $1/p < s < 1+1/p$, the associated contour can be continuously deformed into the \textit{model} contour, and from Theorem \ref{TheoremTranscend}, does this without ever crossing the origin. Hence, the two contours must have the same winding number, namely, zero. \\

Therefore, see Remark \ref{WindingFred}, the operator $W(c_1)+a_2M^0(b_2)W(c_2)$, defined on $L_p(\mathbb{R}_+)$, has Fredholm index equal to zero. This completes the proof of the first theorem. \\

To complete this discussion on winding number, we now give three numerical examples of symbol plots for fixed $\alpha = 0.25$ and $p=4$, with $s$ taking the values $0.3, \, 0.7$ and $1.1$, in turn. See Figures \ref{Phoenixsmall}, \ref{Phoenixmedium} and \ref{Phoenixbig} respectively. The contour is the union of the curves $S_1$ and $S_3$, where $S_3$ forms part of the unit circle. As expected, in each case, it has winding number equal to zero.

\begin{figure}[H]
\centering
\includegraphics[width=6 cm, height=6 cm]{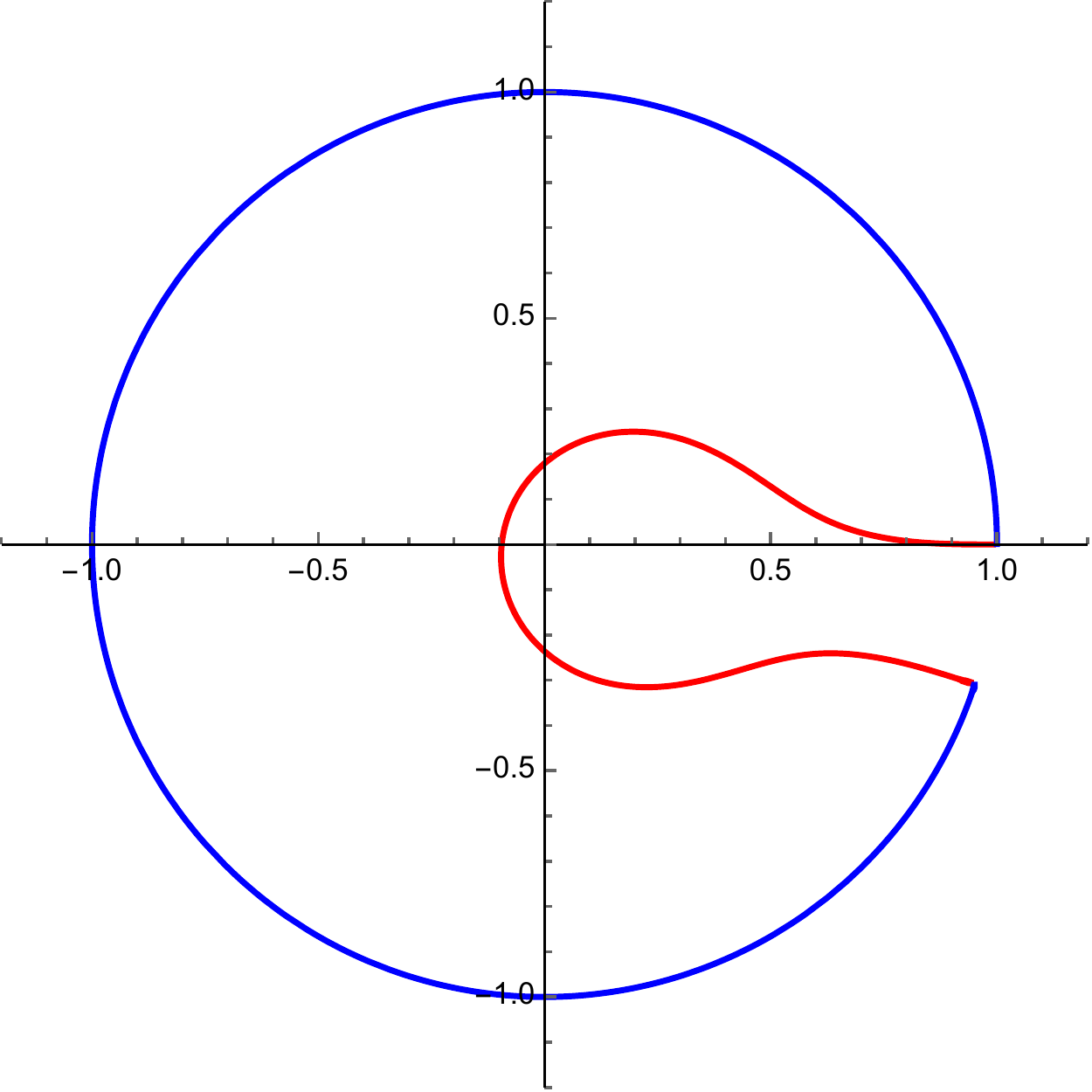} 
\caption{Symbol plot for $\alpha = 0.25, \, p =4$ and $s= 0.3$.}
\label{Phoenixsmall}
\end{figure}

Secondly, we set $s=0.7$.

\begin{figure}[H]
\centering
\includegraphics[width=6 cm, height=6 cm]{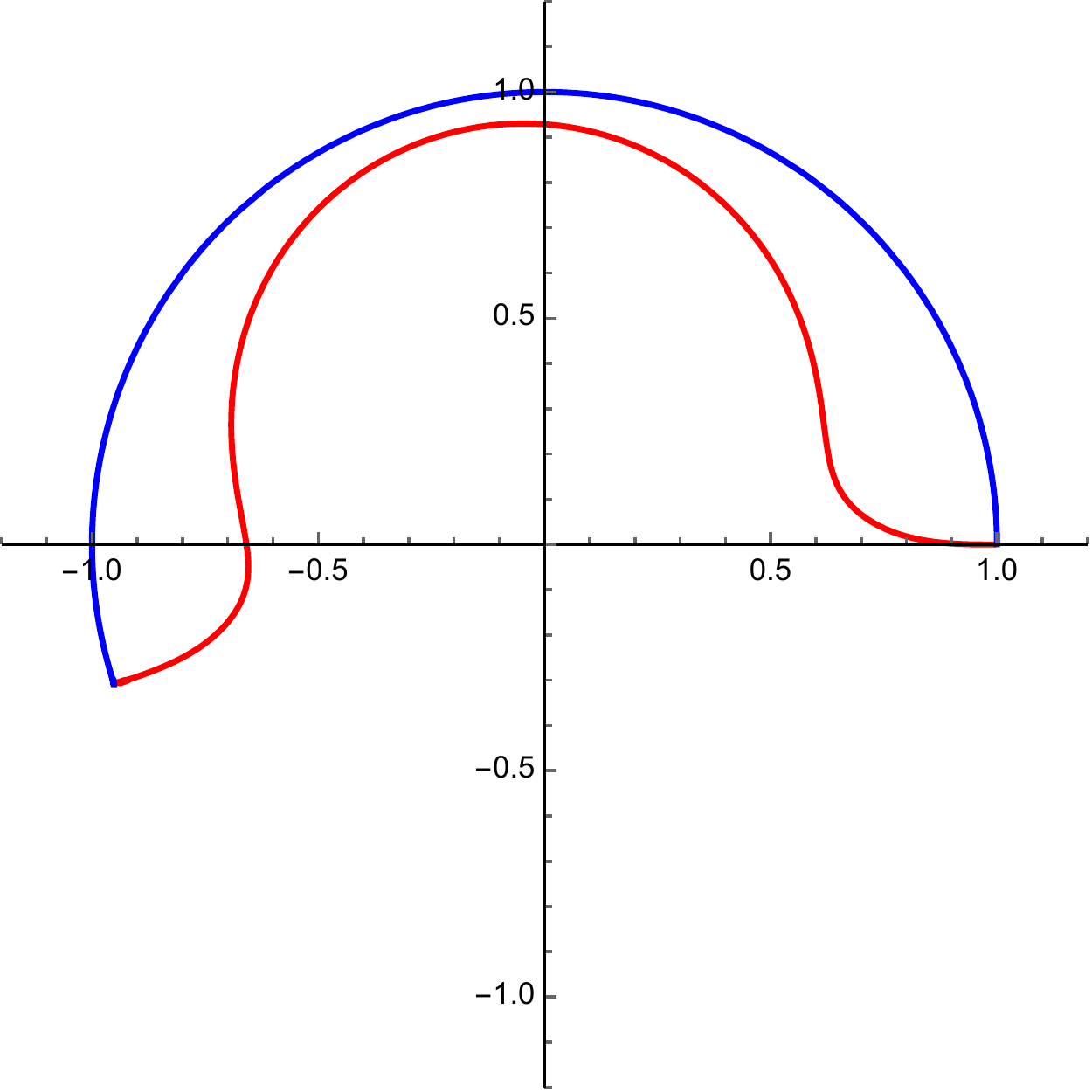} 
\caption{Symbol plot for $\alpha = 0.25, \, p =4$ and $s= 0.7$.}
\label{Phoenixmedium}
\end{figure}

Finally, we take $s=1.1$.

\begin{figure}[H]
\centering
\includegraphics[width=6 cm, height=6 cm]{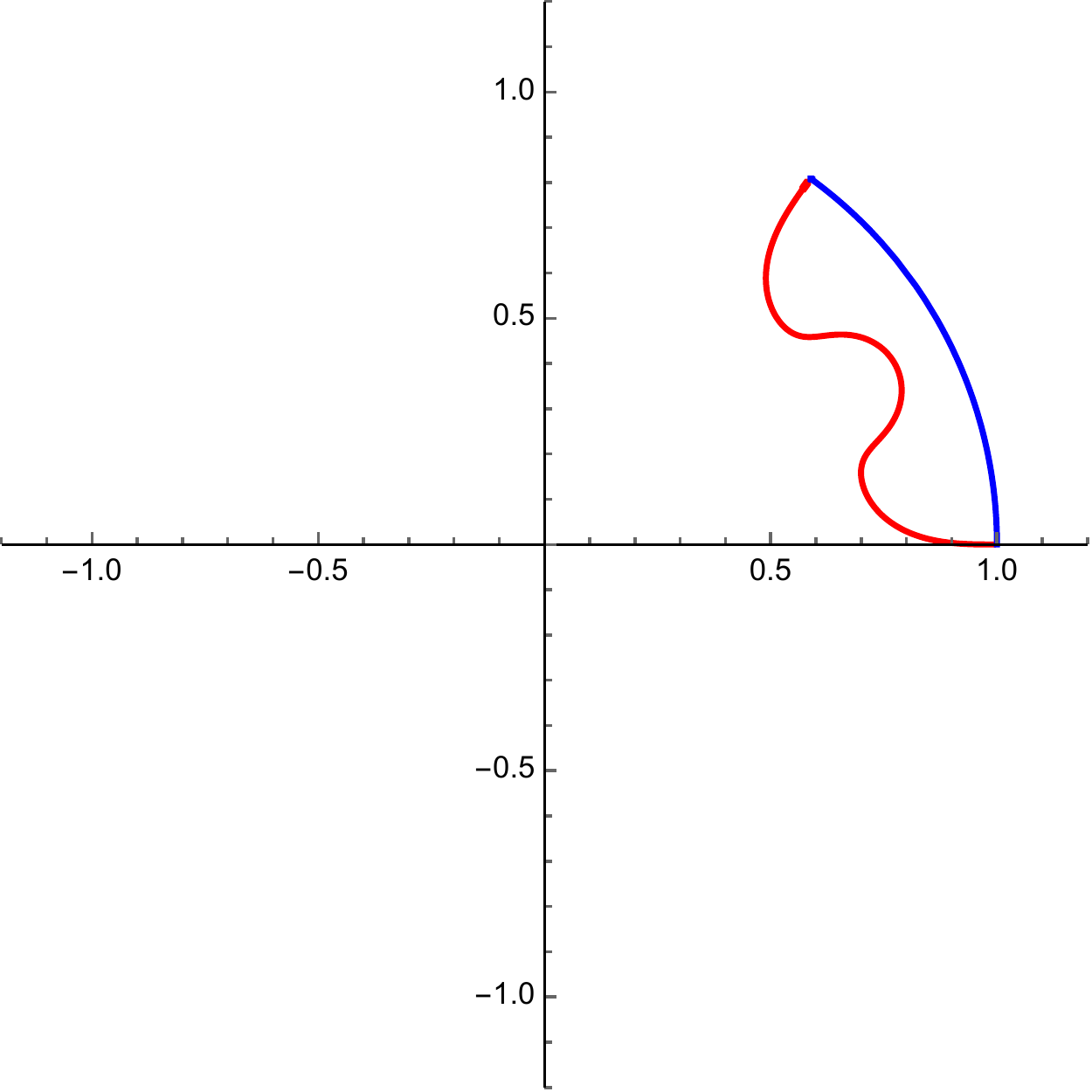}
\caption{Symbol plot for $\alpha = 0.25, \, p =4$ and $s= 1.1$.}
\label{Phoenixbig}
\end{figure}

\section{Proof of second main result} \label{SecondTheoremIndex}
Suppose $0 < \alpha < \tfrac{1}{2}, \, 1 < p < \infty$ and $1/p < s < 1+1/p$. Let $u \in H^s_p(\overline{\mathbb{R}_+})$. Then, see Lemma \ref{lemma:us}, 
\begin{equation*}
u_s := (D+i)^{s-1}e_+(D-i)u,
\end{equation*}
and $u_s \in L_p(\mathbb{R})$ with supp $u_s \in \overline{\mathbb{R}_+}$. Indeed, see Lemma \ref{ustou}, we can also write
\begin{equation*}
u = (r_+(D-i)^{-1}(D+i)^{1-s} e_+)(r_+ u_s).
\end{equation*}

Moreover, from Theorem \ref{TheoremIndexZero}, the operator 
\begin{equation}
\widetilde{A} := W(c_1) + a_2 M^0(b_2)W(c_2),
\end{equation}
defined on $L_p(\mathbb{R}_+)$, has Fredholm index equal to zero. \\

By Theorem \ref{TheoremaWMplusT}, 
\begin{equation*}
\big (r_+ (D-i)^{s-2\alpha} l_+ \big ) \,\, \mathcal{A} \,\,  \big ( r_+(D-i)^{-1}(D+i)^{1-s} e_+ \big ) = \widetilde{A} + T
\end{equation*}
where $T$, defined on $L_p(\mathbb{R}_+)$, is a compact operator. \\

The operators $r_+ (D-i)^{s-2\alpha} l_+ : H^{s-2\alpha}_p(\overline{\mathbb{R}_+}) \to L_p(\mathbb{R}_+)$ and $r_+(D-i)^{-1}(D+i)^{1-s} e_+ : L_p(\mathbb{R}_+) \to H^{s}_p(\overline{\mathbb{R}_+})$ are both invertible. Therefore,
\begin{equation}
\operatorname{ind} \mathcal{A} = \operatorname{ind} (\widetilde{A} + T)  = \operatorname{ind} \widetilde{A} =0,
\end{equation}
since the Fredholm index is stable under compact perturbations. \\ \\

On the other hand, from Theorem \ref{thmabddtrivker}, the operator $\mathcal{A}: H^s_2(\overline{\mathbb{R}_+}) \to H^{s-2\alpha}_2(\overline{\mathbb{R}_+})$ (is bounded and) has a trivial kernel. The following lemma will allow us to to generalise this result from $p=2$ to the full range $1 < p < \infty$. 

\begin{lemma} \footnote{This result was an indirect communication, via a third party, from Vladimir Pilidi to Eugene Shargorodsky.} \label{FredholmKer}
Let $X_1, X_2, Y_1, Y_2$ be Banach spaces such that $X_1$ ($Y_1$) is continuously and densely embedded into $X_2$ (into $Y_2$, respectively).
Suppose $A : X_j \to Y_j$ is Fredholm, $j = 1, 2$, and
$$
\mbox{Ind}_{X_1 \to Y_1} A = \mbox{Ind}_{X_2 \to Y_2} A   .
$$
Then
$$
\mbox{Ker}_{X_1 \to Y_1} A = \mbox{Ker}_{X_2 \to Y_2} A .
$$
\end{lemma}
\begin{proof}
Since the above embeddings are dense,  $X^*_2 \hookrightarrow X^*_1$, $Y^*_2 \hookrightarrow Y^*_1$. The operator $A^* : Y^*_j \to X^*_j$
 is Fredholm, $j = 1, 2$. Let
$$
\alpha_j :=  \mbox{dim Ker}_{X_j \to Y_j} A , \ \ \beta_j :=  \mbox{dim Ker}_{Y^*_j \to X^*_j} A^* , \ \ \  j = 1, 2. 
$$
 Since
\begin{equation}\label{sbs}
\mbox{Ker}_{X_1 \to Y_1} A \subseteq \mbox{Ker}_{X_2 \to Y_2} A , \  \
\mbox{Ker}_{Y^*_2 \to X^*_2} A \subseteq \mbox{Ker}_{Y^*_1 \to X^*_1} A^* ,  
\end{equation}
we have
$$
\alpha_1 \le \alpha_2 , \ \ \beta_1 \ge \beta_2 .
$$
Since
$$
\alpha_1 - \beta_1 = \mbox{Ind}_{X_1 \to Y_1} A = \mbox{Ind}_{X_2 \to Y_2} A = \alpha_2  - \beta_2 ,
$$
we conclude that $\alpha_1 =  \alpha_2$,  $\beta_1 = \beta_2$. Hence the inclusions \eqref{sbs} are in fact equalities. \\
\end{proof}
To complete the proof of the second main result, we now consider (the dimension of) $\operatorname{Ker} \widetilde{A}$, for the cases $p>2$ and $p<2$ respectively. \\

\textbf{Firstly, suppose $p > 2$.} Then, for $0 < \delta <  1$, we define
\begin{equation*}
X_1:= H^{\tfrac{1}{2} + \delta}_{2}(\overline{\mathbb{R}_+}), \quad Y_1:= H^{\tfrac{1}{2} + \delta - 2 \alpha}_{2}(\overline{\mathbb{R}_+})
\end{equation*}
and
\begin{equation*}
X_2:= H^{\tfrac{1}{p} + \delta}_{p}(\overline{\mathbb{R}_+}), \quad Y_2:= H^{\tfrac{1}{p} + \delta - 2 \alpha}_{p}(\overline{\mathbb{R}_+}).
\end{equation*}
Then $X_1$ ($Y_1$) is continuously and densely embedded into $X_2$ (into $Y_2$, respectively). Moreover, $\mathcal{A} : X_j \to Y_j$ is Fredholm, $j = 1, 2$, and
\begin{equation*}
\operatorname{Ind}_{X_1 \to Y_1} \mathcal{A} =\operatorname{Ind}_{X_2 \to Y_2} \mathcal{A} \quad (=0).
\end{equation*}
Therefore, by Lemma \ref{FredholmKer}, 
\begin{equation*}
\operatorname{Ker}_{X_1 \to Y_1} \mathcal{A} =\operatorname{Ker}_{X_2 \to Y_2} \mathcal{A} .
\end{equation*}
That is,
\begin{equation*}
\operatorname{Ker}_{X_2 \to Y_2} \mathcal{A} = \{0\}.
\end{equation*} \\

\textbf{Secondly, suppose $p < 2$.} Then, for $0 < \delta <  1$, we define
\begin{equation*}
X_2:= H^{\tfrac{1}{2} + \delta}_{2}(\overline{\mathbb{R}_+}), \quad Y_2:= H^{\tfrac{1}{2} + \delta - 2 \alpha}_{2}(\overline{\mathbb{R}_+})
\end{equation*}
and
\begin{equation*}
X_1:= H^{\tfrac{1}{p} + \delta}_{p}(\overline{\mathbb{R}_+}), \quad Y_1:= H^{\tfrac{1}{p} + \delta - 2 \alpha}_{p}(\overline{\mathbb{R}_+}).
\end{equation*}
We can now repeat the argument made above, for the case $p>2$, to show that
\begin{equation*}
\operatorname{Ker}_{X_1 \to Y_1} \mathcal{A} = \{0\}.
\end{equation*} \\

So, finally,  the operator $\mathcal{A}: H^s_p(\overline{\mathbb{R}_+}) \to H^{s-2\alpha}_p(\overline{\mathbb{R}_+})$ is invertible.

%% file: KCL_Thesis_Chapter9_v5.tex
\chapter{Higher regularity} \label{ChapterHigherReg}
\section{Problem definition}
Suppose $1 < p < \infty$ and $0 < \alpha < 1$. We now assume $1 + 1/p <  s < 2 + 1/p$. Let $A$ denote the pseudodifferential operator of order $2 \alpha$, with symbol,  see \eqref{symboldef},
\begin{equation*}
A(\xi) = (1+ \xi^2)^\alpha. 
\end{equation*}
Our problem is to investigate the solvability of equation \eqref{rAef}
\begin{equation*}
r_+ \, A \, e_+ u + u \, r_+ A( \chi_{\mathbb{R}_{-}}) = f, 
\end{equation*}
where 
 $u \in H^s_p(\overline{\mathbb{R}_+})$ for a given $f \in H^{s-2\alpha}_p(\overline{\mathbb{R}_+})$, subject to the boundary condition
\begin{equation} \label{bczero}
u'(0)=0. \\
\end{equation}

\section{Reformulation}
As a first step in reformulating equation \eqref{rAef}, it will be convenient to define
\begin{equation} \label{Aminus2}
A^{=}(D) := A(D) (D-i)^{-2}, 
\end{equation}
where $D = i \frac{\partial}{\partial x}$. \\

Let $\delta$ denote the \textit{Dirac delta} function and let $\chi_G$ denote the characteristic function of $G$. Now $\chi'_{\mathbb{R}_{+}}(x) = \delta(x)$, see Example 1.3, p. 10, \cite{Es}, and $\chi_{\mathbb{R}_{+}}(x) + \chi_{\mathbb{R}_{-}}(x) =1$. Therefore, $\chi'_{\mathbb{R}_{-}}(x) = - \delta(x)$, and we can write
\begin{align*}
r_+ \, A( \chi_{\mathbb{R}_{-}}) &= r_+ \, A \, (D - i)^{-2} (D - i)^2  \chi_{\mathbb{R}_{-}} \\
&= r_+ \, A^{=} (D - i)  ( - i \, \delta  - i \chi_{\mathbb{R}_{-}} ) \\
&= r_+ \, A^{=} ( \delta' - 2 \delta  - \chi_{\mathbb{R}_{-}} ). 
\end{align*}

Moreover, from Lemma \ref{lemma:D2eu}, we have the identity, 
\begin{equation*}
(D-i)^2 e_+ u = e_+ (D-i)^2 u -u(0) \, \delta' + 2u(0) \, \delta - u'(0)\, \delta.
\end{equation*}
Using the boundary condition \eqref{bczero} and equation \eqref{Aminus2} 
\begin{align*}
r_+ A e_+ u &=  r_+ A^{=}(D-i)^2 e_+ u \\
&= r_+ A^{=} e_+ (D-i)^2 u - u(0)r_+ A^{=} \, (\delta' - 2\delta).
\end{align*}
Hence, we can rewrite equation \eqref{rAef} as
\begin{equation} \label{rA(s-2)eg}
r_+ A^{=} e_+ (D-i)^2 u  + (u(x)-u(0)) \, r_+ A^{=}  \, (\delta' - 2 \delta) - u(x) r_+ \, A^{=} ( \chi_{\mathbb{R}_{-}} ) = f.  
\end{equation} 

We now define
\begin{equation} \label{A2minuss}
A_{s}(D) := A(D) (D-i)^{-2} (D+i)^{2-s} = A^{=}(D+i)^{2-s}, 
\end{equation}
and hence
\begin{align*}
r_+ A^{=} e_+ (D-i)^2 u & = r_+ A_{s} (D+i)^{s-2}e_+ (D-i)^2 u \\
& = r_+ A_{s} u_s, 
\end{align*}
where we set
\begin{equation} \label{us2defn}
u_s := (D+i)^{s-2}e_+ (D-i)^2u.
\end{equation} 
Moreover, from Lemma \ref{lemma:u2s}, we have $u_s \in L_p(\mathbb{R})$ with supp $u_s \subseteq \overline{\mathbb{R}_+}$. \\

Then, it follows directly from \eqref{rA(s-2)eg}, that equation \eqref{rAef} becomes
\begin{equation} \label{rA=s}
r_+ \, A_{s} \, u_s + (u(x)-u(0)) \, r_+ A^{=}  \, (\delta' - 2 \delta) - u(x) r_+ \, A^{=} ( \chi_{\mathbb{R}_{-}} ) = f. 
\end{equation} 


\section{Operator algebra - initial step} \label{OpAlgInit2}
Our starting point for this section is equation \eqref{rA=s}. We remark that the given function $f \in H^{s-2\alpha}_p(\overline{\mathbb{R}_+})$. In a subsequent section, we shall apply the operator $r_+  (D-i)^{s-2\alpha}l_+$ to each side of equation \eqref{rA=s}, since our ultimate goal is a formulation in $L_p(\mathbb{R}_+)$. \\

For the time being however, we recast equation \eqref{rA=s} in the form
\begin{equation} \label{aMCtilde2}
\tilde{a}_{0}(x) u(0)  + \sum^N_{j=1}  \tilde{a}_j(x) \, M^0(\tilde{b}_j) \, (r_+ \tilde{C}_j e_+)(r_+ u_s) + \tilde{K} u = f,
\end{equation}
where $ \tilde{K}: H^{s}_{p,0}(\overline{\mathbb{R}_+}) \to H^{s-2\alpha}_p(\overline{\mathbb{R}_+})$ is compact. (The definition of the space $H^{s}_{p,0}(\overline{\mathbb{R}_+})$, for $s > 1 +1/p$, is given in \eqref{Hsp0definition}.) \\

In equation \eqref{aMCtilde2}, for $k=0,1, \dots, N$ the functions $\tilde{a}_k(x)$ are known. Moreover, for $j=1,2, \dots , N, \, \tilde{b}_j$ is the symbol of a Mellin convolution operator and $\tilde{C}_j$ is a pseudodifferential operator. We shall denote the symbol of $\tilde{C}_j(D)$ by $\tilde{c}_j(\xi)$. We now examine the individual summands in the left-hand side of equation \eqref{rA=s}. 

\subsection{First term}
Consider the term $r_+ \, A_{s} \, u_s$. From equation \eqref{A2minuss}, $A_{s}(D) := A(D) (D-i)^{-2} (D+i)^{2-s}$ and hence, we can write
\begin{equation*}
r_+ \, A_{s} \, u_s = (r_+ A(D) (D-i)^{-2} (D+i)^{2-s} e_+)(r_+ u_s),
\end{equation*}
since, by Lemma \ref{lemma:u2s}, $u_s \in L_p(\mathbb{R})$ and supp $u_s \subseteq \overline{\mathbb{R}_+}$. \\

Thus, in the notation of equation \eqref{aMCtilde2}, 
\begin{align} \label{2tildea1}
\tilde{a}_1(x) &= 1;  \nonumber \\
\tilde{b}_1(\xi) &= 1;  \\
\tilde{c}_1(\xi) &= (1+\xi^2)^\alpha (\xi-i)^{-2}(\xi+i)^{2-s}. \nonumber  
\end{align} 

\subsection{Middle term}
Now consider the middle term, $(u(x)-u(0)) \, r_+ A^{=}  \, (\delta' - 2 \delta)$. It will be convenient to write 
\begin{equation*}
(u(x)-u(0)) \, r_+ A^{=}  \, (\delta' - 2 \delta) = (u(x)-u(0)) \, r_+ A^{=}  \, (\delta' - \delta) - (u(x)-u(0)) \, r_+ A^{=}  \,  \delta.
\end{equation*}

\subsubsection{Middle term - first part} \label{BigalphaMiddleFirst}
From Lemma \ref{lemma:rA(s-2)deltadashdelta},
\begin{equation*}
(r_+ \, A^{=} (\delta'- \delta)) (x) = - iC_{\alpha}  \, e^{-x} \, U(\alpha+1, 2 \alpha+1, 2x),
\end{equation*}
where the constant $C_\alpha$ only depends on $\alpha$, and is given by equation \eqref{defCalpha} in the statement of Lemma \ref{lemma:rA(s-1)delta} as 
\begin{equation*}
C_{\alpha} = -i \dfrac{\alpha \, 2^{2\alpha}}{\Gamma (1- \alpha)}. \\
\end{equation*}
From Lemma \ref{eUab2x},
\begin{equation*}
(r_+ \, A^{=} (\delta'- \delta)) (x) = -iC_\alpha \big ( \phi(x) + \vartheta(x) \log x + x^{-2 \alpha}  \psi(\alpha + 1, 2\alpha+1, x) \big ),
\end{equation*}
where $ \phi, \vartheta \in C^\infty(\mathbb{R})$ and, together with their derivatives, are bounded and $O(e^{-x})$ as $x \to +\infty$. (We can set $\vartheta$ to be identically zero unless $\alpha = \tfrac{1}{2}$.) Moreover, $ \psi \in C^\infty_0(\mathbb{R})$ with $ \psi(\alpha + 1, 2\alpha+1, x) =0$ for $x >2$. \\

Hence, we can write
\begin{equation*}
(u(x)-u(0)) r_+ \, A^{=} (\delta'- \delta) = -i C_\alpha \, (T_{11} + T_{12} + T_{13}) u
\end{equation*}
where
\begin{align*}
T_{11}u(x) &:= \phi(x) (u(x)- u(0)) \\
T_{12}u(x) &:=  \vartheta(x) \log x \, (u(x)- u(0)) \\
T_{13}u(x) &:= x^{-2\alpha} \psi(\alpha + 1, 2\alpha+1, x) (u(x) - u(0)).
\end{align*}
Firstly, we will show that $T_{11} : H^s_{p,0}(\overline{\mathbb{R}_+}) \to H^{s-\epsilon}_p(\overline{\mathbb{R}_+})$, is compact. Now
\begin{equation*}
\phi(x) (u(x) - u(0)) = \phi(x)  e^{x/2} \cdot e^{-x/2}(u(x)-u(0)).
\end{equation*}
Since $u'(0)=0$, by Remark \ref{phiuu0big}, the map $u \mapsto e^{-x/2}(u(x)-u(0))$ defines a bounded operator from $H^s_{p,0}(\overline{\mathbb{R}_+})$ to $\widetilde{H}^s_p(\overline{\mathbb{R}_+})$. 
Moreover, $\phi(x) e^{x/2} \in H^s_p(\overline{\mathbb{R}_+})$, since it and its derivatives are bounded, smooth and $O(e^{-x/2})$ as $ x \to + \infty$. The compactness of the operator $T_{11}: H^s_{p,0}(\overline{\mathbb{R}_+}) \to H^{s-\epsilon}_p(\overline{\mathbb{R}_+})$ now follows immediately from Lemma \ref{multiplierbHcompact}. \\

Secondly, we will show that $T_{12} : H^s_{p,0}(\overline{\mathbb{R}_+}) \to H^{s-\epsilon}_p(\overline{\mathbb{R}_+})$, is compact. Now
\begin{equation*}
\vartheta(x) \log x (u(x) - u(0)) = \vartheta(x) e^{x/2} \cdot e^{-x/4} \log x \cdot e^{-x/4}(u(x)-u(0)).
\end{equation*}
Since $u'(0)=0$, by Remark \ref{phiuu0big}, the map $u \mapsto e^{-x/4}(u(x)-u(0))$ defines a bounded operator from $H^s_{p,0}(\overline{\mathbb{R}_+})$ to $\widetilde{H}^s_p(\overline{\mathbb{R}_+})$. 
Further, from Corollary \ref{LogMultiplier}, $e^{-x/4} \log x I$ defines a bounded operator from $\widetilde{H}^s_p(\overline{\mathbb{R}_+})$ to $\widetilde{H}^{s-\tfrac{\epsilon}{2}}_p(\overline{\mathbb{R}_+})$. 
Finally, $\vartheta(x) e^{x/2} \in H^{s-\tfrac{\epsilon}{2}}_p(\overline{\mathbb{R}_+})$, since it and its derivatives are bounded, smooth and $O(e^{-x/2})$ as $ x \to + \infty$. The compactness of the operator $T_{12}: H^s_{p,0}(\overline{\mathbb{R}_+}) \to H^{s-\epsilon}_p(\overline{\mathbb{R}_+})$ now follows immediately from Lemma \ref{multiplierbHcompact}. \\

It remains to consider $-i C_\alpha \,T_{13}u(x) = -i C_\alpha x^{-2\alpha} \psi(\alpha + 1, 2\alpha+1, x) (u(x) - u(0))$, and it is convenient to write
\begin{align*}
-i C_\alpha & x^{-2\alpha} \psi(\alpha + 1, 2\alpha+1, x) (u(x) - u(0)) \\
&= -i C_\alpha \, \psi(\alpha + 1, 2\alpha+1, x) \cdot \big \{ x^{-2\alpha} (u(x) - u(0)) \big \} ,
\end{align*}
noting that $ \psi \in C^\infty_0(\mathbb{R})$ with $\psi(\alpha + 1, 2\alpha+1, x) =0$ for $x >2$. \\

On the other hand, from Lemmas \ref{lemma:mellinop1} and \ref{lemma:mellinop12},
\begin{equation*}
x^{-2 \alpha} (u(x) - u(0)) = \int^\infty_0 K_{2\alpha} \bigg( \dfrac{x}{y} \bigg ) h(y) \, \dfrac{dy}{y} \quad := {M}_{2\alpha} h,
\end{equation*}
where $h(x) = (C^{2\alpha}_{0^+}u) (x)$. Moreover, from Lemmas \ref{htous2} and \ref{htous2supplement} 
\[
h=
\begin{cases}
(r_+ C(D) e_+) (r_+ u_s), & \alpha = \tfrac{1}{2}; \\
(r_+ C(D) e_+) (r_+ u_s) + \dfrac{u(0)}{\sqrt{2\pi}} \, r_+ \mathcal{F}^{-1} (-i \xi)^{2\alpha-1}(\xi-i)^{-2}, & \alpha \not = \tfrac{1}{2},
\end{cases}
\]
where $C(D)$ has the symbol $c(\xi) = (-i \xi)^{2\alpha} (\xi + i)^{2-s} (\xi-i)^{-2}$.
From Lemma \ref{lemma:mellinop2}, ${M}_{2\alpha}$ is a Mellin convolution operator with symbol $b(\xi) = B(1/p' + i \xi, 2 \alpha) / \Gamma(2\alpha) $.\\

Thus, in the notation of equation \eqref{aMCtilde2}, we have
\begin{align} \label{2tildea21}
\tilde{a}_{2}(x) &= -i C_\alpha \psi(\alpha + 1, 2\alpha+1, x);\nonumber \\
\tilde{b}_{2}(\xi) &= B(1/p' + i \xi, 2 \alpha) / \Gamma(2\alpha) ;\\
\tilde{c}_{2}(\xi) &= (-i \xi)^{2\alpha} (\xi + i)^{2-s} (\xi-i)^{-2}.\nonumber  
\end{align}
and for $\alpha \not = \tfrac{1}{2}$ 
\begin{equation} \label{2tildea0}
\tilde{a}_{0}(x) = \tilde{a}_{2}(x) M^0(\tilde{b}_{2})  \dfrac{1}{\sqrt{2\pi}} \, r_+ \mathcal{F}^{-1} (-i \xi)^{2\alpha-1}(\xi-i)^{-2}. 
\end{equation}

\subsubsection{Middle term - second part}
From Lemma \ref{lemma:rA(s-2)delta}
\begin{equation*}
(r_+ \, A^{=}  \delta) (x) = \tfrac{1}{2} \, i C_{\alpha}  \, e^{-x} \, U(\alpha+1, 2 \alpha, 2x),
\end{equation*}
where the constant $C_\alpha$ only depends on $\alpha$, and is given by equation \eqref{defCalpha} in the statement of Lemma \ref{lemma:rA(s-1)delta} as 
\begin{equation*}
C_{\alpha} = -i \dfrac{\alpha \, 2^{2\alpha}}{\Gamma (1- \alpha)}. \\
\end{equation*}

From Lemma \ref{eUab2x},
\begin{equation*}
(r_+ \, A^{=} \delta) (x) = \tfrac{1}{2} \, i C_{\alpha} \big ( \phi(x) + \vartheta(x) \log x + x^{1-2 \alpha}  \psi(\alpha + 1, 2\alpha, x) \big ),
\end{equation*}
where $ \phi, \vartheta \in C^\infty(\mathbb{R})$ and, together with their derivatives, are bounded and $O(e^{-x})$ as $x \to +\infty$. (We can set $\vartheta$ to be identically zero unless $\alpha = \tfrac{1}{2}$.) Moreover, $ \psi \in C^\infty_0(\mathbb{R})$ with $ \psi(\alpha + 1, 2\alpha, x) =0$ for $x >2$. \\

Hence, we can write
\begin{equation*}
-(u(x)-u(0)) r_+ \, A^{=} \delta) = -\tfrac{1}{2} \, i C_{\alpha} \, (T_{21} + T_{22})u
\end{equation*}
where
\begin{align*}
T_{21}u(x) &:= (\phi(x) + \vartheta(x) \log x) (u(x)- u(0)) \\
T_{22}u(x) &:= x^{1-2\alpha} \psi(\alpha + 1, 2\alpha, x) (u(x) - u(0)).
\end{align*}

From the earlier part of this section,  $T_{21}: H^s_{p,0}(\overline{\mathbb{R}_+}) \to H^{s-\epsilon}_p(\overline{\mathbb{R}_+})$ is compact. \\

It remains to consider the operator $T_{22}$. \\

Firstly, suppose that $0 < \alpha  \leq \tfrac{1}{2}$. Then we can write
\begin{align*}
x^{1-2\alpha} & \psi(\alpha + 1, 2\alpha, x) (u(x) - u(0)) \\
 & = \psi(\alpha + 1, 2\alpha, x) e^{x/2} \cdot  x^{1-2\alpha} e^{-x/4} \cdot e^{-x/4} (u(x)-u(0)).
\end{align*}
Since $u'(0)=0$, by Remark \ref{phiuu0big}, the map $u \mapsto e^{-x/4}(u(x)-u(0))$ defines a bounded operator from $H^s_{p,0}(\overline{\mathbb{R}_+})$ to $\widetilde{H}^s_p(\overline{\mathbb{R}_+})$. 
Further, from Lemma \ref{multxgammaa}, the operator $x^{1-2\alpha} e^{-x/4} I$ is bounded on $\widetilde{H}^s_p(\overline{\mathbb{R}_+})$. 
Now $\psi(\alpha + 1, 2\alpha, x) e^{x/2}$ is bounded, smooth and has compact support. Finally, for $0 < \alpha  \leq \tfrac{1}{2}$, the compactness of the operator $T_{22}: H^s_{p,0}(\overline{\mathbb{R}_+}) \to H^{s-\epsilon}_p(\overline{\mathbb{R}_+})$ now follows immediately from Lemma \ref{multiplierbcompact}. \\

Secondly, suppose that $\tfrac{1}{2} < \alpha < 1$. Then we can write
\begin{align*}
x^{1-2\alpha} & \psi(\alpha + 1, 2\alpha, x) (u(x) - u(0))\\
& = \psi(\alpha + 1, 2\alpha, x) e^{x/2} \cdot  x^{1-2\alpha} \cdot e^{-x/2} (u(x)-u(0)).
\end{align*}
Since $u'(0)=0$, by Remark \ref{phiuu0big}, the map $u \mapsto e^{-x/2}(u(x)-u(0))$ defines a bounded operator from $H^s_{p,0}(\overline{\mathbb{R}_+})$ to $\widetilde{H}^s_p(\overline{\mathbb{R}_+})$. 
Further, from Lemma \ref{x-alphaI}, the operator $x^{1-2\alpha}I$ is bounded from $\widetilde{H}^s_p(\overline{\mathbb{R}_+})$ to $\widetilde{H}^{s - 2\alpha + 1}_p(\overline{\mathbb{R}_+})$. 
Now $\psi(\alpha + 1, 2\alpha, x) e^{x/2}$ is bounded, smooth and has compact support. Finally, for $\tfrac{1}{2} < \alpha < 1$, the compactness of the operator $T_{22}: H^s_{p,0}(\overline{\mathbb{R}_+}) \to H^{s-2\alpha + 1- \epsilon}_p(\overline{\mathbb{R}_+})$ now follows immediately from Lemma \ref{multiplierbcompact}. \\

In other words, for $0 < \alpha < 1$, the operator $T_{22}: H^s_{p,0}(\overline{\mathbb{R}_+}) \to H^{s - 2\alpha}_p(\overline{\mathbb{R}_+})$ is also compact. \\

\subsection{Final term}
It remains to consider the last term, $- u(x) r_+ \, A^{=} ( \chi_{\mathbb{R}_{-}} )$. From Lemma \ref{a3(x)2}
\begin{equation*}
(r_+ \, A^= \, \chi_{\mathbb{R}_{-}}) (x) = \tfrac{1}{2} \, i C_{\alpha} \int^\infty_x e^{-t} U(\alpha+1, 2\alpha, 2t) \, dt,
\end{equation*}
where the constant $C_\alpha$ only depends on $\alpha$. \\

Suppose $0 < \alpha < 1$ and $\alpha \not = \tfrac{1}{2}$. Then, from Remark \ref{rAminus2chiminus},
\begin{equation*}
(r_+ \, A^= \, \chi_{\mathbb{R}_{-}}) (x) = \tfrac{1}{2} \, i C_{\alpha} \, (\phi_1(x) + x^{2-2\alpha} \phi_2(x)),
\end{equation*}
where $\phi_1, \phi_2 \in C^\infty(\mathbb{R})$ and, together with their derivatives, are bounded and $O(e^{-x})$ as $x \to +\infty$. \\

Hence, we can write
\begin{equation*}
- u r_+ A^=(\chi_{\mathbb{R}_-}) = - \tfrac{1}{2} i C_\alpha \, (T_{31} + T_{32} + T_{33}) u
\end{equation*}
where
\begin{align*}
T_{31} u(x) &:= \phi_1(x) u(x) \\
T_{32} u(x) &:= x^{2-2\alpha} \phi_2(x) (u(x) - u(0))\\
T_{33} u(x) &:= x^{2-2\alpha} \phi_2(x) u(0).
\end{align*}
Since, by assumption, $u'(0)=0$, the compactness of the operators  $T_{31},T_{32}$ and $T_{33}$ now follows in the same manner as Section \ref{Smallalphafinalterm}. \\

It remains to consider the case $\alpha = \tfrac{1}{2}$. \\

The analysis proceeds as for the case $\alpha \not = \tfrac{1}{2}$, but, see Remark \ref{rAminus2chiminus}, we need also to consider a term of the form $\vartheta_1 (x) \, x \log x $, where  $\vartheta_1 \in C^\infty(\mathbb{R})$ and, together with its derivatives, is bounded and $O(e^{-x})$ as $x \to +\infty$. \\

For $u \in H^s_{p,0}(\overline{\mathbb{R}_+})$, let us consider the map
\begin{equation*}
u(x)  \mapsto \vartheta_1(x) \, x \log x \cdot u(x).
\end{equation*}
We write
\begin{equation*}
\vartheta_1 (x) \, x \log x \cdot u(x) = T_{34} u(x) + T_{35} u(x),
\end{equation*}
where
\begin{align*}
T_{34} u(x) &= \vartheta_1 (x) e^{x/2} \, \cdot \, e^{-x/4} x \log x   \, \cdot \, e^{-x/4} (u(x) - u(0)) \\
T_{35} u(x) &= \vartheta_1(x) e^{x/2} \, \cdot \, e^{-x/4} \log x \, \cdot \, x^{-1} \cdot x^2 e^{-x/4} u(0).
\end{align*}

We will show that $T_{34}$ and $T_{35}$ represent compact operators. \\

Firstly, consider $T_{34}$. From Remark \ref{phiuu0big}, $u \mapsto e^{-x/4} (u(x) - u(0))$ defines a bounded operator from $H^s_{p,0}(\overline{\mathbb{R}_+})$ to $\widetilde{H}^s_{p}(\overline{\mathbb{R}_+})$. By Lemma \ref{eepslogbounded}, the operator $e^{-x/4} x \log x I$ is bounded on $\widetilde{H}^s_{p}(\overline{\mathbb{R}_+})$. Moreover, $\vartheta_1 (x) e^{x/2} \in {H}^s_{p}(\overline{\mathbb{R}_+})$, since it and its derivatives are bounded, smooth and $O(e^{-x/2})$ as $x \to \infty$. Finally, the compactness of the operator $T_{34}:H^s_{p,0}(\overline{\mathbb{R}_+}) \to H^{s-\epsilon}_{p}(\overline{\mathbb{R}_+})$, now follows directly from Lemma \ref{multiplierbHcompact}. \\

Secondly, consider $T_{35}$.  Since $1 + 1/p < s < 2 +1/p$, we can choose $\epsilon > 0$ such that $s+ \epsilon < 2 + 1/p$. We note that $ x^2 e^{-x/4} \in \widetilde{H}^{s+ \epsilon}_p (\overline{\mathbb{R}_+})$ because it, and its first derivative, take the value $0$ at $x=0$, and it is smooth with exponential decay. By Lemma \ref{x-alphaI}, the operator $ x^{-1} I$ is bounded from $\widetilde{H}^{s+ \epsilon}_p (\overline{\mathbb{R}_+})$ to $\widetilde{H}^{s+ \epsilon-1}_p (\overline{\mathbb{R}_+})$. Further, by Corollary \ref{LogMultiplier}, the operator $e^{-x/4} \log x I$ is bounded from $\widetilde{H}^{s+\epsilon -1}_p (\overline{\mathbb{R}_+}) \to \widetilde{H}^{s -1}_p (\overline{\mathbb{R}_+})$. Moreover, $\vartheta_1 (x) e^{x/2}$ and its derivatives are bounded, smooth and $O(e^{-x/2})$ as $x \to \infty$, and thus the operator $\vartheta_1 (x) e^{x/2} I$ is bounded on $\widetilde{H}^{s -1}_p (\overline{\mathbb{R}_+})$ by Lemma \ref{multxgammaa}. Finally, $T_{35} : H^s_{p,0}(\overline{\mathbb{R}_+}) \to H^{s-1}_{p}(\overline{\mathbb{R}_+})$ is bounded and rank one, and is therefore compact.

\subsection{Summary}
So, in summary, taking $N=2$, we have the required representation
\begin{equation*} 
\tilde{a}_{0}(x) u(0) + \sum^2_{j=1}  \tilde{a}_j(x) \, M^0(\tilde{b}_j) \, (r_+ \tilde{C}_j e_+)(r_+ u_s) + \tilde{K} u = f,
\end{equation*}
where the symbols $\tilde{a}_0$ and $(\tilde{a}_j, \tilde{b}_j, \tilde{c}_j)$ for $j=1$ and $j=2$ are given by equations \eqref{2tildea0} and \eqref{2tildea1}, \eqref{2tildea21} respectively, and the operator $\tilde{K}: H^s_{p,0}(\overline{\mathbb{R}_+}) \to H^{s-2\alpha}_{p}(\overline{\mathbb{R}_+})$ is compact.


\section{Operator algebra - final step} \label{OpAlgLp2}
We have seen in Section \ref{OpAlgInit2} that equation \eqref{rAef} can be written as (see equation \eqref{aMCtilde2} with $N=2$)
\begin{equation} \label{aMCtilde32}
\tilde{a}_{0}(x) u(0)+ \sum^2_{j=1}  \tilde{a}_j(x) \, M^0(\tilde{b}_j) \, (r_+ \tilde{C}_j e_+)(r_+ u_s) + \tilde{K} u = f,
\end{equation}
where the given function $f \in H^{s-2\alpha}_p(\overline{\mathbb{R}_+})$, and the operator $\tilde{K}:H^s_{p,0}(\overline{\mathbb{R}_+}) \to H^{s-2\alpha}_{p}(\overline{\mathbb{R}_+})$ is compact. \\

In this section, we present a formulation in $L_p(\mathbb{R}_+)$ of the form
\begin{equation} \label{aMW2}
\big ( {a}_1 (x) \, M^0({b}_1) \, W(c_1) + {a}_2 (x) \, M^0({b}_2) \, W(c_2) + T \big ) \, (r_+ u_s) = g,
\end{equation}
where the operator $T$, acting on $L_p(\mathbb{R}_+)$, is compact. The function $g \in L_p(\mathbb{R}_+)$ is defined by
\begin{equation} \label{defng2}
g:= r_+ (D-i)^{s-2\alpha} l_+ f,
\end{equation}
where by Lemma \ref{lemma:rLambdae}, $g$ does not depend on the choice of the extension $l_+$. \\

The subsequent analysis will show that, after the application of the operator $r_+ (D-i)^{s-2\alpha} l_+$, some of the terms in equation \eqref{aMCtilde32} represent compact operators on $L_p(\mathbb{R}_+)$. \\

We now consider the action of the operator $r_+ (D-i)^{s-2\alpha} l_+$ on the individual summands on the left-hand side of equation \eqref{aMCtilde32} in turn. 

\subsection{First term}
We assume $0 < \alpha < 1$ and note, from equation \eqref{2tildea0}, that the first term is only present if  $\alpha \not =\tfrac{1}{2}$. Indeed, in this case, we have 
\begin{align*}
\tilde{a}_{0}(x) & = \tilde{a}_{2}(x) M^0(\tilde{b}_{2})  \dfrac{1}{\sqrt{2\pi}} \, r_+ \mathcal{F}^{-1} (-i \xi)^{2\alpha-1}(\xi-i)^{-2} \\
&= \tilde{a}_2(x) M^0(\tilde{b}_2) \, r_+ \, h_1(x), 
\end{align*}
where
\begin{equation*}
h_1(x):= \dfrac{1}{\sqrt{2\pi}} \, r_+ \mathcal{F}^{-1} (-i \xi)^{2\alpha-1}(\xi-i)^{-2}.\\
\end{equation*}

Note that, from equation \eqref{2tildea21},
\begin{align*} 
\tilde{a}_{2}(x) &= -iC_\alpha \, \psi(\alpha+1, 2\alpha+1,x) ;\nonumber \\
\tilde{b}_{2}(\xi) &= B(1/p' + i \xi, 2 \alpha) / \Gamma(2\alpha).
\end{align*}

Our goal is to show that
\begin{equation*}
\Lambda^{s-2\alpha}_- \tilde{a}_{0}(x) = \Lambda^{s-2\alpha}_- \, \tilde{a}_2(x) M^0(\tilde{b}_2) r_+ h_1(x) \in L_p(\mathbb{R}_+),
\end{equation*}
because then the operator
\begin{equation*}
r_+ u_s(x) \longmapsto \Lambda^{s-2\alpha}_- \tilde{a}_{0}(x) u(0)
\end{equation*}
is bounded on $L_p(\mathbb{R}_+)$. Moreover, it has rank one and is therefore compact.\\\

We note that $\tilde{a}_2 \in r_+ C^\infty_0(\mathbb{R})$ and $\tilde{a}_2(x) = 0$ for $x \geq 2$. Let $\chi \in C^\infty_0(\mathbb{R})$ be such that
\begin{align*}
&\chi(t) :=
\begin{cases} 
	1 &\mbox{if }  |t| \leq 2\\
	0 & \mbox{if } |t| > 3. \\
\end{cases}
\end{align*} 

Then, see Lemmas \ref{lemma:mellinop2} and \ref{M0bchi},
\begin{equation*}
\Lambda^{s-2\alpha}_- \tilde{a}_2 M^0(\tilde{b}_2) r_+ h_1 = \Lambda^{s-2\alpha}_- \tilde{a}_2 M^0(\tilde{b}_2) (r_+ \chi  h_1).
\end{equation*} 

Since $h_1$ is the inverse Fourier transform of an integrable function it is continuous and vanishes at infinity. Hence $r_+ \chi h_1 \in L_p(\mathbb{R}_+)$. Thus, if $s-2\alpha <0$, using Lemma \ref{MellinOpBdd}, the required result follows immediately. \\

It remains to consider the case $s-2\alpha \geq 0$. Set $\mu = 2 \alpha-1, \,\, r=s-2\alpha$, so that $-1 < \mu < 1, \,\,  r < 3 - 2\alpha = 2-(2\alpha-1) = 2 - \mu$. Then, from Lemma \ref{Hrconditionbig}, for any $\chi_1 \in C^\infty_0(\mathbb{R})$, 
\begin{equation*}
\chi_1 (D- i)^{s-2\alpha} h_1 \in L_p(\mathbb{R}), 
\end{equation*}
subject \textit{only} to the condition
\begin{equation} \label{Condpsalpha2}
 s < 2 + \dfrac{1}{p}.
\end{equation}

Hence, using the method of proof in Lemma \ref{lemmapushchi}, $(D-i)^{s-2\alpha} \chi h_1 \in L_p(\mathbb{R})$, and so, after applying the operator $r_+ (D-i)^{2\alpha-s}$, we have
\begin{equation*}
r_+ \chi h_1 \in H^{s-2\alpha}_p(\overline{\mathbb{R}_+}).
\end{equation*}
Therefore, as $s-2\alpha \geq 0$,  again from Lemma \ref{MellinOpBdd},
\begin{equation*}
M^0(\tilde{b}_2) \, r_+ \chi  h_1 \in H^{s-2\alpha}_p(\overline{\mathbb{R}_+}),
\end{equation*}
and hence, 
\begin{equation*}
\Lambda^{s-2\alpha}_- \, \tilde{a}_2(x)  \,M^0(\tilde{b}_2) r_+ h_1(x) \in L_p(\mathbb{R}_+), 
\end{equation*}
as required, since $\tilde{a}_2   \in r_+C^\infty_0(\mathbb{R})$. \\

\subsection{Second term}
We assume $0 < \alpha < 1$. Using \eqref{2tildea1}, we have
\begin{equation*}
\tilde{a}_1(x) \, M^0(\tilde{b}_1) \, (r_+ \tilde{C}_1 e_+) = r_+ \tilde{C}_1 e_+
\end{equation*}
where the pseudodifferential operator $ \tilde{C}_1$ has symbol $(1+\xi^2)^\alpha (\xi-i)^{-2}(\xi+i)^{2-s}$. Hence, by Lemma \ref{lemma:rLambdae}
\begin{equation*}
r_+ (D-i)^{s-2\alpha} l_+ \, r_+ \tilde{C}_1 e_+ = r_+ (D-i)^{s-2\alpha} \tilde{C}_1 e_+.
\end{equation*}
Now $(D-i)^{s-2\alpha} \tilde{C}_1$ has symbol $(1+\xi^2)^\alpha (\xi-i)^{s- 2\alpha-2}(\xi+i)^{2-s}$, which is clearly a Fourier $L_p$ multiplier. (See Lemma \ref{mxiFm}.) Therefore, in the notation of equation \eqref{aMW2}, 
\begin{align} \label{a1final}
{a}_1(x) &= 1;  \nonumber \\
{b}_1(\xi) &= 1;  \\
{c}_1(\xi) &= (1+\xi^2)^\alpha (\xi-i)^{s- 2\alpha-2}(\xi+i)^{2-s}. \nonumber  
\end{align}

\subsection{Third term}
Now from \eqref{2tildea21} we have
\begin{align*}
\tilde{a}_2(x) &= -i C_\alpha \, \psi(\alpha+1, 2\alpha+1, x);\nonumber \\
\tilde{b}_2(\xi) &= B(1/p' + i \xi, 2 \alpha) / \Gamma(2\alpha) ;\\
\tilde{c}_2(\xi) &= (-i \xi)^{2\alpha} (\xi + i)^{2-s} (\xi-i)^{-2}.\nonumber  
\end{align*}

Let us define
\begin{equation*}
r:= s - 2\alpha,
\end{equation*}
so that we need to consider
\begin{equation*}
- 1 + 1/p < r < 2+1/p,
\end{equation*}
since $0 < \alpha < 1$ and $1+ 1/p < s < 2 +1/p$. We note that the pseudodifferential operator $\tilde{C}_2$ has order $-r$. \\

If $r \geq 0$, then from Lemma \ref{MellinOpBdd}, the operator $M^0(\tilde{b}_2) \, (r_+ \tilde{C}_2 e_+) : L_p(\mathbb{R}_+) \to H^r_p(\overline{\mathbb{R}_+})$ is bounded. \\

On the other hand, if $-1 +1/p < r < 0$, we can write
\begin{equation*}
\tilde{c}_2(\xi) =  (i \xi)^{-r} \cdot \tilde{c}_0(\xi) 
\end{equation*}
where
\begin{equation*}
\tilde{c}_0(\xi) :=  (i \xi)^r (-i \xi)^{2 \alpha} (\xi + i)^{2-s} (\xi-i)^{-2}.
\end{equation*}
Since $r + 2\alpha = s > 0, \,\, \tilde{c}_0(0)=0$. Moreover, as $r=s-2\alpha$, the operator $\tilde{C}_0$, with symbol $\tilde{c}_0$, has order $0$. \\

From Lemma \ref{AhomoMellin}, $M_{2\alpha, 0} r_+ (iD)^{-r} l_+ = r_+ (iD)^{-r}l_+ M_{2\alpha,r}$, and from Lemma \ref{lemma:rHomore}, $r_+ (iD)^{-r} l_+: L_p(\mathbb{R}_+) \to H^r_p(\overline{\mathbb{R}_+})$ is bounded. Therefore, the operator 
\begin{align*}
M^0(\tilde{b}_2) \, (r_+ \tilde{C}_2 e_+) & = M_{2\alpha, 0} \, (r_+ \tilde{C}_2 e_+) \quad \text{(see Lemma } \ref{lemma:mellinop2} \text{)} \\
& = M_{2\alpha, 0} \, (r_+ (iD)^{-r} l_+ r_+ \tilde{C}_0 e_+) \quad \text{(see Lemma } \ref{lemma:rHomore} \text{)} \\
& = r_+ (iD)^{-r}l_+ M_{2\alpha,r} (r_+ \tilde{C}_0 e_+)
\end{align*}
is bounded from $L_p(\mathbb{R}_+) \to H^r_p(\overline{\mathbb{R}_+})$. \\

So now, using Lemma \ref{multxgammaa}, each of the three operators in the identity
\begin{align*}
& \Lambda^r_- \, \tilde{a}_2(x) \, M^0(\tilde{b}_2) \, (r_+ \tilde{C}_2 e_+) \\
&= [\Lambda^r_- , \tilde{a}_2(x)] \, M^0(\tilde{b}_2) \, (r_+ \tilde{C}_2 e_+) +\tilde{a}_2(x) \, \Lambda^r_- \, M^0(\tilde{b}_2) \, (r_+ \tilde{C}_2 e_+),
\end{align*}
is bounded on $L_p(\mathbb{R}_+)$. \\

Moreover, the compactness of the operator involving the commutator term follows directly from Lemma \ref{CommutatorLambphi}. Thus, it remains to consider $\tilde{a}_2(x) \, \Lambda^r_- \, M^0(\tilde{b}_2) \, (r_+ \tilde{C}_2 e_+)$. \\

Firstly, suppose that $0 < r < 1$. Then, using Lemma \ref{LambdaMreverse},
\begin{align*}
\Lambda^r_-  & \, M^0(\tilde{b}_2) \, (r_+ \tilde{C}_2 e_+) \\
& = \Lambda^r_- \, M_{2\alpha,0} \, (r_+ \tilde{C}_2 e_+) \\
& = \big ( M_{2\alpha,r} \,  \Lambda^r_- +(-i)^r (M_{2\alpha,0} - M_{2\alpha,r}) +T \big ) \, (r_+ \tilde{C}_2 e_+) \\
& = M_{2\alpha,r} \, (r_+ C_2 e_+) + (-i)^r (M_{2\alpha,0} - M_{2\alpha,r}) \, (r_+ \tilde{C}_2 e_+)
+T \, (r_+ \tilde{C}_2 e_+).
\end{align*}

From Lemma \ref{LambdaMreverse}, $T: H^r_p(\overline{\mathbb{R}_+}) \to L_p(\mathbb{R}_+)$ is compact. Moreover, the pseudodifferential operator $\tilde{C}_2$ has order $-r$, and hence $T  \, (r_+ \tilde{C}_2 e_+)$ is compact on $L_p(\mathbb{R}_+)$. \\

By Remark \ref{remark:mellinop}, the symbols of both $M_{2\alpha,0}$ and $M_{2\alpha,r}$ take the value zero at $\pm \infty$. Hence, $(M_{2\alpha,0} - M_{2\alpha,r}) \, (r_+ \tilde{C}_2 e_+)$ is compact on $L_p(\mathbb{R}_+)$, from Proposition 5.3.4 (i), p. 267, \cite{Roch}. \\

So, in summary, if $0 < r < 1$ then
\begin{equation*}
\Lambda^r_- \, M^0(\tilde{b}_2) \, (r_+ \tilde{C}_2 e_+) = M_{2\alpha,r} \, (r_+ C_2 e_+) + K_1,
\end{equation*}
where $C_2$ has symbol
\begin{equation*}
{c}_2(\xi) = (-i \xi)^{2\alpha}  (\xi-i)^{s - 2 \alpha -2} (\xi + i)^{2-s},
\end{equation*}
and the operator $K_1$, acting on $L_p(\mathbb{R}_+)$, is compact. \\

Similarly, in the case that $-1 + 1/p < r < 0$, then we can again apply Lemma \ref{LambdaMreverse}, noting that the operator $\Lambda^{2 r}_- \, r_+ \tilde{C}_2 e_+$ has order $r < 0$. 

\begin{align*}
\Lambda^r_-  & \, M^0(\tilde{b}_2) \, (r_+ \tilde{C}_2 e_+) \\
& = \Lambda^r_- \, M_{2\alpha,0} \, (r_+ \tilde{C}_2 e_+) \\
& = \big ( M_{2\alpha,r} \,  \Lambda^r_- -(-i)^r (M_{2\alpha,0} - M_{2\alpha,r}) \, \Lambda^{2r}_-+T \big ) \, (r_+ \tilde{C}_2 e_+) \\
& = M_{2\alpha,r} \, (r_+ C_2 e_+) - (-i)^r (M_{2\alpha,0} - M_{2\alpha,r}) \,  \Lambda^{2r}_-\, (r_+ \tilde{C}_2 e_+)
+T \, (r_+ \tilde{C}_2 e_+).
\end{align*}

The compactness of $T  \, (r_+ \tilde{C}_2 e_+)$ on $L_p(\mathbb{R}_+)$ follows exactly as in the case $0 < r < 1$. Moreover, 
\begin{equation*}
(M_{2 \alpha,0} - M_{2 \alpha, r}) \Lambda^{2 r}_- \, r_+ \tilde{C}_2 e_+
\end{equation*}
is compact on $L_p(\mathbb{R}_+)$, from Proposition 5.3.4 (i), p. 267, \cite{Roch}.\\

So, in summary, if $-1 + 1/p < r < 0$ then
\begin{equation*}
\Lambda^r_- \, M^0(\tilde{b}_2) \, (r_+ \tilde{C}_2 e_+) = M_{2\alpha,r} \, (r_+ C_2 e_+) + K_2
\end{equation*}
where the operator $K_2$, acting on $L_p(\mathbb{R}_+)$, is compact. \\

The case $1 < r < 2$ follows similarly, except that we now apply Lemma \ref{Mlambdaextension}. 
In particular, we note that the operator $S_1 \, r_+ \tilde{C}_2 e_+$ has order $1 - r < 0$. Hence, as in the case $0 < r < 1$ discussed above, 
\begin{equation*}
S_1 \, r_+ \tilde{C}_2 e_+
\end{equation*}
is compact on $L_p(\mathbb{R}_+)$, from Proposition 5.3.4 (i), p. 267, \cite{Roch}. \\

For the case $2 < r < 2 +1/p$ we observe that $S_2 \, r_+ \tilde{C}_2 e_+$ has order $2 - r < 0$, and the analysis proceeds as for $1 < r < 2$. \\

Finally, the cases $r=1$ and $r=2$ follow in the same way, and for the case $r=0$, there is nothing to prove.\\

Hence, using Lemma \ref{lemma:mellinop2}, in the notation of equation \eqref{aMW} we have, 
\begin{align} \label{a2final}
{a}_2(x) &=-i C_\alpha \, \psi(\alpha+1, 2\alpha+1, x);\nonumber \\
{b}_2(\xi) &= B(s- 2\alpha +1/p' + i \xi, 2 \alpha) / \Gamma(2\alpha) ;\\
{c}_2(\xi) &= (-i \xi)^{2\alpha}  (\xi-i)^{s - 2 \alpha -2} (\xi + i)^{2-s}.\nonumber  
\end{align} 
Note that a routine application of Lemma \ref{mxiFm} confirms that $c_2$ is a Fourier $L_p$ multiplier.

\subsection{Summary}
Our base assumptions are that 

\begin{equation} \label{Finalcondalphaps2}
1 < p < \infty, \, 1+1/p < s <  2+ 1/p \,\, \text{and} \,\, 0 < \alpha < 1. \\
\end{equation}

So, finally, subject to condition \eqref{Finalcondalphaps2},  the formulation given by equation \eqref{aMW2} becomes 
\begin{equation} \label{WamW2}
\big ( W(c_1) + a_2 M^0(b_2) W(c_2) + T \big ) (r_+ u_s) = g,
\end{equation}
where the operator $T$, acting on $L_p(\mathbb{R}_+)$, is compact and
\begin{align} \label{Finalabc2}
g:&= r_+ (D-i)^{s-2\alpha} l_+ f; \nonumber \\
{c}_1(\xi) &= (1+\xi^2)^\alpha (\xi-i)^{s- 2\alpha-2}(\xi+i)^{2-s}. \nonumber  \\
{a}_2(x) &= -i C_\alpha \, \psi(\alpha+1, 2\alpha+1, x) \quad \text{(see Lemma } \ref{eUab2x}); \nonumber \\
{b}_2(\xi) &= B(s- 2\alpha +1/p' + i \xi, 2 \alpha) / \Gamma(2\alpha) ; \\
{c}_2(\xi) &= (-i \xi)^{2\alpha} (\xi-i)^{s- 2\alpha-2}(\xi+i)^{2-s},\nonumber  
\end{align} 
and the constant $C_\alpha$ is given by
\begin{equation*} 
C_{\alpha} = -i \, \dfrac{\alpha \, 2^{2\alpha}}{\Gamma (1- \alpha)}.
\end{equation*}

\section{Generalised symbol} \label{GenSymbolbig}
We now follow the approach taken in Chapter \ref{GenSymbol} and examine the generalised symbol $A_{\alpha, p, s}(\omega)$ defined on the contour $\Gamma_M$. \\

\subsection{Segment $\Gamma_1$} \label{Gamma1big}
Firstly, we note that
\begin{align*}
a_2(0) & = -i C_\alpha \, \psi(\alpha+ 1, 2 \alpha + 1, 0) \\
& = - i C_\alpha \, 2^{-2 \alpha} \dfrac{\Gamma(2 \alpha)}{\Gamma(\alpha +1)} \quad \text{(see Lemma \ref{eUab2x})}.
\end{align*}

Hence,
\begin{align*}
a_2(0) \, b_2(\xi) & = - i C_\alpha \, 2^{-2 \alpha} \dfrac{\Gamma(2 \alpha)}{\Gamma(\alpha +1)} \cdot \dfrac{B(s- 2\alpha +1/p' + i \xi, 2 \alpha)}{\Gamma(2\alpha)} \\
& = - \dfrac{\alpha \, 2^{2\alpha}}{\Gamma (1- \alpha)} \cdot \dfrac{2^{-2 \alpha}}{\Gamma(\alpha +1)} \cdot B(s- 2\alpha +1/p' + i \xi, 2 \alpha) \\
& = - \dfrac{1}{\Gamma (1- \alpha)\Gamma(\alpha)} \cdot B(s- 2\alpha +1/p' + i \xi, 2 \alpha) \quad (\Gamma(\alpha +1 ) = \alpha \Gamma(\alpha))\\
&= -\dfrac{\sin \pi \alpha}{\pi} \cdot B(s- 2\alpha +1/p' + i \xi, 2 \alpha) \quad \text{(5.5.3, \cite{NIST}).}
\end{align*}

Therefore, on the segment $\Gamma_1$, for $-\infty \leq \xi \leq \infty$, we have 
\begin{align*}
A_{\alpha, p, s}(\omega) & := a_1(0) \, b_1(\xi) \, c_{1p}(\infty, \xi) + a_2(0) \, b_2(\xi) \, c_{2p}(\infty, \xi) \\
& = c_{1p}(\infty, \xi) -  \dfrac{\sin \pi \alpha}{\pi} \cdot B(s- 2\alpha +1/p' + i \xi, 2 \alpha) \, c_{2p}(\infty, \xi). \\
\end{align*}

From Lemmas \ref{c1limits} and \ref{c1beta}, we have
\begin{equation*}
c_{1p}(\infty, \xi)  = e^{ i \pi \nu} \dfrac{\sin [\pi (1/p + \nu - i \xi) ]}{\sin \pi (1/p - i \xi)},  \qquad \nu = 2 - s + \alpha.
\end{equation*}

Similarly, from Lemmas \ref{c2limits} and \ref{c1beta}, we have
\begin{equation*}
c_{2p}(\infty, \xi)  = e^{- i \pi \alpha} e^{ i \pi \nu'} \dfrac{\sin [\pi (1/p + \nu' - i \xi) ]}{\sin \pi (1/p - i \xi)},  \qquad \nu' = 2 - s + 2 \alpha.
\end{equation*} \\

But $e^{- i \pi \alpha} e^{ i \pi \nu'} = e^{i \pi (2-s+\alpha)} = e^{i \pi \nu}$, and thus $c_{1p}(\infty, \xi)$ and $c_{2p}(\infty, \xi)$ have a common factor
\begin{equation*}
\dfrac{e^{i \pi \nu}}{\sin \pi (1/p - i \xi)}.
\end{equation*} \\

So, we are interested in establishing the precise conditions under which the quadruple $(\alpha, p, s, \xi)$ is \textbf{not} a solution of the following transcendental equation
\begin{equation} \label{transcendeqnbig}
\dfrac{\sin (\pi (1/p + \nu - i \xi))}{\sin (\pi (1/p + \nu' - i \xi))} - \dfrac{\sin \pi \alpha}{\pi} \, B(s -2\alpha +1/p' +i\xi, 2\alpha) =0.
\end{equation}

Let us now define
\begin{equation} \label{TBdefinitionbig}
T_B := \dfrac{\sin \pi \alpha}{\pi} B(s- 2 \alpha + 1 - 1/p + i \xi, 2 \alpha),
\end{equation}
and
\begin{equation} \label{Tsdefinitionbig}
T_s := \dfrac{\sin \pi (1/p + m - s + \alpha - i \xi)}{\sin \pi (1/p + m - s + 2 \alpha - i \xi)},
\end{equation}
where $m=2$ and $\xi \in \mathbb{R}$. \\

Then, the transcendental equation \eqref{transcendeqnbig} simply becomes
\begin{equation*}
T_s = T_B.
\end{equation*}

\subsection{Segment $\Gamma^\pm_2$}
Similarly, on $\Gamma^+_2$, for $0 \leq x \leq \infty$, we have
\begin{align*}
A_{\alpha, p, s}(\omega) & :=  a_1(x) \, b_1(\infty) \, c_1(-\infty) + a_2(x) \, b_2(\infty) \, c_2(-\infty) \\
& =  c_1(-\infty) + 0 \\
& = e^{2 \pi \nu i},  
\end{align*}
and on $\Gamma^-_2$, for $\infty \geq x \geq 0$,  
\begin{align*}
A_{\alpha, p, s}(\omega) & :=  a_1(x) \, b_1(-\infty) \, c_1(+\infty) + a_2(x) \, b_2(-\infty) \, c_2(+\infty) \\
& = c_1(+\infty) + 0 \\
& = 1.  
\end{align*}
Hence, 
\begin{equation*}
\inf_{\omega \in \Gamma^+_2 \cup \Gamma^-_2} | A_{\alpha, p, s}(\omega)| =1. \\
\end{equation*}

\subsection{Segment $\Gamma^\pm_3$} \label{Gamma3pmbig}
On $\Gamma^+_3$  for $\infty > \lambda \geq 0$, 
\begin{align*}
A_{\alpha, p, s}(\omega) & := a_1(\infty) \, b_1(\infty) \, c_{1}(-\lambda) + a_2(\infty) \, b_2(\infty) \, c_{2}(-\lambda) \\
& = c_{1}(-\lambda) + 0 \\
& = c_{1}(-\lambda),  
\end{align*}
and on $\Gamma^-_3$, for $0 \leq \lambda < \infty$,
\begin{align*}
A_{\alpha, p, s}(\omega) & :=  a_1(\infty) \, b_1(-\infty) \, c_{1}(\lambda) + a_2(\infty) \, b_2(-\infty) \, c_{2}(\lambda) \\
& = c_{1}(\lambda) + 0 \\
& = c_{1}(\lambda).  
\end{align*}
Note that on $\Gamma^+_3 $ and $\Gamma^-_3 $ the parameter $\lambda$ varies between $0$ and $\infty$ but, of course, in an opposite sense. So, in summary,  
\begin{equation*}
\inf_{\omega \in \Gamma^{\pm}_3} | A_{\alpha, p, s}(\omega)| =1.\\
\end{equation*}

\subsection{Segment $\Gamma_4$}
Finally, on $\Gamma_4$, for $-\infty \leq \xi \leq \infty$,
\begin{align*}
A_{\alpha, p, s}(\omega) & := a_1(\infty) \, b_1(\xi) \, c_{1}(0) + a_2(\infty) \, b_2(\xi) \, c_{2}(0) \\
&=  c_{1}(0) + 0 \\
&= c_1(0). 
\end{align*}
Hence, 
\begin{equation*}
\inf_{\omega \in \Gamma_4} | A_{\alpha, p, s}(\omega)| =1,
\end{equation*}
and this completes the review of the contour $\Gamma_M$.

\subsection{Summary}
Note that the preceding analysis of the segments of the contour has shown that $A_{\alpha, p,s}(\omega)$ is constant on the segments $\Gamma^{\pm}_2$ and $\Gamma_4$. Therefore, it remains to consider $A_{\alpha, p,s}(\omega)$ on $\Gamma_1 \cup \Gamma^+_3 \cup \Gamma^-_3$. But from subsection \ref{Gamma3pmbig}, we can combine $\Gamma^\pm_3$ to give a new segment $\Gamma_3$ (say), where now the parameter $\lambda$ varies from $-\infty$ to $\infty$. (Note that, as expected, the symbol $A_{\alpha, p,s}(\omega)$ is continuous at $\lambda =0$ on the new segment $\Gamma_3$.) \\

By construction, we observe that $A_{\alpha, p,s}(\omega)$ is continuous on $\Gamma_1 \cup \Gamma_3$. Indeed, from subsection \ref{Gamma1big}, on the segment $\Gamma_1$
\begin{equation} \label{AGamma1big}
A_{\alpha, p,s}(\omega)=c_{1p}(\infty, \xi) - \dfrac{\sin \pi \alpha}{\pi} \, B(s -2\alpha +1/p' +i\xi, 2\alpha) \, c_{2p}(\infty, \xi)
\end{equation}
with limits $c_{1p}(\infty, \pm \infty) = c_1(\mp \infty)$ at $\xi = \pm \infty$ respectively. \\

The condition that $A_{\alpha, p,s}(\omega)=0$ on $\Gamma_1$ gives rise to the transcendental equation
\begin{equation*}
\dfrac{\sin (\pi (1/p + \nu - i \xi))}{\sin (\pi (1/p + \nu' - i \xi))} = \dfrac{\sin \pi \alpha}{\pi} \, B(s -2\alpha +1/p' +i\xi, 2\alpha),
\end{equation*}
where $\nu = 2 - s + \alpha$ and $\nu'=2-s + 2 \alpha$. \\

Finally, on $\Gamma_3$ we have
\begin{equation} \label{AGamma3big}
A_{\alpha, p,s}(\omega)= c_1(\lambda)
\end{equation}
with limits $c_1(\pm \infty)$ for $\lambda = \pm \infty$. \\


\section{Index and invertibility}
As previously, we set
\begin{equation*}
\tau := s - 1/p.
\end{equation*}
Then, see Chapter \ref{ChapterTE}, we note the critical importance of the the following transcendental equation, see \eqref{TEatzero}, which, for convenience, we repeat here:
\begin{equation*}
\Gamma(2 \alpha - \tau) \Gamma(\tau+1) \sin \pi(\alpha-\tau) = \Gamma(2 \alpha) \sin \pi \alpha.
\end{equation*} 

From Lemma \ref{lemmaTE5}, if $0 < \alpha < 1$ is fixed, and $1 < \tau < 2$ is considered to vary, then the above equation has a unique solution for $\tau$. Moreover, this solution can be expressed in the form $s = 1 + 1/p + \alpha_c$, where $\alpha_c$ only depends on $\alpha$ and satisfies $0 < \alpha_c < \alpha$. \\

For example, if $\alpha = 0.75, \, p=2$, Figure \ref{Qhoenixmedium} shows that when $s \approx 2.226$, or equivalently $\alpha_c \approx 0.726$, our operator is not Fredholm. \begin{figure}[H]
\centering
\includegraphics[width=6 cm, height=6 cm]{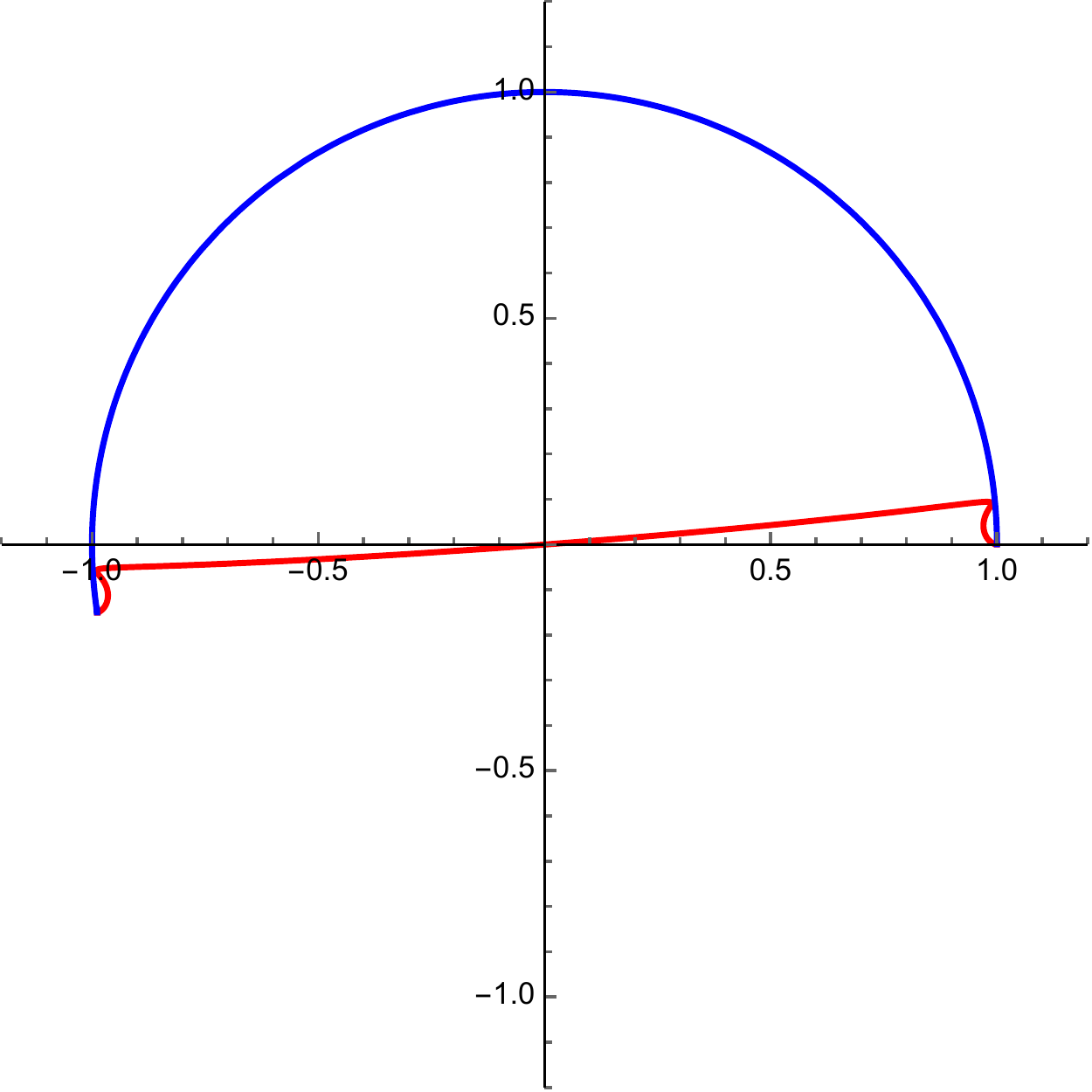}  
\caption{Symbol plot for $\alpha = 0.7500, \, p =2$ and $s=2.2260$.}
\label{Qhoenixmedium}
\end{figure}


Given $\alpha$, we can readily determine a good estimate for $\alpha_c$  using ${\textit{Mathematica}}\textsuperscript{\textregistered}$. Indeed, Figure \ref{AlphacGraph} shows the graph of this  estimate for $\alpha_c$, as $\alpha$ varies over the range $(0 , 1)$. The straight line shown on the plot is simply to highlight the fact that $0 < \alpha < \alpha_c$, and as $\alpha$ tends to $1$, the difference $\alpha - \alpha_c$ tends to $0$. \\

\begin{figure}[H]
\centering
\includegraphics[width=6 cm, height=6 cm]{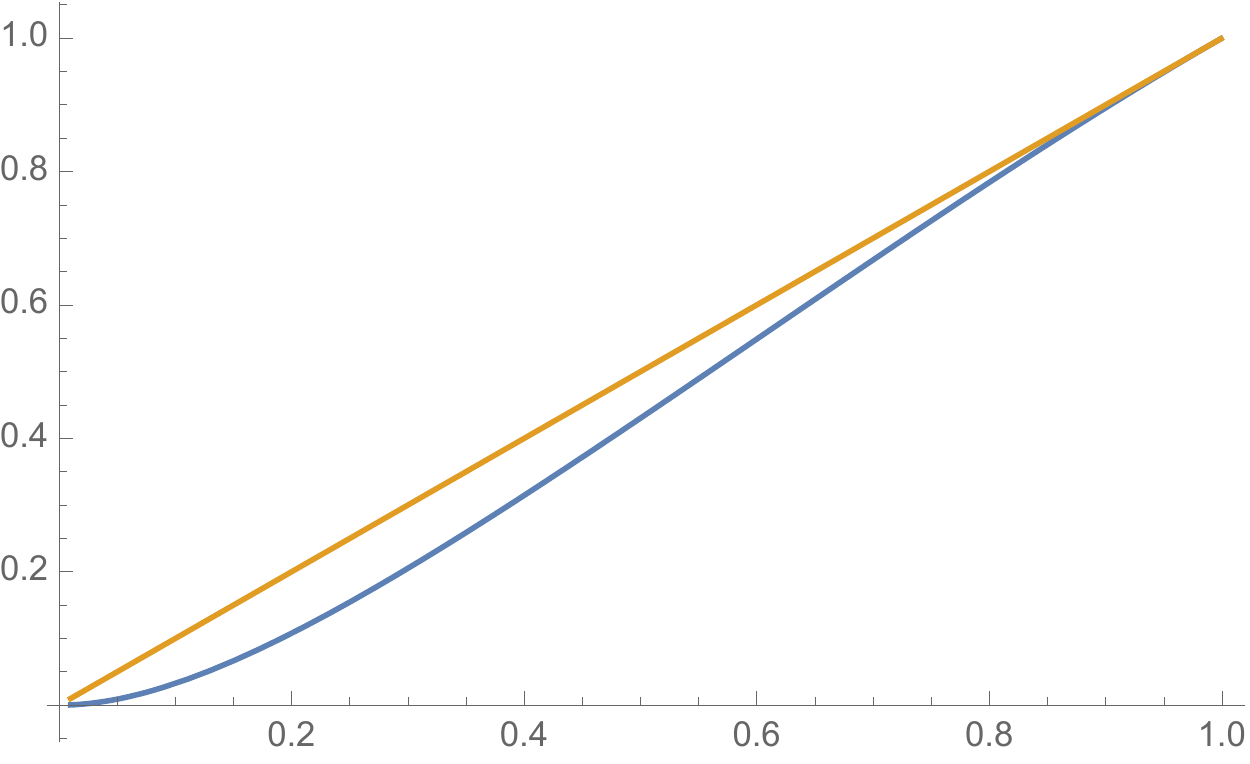} 
\caption{Plot of $\alpha_c$ versus $\alpha$.}
\label{AlphacGraph}
\end{figure}

In the special case that $\alpha = \tfrac{1}{2}$, equation \eqref{TEatzero} reduces to
\begin{equation*}
\tan (\pi \tau ) = \pi \tau,
\end{equation*}
and, in this case, we obtain $\alpha_c \approx 0.4303$. \\

\subsection{Main results}
\begin{theorem} \label{TheoremIndexZerobig}
For all $\alpha, p, s$ satisfying the conditions $0 < \alpha < 1, \,\, 1 < p < \infty$ and $1 + 1/p < s < 1 + 1/p + \alpha_c$, the winding number of the generalised symbol $\big ( A_{\alpha, p,s}, \, \Gamma_M \big )$ in the complex plane is $-1$. Hence, the operator $W(c_1)  + a_2M^0(b_2)W(c_2)$, defined on $L_p(\mathbb{R}_+)$, has Fredholm index equal to $1$. \\

On the other hand, if $1 + 1/p + \alpha_c < s < 2+1/p$, the operator $W(c_1)  + a_2M^0(b_2)W(c_2)$ has Fredholm index equal to $0$. \\
\end{theorem}

\begin{theorem}
Suppose $0 < \alpha < 1, \,\, 1 < p < \infty$ and $1 + 1/p < s < 1+1/p + \alpha_c$. Then the operator $\mathcal{A}: H^s_{p,0}(\overline{\mathbb{R}_+}) \to H^{s-2\alpha}_p(\overline{\mathbb{R}_+})$ is invertible. \\

On the other hand, if $0 < \alpha <1, \,\, 1 < p < \infty$ and $1 + 1/p + \alpha_c < s < 2 +1/p$, then $\mathcal{A}$ has a trivial kernel and is Fredholm with index equal to $-1$. \\
\end{theorem}

\section{Proof of first main result} \label{Windingnumberanalysis}
The constraints are
\begin{equation} \label{alphapsconstraintsbig}
0 < \alpha < 1, \,\, 1 < p < \infty  \,\, \text{and} \,\, 1+1/p < s < 2+1/p.
\end{equation}

Let $\alpha, p ,s$ fall within their admissible ranges and be fixed. From Section \ref{GenSymbolbig}, it is easy to show that the generalised symbol $A_{\alpha,p,s}$ can be represented by a closed contour in the complex plane given by the union of the two curves, $S_1$ and $S_3$. \\

Indeed, from Section \ref{Gamma3pmbig} we have,
\begin{equation*} 
S_3(\xi):= (1+\xi^2)^\alpha \, (\xi -i)^{s-2\alpha-2} \, (\xi +i)^{2-s}, \quad  -\infty \leq \xi \leq \infty.
\end{equation*}

Now
\begin{align*}
S_3(\xi) & = (\xi + i)^\alpha (\xi - i)^\alpha \, (\xi -i)^{s-2\alpha-2} \, (\xi +i)^{2-s} \\
&= (\xi + i)^{2-s+\alpha} (\xi - i)^{-(2-s+\alpha)} .
\end{align*}


From Section \ref{Gamma1big}, for $-\infty \leq \xi \leq \infty$,
\begin{equation*} 
S_1(\xi) := S_{11}(\xi) - \dfrac{\sin \pi \alpha}{\pi} \, B(s-2\alpha + 1-1/p+i \xi, 2\alpha) S_{12} (\xi), 
\end{equation*}
where 
\begin{equation*}
S_{11}(\xi):=  e^{i \pi \nu} \dfrac{\sin [\pi(1/p + \nu - i \xi) ]}{\sin [\pi(1/p - i \xi) ]}, \quad S_{12}(\xi):= e^{i \pi \nu} \dfrac{\sin [\pi(1/p + \nu + \alpha - i \xi) ]}{\sin [\pi(1/p - i \xi) ]},
\end{equation*}
and $\nu = 2-s+\alpha$. \\

\subsection{The case $s > 1 + 1/p + \alpha_c$}
Firstly, we consider the case where $s > 1 + 1/p + \alpha_c$, and we will show that the winding number of the model contour is $0$.\\

Assume $0 < \epsilon <<1$, and let us choose
\begin{equation*}
 \alpha = \tfrac{1}{2}, \quad p =2   \quad \text{and} \quad s=\tfrac{5}{2} - \epsilon,
\end{equation*}
as the set of parameters used to define the first \textit{model} contour. Note that these values lie within the set of admissible constraints given in condition \eqref{alphapsconstraintsbig}. \\

Now $s-2\alpha + 1-1/p = 2 - \epsilon$, and hence by Lemma \ref{lemTBalphasigmaxi}, with $\sigma = 2 - \epsilon$ and $\alpha= \tfrac{1}{2}$,
\begin{align*}
\bigg |\dfrac{\sin \pi \alpha}{\pi} \, B(s-2\alpha + 1-1/p+i \xi, 2\alpha)  \bigg | & \leq \dfrac{1}{\pi} \cdot B(2-\epsilon,1) \\
& = \dfrac{1}{\pi} \cdot \dfrac{\Gamma(2-\epsilon)\Gamma(1)}{\Gamma(3-\epsilon)} \\
& = \dfrac{1}{\pi} \cdot \dfrac{1}{2-\epsilon} \\
& < \dfrac{1}{\pi},  \quad \text{if} \quad 0 < \epsilon << 1.
\end{align*}

Moreover, $\nu = 2 - s + \alpha = \epsilon$, and thus
\begin{equation*}
S_{11}(\xi) =  e^{i \pi \epsilon} \dfrac{\sin \pi(\tfrac{1}{2} + \epsilon - i \xi) }{\sin \pi(\tfrac{1}{2} - i \xi) }.
\end{equation*}
Since 
$\sin \pi(\tfrac{1}{2} + \epsilon - i \xi)  = \sin \pi \epsilon \, \cos \pi (\tfrac{1}{2} - i \xi) + \cos \pi \epsilon \, \sin \pi (\tfrac{1}{2} - i \xi)$ and 
$\cot \pi (\tfrac{1}{2} - i \xi) = i \tanh \pi \xi$, we can write
\begin{equation*}
S_{11}(\xi) =  e^{i \pi \epsilon} \big ( \cos \pi \epsilon + i \tanh \pi \xi \, \sin \pi \epsilon \big ),
\end{equation*}
and similarly
\begin{equation*}
S_{12}(\xi) =  e^{i \pi \epsilon} \big ( -\sin \pi \epsilon + i \tanh \pi \xi \, \cos \pi \epsilon \big ).
\end{equation*}
\\
Hence, we have the following elementary expansions for $0 < \epsilon <<1$
\begin{equation*}
S_{11}(\xi) = 1 + \{ i \pi(1+ \tanh \pi \xi) \} \, \epsilon +O(\epsilon^2),
\end{equation*}
and
\begin{equation*}
S_{12}(\xi) = i \tanh(\pi \xi) - \{ \pi (1 + \tanh(\pi \xi) )\}  \, \epsilon +O(\epsilon^2).
\end{equation*}

Combining these estimates
\begin{align*}
| S_1 (\xi) - 1 | & \leq \{ \pi (1+|\tanh(\pi \xi)|) \} \, \epsilon + |S_{12}|/\pi + O(\epsilon^2) \\
& \leq 1/\pi + 4 \pi \, \epsilon + O(\epsilon^2) \\
& < \tfrac{2}{5} \quad \text{for sufficiently small} \,\, \epsilon >0.
\end{align*}

On the other hand, for any $\xi \in \mathbb{R}$, we have
\begin{equation*}
\xi \pm i = \sqrt{1+ \xi^2} \exp( \pm i \theta), \quad \text{where} \quad 0 < \theta < \pi.
\end{equation*}
Hence, since $2- s + \alpha = \epsilon$, 
\begin{align*}
| S_3 (\xi) - 1 | &= |\exp(i 2 \epsilon \theta) - 1| \\
& \leq |\cos 2 \epsilon \theta + i \sin 2 \epsilon \theta -1| \\
& \leq 2 \pi \epsilon + O(\epsilon^2) \\
& < \tfrac{2}{5} \quad \text{for sufficiently small} \,\, \epsilon >0.
\end{align*}
So, for sufficiently small $\epsilon$,  the model contour, formed by the union of the curves $S_1$ and $S_3$, is wholly contained in the disc of radius $\tfrac{2}{5}$ centred on the point $1$ in the complex plane. Hence, for $s > 1 + 1/p + \alpha_c$, the winding number of the model contour, given by $\alpha = \tfrac{1}{2}, \, p =2$ and $s=\tfrac{5}{2} - \epsilon$, must be $0$. \\

We now give a plot of the model contour, see Figure \ref{Qhoenix52}, where we take $\epsilon = \tfrac{1}{1000}$. As discussed, the contour is contained in a circle with centre $1$ with radius $\tfrac{2}{5}$. \\

\begin{figure}[H]
\centering
\includegraphics[width=6 cm, height=6 cm]{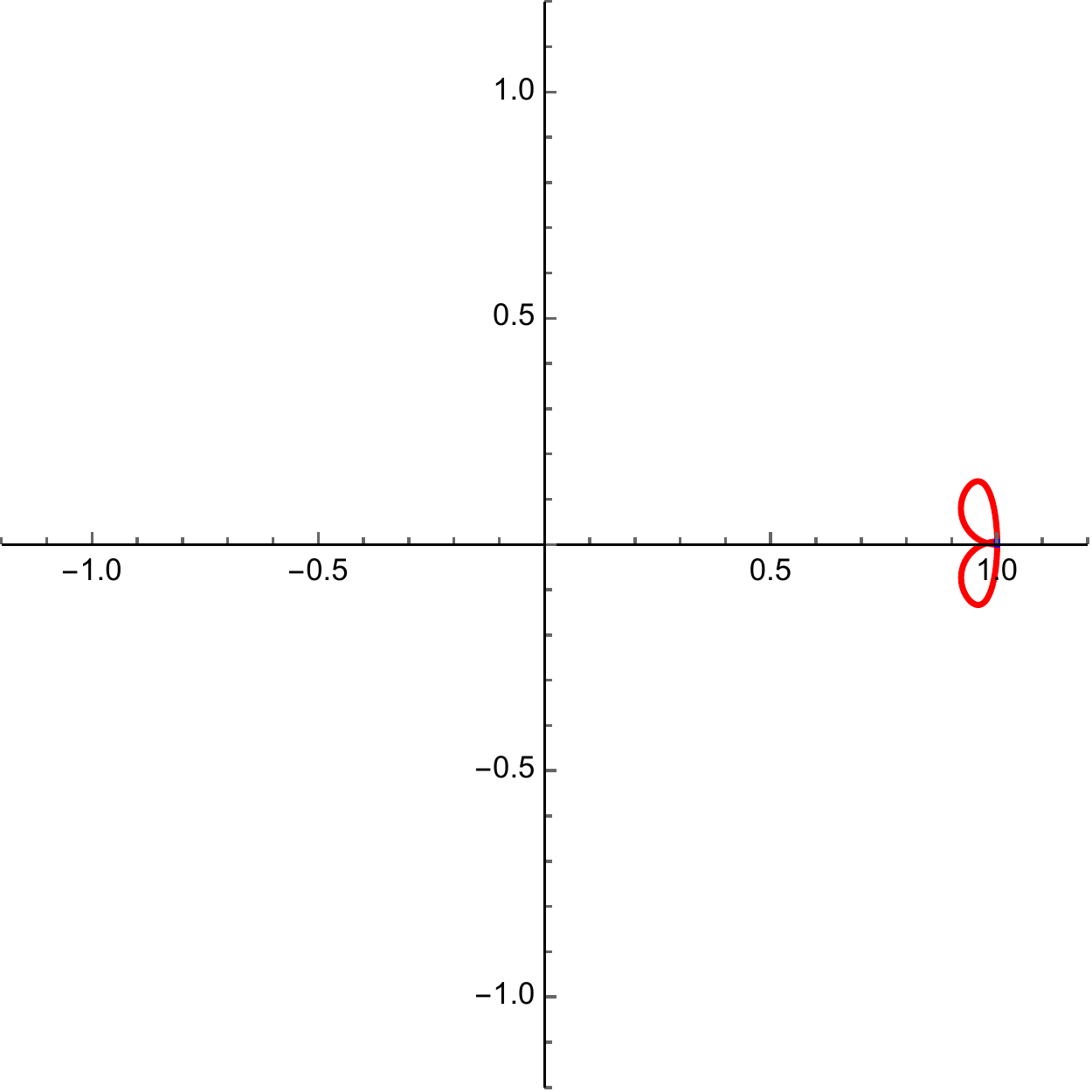}  
\caption{Symbol plot for $\alpha = 0.5000, \, p =2$ and $s=2.5000-\epsilon$.}
\label{Qhoenix52}
\end{figure}

The following four plots, see Figures \ref{Qhoenix54}, \ref{Qhoenix56}, \ref{Qhoenix58} and \ref{Qhoenixlarge} show $\alpha$ increasing in steps of $\tfrac{1}{16}$ and, at the same time, $s$ decreasing by $\tfrac{1}{16}$. In each case, the plot confirms that the winding number of the contour is zero. 
\begin{figure}[H]
\centering
\includegraphics[width=6 cm, height=6 cm]{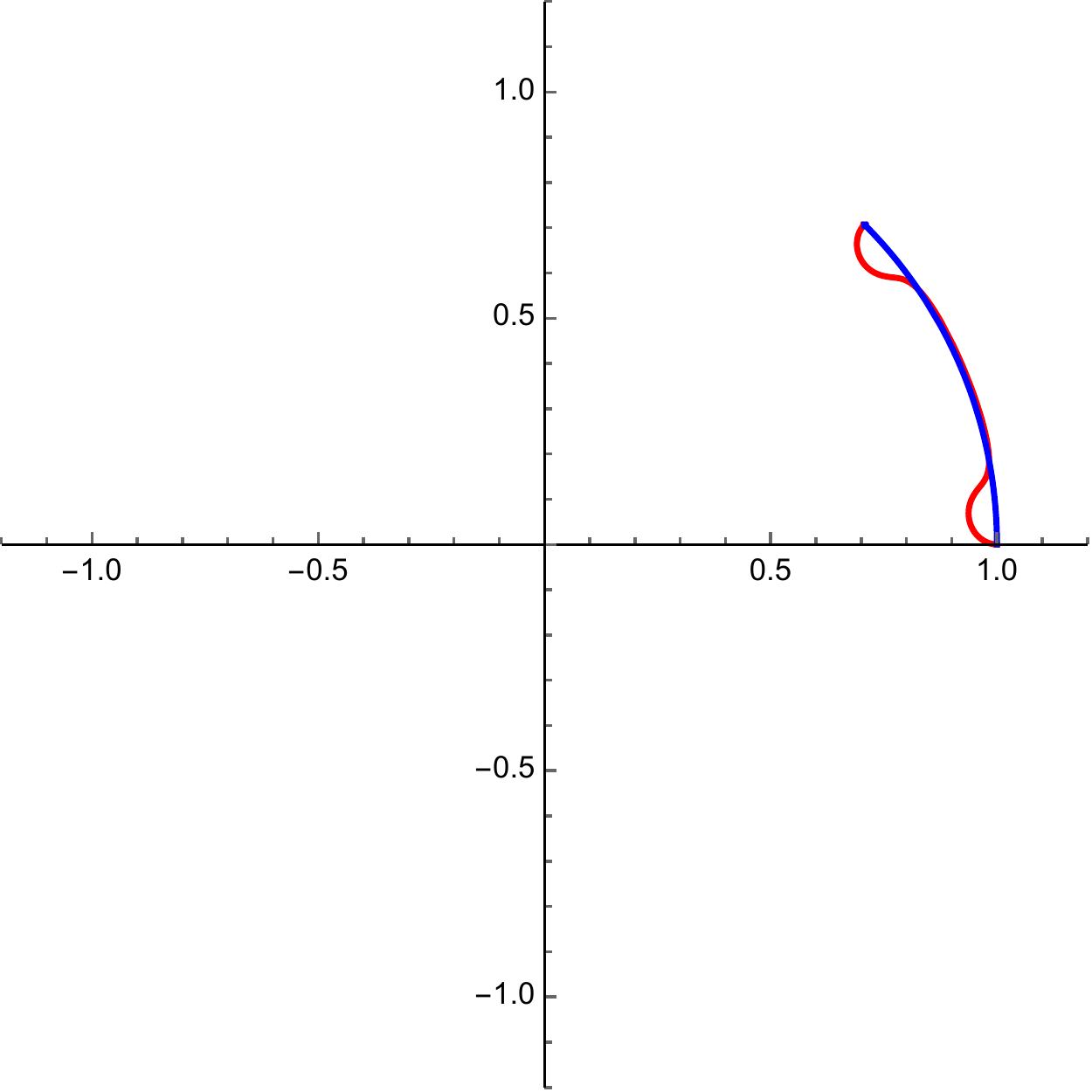}  
\caption{Symbol plot for $\alpha = 0.5625, \, p =2$ and $s=2.4375$.}
\label{Qhoenix54}
\end{figure}

\begin{figure}[H]
\centering
\includegraphics[width=6 cm, height=6 cm]{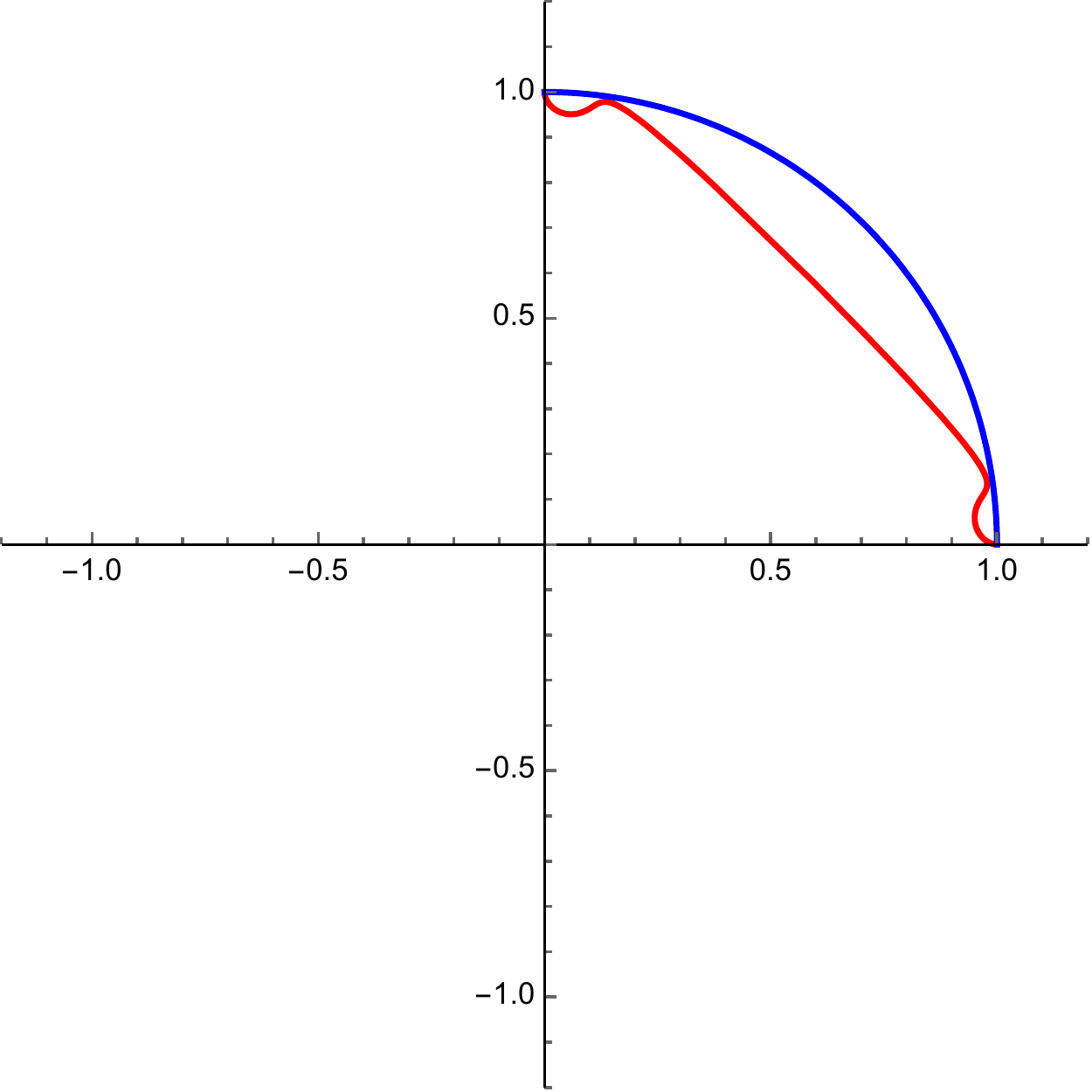}  
\caption{Symbol plot for $\alpha = 0.6250, \, p =2$ and $s=2.3750$.}
\label{Qhoenix56}
\end{figure}

\begin{figure}[H]
\centering
\includegraphics[width=6 cm, height=6 cm]{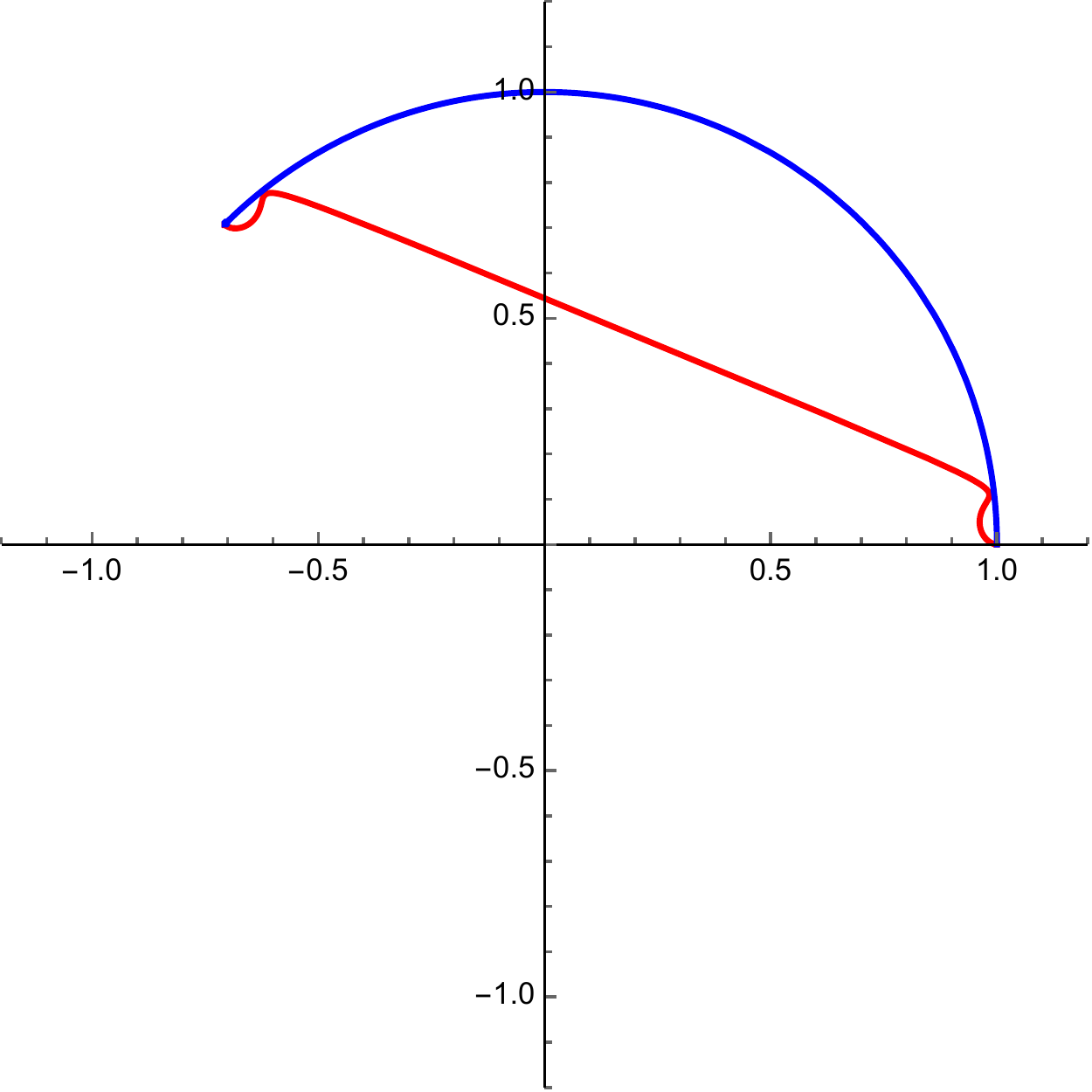}  
\caption{Symbol plot for $\alpha = 0.6875, \, p =2$ and $s=2.3125$.}
\label{Qhoenix58}
\end{figure}

\begin{figure}[H]
\centering
\includegraphics[width=6 cm, height=6 cm]{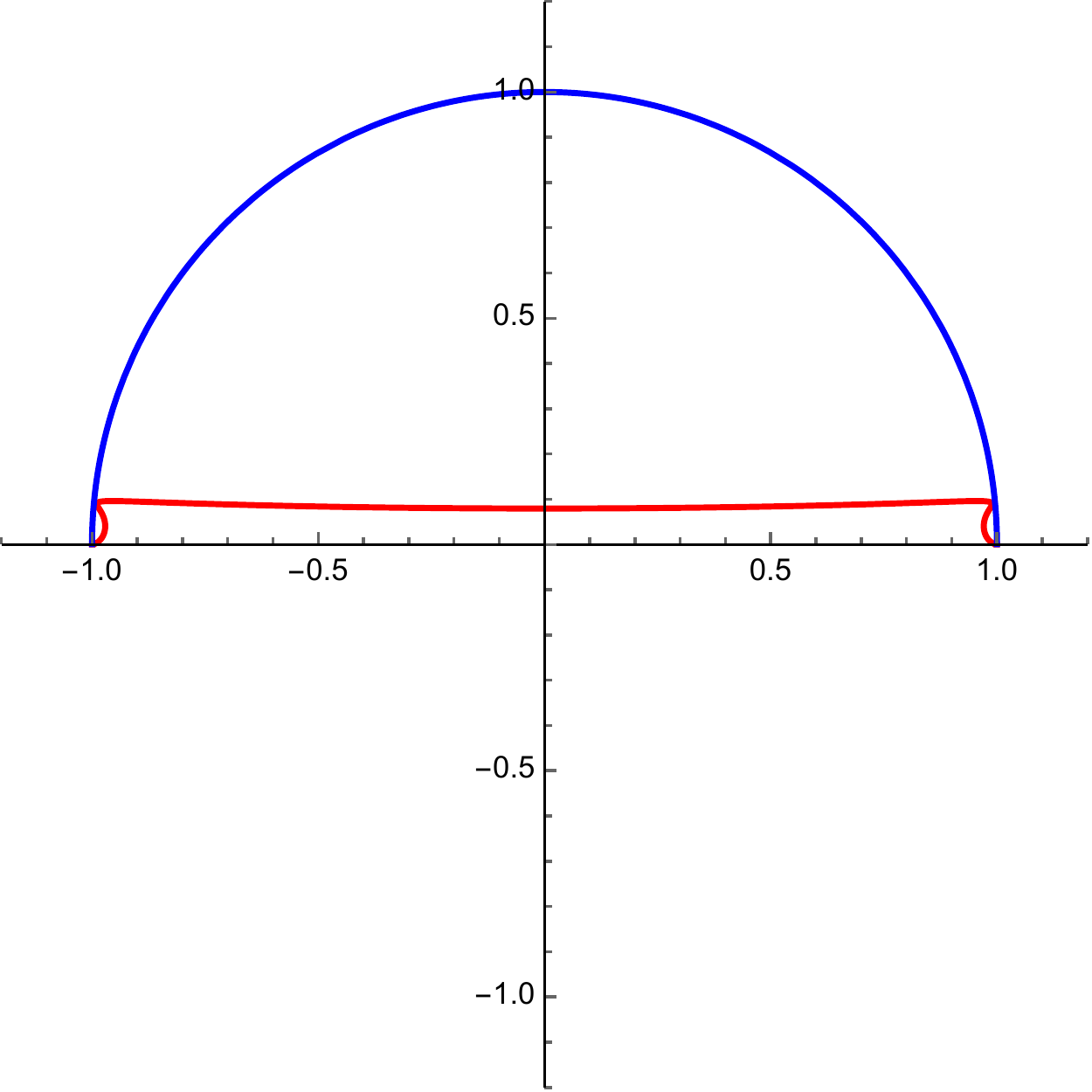}  
\caption{Symbol plot for $\alpha = 0.7500, \, p =2$ and $s=2.2500$.}
\label{Qhoenixlarge}
\end{figure}

Finally, we recall that, for $\alpha = 0.75$, we have $\alpha_c \approx 0.726$. Hence, if $p=2$ and $s=2.25$ we have $s > 1 + 1/p + \alpha_c \approx 2.226$. As expected, the winding number of the contour in Figure \ref{Qhoenixlarge} is $0$.

\subsection{The case $s < 1 + 1/p + \alpha_c$}
Let us first return to our example with $\alpha=0.75$ and $p=2$. We now take $s = 2.2$, so that $s < 1 + \tfrac{1}{2} + \alpha_c \approx 2.226$. 
\begin{figure}[H]
\centering
\includegraphics[width=6 cm, height=6 cm]{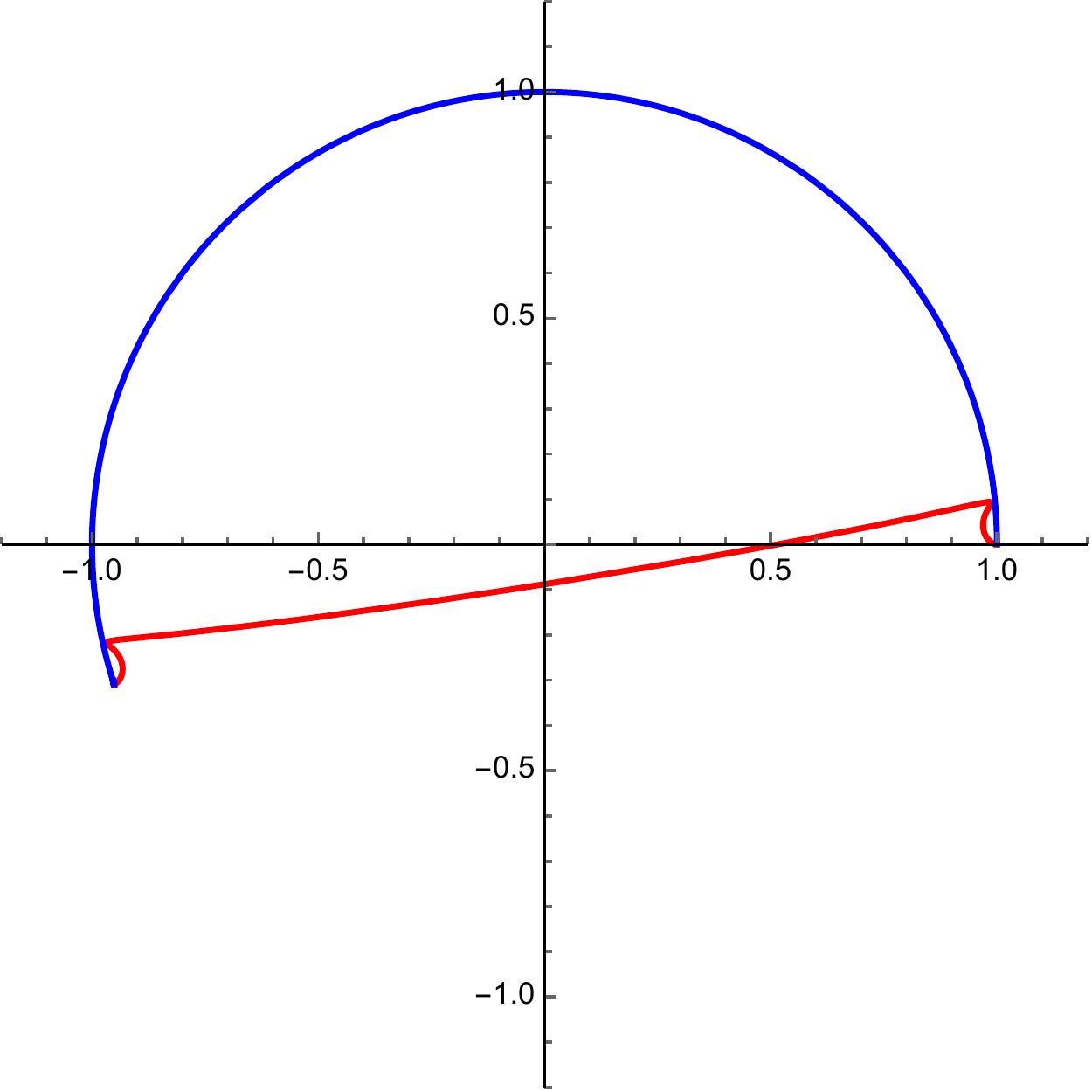}  
\caption{Symbol plot for $\alpha = 0.7500, \, p =2$ and $s=2.2000$.}
\label{Qhoenixsmall}
\end{figure}
Clearly, the winding number of the contour in Figure \ref{Qhoenixsmall} is not zero and, moreover, must take the value $\pm 1$. \\

We now consider the general case where $s < 1 + 1/p + \alpha_c$, and we will show that the winding number is, in fact, $-1$.\\

Again assume $0 < \epsilon <<1$, and let us now choose
\begin{equation*}
 \alpha = \tfrac{1}{2}, \quad p =2   \quad \text{and} \quad s=\tfrac{3}{2} + \epsilon,
\end{equation*}
as the set of parameters used to define the second \textit{model} contour. Note that these values lie within the set of admissible constraints given in condition \eqref{alphapsconstraintsbig}. \\

Now $s-2\alpha + 1-1/p = 1 + \epsilon$, and hence by Lemma \ref{lemTBalphasigmaxi}, with $\sigma = 1+ \epsilon$ and $\alpha= \tfrac{1}{2}$,
\begin{align*}
\bigg |\dfrac{\sin \pi \alpha}{\pi} \, B(s-2\alpha + 1-1/p+i \xi, 2\alpha)  \bigg | & \leq \dfrac{1}{\pi} \cdot B(1+\epsilon,1) \\
& = \dfrac{1}{\pi} \cdot \dfrac{\Gamma(1+\epsilon)\Gamma(1)}{\Gamma(2+\epsilon)} \\
& = \dfrac{1}{\pi} \cdot \dfrac{1}{1+\epsilon} \\
& < \dfrac{1}{\pi},  \quad \text{if} \quad \epsilon >0.
\end{align*}

Moreover, $\nu = 2 - s + \alpha = 1- \epsilon$, and we have the following elementary expansions for $0 < \epsilon <<1$
\begin{equation*}
S_{11}(\xi) = 1 - \{ i \pi(1+ \tanh \pi \xi) \} \, \epsilon +O(\epsilon^2),
\end{equation*}
and
\begin{equation*}
S_{12}(\xi) = i \tanh(\pi \xi) + \{ \pi (1 + \tanh(\pi \xi) )\}  \, \epsilon +O(\epsilon^2).
\end{equation*}

Combining these estimates
\begin{align*}
| S_1 (\xi) - 1 | & \leq \{ \pi (1+|\tanh(\pi \xi)|) \} \, \epsilon + |S_{12}|/\pi + O(\epsilon^2) \\
& \leq 1/\pi + 4 \pi \, \epsilon + O(\epsilon^2) \\
& < \tfrac{2}{5} \quad \text{for sufficiently small} \,\, \epsilon >0.
\end{align*}

On the other hand, since $2 - s +\alpha = 1 - \epsilon$, the curve $S_3$ traverses, in a clockwise direction, the complete unit circle apart from a small neighbourhood near the point $1$ in the complex plane. By choosing $\epsilon$ sufficiently small, we can ensure that the omitted portion lies wholly within the disk of radius $\tfrac{2}{5}$ centred on $1$. \\

Since the model contour, formed by the union of the curves $S_1$ and $S_3$, forms a closed loop, for sufficiently small $\epsilon$, it encircles the origin once, in a clockwise direction, plus an additional component that is wholly contained in the disc of radius $\tfrac{2}{5}$ centred on the point $1$ in the complex plane. Hence, the winding number of the model contour, for $s < 1 + 1/p + \alpha_c$, is equal to $-1$. \\

\subsection{Conclusion}
Given any set of parameters $\alpha,p,s$ satisfying the constraints $0<\alpha<1, \,\, 1 < p < \infty$ and $1+ 1/p < s < 1+1/p + \alpha_c$, the associated contour can be continuously deformed into the \textit{model} contour, and from Theorem \ref{TheoremTranscendbig}, does this without ever crossing the origin. Hence, the two contours must have the same winding number, namely, $-1$. \\

By a similar argument, the winding number is constant, and equal to $0$ in the case that 
$1+ 1/p + \alpha_c < s < 2+1/p $. \\

Therefore, see Remark \ref{WindingFred}, the operator $W(c_1)+a_2M^0(b_2)W(c_2)$, defined on $L_p(\mathbb{R}_+)$, has Fredholm index equal to $1$ if $s < 1 + 1/p + \alpha_c$, and index $0$ if $s > 1 + 1/p + \alpha_c$. This completes the proof of the first theorem. \\

\section{Proof of second main result}
Suppose $0 < \alpha < 1, \,\, 1 < p < \infty$ and $1 + 1/p < s < 2+1/p$. Then, from Theorem \ref{thmabddtrivker}, the operator $\mathcal{A}: H^s_{p,0}(\overline{\mathbb{R}_+}) \to H^{s-2\alpha}_p(\overline{\mathbb{R}_+})$ is bounded, where
\begin{equation*}
H^s_{p,0}(\overline{\mathbb{R}_+}) := \{ u \in H^s_p(\overline{\mathbb{R}_+}) : u'(0)=0\}.
\end{equation*}

Now let $u \in H^s_p(\overline{\mathbb{R}_+})$. Further choose an \textit{arbitrary} $u_0 \in H^s_p(\overline{\mathbb{R}_+})$, with $u'(0)=1$. Then, we can write
\begin{align*}
u(x) &= (u(x) - u'(0)u_0(x)) + u'(0)u_0(x) \\
&:= v(x) + u_0(x),
\end{align*}
where $v(x):= u(x) - u'(0)u_0(x)$ and, clearly, $v'(0)=0$. That is, $v \in H^s_{p,0}(\overline{\mathbb{R}_+})$.  In other words, $H^s_{p,0}(\overline{\mathbb{R}_+})$ has co-dimension $1$ in $H^s_p(\overline{\mathbb{R}_+})$. \\

We now define
\begin{equation*}
u_s:= (D+i)^{s-2} e_+(D-i)^2 u,
\end{equation*}
so that $u_s \in L_p(\mathbb{R}_+)$ with supp $u_s \subseteq \mathbb{R}_+$. Indeed, the operator
\begin{equation*}
I_s:= r_+ (D+i)^{s-2} e_+(D-i)^2 : H^s_p(\overline{\mathbb{R}_+}) \to L_p(\mathbb{R}_+)
\end{equation*}
is an isomorphism for $1+ 1/p < s < 2 +1/p$. (See Lemma \ref{ustou2} and the proof of Lemma \ref{ustou}.) Let $L_{p,0}(\mathbb{R}_+)$ be the image of $H^s_{p,0}(\overline{\mathbb{R}_+})$ under $I_s$. Then $L_{p,0}(\mathbb{R}_+)$ has co-dimension $1$ in $L_p(\mathbb{R}_+)$. Hence, we can write
\begin{equation} \label{Lpdirectsum}
L_p(\mathbb{R}_+) = L_{p,0}(\mathbb{R}_+) \oplus M_{p,0}(\mathbb{R}_+),
\end{equation}
where $M_{p,0}(\mathbb{R}_+)$ has dimension $1$. \\

Let $\widetilde{A}_{ext}: L_p(\mathbb{R}_+) \to L_p(\mathbb{R}_+)$
be the (bounded) operator defined by
\begin{equation*}
\widetilde{A}_{ext} := W(c_1) + a_2 M(b_2)W(c_2).
\end{equation*}

We have shown, in Section \ref{Windingnumberanalysis}, that $\widetilde{A}_{ext}$ is Fredholm,  and 
\begin{enumerate}[\hspace{18pt}(i)]
\item if $s < 1 + 1/p + \alpha_c$ then $\operatorname{ind} \widetilde{A}_{ext} = 1$;
\item if $s > 1 + 1/p + \alpha_c$ then $\operatorname{ind} \widetilde{A}_{ext} = 0$.
\end{enumerate}

Of course, our interest is in the operator
\begin{equation}
\widetilde{A}= W(c_1) + a_2 M(b_2)W(c_2): L_{p,0}(\mathbb{R}_+) \to L_p(\mathbb{R}_+),
\end{equation}
which can usefully be considered as the \textit{restriction} of $\widetilde{A}_{ext}$ to $L_{p,0}(\mathbb{R}_+)$.  \\

Let $\widetilde{A}_0 : L_p(\mathbb{R}_+) \to L_p(\mathbb{R}_+)$ be the linear operator defined by
\[
\widetilde{A}_0 w = 
  \begin{cases} 
   \widetilde{A}_{ext} w, & \text{if } w \in L_{p,0}(\mathbb{R}_+) \\
   0,       & \text{if } w \in M_{p,0}(\mathbb{R}_+). 
  \end{cases}
\]
\\
Then $\widetilde{A}_0 -  \widetilde{A}_{ext}$ has rank one, and is therefore compact. In particular, $\operatorname{ind} \widetilde{A}_0 = \operatorname{ind} \widetilde{A}_{ext}$. On the other hand, it is clear that  $\operatorname{ind} \widetilde{A}_0 = \operatorname{ind} \widetilde{A} + 1$. Hence,
\begin{equation} \label{indexdiff}
\operatorname{ind} \widetilde{A} = \operatorname{ind} \widetilde{A}_{ext} - 1.
\end{equation}
Thus, if $s < 1 + 1/p + \alpha_c$ then $\operatorname{ind} \widetilde{A} =0$, and if $s > 1 + 1/p + \alpha_c$ then $\operatorname{ind} \widetilde{A} = -1$.\\

Repeating the argument from Section \ref{SecondTheoremIndex},
\begin{equation}
\operatorname{ind} \mathcal{A} = \operatorname{ind} \widetilde{A}.
\end{equation} \\

To complete the proof of the second main result, we now consider (the dimension of) $\operatorname{Ker}  \mathcal{A}$, for the cases $p=2 , \, p>2$ and $p<2$ respectively. \\

\textbf{Firstly, suppose $p=2$}. Then, from Theorem \ref{thmabddtrivker}, $\operatorname{dim} \operatorname{Ker}  \mathcal{A} = 0$, for  $1 + \tfrac{1}{2} < s < 2 +\tfrac{1}{2}$.\\

\textbf{Secondly, suppose $p > 2$.} Then, for $0 < \delta <  \alpha_c$ or $\alpha_c < \delta <1$, we define
\begin{equation*}
X_1:= H^{1 + \tfrac{1}{2} + \delta}_{2,0}(\overline{\mathbb{R}_+}), \quad Y_1:= H^{1 + \tfrac{1}{2} + \delta - 2 \alpha}_{2}(\overline{\mathbb{R}_+})
\end{equation*}
and
\begin{equation*}
X_2:= H^{1 + \tfrac{1}{p} + \delta}_{p,0}(\overline{\mathbb{R}_+}), \quad Y_2:= H^{1 + \tfrac{1}{p} + \delta - 2 \alpha}_{p}(\overline{\mathbb{R}_+}).
\end{equation*}
Then $X_1$ ($Y_1$) is continuously and densely embedded into $X_2$ (into $Y_2$, respectively). Moreover, $ \mathcal{A} : X_j \to Y_j$ is Fredholm, $j = 1, 2$, and
\begin{equation*}
\operatorname{Ind}_{X_1 \to Y_1}  \mathcal{A} =\operatorname{Ind}_{X_2 \to Y_2}  \mathcal{A} .
\end{equation*}
Therefore, by Lemma \ref{FredholmKer}, 
\begin{equation*}
\operatorname{Ker}_{X_1 \to Y_1}  \mathcal{A} =\operatorname{Ker}_{X_2 \to Y_2}  \mathcal{A} .
\end{equation*}
That is,
\begin{equation*}
\operatorname{Ker}_{X_2 \to Y_2}  \mathcal{A} = \{0\}.
\end{equation*} 

\textbf{Thirdly, suppose $p < 2$.} We now define
\begin{equation*}
X_2:= H^{1 + \tfrac{1}{2} + \delta}_{2,0}(\overline{\mathbb{R}_+}), \quad Y_2:= H^{1 + \tfrac{1}{2} + \delta - 2 \alpha}_{2}(\overline{\mathbb{R}_+})
\end{equation*}
and
\begin{equation*}
X_1:= H^{1 + \tfrac{1}{p} + \delta}_{p,0}(\overline{\mathbb{R}_+}), \quad Y_1:= H^{1 + \tfrac{1}{p} + \delta - 2 \alpha}_{p}(\overline{\mathbb{R}_+}).
\end{equation*}
We can now repeat the argument made above, for the case $p>2$, to show that
\begin{equation*}
\operatorname{Ker}_{X_1 \to Y_1}  \mathcal{A} = \{0\}.
\end{equation*}

So, finally, if $0 < \alpha <1, \,\, 1 < p < \infty$ and $1 + 1/p < s < 1 + 1/p + \alpha_c$, then the operator $\mathcal{A}: H^s_{p,0}(\overline{\mathbb{R}_+}) \to H^{s-2\alpha}_p(\overline{\mathbb{R}_+})$ is invertible. \\

On the other hand, if $0 < \alpha <1, \,\, 1 < p < \infty$ and $1 + 1/p + \alpha_c < s < 2 +1/p$, then $\mathcal{A}$ has a trivial kernel and is Fredholm with index equal to $-1$. \\

%% file: KCL_Thesis_Chapter7_v5.tex
\chapter{Transcendental equation} \label{ChapterTE}
\section{Main results}
Following Chapter \ref{GenSymbol}, let $\big ( A_{\alpha, p,s}, \, \Gamma_M \big )$ denote the generalised symbol and associated contour of the operator 
\begin{equation*}
\widetilde{A}:= W(c_1)  + a_2M^0(b_2)W(c_2),
\end{equation*}
defined on $L_p(\mathbb{R}_+)$. Then, see Theorem \ref{AtilldeFredholm}, $\widetilde{A}$ is Fredholm if and only if
\begin{equation*}
\inf_{\omega \in \Gamma_M} |A_{\alpha, p,s}(\omega)| > 0.
\end{equation*} 
In the case $1/p < s < 1+1/p$, from Chapter \ref{GenSymbol}, we have the following constraints on the values of $\alpha, p$ and $s$:  
\begin{equation}  \label{SimpleConstrSmall}
0 < \alpha < \tfrac{1}{2}, \quad 1< p < \infty \quad \text{and} \quad 1/p < s < 1+1/p.
\end{equation}

Similarly, from Section \ref{GenSymbolbig}, for higher regularity, we will assume
\begin{equation}  \label{SimpleConstrBig}
0 < \alpha < 1, \quad 1< p < \infty \quad \text{and} \quad 1+ 1/p < s < 2+1/p.
\end{equation}

\begin{theorem} \label{TheoremTranscend}
For all $\alpha, p, s$ satisfying the conditions $0 < \alpha < \tfrac{1}{2}, \, 1< p < \infty \, $ and $\, 1/p < s < 1+ 1/p$, we have
\begin{equation*}
\inf_{\omega \in \Gamma_M} |A_{\alpha, p,s}(\omega)| > 0.
\end{equation*} 
\end{theorem}

\begin{theorem} \label{TheoremTranscendbig}
For all $\alpha, p, s$ satisfying the conditions $0 < \alpha < 1, \, 1< p < \infty \, $ and $\, 1+ 1/p < s <  2+ 1/p$, we have
\begin{equation*}
\inf_{\omega \in \Gamma_M} |A_{\alpha, p,s}(\omega)| > 0,
\end{equation*} 
unless $\xi=0$ and $s = 1 + 1/p + \alpha_c$, where $\alpha_c$ only depends on $\alpha$ and satisfies $0 < \alpha_c < \alpha$. 
\end{theorem}

\section{Background} \label{TEbckground}
We have seen in Chapter \ref{GenSymbol} that as $\omega$ varies over $\Gamma_M$, the symbol $A_{\alpha, p,s}(\omega)$ forms a closed loop in the complex plane. This loop comprises two components. The first lies on the unit circle, and the second is given in terms of certain transcendental functions.  \\

Let us  define
\begin{equation} \label{TBrecap}
T_B := \dfrac{\sin \pi \alpha}{\pi} B(s- 2 \alpha + 1 - 1/p + i \xi, 2 \alpha),
\end{equation}
and
\begin{equation} \label{Tsrecap}
T_s := \dfrac{\sin \pi (1/p + m - s + \alpha - i \xi)}{\sin \pi (1/p + m - s + 2 \alpha - i \xi)},
\end{equation}
where $m=1$ or $2$, as $1/p < s < 1 + 1/p$ or $1+ 1/p <  s < 2 + 1/p$ respectively, and $\xi \in \mathbb{R}$. It is easy to see that $T_s$ is independent of the choice of $m$, and it will be convenient for us to assume that $m=1$.\\

From Chapter \ref{GenSymbol} and Section \ref{GenSymbolbig}, for the cases $1/p < s < 1 + 1/p$ and $1+ 1/p <  s < 2 + 1/p$ respectively, in order to show that
\begin{equation*} 
\inf_{\omega \in \Gamma_M} |A_{\alpha, p,s}(\omega)| > 0,
\end{equation*} 
it is enough to show that the transcendental equation
\begin{equation}
T_s = T_B,
\end{equation}
has no solutions for $\alpha, p$ and $s$ varying subject to either the constraints \eqref{SimpleConstrSmall} or \eqref{SimpleConstrBig}.\\

It will be convenient to begin by considering the transcendental equation when $\xi = 0$. Moreover, from the simple relationships 
\begin{equation*}
\overline{\sin (z)} = \sin(\overline{z}) \quad \text{and} \quad \overline{\Gamma(z)} = \Gamma(\overline{z}) \quad \text{(6.1.23, p. 256, \cite{AandS})} ,
\end{equation*}
it is easy to see that if $(\alpha_0, p_0, s_0, \xi_0)$ is a solution quadruple, then so is $(\alpha_0, p_0, s_0, -\xi_0)$. Therefore, having the required result for $\xi =0$, it remains to consider the case $\xi >0$. \\

Finally, we note that $T_s$ and $T_B$ depend only the difference $(s-1/p)$, rather than $s$ and $p$ independently. Accordingly, we define
\begin{equation} \label{taudefinition}
\tau := s - 1/p.
\end{equation}
Of course, we are interested in either $0 < \tau < 1$ or $1< \tau < 2$. \\

From equation \eqref{TBrecap} with $\xi=0$,
\begin{equation*} 
T_B = \dfrac{\sin \pi \alpha}{\pi} B(\tau - 2 \alpha + 1, 2 \alpha) = \dfrac{\sin \pi \alpha}{\pi} \, \dfrac{\Gamma(\tau - 2\alpha +1)\Gamma(2 \alpha)}{\Gamma(\tau+1)}.
\end{equation*}

From equation \eqref{Tsrecap} with $\xi=0$,
\begin{equation*} 
T_s = \dfrac{\sin \pi ( 1 - \tau + \alpha)}{\sin \pi ( 1 - \tau + 2 \alpha)} = \dfrac{\sin \pi ( \alpha - \tau)}{\sin \pi ( \tau - 2 \alpha +1)}.
\end{equation*}
Now $\Gamma(z) \Gamma(1-z) = \pi /\sin \pi z$, see 5.5.3, \cite{NIST}. Hence, taking $z= \tau - 2\alpha +1$, we can re-write the equation $T_s = T_B$ as
\begin{equation} \label{TEatzero}
\Gamma(2 \alpha - \tau) \Gamma(\tau+1) \sin \pi(\alpha-\tau) = \Gamma(2 \alpha) \sin \pi \alpha.
\end{equation} 

If $0 < \alpha < 1$ and $1 < \tau <2$, it turns out, see Lemma \ref{lemmaTE5}, that equation \eqref{TEatzero} has a unique solution of the form $\tau = 1 + \alpha_c$, where $\alpha_c$ only depends on $\alpha$ and satisfies $0 < \alpha_c < \alpha$. \\

\begin{remark}
In this chapter we will do several computations involving the function $\arg(\cdot)$. For $z \in \mathbb{C} \setminus \{ 0 \}$ and $c \in \mathbb{R}$, we shall write
\begin{equation*}
\arg z  \equiv c,
\end{equation*}
to indicate that 
\begin{equation*}
\arg z = c + 2 \pi k,
\end{equation*}
for some $k \in \mathbb{Z}$. (Of course, if $-\pi < c \leq \pi$, then $k=0$. See \eqref{argsingleval}.)\\
\end{remark}

\section{Proof of main results}
\begin{theorem}
Suppose $0<\alpha < \tfrac{1}{2}, \, 0 < \tau < 1$ and $\xi \in \mathbb{R}$. Define
\begin{equation} \label{Tsrecaptau}
T_s := \dfrac{\sin \pi ( 1 - \tau + \alpha - i \xi)}{\sin \pi ( 1 - \tau + 2 \alpha - i \xi)},
\end{equation}
and
\begin{equation} \label{TBrecaptau}
T_B := \dfrac{\sin \pi \alpha}{\pi} B(\tau - 2 \alpha + 1 + i \xi, 2 \alpha).
\end{equation}
\begin{enumerate} [\hspace{18pt}(a)]
\item If $\xi =0$ and $0 < \tau <1$, then $T_s \not = T_B$. \\
\item If $\xi \geq \tfrac{1}{4}$ and $\alpha \leq \tau <1$, then $|T_s| > |T_B|$. \\
\item If $0 < \xi <  \tfrac{1}{4}$ and $\alpha \leq \tau <1$, then $\arg T_s < \arg T_B$. \\
\item If $ \xi >0$ and $0 < \tau < \alpha$, then $\arg T_s \not = \arg T_B$. 
\end{enumerate}
In other words, the transcendental equation $T_s=T_B$ has no solutions for $0< \alpha < \tfrac{1}{2}, \, 0 < \tau < 1$ and $\xi \in \mathbb{R}$.
\end{theorem}
\begin{proof}
See Lemmas \ref{lemmaTE1}, \ref{lemmaTE2}, \ref{lemmaTE3} and \ref{lemmaTE4} for the proof of cases $(a), (b), (c)$ and $(d)$ respectively. \\
\end{proof}

\begin{theorem}
Suppose $0 < \alpha < 1$ and $\xi \in \mathbb{R}$. Let $T_s$ and $T_B$ be as defined previously in \eqref{Tsrecaptau} and \eqref{TBrecaptau} respectively.

\begin{enumerate} [\hspace{18pt}(a)]
\item If $\xi =0$ and $1 < \tau < 2$, then there exists a unique $\alpha_c$ such that $T_s  = T_B$, with $\tau = 1 + \alpha_c$. \\
\item If $\xi \geq \tfrac{1}{4}$ and $1 + \alpha \leq \tau < 2$, then $|T_s| > |T_B|$. \\
\item If $0 < \xi <  \tfrac{1}{4}$ and $1 + \alpha \leq \tau < 2$, then $\arg T_s \not = \arg T_B$. \\
\item If $\xi >0$ and $1 < \tau < 1+ \alpha$ then $\arg T_s \not = \arg T_B$. 
\end{enumerate}
In other words, for a given $0 < \alpha < 1$, the transcendental equation $T_s=T_B$ has a unique solution in the range $1 < \tau < 2$ and $\xi \in \mathbb{R}$, which occurs when $\xi=0$ and $\tau = 1 + \alpha_c$. (In particular, if $\xi \not =0$, then the equation $T_s=T_B$ has no solutions for $0 < \alpha < 1$ and $1 < \tau < 2$.)
\end{theorem}
\begin{proof}
See Lemmas \ref{lemmaTE5} and \ref{lemmaTE6} for the proof of cases $(a)$ and $(b)$ respectively.  For case $(c)$ we use Lemma \ref{lemmaTE3} if $0 < \alpha < \tfrac{1}{2}$, and Lemma \ref{lemmaTE7} if $\tfrac{1}{2} \leq \alpha < 1$. Finally, for case $(d)$, see Lemma \ref{lemmaTE8}.\\
\end{proof}

\begin{lemma} \label{lemmaTE1}
If $0<\alpha < \tfrac{1}{2}, \, \xi =0$ and $0 < \tau <1$, then $T_s \not = T_B$.
\end{lemma}
\begin{proof}
Suppose that $\alpha \in (0, \tfrac{1}{2})$ is fixed and $\xi =0$. We note that if $\tau = 2\alpha$  then the left-hand side of equation \eqref{TEatzero} becomes infinite, whilst the right-hand side is finite. Moreover, if $\alpha \leq \tau < 2 \alpha$, then the left-hand side is bounded above by zero, whereas the right-hand side is strictly positive. In other words, equation \eqref{TEatzero} can have no solutions for $\tau$ in the range $\alpha \leq \tau \leq 2 \alpha$. \\

If $\tau \in (0, \alpha) \cup (2\alpha, 1)$, it easy to see that
\begin{equation*}
0 < \Gamma(2 \alpha - \tau) \Gamma(\tau+1)\sin \pi (\alpha- \tau) < \infty.
\end{equation*}

Now suppose that $0 \leq \tau < \alpha$. It is trivially obvious that $\tau =0$ is an (inadmissible) solution of equation \eqref{TEatzero}. But since, from Lemma \ref{ftaualphaderiv}, the derivative of the left-hand side, with respect to $\tau$, is strictly negative we see immediately that
\begin{equation*}
\Gamma(2 \alpha - \tau) \Gamma(\tau+1) \sin \pi(\alpha-\tau) < \Gamma(2 \alpha) \sin \pi \alpha.
\end{equation*}

On the other hand, suppose that $2 \alpha < \tau \leq 1$. Then, by Lemma \ref{ftaualphaderiv}, the left-hand side of equation \eqref{TEatzero} is a strictly decreasing function of $\tau$ over this range. But if $\tau = 1$, then 
\begin{align*}
\Gamma(2 \alpha - \tau) \Gamma(\tau+1) \sin \pi(\alpha-\tau) & = \Gamma(2 \alpha - 1) \Gamma(2) \sin \pi(\alpha-1) \\
&= \dfrac{\Gamma(2 \alpha)}{1-2\alpha} \sin \pi \alpha  \quad (\Gamma(z+1) = z \Gamma(z))\\
& >\Gamma(2 \alpha) \sin \pi \alpha.
\end{align*}
Hence, if $\xi =0$ and $0 < \tau <1$, then $T_s \not = T_B$. \\
\end{proof}

Given that we have proved there are no solutions to the transcendental equation for $\xi=0$, for the case $0 < \tau < 1$, it remains to consider the case $\xi >0$.  (See Section \ref{TEbckground}.) \\

From equation \eqref{TBrecap} and Remark \ref{remark:mellinop}, we have
\begin{equation*}
0 < |T_B| < \infty \text{  for all finite  } \xi >0, \text{  and  } |T_B| \to 0 \text{  as  } \xi \to \infty.
\end{equation*}
On the other hand, from Lemma \ref{lemSabsarg}
\begin{equation*}
0 < |T_s| < \infty \text{  for all finite  } \xi >0, \text{  and } |T_s| \to 1 \text{  as  } \xi \to \infty.
\end{equation*}
\\

\begin{lemma} \label{lemmaTE2}
If $0<\alpha < \tfrac{1}{2}, \, \, \xi \geq \tfrac{1}{4}$ and $\alpha \leq \tau <1$, then 
\begin{equation*}
|T_s| > \dfrac{2}{\pi} >  |T_B|.
\end{equation*}
\end{lemma}
\begin{proof}
Firstly, we find an upper bound for $|T_B|$. Define
\begin{equation*}
\sigma := \tau - 2\alpha +1,
\end{equation*}
so that $\sigma \geq 1- \alpha$. Now
\begin{align*}
|T_B| & = \dfrac{\sin \pi \alpha}{\pi} \, | B(\sigma+i \xi, 2\alpha) | \\
& = \dfrac{\sin \pi \alpha}{\pi} \, \Gamma(2\alpha) \bigg | \dfrac{\Gamma(\sigma +i \xi)}{\Gamma(\sigma + 2 \alpha +i \xi)} \bigg |.\\
& \leq \dfrac{\sin \pi \alpha}{\pi} \, \Gamma(2\alpha) \cdot \dfrac{\Gamma(\sigma)}{\Gamma( \sigma + 2 \alpha)}  \quad \text{by Lemma } \ref{lemTBalphasigmaxi}\\
& \leq \dfrac{\sin \pi \alpha}{\pi} \, \Gamma(2\alpha) \cdot \dfrac{\Gamma(1 - \alpha)}{\Gamma( 1 + \alpha)} \quad  \text{by Lemma } \ref{lemgammasigma}.
\end{align*}
Since $(\sin \pi \alpha)/\pi = 1/(\Gamma(1- \alpha)\Gamma(\alpha))$, we have
\begin{align*}
|T_B| & \leq \dfrac{\Gamma(2\alpha)}{\Gamma(1- \alpha)\Gamma(\alpha)} \,  \cdot \dfrac{\Gamma(1 - \alpha)}{\Gamma( 1 + \alpha)} \\
& = \dfrac{\Gamma(2\alpha)}{\Gamma(1+ \alpha)\Gamma(\alpha)} \\
& = \dfrac{2 \alpha \, \Gamma(2\alpha)}{2 \alpha \, \Gamma(1+ \alpha)\Gamma(\alpha)} \\
& = \dfrac{\Gamma(2\alpha+1)}{2 \, (\Gamma(1+ \alpha))^2} \\
& < \dfrac{\Gamma(2)}{2( \Gamma(\tfrac{3}{2}))^2}  \quad \text{by Lemma } \ref{lemfincr}\\
& = \dfrac{2}{\pi}, \quad \text{since  }  \Gamma(\tfrac{3}{2}) = \sqrt{\pi}/2.
\end{align*}
In other words, we can find a uniform upper bound for $|T_B|$ by taking $\alpha = \tfrac{1}{2}, \, \tau = \alpha$ and $\xi =0$. \\

Secondly, we determine a lower bound for $|T_s|$. By Lemma \ref{lemSabsarg}
\begin{align*}
|T_s| &= \bigg ( \dfrac{\cosh 2 \pi \xi - \cos 2 \pi (1-\tau + \alpha)}{\cosh 2 \pi \xi - \cos 2 \pi (1-\tau + 2\alpha)} \bigg)^{\frac1{2}} \\
& \geq \bigg ( \dfrac{\cosh 2 \pi \xi - 1}{\cosh 2 \pi \xi +1} \bigg)^{\frac1{2}}  \\
& = \tanh (\pi \xi) \\
& \geq \tanh(\pi/4) \quad \text{for } \xi \geq \tfrac{1}{4}.
\end{align*}

So finally, 
\begin{equation*}
|T_s| \geq \tanh (\pi/4) > 0.655 > \dfrac{2}{\pi} > |T_B|.
\end{equation*} 
That is, for $\xi \geq \tfrac{1}{4}$ and $\alpha \leq \tau < 1$, we have $|T_s| > \dfrac{2}{\pi} > |T_B|$, as required. \\
\end{proof}

\begin{lemma} \label{lemmaTE3}
Let $0<\alpha < \tfrac{1}{2}$ and  $0 < \xi <  \tfrac{1}{4}$. If $\alpha \leq \tau <1$ or $1+ \alpha \leq \tau <2$, then 
\begin{equation*}
\arg T_B > - 4\alpha \xi > \arg T_s.
\end{equation*}
\end{lemma}
\begin{proof}
Let us define
\begin{equation*}
\sigma := \tau - 2\alpha + 1,
\end{equation*}
so that $\sigma \geq 1- \alpha$, if $\alpha \leq \tau <1$ or $1+ \alpha \leq \tau <2$. Hence, from Lemma \ref{lemBarg}, 
\begin{equation*}
\arg T_B > - 4\alpha \xi. 
\end{equation*} \\

We now find an upper bound for $\arg T_s$. Let $b:=1 - \tau + 2\alpha$, so that
\begin{equation*}
T_s = \dfrac{\sin \pi (b- \alpha - i \xi)}{\sin \pi (b -i \xi)}.
\end{equation*}

Then, since $\sin (x + i y) = \sin x\cosh y  + i \cos x \sinh y$, see, for example, 4.21.37, \cite{NIST}, a routine calculation gives
\begin{equation*}
\operatorname{Im } T_s = -\dfrac{\sin \pi \alpha \sinh 2 \pi \xi}{\cosh 2\pi \xi - \cos 2 \pi b} \quad < 0,
\end{equation*}
and

\begin{equation*}
\operatorname{Re} \, T_s = \dfrac{-\cos \pi(\alpha-2 b) + \cos \pi \alpha \cosh 2 \pi \xi}{\cosh 2 \pi \xi - \cos 2 \pi b}.
\end{equation*} \\

Since $\operatorname{Im} \,  T_s < 0$, we must have $- \pi < \arg T_s < 0$. As $\xi > 0$, $\cosh 2 \pi \xi - \cos 2 \pi b > 0$, and we can determine an upper bound for $\arg T_s$ by finding an upper bound for $-\cos \pi(\alpha-2 b) + \cos \pi \alpha \cosh 2 \pi \xi $. \\

But $-\cos \pi(\alpha-2 b) + \cos \pi \alpha \cosh 2 \pi \xi \leq 1 + \cosh 2 \pi \xi$, and thus 
\begin{align*}
\arg T_s &\leq - \arctan \bigg (\dfrac{\sin \pi \alpha  \sinh 2 \pi \xi}{1 + \cosh 2 \pi \xi} \bigg) \\
& = - \arctan ( \sin \pi \alpha \tanh \pi \xi) \\
& < - 4 \alpha \xi \quad \text{  by Lemma  } \ref{sinpialphatanh}.
\end{align*}

So, finally
\begin{equation*}
\arg T_B > - 4\alpha \xi > \arg T_s.
\end{equation*}
\end{proof}

\begin{lemma} \label{lemmaTE4}
If $0<\alpha < \tfrac{1}{2}, \,\, \xi >0$ and $0 < \tau < \alpha$, then 
\begin{equation*}
 \arg T_s \not = \arg T_B.
\end{equation*}
\end{lemma}
\begin{proof} We have
\begin{equation*}
T_B = \dfrac{\sin \pi \alpha}{\pi} \, \dfrac{\Gamma(\tau+1 - 2\alpha + i \xi) \,\Gamma(2\alpha)}{\Gamma(\tau +1+ i \xi)}.
\end{equation*}
On the other hand,
\begin{align*}
T_s & = \dfrac{\sin \pi( 1 - \tau + \alpha - i \xi)}{\sin \pi(1 - \tau + 2\alpha - i \xi)} \\
& = \dfrac{\sin \pi(\tau + 1 - \alpha + i \xi)}{\sin \pi(\tau + 1 - 2 \alpha + i \xi)}.
\end{align*}
Using the identity $\sin \pi z = \pi / (\Gamma(z) \Gamma(1-z))$, \text{(see 5.5.3, \cite{NIST})},
\begin{equation*}
T_s = \dfrac{\Gamma(\tau + 1 - 2 \alpha + i \xi)\Gamma(2 \alpha - \tau - i \xi)}{\Gamma(\tau + 1 - \alpha + i \xi)\Gamma(\alpha - \tau - i \xi)}.
\end{equation*}
Hence, noting that $\overline{\Gamma(z)} = \Gamma(\overline{z})$, we have 
\begin{align*}
& \arg T_B - \arg T_s \\
& = \arg \bigg ( \dfrac{\Gamma(\tau+1 - 2\alpha + i \xi)}{\Gamma(\tau +1+ i \xi)}\bigg ) -\arg \bigg (   \dfrac{\Gamma(\tau + 1 - 2 \alpha + i \xi)\Gamma(2 \alpha - \tau - i \xi)}{\Gamma(\tau + 1 - \alpha + i \xi)\Gamma(\alpha - \tau - i \xi)}\bigg )  \\
& \equiv \arg \bigg ( \dfrac{\Gamma(\tau + 1 - \alpha + i \xi)\, \Gamma(\alpha - \tau - i \xi)}{\Gamma(\tau +1+ i \xi)\, \Gamma(2 \alpha - \tau - i \xi)} \bigg )  \\
& \equiv \arg \bigg ( \dfrac{\Gamma(\tau + 1 - \alpha + i \xi)}{\Gamma(\tau + 1 + i \xi)} \bigg ) + \arg \bigg ( \dfrac{\Gamma(\alpha - \tau - i \xi)}{\Gamma(2\alpha - \tau - i \xi)} \bigg )  \\
& \equiv \arg \bigg ( \dfrac{\Gamma(\tau + 1 - \alpha + i \xi)}{\Gamma(\tau + 1 + i \xi)} \bigg ) -\arg \bigg ( \dfrac{\Gamma(\alpha - \tau + i \xi)}{\Gamma(2\alpha - \tau + i \xi)} \bigg )  \quad (\arg \overline{z} = - \arg z)\\
& \equiv \arg B(\tau+1-\alpha + i \xi, \alpha) -  \arg B(\alpha - \tau + i \xi, \alpha). 
\end{align*}

In other words, $\arg T_B - \arg T_s = T_\Delta + 2 \pi k$, for some $k \in \mathbb{Z}$, where
\begin{equation*}
T_\Delta := \arg B(\tau+1-\alpha + i \xi, \alpha) - \arg B(\alpha - \tau+ i \xi, \alpha).
\end{equation*}
To find an upper bound for $T_\Delta$ we note that since $0 < \alpha <  \tfrac{1}{2} < 1$, from Corollary \ref{Corrbetasigmagammaest}, 
$- \pi/2 < \arg B(\tau+1-\alpha + i \xi, \alpha), \, \arg B(\alpha - \tau+ i \xi, \alpha) <0$ and therefore,
\begin{equation*}
T_\Delta < \dfrac{\pi}{2}.
\end{equation*}

We now use the identity (1.625 9, p. 59, \cite{GR})
\begin{equation*}
\arctan(x) - \arctan(y) = \arctan \bigg ( \dfrac{x-y}{1+ xy}\bigg ) \quad \text{if} \quad xy > -1,\\
\end{equation*}
in conjunction with the result from Lemma \ref{argBarctan}, to compute $\arg B(\tau+1-\alpha + i \xi, \alpha)$ and $\arg B(\alpha - \tau + i \xi, \alpha)$ in turn.  \\

Since $0 < \alpha <  \tfrac{1}{2} < 1$, from Corollary \ref{Corrbetasigmagammaest}, we can write
\begin{equation*}
T_\Delta = \sum^\infty_{n=0} \bigg ( \arctan \dfrac{\xi \alpha}{(\sigma_s + n)(\sigma_s + n + \alpha)+ \xi^2} \,  - \,  \arctan \dfrac{\xi \alpha}{(\sigma_B + n)(\sigma_B + n + \alpha)+ \xi^2} \bigg ),
\end{equation*}
where $\sigma_s := \alpha - \tau$ and $\sigma_B := \tau + 1 - \alpha$. \\

But $0 < \alpha - \tau < \alpha < 1- \alpha < \tau + 1 - \alpha $ and hence
\begin{equation*}
0 < \sigma_s < \sigma_B,
\end{equation*}
and, thus, $T_\Delta >0$. \\

In summary, since $\arg T_B - \arg T_s = T + 2 \pi k$, with $0 < T_\Delta < \pi/2$, we have 
 \begin{equation*}
 \arg T_s \not = \arg T_B,
 \end{equation*}
 as required. \\
\end{proof}

In the case that $1 < \tau < 2$, the transcendental equation \eqref{TEatzero} always has a unique root $\tau_c$, where $1 < \tau_c < 1+\alpha$. (Lemmas \ref{ftaualpharoot} and \ref{ftaualpharootbig} provide the details for $0 < \alpha < \tfrac{1}{2}$ and $\tfrac{1}{2} \leq \alpha < 1$ respectively.) \\

\begin{lemma} \label{lemmaTE5}
Suppose $0 < \alpha < 1$. If $\xi =0$ and $1 < \tau < 2$, then $T_s \not = T_B$, unless $\tau = 1 + \alpha_c$.
\end{lemma}
\begin{proof}
From equation \eqref{TEatzero}, we can re-write equation $T_s = T_B$ as
\begin{equation*}
\Gamma(2 \alpha - \tau) \Gamma(\tau+1) \sin \pi(\alpha-\tau) = \Gamma(2 \alpha) \sin \pi \alpha.
\end{equation*}

We define 
\begin{equation*}
f(\tau ; \alpha) := \Gamma (2 \alpha - \tau) \Gamma(\tau +1)  \sin  \pi(\alpha - \tau). 
\end{equation*}
\\
\textbf{Firstly, suppose $0 < \alpha < \tfrac{1}{2}$.} \\
If $1 < \tau \leq 1 + \alpha$ then, by Lemma \ref{ftaualpharoot}, $T_s \not = T_B$ unless $\tau = 1 + \alpha_c$. \\

If $1 + \alpha < \tau < 1 + 2 \alpha$ then, by a routine calculation, 
\begin{equation*}
f(\tau ; \alpha) < 0.
\end{equation*}
On the other hand, $\Gamma(2 \alpha) \sin \pi \alpha > 0$ and hence, $T_s \not = T_B$. \\

If $\tau = 1 + 2\alpha$ the the left-hand side of equation \eqref{TEatzero} is infinite, the right-hand side is finite and, again, the required result follows. \\

Finally, if $1 + 2\alpha < \tau < 2$ then $f(\tau ; \alpha) > 0$. We now apply Lemma \ref{ftaualphaderiv}, and the required result follows if we can show that
\begin{equation*}
f(2;\alpha) > \Gamma(2 \alpha) \sin \pi \alpha.
\end{equation*}
But, for $0 < \alpha < \tfrac{1}{2}$,
\begin{align*}
f(2;\alpha) & = \Gamma(2\alpha-2)\Gamma(3)\sin\pi (\alpha -2) \\
& = \dfrac{2}{(2\alpha- 2 )(2\alpha -1)} \cdot \Gamma(2 \alpha) \sin \pi \alpha \\
& > \Gamma(2 \alpha) \sin \pi \alpha. \\
\end{align*}

\begin{figure}[H]
\centering
\includegraphics[width=10 cm, height=6 cm]{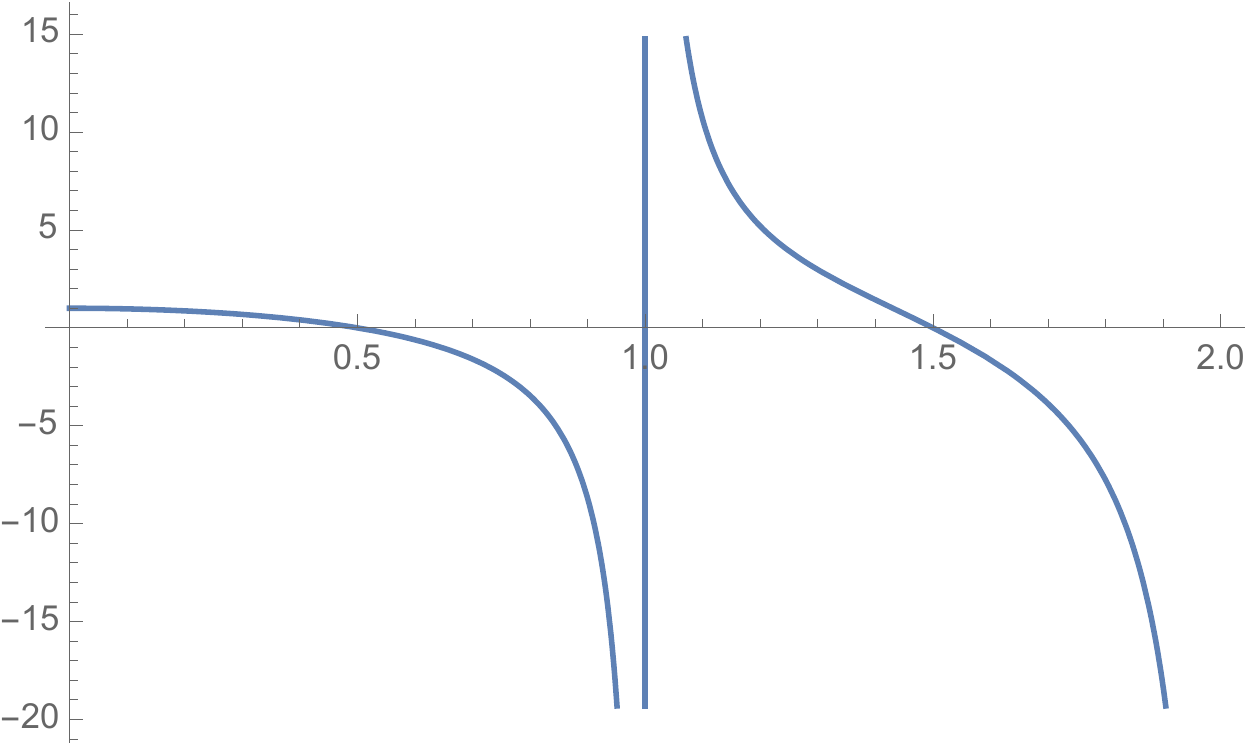} 
\caption{Graph of $f(\tau,\alpha)$ for $\alpha = 0.5$ and $0< \tau < 2$.}
\label{TEGraph}
\end{figure}

\textbf{Now suppose $\tfrac{1}{2} \leq \alpha < 1$.} \\
If $1 < \tau <  2\alpha$ or $1+ \alpha < \tau < 2$ then, by a routine calculation, 
\begin{equation*}
f(\tau ; \alpha) < 0.
\end{equation*}
On the other hand, $\Gamma(2 \alpha) \sin \pi \alpha > 0$ and hence, $T_s \not = T_B$. \\

If $\tau = 2\alpha$ the the left-hand side of equation \eqref{TEatzero} is infinite, the right-hand side is finite and, again, the required result follows. \\

Finally, if $2 \alpha < \tau \leq 1 + \alpha$ then, by Lemma \ref{ftaualpharootbig}, $T_s \not = T_B$ unless $\tau = 1 + \alpha_c$. \\
\end{proof}

\begin{lemma} \label{lemmaTE6}
Suppose $0 < \alpha < 1$. If $\xi \geq \tfrac{1}{4}$ and $1 + \alpha \leq \tau < 2$, then $|T_s| > |T_B|$.
\end{lemma}
\begin{proof}
Firstly, we find an upper bound for $|T_B|$. Define
\begin{equation*}
\sigma := \tau - 2\alpha +1,
\end{equation*}
so that $\sigma \geq 2- \alpha >0$. Now
\begin{align*}
|T_B| & = \dfrac{\sin \pi \alpha}{\pi} \, | B(\sigma+i \xi, 2\alpha) | \\
& = \dfrac{\sin \pi \alpha}{\pi} \, \Gamma(2\alpha) \bigg | \dfrac{\Gamma(\sigma +i \xi)}{\Gamma(\sigma + 2 \alpha +i \xi)} \bigg |.\\
& \leq \dfrac{\sin \pi \alpha}{\pi} \, \Gamma(2\alpha) \cdot \dfrac{\Gamma(\sigma)}{\Gamma( \sigma + 2 \alpha )}  \quad \text{by Lemma } \ref{lemTBalphasigmaxi}\\
& \leq \dfrac{\sin \pi \alpha}{\pi} \, \Gamma(2\alpha) \cdot \dfrac{\Gamma(2 - \alpha)}{\Gamma( 2 + \alpha)} \quad  \text{by Lemma } \ref{lemgammasigma}.
\end{align*}
Therefore, using the identity $\sin \pi z = \pi / (\Gamma(z) \Gamma(1-z))$, \text{see 5.5.3, \cite{NIST}},
\begin{align*}
|T_B| & \leq \dfrac{\Gamma(2\alpha)}{\Gamma(1- \alpha)\Gamma(\alpha)} \,  \cdot \dfrac{\Gamma(2 - \alpha)}{\Gamma( 2 + \alpha)} \\
& = \dfrac{\Gamma(2\alpha)}{\Gamma(1- \alpha)\Gamma(\alpha)} \,  \cdot \dfrac{(1- \alpha)\Gamma(1 - \alpha)}{\Gamma( 2 + \alpha)} \quad (\Gamma(z+1) = z \Gamma(z))\\
& = \dfrac{\Gamma(2\alpha) \cdot (1-\alpha)}{\Gamma(\alpha)\Gamma(2+\alpha)} \\
& = \dfrac{\Gamma(2\alpha) \cdot (1-\alpha)}{(\alpha-1) \Gamma(\alpha-1)\Gamma(2+\alpha)} \\
& = -\dfrac{\Gamma(2\alpha)}{ \Gamma(\alpha-1)\Gamma(2+\alpha)}.
\end{align*}

From Lemma \ref{lemgdecr}, 
\begin{align*}
|T_B|  \leq & \lim_{\alpha \searrow 0^+} \bigg ( -\dfrac{\Gamma(2\alpha)}{ \Gamma(\alpha-1)\Gamma(2+\alpha)}\bigg) \\
=  & \lim_{\alpha \searrow 0^+} \bigg ( -\dfrac{\Gamma(2\alpha+1)}{2 \alpha} \cdot \dfrac{(\alpha-1)\alpha}{ \Gamma(\alpha+1)\Gamma(2+\alpha)}\bigg) \\
= & \lim_{\alpha \searrow 0^+} \bigg ( \dfrac{(1-\alpha) \, \Gamma(2\alpha+1)}{2 \,\Gamma(\alpha+1)\Gamma(2+\alpha)} \bigg) = \dfrac{1}{2}.
\end{align*}

In other words, we can find a uniform upper bound for $|T_B|$ by taking $\alpha = 0, \, \tau = 1+ \alpha$ and $\xi =0$. \\

Secondly, we determine a lower bound for $|T_s|$. As in the proof of Lemma \ref{lemmaTE2},
\begin{equation*}
|T_s|  \geq \tanh(\pi/4) \quad \text{for } \xi \geq \tfrac{1}{4}.
\end{equation*}

So finally, 
\begin{equation*}
|T_s| \geq \tanh (\pi/4) > 0.655 > 0.5 \geq |T_B|.
\end{equation*} 
That is, for $\xi \geq \tfrac{1}{4}$ and $1 + \alpha \leq \tau < 2$, we have $|T_s| > |T_B|$, as required. \\
\end{proof}

\begin{lemma} \label{lemmaTE7}
Suppose $\tfrac{1}{2} \leq \alpha < 1$. If $0 < \xi <  \tfrac{1}{4}$ and $1 + \alpha \leq \tau < 2$, then 
\begin{equation*}
-\pi < \arg T_s \leq -\pi/2 < \arg T_B < 0.
\end{equation*}
\end{lemma}
\begin{proof}
Let $\sigma:= \tau + 1 - 2 \alpha$. Then
\begin{equation*}
\sigma \geq (1+ \alpha)+ 1 - 2\alpha = 2 - \alpha > \tfrac{1}{4} > \xi.
\end{equation*}
Hence, from Remark \ref{betasigmagammaest} with $\gamma = 2 \alpha \geq 1$,
\begin{equation*}
-\pi/2 < \arg B(\sigma + i \xi, 2\alpha) < 0.
\end{equation*}
Therefore, $-\pi/2 < \arg T_B < 0$. \\

We now determine bounds for $\arg T_s$. Let $b:= 1 - \tau + 2 \alpha$.
Then
\begin{equation*}
T_s = \dfrac{\sin \pi(b - \alpha - i \xi)}{\sin \pi (b - i \xi)}.
\end{equation*}

Then, as in the proof of Lemma \ref{lemmaTE3}, a routine calculation gives
\begin{equation*}
\operatorname{Im } T_s = -\dfrac{\sin \pi \alpha \sinh 2 \pi \xi}{\cosh 2\pi \xi - \cos 2 \pi b} \quad < 0,
\end{equation*}
and
\begin{equation*}
\operatorname{Re } T_s = \dfrac{(\cos \pi \alpha - \cos \pi (\alpha - 2b)) + (\cosh 2 \pi\xi - 1) \cos \pi \alpha}{\cosh 2\pi \xi -  \cos 2 \pi b}.
\end{equation*}

Of course, for $\xi >0$ we have $\cosh 2 \pi \xi -1 >0$. Moreover, if $\tfrac{1}{2} \leq \alpha < 1$ then $\cos \pi \alpha \leq 0$. Hence, $\operatorname{Re } T_s \leq 0$ if $\cos \pi \alpha - \cos \pi (\alpha - 2b) \leq 0$, which follows directly from Lemma \ref{coscosest}. \\

So, finally
\begin{equation*}
- \pi < \arg T_s \leq -\pi /2,
\end{equation*}
which completes the proof of the lemma.
\end{proof}

\begin{lemma} \label{lemmaTE8}
If $0<\alpha < 1, \,\, \xi >0$ and $1 < \tau < 1+ \alpha$, then 
\begin{equation*}
\arg T_s \not= \arg T_B.
\end{equation*}
\end{lemma}
\begin{proof} We follow the method taken in Lemma \ref{lemmaTE4}, and repeat the result
\begin{equation*}
\arg T_B - \arg T_s = \arg \bigg ( \dfrac{\Gamma(\tau + 1 - \alpha + i \xi)}{\Gamma(\tau + 1 + i \xi)} \bigg ) -\arg \bigg ( \dfrac{\Gamma(\alpha - \tau + i \xi)}{\Gamma(2\alpha - \tau + i \xi)} \bigg ),
\end{equation*}
but now $\alpha - \tau < 0$. However, we can write\\
\begin{align*}
\arg \bigg ( \dfrac{\Gamma(\alpha - \tau + i \xi)}{\Gamma(2\alpha - \tau + i \xi)} \bigg ) &= \arg \bigg ( \dfrac{(\alpha - \tau + i \xi) \Gamma(\alpha - \tau + i \xi)}{(2\alpha - \tau + i \xi)\Gamma(2\alpha - \tau + i \xi)}  \cdot \dfrac{(2\alpha - \tau + i \xi)}{(\alpha - \tau + i \xi)} \bigg ) \\
&\equiv \arg \bigg ( \dfrac{\Gamma(1+ \alpha - \tau + i \xi)}{\Gamma(1+2\alpha - \tau + i \xi)} \bigg )  + \arg \bigg ( \dfrac{2\alpha - \tau + i \xi}{\alpha - \tau + i \xi} \bigg ). 
\end{align*}

Therefore,
\begin{align*}
\arg T_B - \arg T_s  \equiv \bigg ( & \arg B(\tau + 1 - \alpha  + i \xi ,\alpha)- \arg B(1+ \alpha - \tau + i \xi , \alpha) \bigg )\\
& - \arg \bigg ( \dfrac{2\alpha - \tau + i \xi}{\alpha - \tau + i \xi} \bigg ).
\end{align*}

Noting that $0 < 1+\alpha- \tau < \tau + 1 - \alpha$,  and using the approach of Lemma \ref{lemmaTE4},
\begin{equation*}
\arg B(\tau + 1 - \alpha  + i \xi ,\alpha)- \arg B(1+ \alpha - \tau + i \xi , \alpha) = T_\Delta,
\end{equation*}
where $0 < T_\Delta < \pi/2$. \\

On the other hand, by a routine calculation, 
\begin{equation*}
\operatorname{Im} \bigg ( \dfrac{2\alpha - \tau + i \xi}{\alpha - \tau + i \xi} \bigg ) = \dfrac{-\alpha \, \xi}{(\alpha-\tau)^2 + \xi^2} < 0,
\end{equation*}
and thus
\begin{equation*}
0 < -\arg \bigg ( \dfrac{2\alpha - \tau + i \xi}{\alpha - \tau + i \xi} \bigg ) < \pi.
\end{equation*}

Therefore,
\begin{equation*}
\arg T_s \not= \arg T_B .
\end{equation*}
This completes the proof of the lemma. \\
\end{proof}

\section{Supporting lemmas}
\begin{remark} \label{remgammapsi}
In the lemmas that follow we will be considering various derivatives of combinations of gamma functions of real and complex arguments.  Suppose $z \in \mathbb{C}$. Then, see   6.1.23, p. 256 and 6.3.1, p. 258, \cite{AandS},
\begin{equation*}
\overline{\Gamma (z)} = \Gamma(\overline{z}); \qquad \Gamma'(z)= \Gamma(z) \psi(z),
\end{equation*}
where $\psi$ denotes the \textit{digamma} function. \\

Suppose $x>0$. Then from 6.4.1, p. 260, \cite{AandS}, we have
\begin{equation*}
\psi^{'}(x) > 0, \,\, \psi^{''}(x) < 0 \,\, \text{and} \,\, \psi^{'''}(x)>0.
\end{equation*}
In particular, the function $\psi(x)$ is concave and strictly increasing. \\
\end{remark}

For the purposes of Lemmas \ref{ftaualphaderiv}, \ref{ftaualpharoot} and \ref{ftaualpharootbig} we define 
\begin{equation*}
f(\tau ; \alpha) := \Gamma (2 \alpha - \tau) \Gamma(\tau +1)  \sin  \pi(\alpha - \tau).
\end{equation*}

\begin{lemma} \label{ftaualphaderiv}
Suppose that one of the following three conditions hold:
\begin{enumerate}[\hspace{18pt}(a)]
\item $0 < \alpha < \tfrac{1}{2} \,\, \text{and} \,\, 0 < \tau < \alpha$;
\item $0 < \alpha < 1 \,\, \text{and} \,\, 2\alpha < \tau < 1+ \alpha$;
\item $0 < \alpha < \tfrac{1}{2} \,\, \text{and} \,\, 1 + 2 \alpha < \tau < 2$.
\end{enumerate}
Then
\begin{equation*}
f(\tau ; \alpha) > 0 \quad \text{and} \quad \dfrac{\partial}{\partial \tau} f(\tau ; \alpha) < 0.
\end{equation*}
\end{lemma}
\begin{proof}
The assertion that $f(\tau ; \alpha) > 0$ follows immediately from the observation that if
$x \in (-2,-1) \cup (0, \infty)$ then $\Gamma(x) > 0$, and if $x \in (-1,0)$ then $\Gamma(x) < 0$. \\

Since $\Gamma'(z) = \Gamma(z) \psi(z)$, we have
\begin{align*}
\dfrac{\partial}{\partial \tau} &  f(\tau ; \alpha)\\
&= \Gamma(2 \alpha - \tau) \Gamma(\tau+1) \big \{(\psi(\tau+1) - \psi(2 \alpha- \tau)) \sin \pi (\alpha- \tau) - \pi \cos \pi (\alpha - \tau) \big \} \\
&= f(\tau ; \alpha)(\psi(\tau+1) - \psi(2 \alpha- \tau)) - \Gamma(2 \alpha - \tau) \Gamma(\tau+1)\pi \cos \pi (\alpha - \tau).
\end{align*}

We will now prove that if one of the conditions (a), (b) or (c) holds then
\begin{equation} \label{conditionpsiabc}
\psi(\tau+1) - \psi(2 \alpha- \tau)  < \psi(\tau+1- \alpha) - \psi(\alpha- \tau ). 
\end{equation} 
\\
Firstly, suppose $0 < \alpha < \tfrac{1}{2} \,\, \text{and} \,\, 0 < \tau < \alpha$.
Then
\begin{equation*}
\psi(\tau + 1) - \psi(2 \alpha - \tau) < \psi(\tau+1-\alpha) - \psi(\alpha- \tau),
\end{equation*}
follows directly from the concavity of $\psi(x)$ for $x>0$, since $\tau+1-\alpha > 1- \alpha > \alpha > \alpha - \tau >0$. \\

For the remaining two cases we note that
\begin{equation*}
\psi(z+1) = \psi(z) + 1/z, \quad z \in \mathbb{C}\setminus\{0\},
\end{equation*}
see 6.3.5, p. 258, \cite{AandS}. \\

Hence, for $0 < \alpha < 1 \,\, \text{and} \,\, 2\alpha < \tau < 1+ \alpha$,
\begin{align*}
\psi(\tau + 1) - \psi(2 \alpha - \tau) & = \psi(\tau + 1) - \psi(1+ 2 \alpha - \tau) + 1 /(2 \alpha - \tau) \\
& < \psi(\tau + 1-\alpha) - \psi(1+  \alpha - \tau) + 1 /(2 \alpha - \tau) \\
& = \psi(\tau + 1-\alpha) - \psi( \alpha - \tau) -1/(\alpha- \tau) + 1 /(2 \alpha - \tau) \\
& = \psi(\tau + 1-\alpha) - \psi( \alpha - \tau) + \{ 1/(\tau - \alpha) - 1 /(\tau -2 \alpha) \}  \\
& < \psi(\tau + 1-\alpha) - \psi( \alpha - \tau),
\end{align*}
noting that $\tau + 1-\alpha > 1+ \alpha > 1+  \alpha - \tau >0$. \\

Finally, for $0 < \alpha < \tfrac{1}{2} \,\, \text{and} \,\, 1 + 2 \alpha < \tau < 2$,
\begin{align*}
& \psi(\tau + 1) - \psi(2 \alpha - \tau) \\
& = \psi(\tau + 1) - \psi(1+ 2 \alpha - \tau) + 1 /(2 \alpha - \tau) \\
& = \psi(\tau + 1) - \psi(2+ 2 \alpha - \tau) + 1 /(2 \alpha - \tau) + 1 /(1+ 2 \alpha - \tau) \\
& < \psi(\tau + 1- \alpha) - \psi(2+  \alpha - \tau) + 1 /(2 \alpha - \tau) + 1 /(1+ 2 \alpha - \tau) \\
& = \psi(\tau + 1- \alpha) - \psi(\alpha - \tau) \\
& \qquad -1/(1+\alpha-\tau) - 1/(\alpha - \tau) + 1 /(2 \alpha - \tau)   + 1 /(1+ 2 \alpha - \tau) \\
& = \psi(\tau + 1- \alpha) - \psi(\alpha - \tau)  \\
& \qquad + \{ 1/(\tau - \alpha) - 1 /(\tau -2 \alpha) \}   + \{1/(\tau -1-\alpha)- 1 /(\tau- 1- 2 \alpha) \}\\
& < \psi(\tau + 1-\alpha) - \psi( \alpha - \tau),
\end{align*}
noting that $\tau + 1-\alpha > 2 + \alpha > 2+  \alpha - \tau >0$. \\

Hence, assuming that one of the conditions (a), (b) or (c) holds, then from \eqref{conditionpsiabc},
\begin{align*}
\psi(\tau+1) - \psi(2 \alpha- \tau) & < \psi(\tau+1- \alpha) - \psi( \alpha- \tau) \\
& = - \big ( \psi( \alpha- \tau)  -  \psi(1- \alpha + \tau) \big )\\
& = \dfrac{\pi}{\tan \pi(\alpha-\tau)},
\end{align*}
since $\psi(z) - \psi(1-z) = -\pi/ \tan \pi z,  \text{(see 6.3.7, p. 259, \cite{AandS})}$. \\

Since $f(\tau ; \alpha) >0$,
\begin{align*}
\dfrac{\partial}{\partial \tau} f(\tau ; \alpha) & < f(\tau ; \alpha)\dfrac{\pi}{\tan \pi(\alpha-\tau)} - \Gamma(2 \alpha - \tau) \Gamma(\tau+1)\pi \cos \pi (\alpha - \tau) \\
& = \Gamma(2 \alpha - \tau) \Gamma(\tau+1) \bigg \{ \dfrac{\pi}{\tan \pi(\alpha-\tau)} \cdot \sin \pi (\alpha- \tau) - \pi \cos \pi (\alpha- \tau)\bigg \} \\
& =0.
\end{align*}
This completes the proof of the lemma. \\
\end{proof}

\begin{lemma} \label{ftaualpharoot}
Suppose $0 < \alpha < \tfrac{1}{2}$. Then, for any given $\alpha$, there exists a unique $\tau_\alpha$, depending only on $\alpha$, such that
\begin{equation*}
f(\tau_\alpha ; \alpha) - \Gamma(2 \alpha) \sin \pi \alpha =0  \quad \text{and} \quad 1 < \tau_\alpha < 1 + \alpha.
\end{equation*}
\end{lemma}
\begin{proof}
Suppose $1 \leq \tau \leq 1+ \alpha$. Hence, $ -1-\alpha \leq - \tau \leq -1$, and thus 
$ -1 < -1+\alpha \leq 2\alpha - \tau \leq 2 \alpha -1 < 0$. Therefore, 
\begin{equation*}
-1 < 2\alpha - \tau < 0.
\end{equation*}
Hence, for any given $\alpha \in (0, \tfrac{1}{2})$, the function $f(\tau;\alpha)$ is continuous for all $\tau \in [1, 1+ \alpha]$. Moreover, 
\begin{align*}
f(1; \alpha) - \Gamma(2 \alpha) \sin \pi \alpha &= \Gamma (2\alpha-1) \cdot (-\sin \pi \alpha) - \Gamma(2 \alpha) \sin \pi \alpha\\
&= \sin \pi \alpha \cdot \Gamma(2 \alpha-1) \{ -1 - (2\alpha-1)\}\\
&= 2 \alpha \cdot \sin \pi \alpha \cdot (-\Gamma(2 \alpha-1))\\
&> 0.
\end{align*}
On the other hand,
\begin{equation*}
f(1+\alpha; \alpha) - \Gamma(2 \alpha) \sin \pi \alpha = - \Gamma(2 \alpha) \sin \pi \alpha < 0.
\end{equation*}
The existence of $\tau_\alpha$ now follows directly from the Intermediate Value Theorem, and Lemma \ref{ftaualphaderiv} guarantees its uniqueness. \\
\end{proof}

\begin{lemma} \label{ftaualpharootbig}
Suppose $\tfrac{1}{2} \leq \alpha < 1$. Then, for any given $\alpha$, there exists a unique $\tau_\alpha$, depending only on $\alpha$, such that
\begin{equation*}
f(\tau_\alpha ; \alpha) - \Gamma(2 \alpha) \sin \pi \alpha =0  \quad \text{and} \quad 2\alpha  < \tau_\alpha < 1 + \alpha.
\end{equation*}
\end{lemma}
\begin{proof}
Choose any $\delta$ such that $0 < \delta < 1- \alpha$. \\

Further, assume that $\tau$ satisfies $2\alpha + \delta \leq \tau \leq 1+ \alpha$. Then, 
$ -1-\alpha \leq - \tau \leq -2\alpha - \delta$ and 
$-1 <  -1+\alpha \leq 2\alpha - \tau \leq -\delta <0$. Therefore,
\begin{equation*}
-1 < 2\alpha - \tau < 0.
\end{equation*}
Hence, for any given $\alpha \in [\tfrac{1}{2},1)$, the function $f(\tau;\alpha)$ is continuous for all $\tau \in [2\alpha + \delta, 1+ \alpha]$. Moreover, 
\begin{align*}
f(2\alpha & + \delta; \alpha)  - \Gamma(2 \alpha) \sin \pi \alpha \\
&= \Gamma (-\delta) \cdot \Gamma(2\alpha + \delta + 1) \cdot \sin \pi (-\alpha - \delta) - \Gamma(2 \alpha) \sin \pi \alpha\\
&= (-\Gamma (-\delta)) \cdot \Gamma(2\alpha + \delta + 1) \cdot \sin \pi (\alpha + \delta) - \Gamma(2 \alpha) \sin \pi \alpha\\
&> 0, \quad \text{for sufficiently small } \delta.
\end{align*}

On the other hand,
\begin{equation*}
f(1+ \alpha; \alpha) - \Gamma(2 \alpha) \sin \pi \alpha = - \Gamma(2 \alpha) \sin \pi \alpha < 0.
\end{equation*}
The existence of $\tau_\alpha$ now follows directly from the Intermediate Value Theorem, and Lemma \ref{ftaualphaderiv} guarantees its uniqueness. \\
\end{proof}

\begin{lemma} \label{lemDxGammaMod2}
Suppose $z_0 = x + i y, \, z_1 = x + a + iy$ where $a,x,y \in \mathbb{R}$. Then
\begin{equation*}
\dfrac{\partial}{\partial x}  \bigg | \dfrac{\Gamma(z_0)}{\Gamma(z_1)}\bigg |^2  = \bigg | \dfrac{\Gamma(z_0)}{\Gamma(z_1)}\bigg |^2 \big \{ \psi(z_0) - \psi(z_1) + \psi(\overline{z_0}) - \psi(\overline{z_1})\big \}.
\end{equation*}
\end{lemma}
\begin{proof}
Since $\Gamma'(z)= \Gamma(z) \psi(z)$, it is easy to show that
\begin{equation*}
\dfrac{\partial}{\partial x}  \bigg ( \dfrac{\Gamma(z_0)}{\Gamma(z_1)} \bigg ) =  \dfrac{\Gamma(z_0)}{\Gamma(z_1)} \big \{ \psi(z_0) - \psi(z_1) \big \}.
\end{equation*}
Finally, since $\overline{\Gamma (z)} = \Gamma(\overline{z})$, we have
\begin{equation*}
\bigg | \dfrac{\Gamma(z_0)}{\Gamma(z_1)}\bigg |^2 = \dfrac{\Gamma(z_0)\Gamma(\overline{z_0})}{\Gamma(z_1)\Gamma(\overline{z_1})}
\end{equation*}
and the required result follows immediately. \\
\end{proof}

\begin{remark} \label{remDyGammaMod2}
Under the same hypotheses as Lemma \ref{lemDxGammaMod2}, we can similarly show that
\begin{equation*}
\dfrac{\partial}{\partial y}  \bigg | \dfrac{\Gamma(z_0)}{\Gamma(z_1)}\bigg |^2  = i \, \bigg | \dfrac{\Gamma(z_0)}{\Gamma(z_1)}\bigg |^2 \big \{ \psi(z_0) - \psi(z_1) - \psi(\overline{z_0}) + \psi(\overline{z_1})\big \}.
\end{equation*}
\end{remark}

\begin{lemma} \label{lemTBalphasigmaxi}
Suppose $0 < \alpha < 1, \, \sigma >0$ and $\xi > 0$. Then
\begin{equation*}
\bigg | \dfrac{\Gamma(\sigma + i \xi)}{\Gamma (\sigma + 2 \alpha + i \xi)}  \bigg | \leq  \dfrac{\Gamma(\sigma)}{\Gamma (\sigma + 2 \alpha)} .
\end{equation*}
\end{lemma}
\begin{proof}
To simplify the exposition, let us define
\begin{equation*}
R(\sigma + i \xi, 2 \alpha) := \dfrac{\Gamma(\sigma + i \xi)}{\Gamma (\sigma + 2 \alpha + i \xi)} .
\end{equation*}
Then, since $\overline{\Gamma(z)}=\Gamma(\overline{z})$ we have 
\begin{align*}
| R(\sigma + i \xi, 2 \alpha) |^2 :=  \dfrac{\Gamma(\sigma + i \xi)\Gamma(\sigma - i \xi)}{\Gamma (\sigma + 2 \alpha + i \xi) \Gamma (\sigma + 2 \alpha - i \xi)}.
\end{align*}
Further let
\begin{equation*}
z_0 = \sigma + i \xi; \quad z_1 = \sigma + 2 \alpha + i \xi.
\end{equation*} 

Then, since $\overline{\psi(z)}=\psi(\overline{z})$ (6.3.9, p.  259, \cite{AandS}), by Remark \ref{remDyGammaMod2},  
\begin{align*}
\frac{\partial}{\partial \xi} \, | R(\sigma + i \xi, 2 \alpha) |^2 &= i \, | R(\sigma + i \xi, 2 \alpha) |^2 \big \{ \psi(z_0) - \psi(z_1) - \psi(\overline{z_0}) + \psi(\overline{z_1})\big \} \\
&= i \, | R(\sigma + i \xi, 2 \alpha) |^2 \big \{ 2i \, \operatorname{Im} \psi(z_0) - 2i \,\operatorname{Im} \psi(z_1)\big \} \\
&=  2 \, | R(\sigma + i \xi, 2 \alpha) |^2 \big \{ \operatorname{Im} \psi(z_1) -  \,\operatorname{Im} \psi(z_0)\big \} \\
& < 0, \quad \text{by Lemma } \ref{lemImpsi}.
\end{align*}
This completes the proof of the lemma. \\
\end{proof}

\begin{lemma} \label{lemgammasigma}
Suppose $\alpha >0$ and $\sigma \geq \sigma_{min} >0$. Then 
\begin{equation*}
\dfrac{\Gamma(\sigma)}{\Gamma(\sigma + 2 \alpha)} \leq \dfrac{\Gamma(\sigma_{min})}{\Gamma(\sigma_{min} + 2 \alpha)}.
\end{equation*}

\end{lemma}
\begin{proof}
It is easy to see that 
\begin{equation*}
\dfrac{\partial}{\partial \sigma}\dfrac{\Gamma(\sigma)}{\Gamma(\sigma + 2 \alpha)} = \dfrac{\Gamma(\sigma)}{\Gamma(\sigma + 2 \alpha)} \big ( \psi(\sigma) - \psi(\sigma+ 2\alpha) \big ) < 0.
\end{equation*}
This completes the proof of the lemma. \\
\end{proof}

\begin{lemma} \label{lemfincr}
Suppose $\alpha >0$. Then the function
\begin{equation*}
f(\alpha):= \dfrac{\Gamma(2 \alpha + 1)}{(\Gamma(1 + \alpha ))^2}
\end{equation*}
strictly increases as $\alpha$ increases.
\end{lemma}
\begin{proof}
It is easy to see that 
\begin{equation*}
\dfrac{d}{d\alpha} \dfrac{\Gamma(2 \alpha + 1)}{(\Gamma(1 + \alpha ))^2} = \dfrac{2 \, \Gamma(2 \alpha + 1)}{(\Gamma(1 + \alpha ))^2} \big ( \psi(1+2 \alpha) - \psi(1+\alpha) \big ) > 0.
\end{equation*}
This completes the proof of the lemma. \\
\end{proof}

\begin{lemma} \label{lemgdecr}
Suppose $0 < \alpha < 1$. Then the function
\begin{equation*}
g(\alpha):= -\dfrac{\Gamma(2\alpha)}{ \Gamma(\alpha-1)\Gamma(2+\alpha)}
\end{equation*}
strictly decreases as $\alpha$ increases.
\end{lemma}
\begin{proof}
It is easy to see that
\begin{equation*}
\dfrac{dg}{d\alpha} = \dfrac{\Gamma(2\alpha)}{ \Gamma(\alpha-1)\Gamma(2+\alpha)} \cdot \big \{ \psi(2+\alpha) - 2 \psi(2\alpha) + \psi(\alpha-1) \big \}. \\
\end{equation*}

Since $\Gamma(\alpha-1) < 0$ for $0 < \alpha < 1$, it is enough to show that 
$ \psi(2+\alpha) - 2 \psi(2\alpha) + \psi(\alpha-1) > 0$. We note the identity
\begin{equation*}
\psi(z+1) = \psi(z) +\dfrac{1}{z} \quad \text{(6.3.5, p. 258, \cite{AandS})},
\end{equation*}
and that $\psi(x)$ is increasing for $x > 0$. 
Since $2 \, \psi(2\alpha) = \psi(\alpha) + \psi(\alpha + \tfrac{1}{2}) + \log 4$, see 6.3.8, p. 259, \cite{AandS}, we have
\begin{align*}
\psi(2+\alpha) - 2 \psi(2\alpha) + \psi(\alpha-1) & = \psi(2+\alpha) -  [\psi(\alpha) + \psi(\alpha + \tfrac{1}{2}) + \log 4 ] + \psi(\alpha -1) \\
& = \psi(2+\alpha) + \dfrac{1}{1-\alpha} - \psi(\alpha + \tfrac{1}{2}) - \log 4 \\
& > \psi(2) + \dfrac{1}{1-\alpha} - \psi(\tfrac{3}{2}) - \log 4.
\end{align*}

But, see  6.3.2, 6.3.3, p. 258, \cite{AandS}, 
\begin{align*}
\psi(2) - \psi(\tfrac{3}{2}) - \log 4 & = [ \psi(1) +1 ] - [\psi(\tfrac{1}{2}) +2] - \log 4 \\
& = [ \psi(1) +1] - [ \psi(1) - \log 4 +2] - \log 4 \\
& = -1.
\end{align*}

Hence, finally,
\begin{align*}
\psi(2+\alpha) - 2 \psi(2\alpha) + \psi(\alpha-1) & > \dfrac{1}{1-\alpha} -1 \\
& = \dfrac{\alpha}{1 - \alpha} \\
& > 0, 
\end{align*}
for $0 < \alpha < 1$. This completes the proof of the lemma. \\
\end{proof}

\begin{lemma} \label{lemSabsarg}
Suppose $a,b \in \mathbb{R}$. Define
\begin{equation*}
S(a,b ; \xi) := \dfrac{\sin[\pi (a-i\xi)]}{\sin[\pi (b-i\xi)]}.
\end{equation*}

Then
\begin{equation*}
|S(a,b ; \xi)| = \bigg ( \dfrac{\cosh 2 \pi \xi - \cos 2 \pi a}{\cosh 2 \pi \xi - \cos 2 \pi b}\bigg)^{\tfrac{1}{2}}.
\end{equation*}
\end{lemma}
\begin{proof}
From 4.21.37, \cite{NIST},
\begin{equation*}
\sin[\pi (a-i\xi)] = \cosh \pi \xi \sin \pi a - i \cos \pi a \sinh \pi \xi.
\end{equation*}
Therefore $|\sin[\pi (a-i\xi)]|^2$
\begin{align*}
&= \cosh^2 \pi \xi \, \sin^2 \pi a+ \cos^2 \pi a \, \sinh^2 \pi \xi \\
& = \tfrac{1}{2} \, (\cosh 2 \pi \xi +1) \, \sin^2 \pi a + \tfrac{1}{2} \, (\cosh 2 \pi \xi -1) \, \cos^2 \pi a \\
& = \tfrac{1}{2} \, ( \cosh 2 \pi \xi - \cos 2 \pi a).
\end{align*}
Hence
\begin{equation*}
|S(a,b ; \xi)| = \bigg ( \dfrac{\cosh 2 \pi \xi - \cos 2 \pi a}{\cosh 2 \pi \xi - \cos 2 \pi b}\bigg)^{\tfrac{1}{2}}.
\end{equation*} \\

\end{proof}

\begin{lemma} \label{lemImpsi}
Suppose  $\xi > 0$ and $\sigma > 0$. Then, for fixed $\xi$,
\begin{align*}
\operatorname{Im} \, \psi (\sigma + i \xi) \quad \text{decreases as } \sigma \text{ increases}. \end{align*}
\end{lemma}
\begin{proof}
From  Section 44.11, p. 455, \cite{Oldham},
\begin{equation*}
\operatorname{Im} \, \psi (\sigma + i \xi) = \sum^\infty_{j=0} \dfrac{\xi}{( j+\sigma)^2+\xi^2}.
\end{equation*}
Since
\begin{equation*}
\dfrac{\partial}{\partial \sigma} \bigg ( \dfrac{\xi}{( j+\sigma)^2+\xi^2} \bigg) = \dfrac{-2(j+\sigma) \xi}{[( j+\sigma)^2+\xi^2]^2} <0, \quad j=0,1,2, \dots,
\end{equation*}
the required result follows immediately. \\
\end{proof}


\begin{lemma} \label{argBarctan}
Suppose $\gamma >0, \sigma >0$ and $\xi >0$. Then
\begin{equation} \label{argbetasigma}
\arg B(\sigma + i \xi, \gamma) = \sum^\infty_{n=0} \bigg ( \arctan \dfrac{\xi}{\sigma + \gamma +n} -\arctan \dfrac{\xi}{\sigma + n} \bigg) + 2 \pi k,
\end{equation}
for some $k \in \mathbb{Z}$.
\end{lemma}
\begin{proof}
From  6.1.27, p. 256, \cite{AandS},
\begin{equation*}
\arg \Gamma (\sigma + i \xi) \equiv \xi \, \psi(\sigma) +\sum^\infty_{n=0} \bigg ( \dfrac{\xi}{\sigma + n}- \arctan  \dfrac{\xi}{\sigma + n}\bigg).
\end{equation*}
Now
\begin{align} 
\arg B(\sigma + i \xi, \gamma) &= \arg \bigg ( \dfrac{\Gamma(\sigma + i \xi) \, \Gamma (\gamma)}{\Gamma(\sigma + \gamma + i \xi) } \bigg ) \nonumber \\
&\equiv \arg \Gamma(\sigma + \gamma - i \xi) + \arg \Gamma(\sigma +  i \xi) \nonumber \\
&\equiv -\xi \, \psi(\sigma+ \gamma ) - \sum^\infty_{n=0} \bigg ( \dfrac{\xi}{\sigma + \gamma +n}- \arctan  \dfrac{\xi}{\sigma + \gamma + n}\bigg)  \label{Bsigmagamma}\\
& \quad +\xi \, \psi(\sigma) +\sum^\infty_{n=0} \bigg ( \dfrac{\xi}{\sigma + n}- \arctan  \dfrac{\xi}{\sigma + n}\bigg).  \nonumber 
\end{align}

We note from  8.363 3, p. 903, \cite{GR}, that
\begin{equation*}
\psi(x) - \psi(y) = \sum^\infty_{n=0} \bigg ( \dfrac{1}{y+n} - \dfrac{1}{x+n}\bigg ).
\end{equation*}
Using this result, with equation \eqref{Bsigmagamma}, we can write
\begin{equation*} 
\arg B(\sigma + i \xi, \gamma) = \sum^\infty_{n=0} \bigg ( \arctan \dfrac{\xi}{\sigma + \gamma +n} -\arctan \dfrac{\xi}{\sigma + n} \bigg) + 2 \pi k,
\end{equation*}
as required. \\
\end{proof}

\begin{lemma} \label{Sgamsigxi}
Suppose  $ 0 < \gamma < 1, \, \sigma > 0$ and $\xi >0$. Define
\begin{equation*}
S(\gamma, \sigma, \xi) := \sum^\infty_{n=0} \bigg ( \arctan \dfrac{\xi}{\sigma+\gamma+n} - \arctan \dfrac{\xi}{\sigma+n}\bigg).
\end{equation*}
Then
\begin{equation*}
-\dfrac{\pi}{2} < - \arctan \dfrac{\xi}{\sigma} < S(\gamma, \sigma, \xi) < \arctan \dfrac{\xi}{\sigma+\gamma} - \arctan \dfrac{\xi}{\sigma} < 0.
\end{equation*}
\end{lemma}
\begin{proof}
Since $0< \gamma < 1$, we can determine a lower bound for $S(\gamma, \sigma, \xi)$ by writing
\begin{align*}
S(\gamma, \sigma, \xi) & = \sum^\infty_{n=0} \bigg ( \arctan \dfrac{\xi}{\sigma+\gamma+n} - \arctan \dfrac{\xi}{\sigma+n}\bigg) \\
& > \sum^\infty_{n=0} \bigg ( \arctan \dfrac{\xi}{\sigma+ 1+n } -\arctan \dfrac{\xi}{\sigma+n}\bigg) \\
& = - \arctan \dfrac{\xi}{\sigma}.
\end{align*}

On the other hand, to determine an upper bound we note that
\begin{align*}
S(\gamma, \sigma, \xi) & = \sum^\infty_{n=0} \bigg ( \arctan \dfrac{\xi}{\sigma+\gamma+n} - \arctan \dfrac{\xi}{\sigma+n}\bigg) \\
& = \arctan \dfrac{\xi}{\sigma+\gamma} - \arctan \dfrac{\xi}{\sigma} + \sum^\infty_{n=1} \bigg ( \arctan \dfrac{\xi}{\sigma+\gamma+n} - \arctan \dfrac{\xi}{\sigma+n}\bigg) \\
& < \arctan \dfrac{\xi}{\sigma+\gamma} - \arctan \dfrac{\xi}{\sigma}.
\end{align*}
\end{proof}

We now give a corollary of Lemma \ref{argBarctan}, in the case $0< \gamma < 1$. 

\begin{corollary} \label{Corrbetasigmagammaest}
If we assume $0 < \gamma < 1$ in Lemma \ref{argBarctan}, then the integer $k=0$, and we have the estimate
\begin{equation*}
-\pi/2 <  \arg B(\sigma + i \xi, \gamma) < 0. 
\end{equation*}
\end{corollary}
\begin{proof}
See equation \eqref{argbetasigma} and Lemma \ref{Sgamsigxi}. \\
\end{proof}

On the other hand if $1 \leq \gamma < 2$ we need to add to an extra condition. 

\begin{remark} \label{betasigmagammaest}
Similarly, if $1 \leq \gamma < 2$ then
\begin{equation*}
- \arctan \dfrac{\xi}{\sigma} -\arctan \dfrac{\xi}{\sigma+1} < \arg B(\sigma + i \xi, \gamma)  <  \arctan \dfrac{\xi}{\sigma + \gamma} - \arctan \dfrac{\xi}{\sigma}.
\end{equation*}
Moreover, if $\sigma \geq \xi$, then
\begin{equation*}
- \arctan \dfrac{\xi}{\sigma} -\arctan \dfrac{\xi}{\sigma+1} = - \arctan \bigg (  \dfrac{\xi(2 \sigma+1)}{\sigma(\sigma +1) - \xi^2}\bigg), 
\end{equation*}
so that, in particular, if $1 \leq \gamma < 2$ and $\sigma \geq \xi$ then 
\begin{equation*}
-\pi/2 <  \arg B(\sigma + i \xi, \gamma) < 0, 
\end{equation*}
as previously. \\
\end{remark}

\begin{lemma} \label{lemBarg}
Suppose $0 < \alpha < \tfrac{1}{2}, \, \sigma \geq 1-  \alpha,$ and $0 < \xi < \tfrac{1}{4}$. Then
\begin{equation*}
-\pi/2 < -4 \alpha \xi <  \arg B(\sigma + i \xi, 2 \alpha) < 0.
\end{equation*}
\end{lemma}
\begin{proof} 
From equation \eqref{Bsigmagamma},  
\begin{align*} \label{Bsigmaalpha}
 \arg B(\sigma + i \xi, 2 \alpha) &= -\xi \, \psi(\sigma+ 2 \alpha ) - \sum^\infty_{n=0} \bigg ( \dfrac{\xi}{\sigma + 2\alpha +n}- \arctan  \dfrac{\xi}{\sigma + 2\alpha + n}\bigg)  \nonumber \\
& \quad +\xi \, \psi(\sigma) +\sum^\infty_{n=0} \bigg ( \dfrac{\xi}{\sigma + n}- \arctan  \dfrac{\xi}{\sigma + n}\bigg) + 2 \pi k, \nonumber 
\end{align*}
for some $k \in \mathbb{Z}$. We now use the fact that $ \arg B(\sigma + i \xi, 2 \alpha) \in (-\pi, \pi]$ to show that $k=0$. \\

Let us define
\begin{equation*}
t_{\alpha, n} := \dfrac{\xi}{\sigma + 2\alpha +n}- \arctan  \dfrac{\xi}{\sigma + 2\alpha + n}.
\end{equation*}
Then, it is easy to see that
\begin{equation*}
0 < \sum^\infty_{n=0} \big ( t_{0,n} - t_{\alpha,n} \big ) < t_{0,0},
\end{equation*}
since $0 < 2 \alpha < 1$, and the function $x - \arctan x$ is strictly increasing for $x>0$. But
\begin{equation*}
t_{0,0} = \dfrac{\xi}{\sigma}- \arctan  \dfrac{\xi}{\sigma},
\end{equation*}
and because $\xi/\sigma < \tfrac{1}{2}$, we have
\begin{equation*}
t_{0,0} < \tfrac{1}{2} - \arctan \tfrac{1}{2} < \tfrac{1}{25}.
\end{equation*}
Using the relationship
\begin{equation*}
\psi(z+1) = \psi(z) + 1/z, \quad z \in \mathbb{C}\setminus\{0\},
\end{equation*}
see 6.3.5, p. 258, \cite{AandS}, we have the estimate
\begin{equation*}
0 < \psi(\sigma+ 2\alpha) - \psi(\sigma) < \psi(\sigma+1) - \psi(\sigma) = 1/\sigma < 2.
\end{equation*}
Therefore
\begin{equation*}
-\tfrac{1}{2} < \xi \big ( \psi(\sigma) - \psi(\sigma + 2\alpha) ) < 0.
\end{equation*}
Hence, $k=0$. \\

Noting that $\psi(x)$ is increasing and $\psi'(x)$ is decreasing,  
\begin{align*}
\arg B(\sigma + i \xi, 2 \alpha) & > - \xi  \big \{ \psi(\sigma+ 2\alpha) - \psi(\sigma)\big \}   \\
& \geq - \xi  \big \{ \psi(1+\alpha) - \psi(1- \alpha)\big \}   \quad \text{since } \sigma \geq 1 -\alpha \\
& = - \xi \, h(\alpha),
\end{align*}

where the function 
\begin{equation*}
h(\alpha) := \psi(1+\alpha) - \psi(1- \alpha) \quad \text{for } 0 < \alpha < \tfrac{1}{2}.
\end{equation*}
Clearly $h(0) = 0$. On the other hand, from  6.3.5, p. 258, \cite{AandS},
\begin{equation*}
h(\tfrac{1}{2}) = \psi(\tfrac{3}{2}) - \psi(\tfrac{1}{2}) =2.
\end{equation*}
Noting that $(-\psi(1-\alpha))^{''} = - \psi^{''}(1-\alpha)$, we have
\begin{equation*}
h^{''}(\alpha) = \psi^{''}(1+\alpha) - \psi^{''}(1- \alpha) > 0,
\end{equation*}
because $\psi^{'''} >0$ by 6.4.1, p.260, \cite{AandS}. Hence $h$ is convex and, therefore,
\begin{equation*}
h(\alpha) \leq \bigg ( \dfrac{2 - 0}{\tfrac{1}{2}-0} \bigg ) \, \alpha = 4 \alpha.
\end{equation*}
So, now we have the lower bound
\begin{equation*}
\arg B(\sigma + i \xi, 2 \alpha) > -4 \alpha \xi. \\
\end{equation*}

Finally, combining this result and Corollary \ref{Corrbetasigmagammaest}, with $0< \gamma = 2\alpha < 1$, we have the final estimate
\begin{equation*}
-\pi/2 < -4 \alpha \xi <  \arg B(\sigma + i \xi, 2 \alpha) < 0.
\end{equation*}
\end{proof}

\begin{lemma} \label{sinpialphatanh}
Suppose $0 < \alpha < \tfrac{1}{2}$ and $0 < \xi < \tfrac{1}{4}$. Then
\begin{equation*}
\arctan ( \sin \pi \alpha \tanh \pi \xi) > 4 \alpha \xi.
\end{equation*}
\end{lemma}
\begin{proof}
Firstly, by Jordan's inequality
\begin{equation*}
\dfrac{\sin x}{x} \geq \dfrac{2}{\pi} \qquad  \text{for  } 0 < x < \pi/2.
\end{equation*}
So, setting $x = \pi \alpha$ gives
\begin{equation} \label{sinpialphaest}
\sin \pi \alpha \geq 2 \alpha \qquad  \text{for  } 0 < \alpha < \tfrac{1}{2}. \\
\end{equation}

Secondly, $\tanh \pi \xi$ is concave function for $0 < \xi < \tfrac{1}{4}$, with $\tanh (0) =0$ and $\tanh (\pi/4) > 0.655 > \tfrac{5}{8}$. Hence
\begin{equation} \label{tanhpixiest}
\tanh \pi \xi > \tfrac{5}{2} \xi \qquad  \text{for  } 0 < \xi < \tfrac{1}{4}. \\
\end{equation}

Thirdly, $\arctan y$ is concave function for $0 < y < \tfrac{5}{8}$, with $\arctan (0) =0$ and $\arctan (\tfrac{5}{8}) > 0.558 > \tfrac{1}{2}$. Hence
\begin{equation} \label{arctanhyest}
\arctan y > \tfrac{4}{5} y \qquad  \text{for  } 0 < y < \tfrac{5}{8}. \\
\end{equation}

Finally, using estimates \eqref{sinpialphaest}, \eqref{tanhpixiest} and \eqref{arctanhyest} in turn,
\begin{align*}
 \arctan ( \sin \pi \alpha \tanh \pi \xi)
& \geq \arctan (2 \alpha \tanh \pi \xi)  \\
& > \arctan (2 \alpha \tfrac{5}{2} \xi)  \\
& = \arctan (  5 \alpha  \xi) \\
& > 4 \alpha \xi.
\end{align*}
\end{proof}

\begin{lemma} \label{coscosest}
Suppose $\tfrac{1}{2} \leq \alpha < 1$ and $1 + \alpha \leq \tau < 2$. Define $b:= 1- \tau + 2 \alpha$. Then
\begin{equation*}
\cos \pi \alpha - \cos \pi (\alpha - 2 b) \leq 0.
\end{equation*}
\end{lemma}
\begin{proof} \begin{align*}
\cos \pi \alpha - \cos \pi (\alpha - 2 b) &= \cos \pi \alpha - \cos \pi (\alpha - 2 + 2\tau - 4 \alpha) \\
&= \cos \pi \alpha - \cos \pi (2\tau - 3\alpha) \\
&= -2 \sin \pi (\tau - \alpha) \sin \pi (2 \alpha - \tau) \\
&= 2 \sin \pi (\tau - \alpha) \sin \pi (\tau - 2 \alpha).
\end{align*}
But since $1 + \alpha \leq \tau < 2$, we have $1 \leq \tau - \alpha <2$ and thus, $\sin \pi (\tau - \alpha) \leq 0$. \\

Similarly given $1 + \alpha \leq \tau < 2$, we have $0 < 1-\alpha \leq \tau - 2 \alpha < 2 - 2 \alpha \leq 1$ and thus $\sin \pi (\tau - 2 \alpha) > 0$. Therefore,
\begin{equation*}
\cos \pi \alpha - \cos \pi (\alpha - 2 b) \leq 0.
\end{equation*}
\end{proof}

%% file: KCL_Thesis_Chapter10_v5.tex
\chapter{Future research} \label{ChapterFutureResearch}
The backdrop for this research is the goal of developing a theory of boundary value problems for operators which are sums of pseudodifferential operators and ``fine-tuned potentials", which are less singular than the original (unperturbed) pseudodifferential operators. As we have seen, the generators of L\'evy processes in domains are a rich source of interesting models. Indeed, in this thesis, we have taken a first step by studying, in detail, a particular one-dimensional operator on the half-line. In passing, we note again the usefulness of Mellin operators in this work. \\

Even in one spatial dimensional, there are further significant opportunities for research.
These include both consideration of a wider class of elliptic operators and the analysis of  the related problem on a bounded domain. It is worth remarking that even one-dimensional models have important applications in various fields, including non-Gaussian market models in financial mathematics. \\

The second major future phase of this research is to extend the work done in one dimension to the $n$-dimensional case, where the region of interest becomes a half-space. In addition, it will be interesting to consider the natural generalisation to systems of equations. \\

In summary, this thesis is simply a (promising) beginning... Many useful techniques have been developed which, no doubt, will be reusable in a wider context. However, there remains much to do, and many challenges lie ahead!

%% file: KCL_Thesis_Appendices_v5.tex
\chapter{Simple example} \label{AppendixYZmapping} 
We define the homogeneous fractional Laplacian $(-\Delta)^\alpha$ by
\begin{equation*}
 (-\Delta)^{\alpha} := \mathcal{F}^{-1} |\xi|^{2\alpha} \mathcal{F}, \quad \text{where  } 0 < \alpha < 1.
\end{equation*}
Our goal in this appendix is to consider $(-\Delta)^\alpha$, for the range $\tfrac{1}{2} < \alpha < 1$, acting in one spatial dimension, in some detail. Indeed, we will critically examine both the difficulties caused by truncation, see Section \ref{sectiontruncation}, and, on the other hand, the improvements offered by adding a potential term, as described in Section \ref{sectionpotential}. \\

Let $u \in C^\infty_0(\mathbb{R})$, so that $\mathcal{F} u \in S(\mathbb{R})$. Since $2 \alpha-2 > -1$, we have $(|\xi|^{2\alpha} \mathcal{F}u)'' \in L_1(\mathbb{R})$, and hence 
$\mathcal{F}^{-1} |\xi|^{2\alpha} \mathcal{F}u $ is continuous and $O(|x|^{-2})$  as $|x| \to \infty$. Therefore,  
\begin{equation*}
(-\Delta)^\alpha: C^\infty_0(\mathbb{R}) \to L_1(\mathbb{R}) \cap C^\infty(\mathbb{R}),
\end{equation*}
since $(-\Delta)^\alpha: C^\infty_0(\mathbb{R}) \to C^\infty(\mathbb{R})$. (See, for example, Theorem 3.1, p. 47, \cite{XSR}.) \\

Let $G = \mathbb{R}_+$.  Our first goal is to construct $u \in r_+ C^\infty_0(\mathbb{R})$ such that $r_+ (-\Delta)^\alpha e_+ u \not \in L_{1, loc}(\mathbb{R}_+)$. Indeed, let $U \in C^\infty_0(\mathbb{R})$ be such that
\begin{equation*}
U= 1 \text{ if  } |x| \leq 1 \text{ and } U = 0 \text{ if } |x| \geq 2, 
\end{equation*}
and set
\begin{equation*}
u = r_+ U.
\end{equation*}

Since $|\xi|^{2\alpha} = \xi |\xi|^{2(\alpha -1)} \xi$, we have
\begin{equation*}
r_+ (-\Delta)^\alpha e_+ u = r_+ \, \mathcal{F}^{-1} |\xi|^{2\alpha} \mathcal{F} (e_+ u) = - r_+ \, \dfrac{d}{dx} \mathcal{F}^{-1} |\xi|^{2(\alpha -1)} \mathcal{F} (e_+ u)^\prime.  
\end{equation*}
But, see Lemma \ref{lemma:Deu}, 
\begin{equation*}
(e_+ u)^\prime = u(0)\delta + e_+ u' = \delta + e_+u'.
\end{equation*}
Of course,  $e_+u' \in C^\infty_0(\mathbb{R})$, and thus $\mathcal{F} (e_+u') \in S(\mathbb{R})$. Hence, $P(\xi) |\xi|^{2\alpha-2} \mathcal{F} (e_+u') \in L_1(\mathbb{R})$, for any polynomial $P$. Therefore, $\mathcal{F}^{-1} P(\xi) |\xi|^{2\alpha-2} \mathcal{F} (e_+u')$ is continuous and vanishes at infinity. So,
\begin{equation*}
w(x) := -r_+ \, \dfrac{d}{dx} \mathcal{F}^{-1} |\xi|^{2\alpha-2} \mathcal{F} (e_+u'),
\end{equation*}
and all its derivatives, are continuous and vanish at infinity. \\

It remains to consider
$- r_+ \, \dfrac{d}{dx} \mathcal{F}^{-1} |\xi|^{2(\alpha -1)} \mathcal{F} \delta$. \\

Now $\mathcal{F} \delta = 1/\sqrt{2 \pi}$, see Appendix \ref{Appendix FT}, and from Section 17.23, p. 1119, \cite{GR} or Example 3.1, equation (3.14), p. 38, \cite{Es},
\begin{equation} \label{FTmodxpowera}
\mathcal{F} ( |x|^{-a} ) = \sqrt{\frac{2}{\pi} }\, \Gamma(1-a) \sin \bigg (\frac{\pi a}{2} \bigg ) \, |\xi|^{a-1}, \quad 0 < a <1.
\end{equation}
Applying the inverse Fourier transform, and taking $a = 2 \alpha - 1$,
\begin{align*}
\mathcal{F}^{-1} |\xi|^{2(\alpha -1)} \mathcal{F} \delta & =
\dfrac{\sqrt{\tfrac{\pi}{2}} \, \, |x|^{1-2\alpha}}{\Gamma(2-2\alpha)\sin \big ( \tfrac{\pi(2\alpha-1)}{2} \big )} \cdot \dfrac{1}{\sqrt{2 \pi}} \\
& = \dfrac{\tfrac{1}{2} \, \, |x|^{1-2\alpha}}{\Gamma(2-2\alpha)} \cdot \dfrac{\Gamma(\tfrac{2\alpha-1}{2}) \Gamma(1- \tfrac{2\alpha-1}{2})}{\pi} \quad 5.5.3, \cite{NIST} \\
&= \dfrac{\tfrac{1}{2} \,\,|x|^{1-2\alpha} \,\, \Gamma(\alpha - \tfrac{1}{2}) \Gamma(\tfrac{3}{2} - \alpha) }{\tfrac{1}{\sqrt{2 \pi}} \,\, 2^{2-2\alpha- \frac{1}{2}} \,\, \Gamma(1- \alpha) \Gamma(\tfrac{3}{2} - \alpha) \,\, \pi} \quad 5.5.5, \cite{NIST} \\
&= - \dfrac{2^{2\alpha-1} \,\, \Gamma(\tfrac{1}{2} + \alpha)}{\pi^{\frac{1}{2}} \,\, \Gamma(1-\alpha)}  \cdot \dfrac{|x|^{1-2\alpha}}{1-2\alpha}
\end{align*}
since $(\alpha - \tfrac{1}{2}) \Gamma(\alpha - \tfrac{1}{2})= \Gamma(\alpha + \tfrac{1}{2})$, see 5.5.1, \cite{NIST}. \\

Hence
\begin{equation} \label{leadingsingularterm}
r_+ (-\Delta)^\alpha e_+ u =  \dfrac{2^{2\alpha-1} \,\, \Gamma(\tfrac{1}{2} + \alpha)}{\pi^{\frac{1}{2}} \,\, \Gamma(1-\alpha)} \,\, x^{-2\alpha} + w.
\end{equation}

But since $\tfrac{1}{2} < \alpha < 1$, we have $r_+ (-\Delta)^\alpha e_+ u \not \in L_{1,loc}(\mathbb{R}_+)$, as required. \\ \\

As before, we assume $G = \mathbb{R}_+$. Our second goal is to show that the leading singular terms
of $r_+ (-\Delta)^\alpha e_+ u$ and $\kappa^\alpha_G u$, near $x=0$, cancel each other, and thus, see equation \eqref{regFLalt},
\begin{equation*}
(-\Delta)_G^\alpha u = r_+ (-\Delta)^\alpha e_+ u + \kappa^\alpha_G u, 
\end{equation*}
is less singular than either of $r_+ (-\Delta)^\alpha e_+ u$ and $\kappa^\alpha_G u$, when considered separately. \\

Now, from equation \eqref{killingpot}, for $x>0$, 
\begin{align*}
\kappa^\alpha_G(x) & = - c_{1, \alpha} \int^0_{-\infty} \dfrac{1}{|x-y|^{1+2\alpha}} \, dy \\
&=  - c_{1, \alpha} \int^\infty_0 \dfrac{1}{(x+\tau)^{1+2\alpha}} \, d\tau \\
&= - c_{1, \alpha} \Bigg [ \dfrac{(x+\tau)^{-2\alpha}}{(-2\alpha)} \Bigg ]^\infty_0\\
&= - c_{1, \alpha} \dfrac{x^{-2\alpha}}{2\alpha}.
\end{align*}
But, see for example \cite{Kwas}, the positive constant $c_{n, \alpha}$ is given by
\begin{equation*}
c_{n,\alpha} = - \dfrac{2^{2\alpha} \Gamma(\alpha + \frac{n}{2})}{\pi^\frac{n}{2} \Gamma(- \alpha)}.
\end{equation*}
Therefore, noting that $(-\alpha) \Gamma(-\alpha) = \Gamma (1 - \alpha)$, for $x>0$,
\begin{equation} \label{leadingKappa}
\kappa^\alpha_G(x) = - \dfrac{2^{2\alpha -1} \Gamma (\alpha + \frac{1}{2})}{{\pi}^\frac{1}{2} \Gamma(1-\alpha)} x^{-2\alpha}.
\end{equation}
So, comparing equations \eqref{leadingsingularterm} and \eqref{leadingKappa}, we see that the leading singular terms of $r_+ (-\Delta)^\alpha e_+ u$ and $\kappa^\alpha_G u$, near $x=0$, cancel each other, as required.

\chapter{Feller and L\'{e}vy processes} \label{FellerLevy} 

Useful background on the material in this appendix can be found in \cite{Ja, Sh,Ta}. Our starting point is a \textit{probability space}. Let $\Omega$ be a non-empty set, and let $\mathcal{A}$ denote a $\sigma$-$\it{field}$ on $\Omega$. Further, suppose that $P$ is a \textit{probability measure} defined on $\mathcal{A}$. Then the triple $(\Omega, \mathcal{A}, P)$ is called a probability space. \\

Suppose $G \subset \mathbb{R}^n$ is a Borel set. Then we can define a \textit{stochastic process} by the quadruple $(\Omega, \mathcal{A}, P, (X_t)_{t \geq 0})$, where $X_t:\Omega \rightarrow G, \, t \geq 0$ is a random variable. We call $G$ the \textit{state space}. The mappings $t \rightarrow X_t(\omega)$, for $\omega \in \Omega$, are called the \textit{paths} of the process. We will be interested in families of stochastic processes indexed by the state space, sometimes called \textit{universal processes}. More precisely, we will consider for each $x \in G$, the stochastic process $(\Omega, \mathcal{A}, P^x, (X_t)_{t \geq 0})_{x \in G}$, where $P^x \{ X_0 = x \} =1$. \\

Let $B_b(G)$ denote the set of bounded Borel functions on $G$. Then, given a universal process,  we can define a family $(T_t)_{t\geq0}$ of operators acting on $B_b(G)$ by
\begin{equation*}
(T_t u)(x) = \mathbb{E}^x(u(X_t)).
\end{equation*}
In particular, given a Borel set $A \subset G$, we can define the \textit{transition function} $p_t(x,A)$, for $x \in G$, to be
\begin{equation*}
p_t(x,A) := (T_t \chi_A)(x) = \mathbb{E}^x(\chi_A(X_t)) = P^x \{ X_t \in A \}.
\end{equation*}
Intuitively, the transition function $p_t(x,A)$ is the probability of being in the set $A$ at time $t$, starting at time 0 from a point $x \in G$. \\

Let $\mathcal{B}(G)$ denote the Borel $\sigma$-field over $G$. Then it is easy to see that, for fixed $t \geq 0$, the mapping $A \rightarrow p_t(x,A)$ is a probability measure on $\mathcal{B}(G)$. Hence, the operator $T_t, \, t\geq 0$ can be represented as
\begin{equation*}
(T_t u)(x) = \displaystyle \int_G u(y) p_t(x,dy).
\end{equation*}
Similarly, let us define
\begin{equation*}
p_{s,t}(x,A) = \displaystyle \int_G p_s(y,A) p_t(x,dy).
\end{equation*}
We call the family $(p_t(x,A))_{t \geq 0, x \in G}$ a \textit{semigroup of Markovian kernels} if for all $s,t \geq 0, x \in G$, and any Borel set $A \subset G$ we have
\begin{equation*}
p_{s,t}(x,A) = p_{s+t}(x,A). 
\end{equation*}
In other words, the \textit{Chapman-Kolmogorov} equations
\begin{equation} \label{ChapKol}
p_{s+t}(x,A) = \displaystyle \int_G p_s(y,A) p_t(x,dy)
\end{equation}
hold. In this case, it follows that $(T_t)_{t \geq 0}$ is a \textit{semigroup} of linear operators on $B_b(G)$ with
\begin{equation*}
T_{s+t} = T_s \circ T_t
\end{equation*}
valid for all $s,t \geq 0$. Since $p_0(\cdot, \{ x \}) = \delta_x$, we always have $T_0$ is the identity map. A universal process is called a \textit{Markov Process} when its transition function satisfies equation \eqref{ChapKol}. \\

Suppose we have a Markov process $((X_t)_{t \geq 0}, P^x)_{x \in G}$, and the associated semigroup of linear operators $T_t:B_b(G) \rightarrow B_b(G)$. The it is easy to show that the operator $T_t$ is a contraction on $B_b(G)$. It is also useful to consider the action of $T_t$ on the Banach space $C_\infty(G)$, equipped with the supremum norm, consisting of all continuous functions on $G$ that vanish at infinity. We say that a semigroup $(T_t)_{t \geq 0}$ of linear operators on $C_\infty(G)$ is a \textit{Feller} semigroup if it meets the following three conditions:
\begin{enumerate}[(i)]
\item $T_t:C_\infty(G) \rightarrow C_\infty(G)$  is a linear contraction; 
\item $\lim_{t \rightarrow 0} \| T_tu - u \|_\infty = 0$,   i.e. the semigroup is strongly continuous; 
\item $0 \leq u \leq 1 \,\,$ implies $\,\, 0 \leq T_tu \leq 1$.
\end{enumerate}
A Markov process is called a \textit{Feller process} when its corresponding semigroup is a Feller semigroup. \\

Let $X = (X_t)_{t \geq 0}$ be a stochastic process defined on a probability space $(\Omega, \mathcal{A}, P)$. We say that $X$ has \textit{independent increments} if for each $n \in \mathbb{N}$, and $0 \leq t_1 < t_2 < \dots < t_{n+1} < \infty$, the random variables
$\{ X_{t_{j+1}} - X_{t_j} \}^n_{j=1}$ are independent. Moreover, we say that it has \textit{stationary increments} if, for each $1 \leq j \leq n$, the random variables
$X_{t_{j+1}} - X_{t_j}$ and $X_{t_{j+1} - t_j} - X_0$ are equal in distribution. \\

A \textit{L\'{e}vy process} is a Feller process with independent and stationary increments which is continuous in probability. That is, for all $a > 0$ and all $s \geq 0$,
\begin{equation*}
\lim_{t \to s} P( |X_t - X_s | > a) =0.
\end{equation*}
It is sometimes helpful to think of a L\'{e}vy process as a continuous analogue of a random walk \cite{Va}. The most well known examples of L\'{e}vy processes are Brownian motion and the Poisson process. L\'{e}vy processes whose paths are almost surely non-decreasing are called \textit{subordinators}. See, for example, \cite{Kyp}. \\

The generator of the symmetric $\alpha$-stable subordinator is the fractional Laplacian, which may be defined as
\begin{equation*}
(-\Delta)^\alpha := \mathcal{F}^{-1} |\xi|^{2 \alpha} \mathcal{F}, \quad 0 < \alpha <1.
\end{equation*}
For more details, see Example 5.8, p. 96 \cite{GS}. \\

Similarly, see Example 5.9, p. 97, \cite{GS}, the generator of the (killed) relativistic $\alpha$-stable subordinator is the pseudodifferential operator
\begin{equation*}
\mathcal{F}^{-1} (1+ |\xi|^2)^{\alpha} \mathcal{F}, \quad 0 < \alpha <1.
\end{equation*}
This class of L\'{e}vy processes is also discussed in Remark \ref{relativiststabsubs}.

\chapter{H\"ormander space} \label{GGrubb}
In Section \ref{sectiontransmission}, we briefly reviewed the approach taken in \cite{Grubb2014,Grubb2015}, using the $\mu$-transmission condition and H\"ormander spaces. The goal of this appendix is to consider a particular example to explore these ideas in more detail. \\

The following material is taken from a presentation entitled \enquote{Boundary problems for fractional Laplacians and other fractional-order operators}, given by Gerd Grubb in March 2016. The main changes shown here are due to notational differences, 
differing Fourier transform conventions and, where appropriate, further explanation. \\

We suppose that $0 < \alpha < 1$, and let $A$ be a pseudo-differential operator of fractional order.  In general, given an open set $\Omega \subset \mathbb{R}^n, \, n \geq 2$, with a suitably smooth boundary, we are interested in the \textit{Dirichlet problem}
\begin{equation} \label{Dirichletprobdef}
A u = f \quad \text{in } \Omega, \quad \operatorname{supp} u \subset \overline{\Omega},
\end{equation}
where $f \in H^t(\overline{\Omega})$ is given. The novel aspect here, in our terms, is that  the solution $u$, if it exists, is sought in the so-called \textit{H\"ormander space} $H^{\alpha(t + 2\alpha)} (\overline{\mathbb{R}^n_+})$. (See Section  \ref{sectiontransmission} for the definition of $H^{\alpha(t + 2\alpha)} (\overline{\mathbb{R}^n_+}), \,\, t \geq 0$.)\\

\begin{theorem}
Let $A = (I - \Delta)^\alpha$ on $\Omega = \mathbb{R}^n_+$ and suppose $t \geq 0$. Then the Dirichlet problem \eqref{Dirichletprobdef} has a unique solution $u \in H^{\alpha(t + 2\alpha)} (\overline{\mathbb{R}^n_+})$. \\
\end{theorem}

\begin{proof}
The operator $A = (I - \Delta)^\alpha$ has symbol $(1 + |\xi|^2)^\alpha = \langle \xi \rangle^{2\alpha}$. We have the factorisation
\begin{equation*}
(1 + |\xi|^2)^\alpha = (\langle  \xi' \rangle + i \xi_n)^\alpha \, (\langle \xi' \rangle - i \xi_n)^\alpha.
\end{equation*}
where $\xi = (\xi', \xi_n)$ and $\xi' = (\xi_1, \dots , \xi_{n-1})$. \\

We set $\mu^t_\pm := ( \langle \xi' \rangle \mp  i \xi_n)^t$. (The reason for this counter-intuitive definition will become clear later. See \eqref{hatLambdaminus}.)  Then we define $\hat{\Lambda}^t_\pm := \mathcal{F}^{-1} (\mu^t_\pm) \mathcal{F}$, and we can write
\begin{equation*}
(I - \Delta)^\alpha = \hat{\Lambda}^\alpha_- \, \hat{\Lambda}^\alpha_+.
\end{equation*}

Since $(\big \langle \xi' \rangle - i ( \xi_n + i \tau)  \big)^t = (\big \langle \xi' \rangle + \tau - i \xi_n  \big)^t$ is analytic for $\tau > 0$, from Theorem 1.9, p. 52, \cite{Shar}, for any $s, \, t \in \mathbb{R}$,
\begin{equation} \label{hatLambdaminus}
\hat{\Lambda}^t_+: \widetilde{H}^{s}(\overline{\mathbb{R}^n_+}) \to \widetilde{H}^{s-t}(\overline{\mathbb{R}^n_+}),
\end{equation}
is bounded and, moreover, has inverse $\hat{\Lambda}^{-t}_+$. In particular, we note that $\hat{\Lambda}^t_+$ preserves support in ${\overline{\mathbb{R}^n_+}}$. \\

Similarly, as $(\big \langle \xi' \rangle + i ( \xi_n - i \tau)  \big)^t = (\big \langle \xi' \rangle + \tau + i \xi_n  \big)^t$ is analytic for $\tau > 0$, from Theorem 1.10, p. 53, \cite{Shar}, for any $s, \, t \in \mathbb{R}$,
\begin{equation} \label{hatLambdaplus}
r_+ \, \hat{\Lambda}^t_- l_+ : H^s (\overline{\mathbb{R}^n_+}) \to H^{s-t}(\overline{\mathbb{R}^n_+}),
\end{equation}
is bounded, where $l_+$ is the extension operator. In particular, we note from Remark 1.11, p. 53, \cite{Shar}, that $(r_+ \, \hat{\Lambda}^t_- l_+ ) \, r_+ = r_+ \, \hat{\Lambda}^t_-$.\\

The model Dirichlet problem is
\begin{equation} \label{modelDirichlet}
r_+ \, (I - \Delta)^\alpha u = f  \quad \text{on } \mathbb{R}^n_+, \quad \text{supp } u \subset \overline{\mathbb{R}^n_+},
\end{equation}
where, by hypothesis, $f \in H^t(\overline{\mathbb{R}^n_+})$ for some $t \geq 0$, and we seek a solution $u \in H^{\alpha (t + 2 \alpha)}(\overline{\mathbb{R}^n_+})$. Now
\begin{equation*}
r_+ \, (I - \Delta)^\alpha u  = r_+ \, \hat{\Lambda}^\alpha_- \, \hat{\Lambda}^\alpha_+ u 
= r_+ \, \hat{\Lambda}^\alpha_- \, l_+r_+ \, \hat{\Lambda}^\alpha_+ u.
\end{equation*}

But, it is easy to see from Remark 1.11, p. 53, \cite{Shar}, that $(r_+ \, \hat{\Lambda}^t_- l_+)^{-1} = r_+ \, \hat{\Lambda}^{-t}_- l_+$ and thus \eqref{modelDirichlet} can be reduced to 
\begin{equation} \label{modelDirichletreduced}
r_+ \,   \hat{\Lambda}^\alpha_+ u =  g  \quad \text{on } \mathbb{R}^n_+, \quad \text{supp } u \subset \overline{\mathbb{R}^n_+},
\end{equation}
where 
\begin{equation*}
g := r_+ \, \hat{\Lambda}^{-\alpha}_- l_+ f \in H^{t+\alpha}(\overline{\mathbb{R}^n_+}).
\end{equation*}
We note that, by Theorem 1.10, p. 53, \cite{Shar}, $g$ is independent of the choice of the extension $l_+$.  \\

Now suppose that \eqref{modelDirichletreduced} has two solutions $u_1$ and $u_2$ with $\operatorname{supp} u_1, \,\, \operatorname{supp} u_2 \subset \overline{\mathbb{R}^n_+}$. Let $v = u_1 - u_2$. Then 
\begin{equation*} 
r_+ \,   \hat{\Lambda}^\alpha_+ v =  0  \quad \text{on } \mathbb{R}^n_+, \,\, \text{ and  } \operatorname{supp} v \subset \overline{\mathbb{R}^n_+}.
\end{equation*}
Hence, see \eqref{hatLambdaminus}, $\hat{\Lambda}^\alpha_+ v =  0$ on $\mathbb{R}^n$, and thus $v=0$. In other words, if a solution to \eqref{modelDirichletreduced} does exist, then it is unique. \\

But now it is easy to see, by direct substitution, that \eqref{modelDirichletreduced} has the solution
\begin{equation*}
u = \hat{\Lambda}^{-\alpha}_+ e_+ g.
\end{equation*}

Thus, \eqref{modelDirichlet} has a unique solution $u$, and it lies in 
\begin{equation} \label{Hormanderspace}
H^{\alpha (t + 2 \alpha)}(\overline{\mathbb{R}^n_+}) := \hat{\Lambda}^{-\alpha}_+ (e_+ H^{t+\alpha}(\overline{\mathbb{R}^n_+})) ,
\end{equation}
which is known as \textit{H\"ormander's space}. In particular, we note that if $t + \alpha > \tfrac{1}{2}$ then functions in the space $e_+ \,  H^{t+\alpha} (\overline{\mathbb{R}^n_+})$ may have a jump at $x_n=0$. This gives rise to a singularity when the operator $\hat{\Lambda}^{-\alpha}_+$ is applied.\\

\end{proof}

\begin{remark} \label{HormanderspaceRemark}
The following relationships, see \cite{Grubb2014,Grubb2015}, provide a useful characterisation of $H^{\alpha(t + 2\alpha)} (\overline{\mathbb{R}^n_+})$:
\[
H^{\alpha(t + 2\alpha)} (\overline{\mathbb{R}^n_+})
\begin{cases} 
= \widetilde{H}^{t+2\alpha}(\overline{\mathbb{R}^n_+})  & \mbox{if } -\tfrac{1}{2} < t + \alpha < \tfrac{1}{2}; \\ 
\subset  e_+ \, x^\alpha_n \, H^{t+\alpha} (\overline{\mathbb{R}^n_+}) + \widetilde{H}^{t+2\alpha - \epsilon}(\overline{\mathbb{R}^n_+})  & \mbox{if }  t + \alpha > \tfrac{1}{2}, \\ 
\end{cases} 
\]
where the term $-\epsilon$ only applies if $t + \alpha - \tfrac{1}{2} \in \mathbb{N}$. \\
\end{remark}

\chapter{Fractional Calculus} \label{Appendix FC}
The goal of this appendix is summarise some useful components of the Fractional Calculus. The authoritative text on this subject is Samko et al., \cite{Samko}. However, given that this book is out of print, supplementary technical references are taken from a more recent work by Diethelm \cite{Diethelm}. \\

Suppose $n \in \mathbb{N}$. Then, from equation (2.16), p. 33, \cite{Samko}, we have the following formula for the $n$-fold integral
\begin{equation*}
\int^x_a \int^{\sigma_1}_a \cdots \int^{\sigma_{n-1}}_a f(\sigma_n) \, d\sigma_n  \cdots d\sigma_1
= \dfrac{1}{(n-1)!} \int^x_a (x-t)^{n-1} \varphi(t) \, dt,
\end{equation*} 
where, of course, $(n-1)! = \Gamma(n)$. This provides the motivation for the following definition. \\

Let $-\infty \leq a < x < b \leq \infty$. The \textit{Riemann-Liouville} fractional integral of order $\alpha > 0$ with lower limit $a$ is defined for locally integrable functions $f:[a,b] \to \mathbb{R}$ as
\begin{equation} \label{defineIalpha}
(I^\alpha_{a^+} f)(x) := \dfrac{1}{\Gamma(\alpha)} \int^x_a \dfrac{f(y)}{(x-y)^{1-\alpha}} \, dy, \quad x>a.
\end{equation}
See Definition 2.1, equation (2.17), p. 33, \cite{Samko}. (Also, Definition 2.1, p. 13,  \cite{Diethelm}.) \\

In particular, as expected, we have
\begin{equation} \label{I1andI2}
(I^1_{a^+} f)(x) = \int^x_a f(y) \, dy \quad \text{and} \quad (I^2_{a^+} f)(x) = \int^x_a \bigg (\int^y_a f(t) \, dt \bigg ) \, dy.
\end{equation}

Moreover, fractional integration has the property that
\begin{equation} \label{IaIb}
I^\alpha_{a^+} I^\beta_{a^+} f = I^{\alpha+\beta}_{a^+} f,  \quad \text{for} \,\,  \alpha, \beta >0,
\end{equation}
see equation (2.21), p. 34, \cite{Samko}. (See also Corollary 2.3, p. 14, \cite{Diethelm}.) \\


Suppose $\alpha > 0$ and $n = [\alpha]+1$. We now define the \textit{Caputo} fractional derivative of order $\alpha$ with lower limit $a$ as
\begin{equation} \label{Caputo2}
(C^\alpha_{a^+}f)(x) := (I^{n-\alpha}_{a^+} f^{(n)})(x),
\end{equation}
for sufficiently smooth functions $f$. (See Definition 3.1, p. 49, \cite{Diethelm}.) In the special case that $\alpha =1$, we have 
\begin{equation} \label{C1af}
(C^1_{a^+}f)(x) = I^{1}_{a^+} f''(x) = \int^x_a f''(t) \, dt = f'(x) - f'(a). 
\end{equation}

Suppose $0 < \alpha < 1$. Then
\begin{equation*}
I^\alpha_{a^+} C^\alpha_{a^+} f = I^\alpha_{a^+} \big (  I^{1-\alpha}_{a^+} f'  \big )
= I^1_{a^+} f' 
= \int^x_a f'(t) \, dt 
= f(x) -f(a).
\end{equation*}

Similarly, if $1 \leq \alpha < 2$, then 
\begin{equation*}
I^\alpha_{a^+} C^\alpha_{a^+} f 
= I^2_{a^+} f'' \\
= \int^x_a \bigg (\int^y_a f''(t) \, dt \bigg ) \, dy
= f(x) - f(a) - (x-a) f'(a).
\end{equation*} \\

In summary, 
\begin{equation} \label{TaylorCase1}
f(x)-f(a) = I^\alpha_{a^+}C^\alpha_{a^+}f(x) \qquad (0 < \alpha < 1),
\end{equation}
and 
\begin{equation} \label{TaylorCase2}
f(x)-f(a)-f'(a)(x-a) = I^\alpha_{a^+}C^\alpha_{a^+}f(x) \qquad (1 \leq \alpha < 2). 
\end{equation} \\

Finally, let $I^\alpha_+ := I^\alpha_{(-\infty)^+}$. Then, for $0 < \alpha < 1$,
\begin{equation} \label{FTI+alpha2}
\mathcal{F} (I^\alpha_+ f) = \dfrac{ (\mathcal{F}f)(\xi)} {{(-i\xi)}^\alpha},
\end{equation}
as given in equation (7.1), p. 137, \cite{Samko}.


\chapter{Fourier transform results} \label{Appendix FT}
As previously, we define the \textit{Fourier Transform} on the Schwartz space, $S(\mathbb{R})$, of rapidly decaying infinitely differentiable functions $\varphi$ by
\begin{equation*}
(\mathcal{F} \varphi)(\xi) = \dfrac{1}{\sqrt{2\pi}} \int_{\mathbb{R}^n} e^{+i\xi \cdot x} \varphi(x) \, dx, \quad \xi \in \mathbb{R}. 
\end{equation*}

Let $S'(\mathbb{R})$ denote the corresponding space of tempered distributions. Then, see equation (2.28), p. 22, \cite{Es}, the Fourier transform of $f \in S'(\mathbb{R})$ is the tempered distribution $\widehat{f} \in S'(\mathbb{R})$, such that
\begin{equation*}
\langle \, \widehat{f}, \, \widehat{\varphi} \, \rangle = \langle f, \varphi \rangle \quad \text{for all  } \varphi \in S(\mathbb{R}).
\end{equation*}
As \cite{Es}, for $f \in S'(\mathbb{R})$, we adopt the convention that  $\mathcal{F} f := \widehat{f}$. \\

From 17.23, p. 1118, \cite{GR}, we have
\begin{align}
\mathcal{F} (1) & = \sqrt{2 \pi} \delta(\xi) \\
\mathcal{F} (\delta(x)) & = 1/ \sqrt{2 \pi} \\
\mathcal{F} (\chi_{\mathbb{R}_+}(x)) & = \dfrac{i}{\sqrt{2 \pi}} \, \dfrac{1}{\xi} + \sqrt{\dfrac{\pi}{2}} \, \delta(\xi).
\end{align}

Hence, given the identities $\chi_{\mathbb{R}_-} = 1 - \chi_{\mathbb{R}_+}$ and $\sgn = 2 \chi_{\mathbb{R}_+} - 1$, we can easily deduce
\begin{align}
\mathcal{F}(\chi_{\mathbb{R}_-}(x)) & =  - \dfrac{i}{\sqrt{2 \pi}} \, \dfrac{1}{\xi} + \sqrt{\dfrac{\pi}{2}} \, \delta(\xi) \\
\mathcal{F}(\sgn(x)) & =  i \sqrt{\dfrac{2}{\pi}} \, \dfrac{1}{\xi}.
\end{align}

\chapter{Technical lemma} \label{AppendixTechLemm} 
Suppose $0 < \mu < 1$. It will be convenient to define
\begin{equation} \label{defwminus}
w_{-, \mu}(x) := \mathcal{F}^{-1} (- i \xi)^{-\mu} (\xi - i)^{-1}, \quad  x \in \mathbb{R}.
\end{equation} 

Further, let $\chi \in C^\infty_0(\mathbb{R})$ be such that
\begin{align*}
&\chi(t) :=
\begin{cases} 
	1 &\mbox{if }  |t| \leq 2\\
	0 & \mbox{if } |t| > 3. \\
\end{cases}
\end{align*}
Further take $\chi_1, \chi_2 \in C_0^\infty(\mathbb{R})$ such that $\chi_1\chi = \chi$, $\chi_2\chi_1 = \chi_1$. (That is, $\chi_1 = 1$ on supp $\chi$, and $\chi_2 = 1$ on supp $\chi_1$.) \\

Given these definitions, the key result of this Appendix is:
\begin{lemma} \label{lemmapushchi}
Suppose $0 < \mu < 1, \,\,1 < p < \infty$ and $r < \mu + 1/p$. Then
\begin{equation*}
(D - i)^r \chi w_{-, \mu} \in L_p(\mathbb{R}).
\end{equation*}
\end{lemma}
\begin{proof}
We have
\begin{align*}
(D - i)^r \chi w_{-, \mu}  =  \,\, (1 - & \chi_1)  (D - i)^r \chi w_{-, \mu} + \chi_1(D - i)^r \chi w_{-, \mu} \\
=  \,\, (1 - & \chi_1)  (D - i)^r \chi w_{-, \mu} + \chi_1(D - i)^r (\chi - \chi_2) w_{-, \mu} \\
&+ \chi_1(D - i)^r (\chi_2 - 1) w_{-, \mu} + \chi_1(D - i)^r  w_{-, \mu} .
\end{align*}
From Lemma \ref{1mchiDw},
\begin{equation*}
(1 - \chi_1) (D - i)^r \chi w_{-, \mu} \,\, \text{and} \,\, \chi_1(D - i)^r (\chi_2 - 1) w_{-, \mu} \in L_p(\mathbb{R}).
\end{equation*} 
Moreover, from Lemma \ref{chiDrchi},
\begin{equation*}
 \chi_1(D - i)^r (\chi - \chi_2) w_{-, \mu} \in L_p(\mathbb{R}).
\end{equation*}
But, from Lemma \ref{Hrcondition},
\begin{equation*}
\chi_1 (D - i)^r w_{-, \mu} \in L_p(\mathbb{R}),
\end{equation*}
and the required result follows immediately. \\
\end{proof}

Our first task is to prove Lemma \ref{1mchiDw}. We begin with a definition.
\begin{definition}
Let $m \in \mathbb{R}$. Then, we say that $a \in S^m$, if $a= a(\xi)$ is smooth on $\mathbb{R}$ and if
\begin{equation*}
|\partial^\beta a(\xi)| \leq C_\beta (1+ |\xi|)^{m-\beta} \quad \text{for all} \quad \beta \in \mathbb{N}\cup\{0\}, \,\, \xi \in \mathbb{R},
\end{equation*}
for certain constants $C_\beta$, that only depend on $\beta$. \\
\end{definition}

\begin{remark} \label{kernelinverseFT}
Let $A$ be a pseudodifferential operator with symbol $a \in S^m$, for some $m \in \mathbb{R}$. Then, for $u \in S(\mathbb{R})$,
\begin{equation*}
A u  =  \mathcal{F}^{-1} a \mathcal{F} u
 = \mathcal{F}^{-1} \big ( \mathcal{F}(\mathcal{F}^{-1}a )\mathcal{F} u \big )
=  k * u
\end{equation*}
where 
\begin{equation} \label{Kernelkdefn}
k := \dfrac{1}{\sqrt{2 \pi}} \mathcal{F}^{-1}a.
\end{equation} \\
\end{remark}

\begin{lemma} \label{Kerneldecay}
Suppose $a \in S^m$ for some $m \in \mathbb{R}$.  Then the kernel,  $k(z)$, satisfies
\begin{equation*}
|\partial^\beta k(z) | \leq \text{ const } |z|^{-N}
\end{equation*}
for $N > m+1+ \beta$ and $z \not =0 $. Thus, for $z \not = 0$, the kernel $k(z)$ is a smooth function which is rapidly decaying as $|z| \to \infty$. \\
\end{lemma}

\begin{proof}
Let $N \in \mathbb{N} \cup \{ 0 \}$. Then, from \eqref{Kernelkdefn},
\begin{equation*}
(i z)^N \partial^\beta k(z) = \tfrac{1}{\sqrt{2 \pi}}{\mathcal{F}}^{-1} \partial^N [ (-i \xi)^\beta a(\xi) ]. \\
\end{equation*}

Since $(-i \xi)^\beta a(\xi) \in S^{m+\beta}$, we have  the upper bound
\begin{equation*}
|\partial^N [ (-i \xi)^\beta a(\xi) ]| \leq C_{N, \beta} (1 + |\xi|)^{m+\beta - N}.
\end{equation*}
and hence, $\partial^N [ (-i \xi)^\beta a(\xi) ] \in L_1(\mathbb{R})$, provided $N > m + 1 + \beta$. \\

Therefore, its inverse Fourier transform is bounded. In other words,
\begin{equation*}
(i z)^N \partial^\beta k(z) \in L_\infty(\mathbb{R}) \quad \text{for} \quad N > m + 1 + \beta.
\end{equation*}
\end{proof}

\begin{remark} \label{1minuschiA}
For $u \in S(\mathbb{R})$,
\begin{equation*}
((1-\chi_1) A \chi u)(x)
=  \int_\mathbb{R} (1- \chi_1(x)) \chi(y) \, k(x-y) u(y) \, dy. 
\end{equation*}
By the definition of $\chi$ and $\chi_1$, if $x=y$ then $(1- \chi_1(x)) \chi(y) \, k(x-y) = 0$. Therefore, from Lemma \ref{Kerneldecay},  the integral kernel $(1- \chi_1(x)) \chi(y) k(x-y)$ is smooth, bounded and rapidly decaying as $|x-y| \to \infty$. \\

Similarly, 
\begin{equation*}
(\chi_1 A (\chi_2-1) u)(x)
=  \int_\mathbb{R} \chi_1(x) (\chi_2(y)-1)  \, k(x-y) u(y) \, dy, 
\end{equation*}
and the integral kernel $\chi_1(x) (\chi_2(y)-1) k(x-y)$ is smooth, bounded and rapidly decaying as $|x-y| \to \infty$. \\
\end{remark}

\begin{lemma} \label{1mchiDw}
Suppose $r \in \mathbb{R}$ and $1 < p < \infty$. Then
\begin{equation*}
(1-\chi_1)(D-i)^r \chi w_{-, \mu} \,\, \text{and} \,\,  \chi_1(D - i)^r (\chi_2 - 1) w_{-, \mu} \in L_p(\mathbb{R}).
\end{equation*}
\end{lemma}
\begin{proof}
Since $w_{-, \mu}(x)$ is the inverse Fourier transform of an integrable function, it is bounded, continuous and tends to zero as $|x| \to \infty$. \\

From Remark \ref{1minuschiA}, the kernels of the integral operators $(1-\chi_1)(D-i)^r \chi$ and $ \chi_1(D - i)^r (\chi_2 - 1)$ are smooth, bounded and rapidly decaying as $|x-y| \to \infty$. Hence result. \\
\end{proof}

Finally, we prove Lemma \ref{chiDrchi}. We begin with a simple result.

\begin{lemma} \label{chichi2w}
\begin{equation*}
(\chi - \chi_2)w_{-, \mu} \in C^\infty_0(\mathbb{R}_+).
\end{equation*}
\end{lemma}
\begin{proof}
By definition,
\begin{equation*}
\mathcal{F} w_{-, \mu} = \dfrac{(-i \xi)^{-\mu}}{\xi - i},
\end{equation*}
where $0<\mu <1$. Since $w_{-, \mu}(x)$ is the inverse Fourier transform of an integrable function, it is continuous for all $x \in \mathbb{R}$. Now
\begin{equation*}
\mathcal{F} w'_{-, \mu} = \dfrac{(-i \xi) (-i \xi)^{-\mu}}{\xi - i} = \dfrac{(-i (\xi -i) + 1) (-i \xi)^{-\mu}}{\xi - i}= -i (-i\xi)^{-\mu} +  \mathcal{F} w_{-, \mu},
\end{equation*}
so that, from equation \eqref{FTmodxpowera}, for $x>0$,
\begin{equation*}
w'_{-, \mu}(x) = C x^{\mu-1} + w_{-, \mu}(x),
\end{equation*}
But since $\chi - \chi_2$ equals zero in a neighbourhood of $x=0$, we see immediately that 
$(\chi - \chi_2)w'_{-, \mu}$ is continuous with compact support in $\mathbb{R}_+$. \\

Finally, since for any $m \in \mathbb{N}$,
\begin{equation*}
\mathcal{F} w^{(m+1)}_{-, \mu} = \dfrac{(-i \xi)^{m+1} (-i \xi)^{-\mu}}{\xi - i} = \dfrac{ (-i (\xi- i) +1) (-i \xi)^{m-\mu}}{\xi - i}= -i (-i\xi)^{m-\mu} + \mathcal{F} w^{(m)}_{-, \mu},
\end{equation*}
the required result follows directly by induction on $m$. \\
\end{proof}

\begin{lemma} \label{chiDrchi}
Suppose $r \in \mathbb{R}$. Then
\begin{equation*}
\chi_1(D-i)^r (\chi - \chi_2) w_{-, \mu} \in C^\infty_0(\mathbb{R}).
\end{equation*}
\end{lemma}
\begin{proof}
From Lemma \ref{chichi2w}, $(\chi - \chi_2) w_{-, \mu} \in C^\infty_0(\mathbb{R})$. Hence, see Theorem 3.1, p. 47, \cite{XSR},
\begin{equation*}
(D-i)^r (\chi - \chi_2) w_{-, \mu} \in C^\infty(\mathbb{R}).
\end{equation*}
Finally, $\chi_1(D-i)^r (\chi - \chi_2) w_{-, \mu} \in C^\infty_0(\mathbb{R})$, as required. \\
\end{proof}

\chapter{A certain integral} \label{AppendixIntLemm} 
Suppose $-1 < \mu < 1$ and $x>0$. Then we define
\begin{equation*}
I^\pm_{[a,b]} (\mu; x) := \int^b_a e^{-i \xi} (-i \xi)^\mu (\xi \pm i x)^{-\mu-1} \, d\xi.
\end{equation*} 

\begin{lemma} \label{InverFTmulog}
Suppose $-1 < \mu < 1$ and $x>0$. Then 
\begin{equation*}
I^+_{[-1,1]} (\mu; x) = O(1), \quad x \searrow 0^+,
\end{equation*}
and 
\[
I^-_{[-1,1]} (\mu; x) = 
\begin{cases} 
O(1)  & \mbox{if }  \mu =0 \\
O(\log x) &\mbox{if } \mu \not = 0 
\end{cases}
\quad \mbox{as  } x \searrow 0^+.
\] \\
\end{lemma}


\begin{proof}
We begin the proof with some observations that will prove useful. \\

Suppose $z_1, z_2 \in \mathbb{C}$ and $\nu \in \mathbb{R}$. Then, if $ -\pi < \arg z_1 + \arg z_2 \leq \pi$,
\begin{equation} \label{z1z2fracpower}
(z_1 z_2)^\nu = z^\nu_1 z^\nu_2, \quad \text{(see Remark \ref{exponentrule}).}
\end{equation} 

If $| \arg (1 + b) | < \pi$ and $a > 0$, then
\begin{equation} \label{int2F1}
\int^1_0 \dfrac{\xi^{a-1}}{(1+ b \xi)^\nu} = \dfrac{1}{a} \, \, {}_2 F_1(\nu, a; 1+a; -b) \quad \text{(3.194 1, p. 318, \cite{GR}).}
\end{equation}

\begin{equation} \label{transform2F1}
{}_2 F_1(a, b; c; z) =(1-z)^{c-a-b} {}_2 F_1(c-a, c-b; c; z) \quad \text{(9.131 1, p. 1018, \cite{GR}).}
\end{equation}

We note that ${}_2F_1(a,b;c;z) = \Gamma (c) \, {}_2\mathbf{F}_1(a,b;c;z)$, see 15.1.2, \cite{NIST}. Hence, if $| \arg (-z) | < \pi$, then
\begin{equation} \label{2F1infinity}
z \,\, {}_2 F_1(1, 1; c; z) = -(c-1) \log (-z) +O(1), \quad |z| \to \infty \quad \text{(15.8.8, \cite{NIST}).}
\end{equation} 

If we take $z= \pm i /x$ in \eqref{2F1infinity}, then 
\begin{equation} \label{2F1infinityx}
\dfrac{{}_2 F_1(1, 1; c; \pm i /x)}{x (c-1)} = \mp i \log x +O(1), \quad x \searrow 0^+.
\end{equation} \\

Since the integrand of $I^+_{[-1,1]} (\mu; x) $ admits an analytic continuation to the upper complex half-plane, it is easy to see, from Cauchy's theorem, that
\begin{equation*}
I^+_{[-1,1]} (\mu; x) = - \int_{\mathbb{T}_+} e^{-i\xi} (-i\xi)^{\mu}\, (\xi + i x)^{-\mu - 1}\, d\xi  = O\left(1\right) \ \mbox{ as } \ x \searrow 0^+ ,
\end{equation*}
where $\mathbb{T}_+$ is the upper unit semicircle. \\

Similarly,
\begin{equation*}
I^-_{[-1,1]} (0; x) =  \int_{-1}^1 e^{-i\xi}  (\xi - i x)^{-1}\, d\xi  = \int_{\mathbb{T}_-} e^{-i\xi}  (\xi - i x)^{-1}\, d\xi = 
O\left(1\right) \ \mbox{ as } \ x \searrow 0^+ . 
\end{equation*}

On the other hand, we will now show that $I^-_{[-1,1]} (\mu; x)$, $\mu \not= 0$  is unbounded as $x \searrow 0^+$. Indeed,
\begin{eqnarray*}
&& I^-_{[-1,1]} (\mu; x) = \int_{-1}^1 e^{-i\xi} (-i\xi)^{\mu}\, (\xi - i x)^{-\mu - 1}\, d\xi  \\
&& = \int_{-1}^1  (-i\xi)^{\mu}\, (\xi - i x)^{-\mu - 1}\, d\xi
+ \int_{-1}^1 O(\xi) (-i\xi)^{\mu}\, (\xi - i x)^{-\mu - 1}\, d\xi \\
&& = J^-_{[-1,1]} (\mu; x) + O\left(1\right) \ \mbox{ as } \ x \searrow 0^+ ,
\end{eqnarray*}
where
\begin{equation*}
J^-_{[a,b]} (\mu; x) := \int_{a}^b  (-i\xi)^{\mu}\, (\xi - i x)^{-\mu - 1}\, d\xi. 
\end{equation*} 

We can write
\begin{align*}
J^-_{[0,1]} (\mu; x) &= (-i)^\mu \int^1_0 \xi^\mu [(-ix) (1+i\xi/x)]^{-\mu-1} \, d\xi \qquad & \text{by } \eqref{z1z2fracpower}\\
&= (-i)^\mu (-ix)^{-\mu-1} \int^1_0 \xi^\mu (1+i\xi/x)^{-\mu-1} \, d\xi \qquad &\text{by } \eqref{z1z2fracpower}\\
&= (-i)^\mu \dfrac{(-ix)^{-\mu}}{(-ix)} \dfrac{{}_2F_1(1+\mu, 1+\mu; 2+ \mu; -i/x) }{1+\mu} \\
&= \dfrac{(-i)^\mu}{-i} (-ix)^{-\mu} \, (1+i/x)^{-\mu} \,  \dfrac{{}_2F_1(1, 1; 2+ \mu; -i/x) }{x(1+\mu)} \\
&= i (-i)^\mu \, (1 - ix)^{-\mu} \,  \dfrac{{}_2F_1(1, 1; 2+ \mu; -i/x) }{x(1+\mu)} \\
&= - (-i)^\mu \, \log x + O(1) \quad \text{as  } x \searrow 0^+, \quad  & \text{by } \eqref{2F1infinityx}.
\end{align*}
Since $ -(-i)^\mu = \exp (i \pi)  \exp (-i \pi \mu / 2) = \exp (i \pi (2-\mu)/2)$, we have
\begin{equation} \label{Jminus01}
J^-_{[0,1]} (\mu; x) = \exp (i \pi (2-\mu)/2) \, \log x + O(1) .
\end{equation} \\

Similarly,
\begin{align*}
J^-_{[-1,0]} (\mu; x) &= \int^0_{-1} (-i \eta)^\mu (\eta - i x)^{-\mu-1} \, d\eta \\
&= \int^1_0 (i \xi)^\mu (-\xi - i x)^{-\mu-1} \, d\xi \\
&= i^\mu \int^1_0 \xi^\mu [(-ix) (1-i\xi/x)]^{-\mu-1} \, d\xi \qquad & \text{by } \eqref{z1z2fracpower}\\
&= i^\mu \dfrac{(-ix)^{-\mu}}{(-ix)} \int^1_0 \xi^\mu (1-i\xi/x)^{-\mu-1} \, d\xi \qquad &\text{by } \eqref{z1z2fracpower}\\
&= \dfrac{i^\mu}{-i} (-ix)^{-\mu} \dfrac{{}_2F_1(1+\mu, 1+\mu; 2+ \mu; i/x) }{x(1+\mu)} \\
&=i^{1+\mu} \, (-ix)^{-\mu} \, (1-i/x)^{-\mu} \,  \dfrac{{}_2F_1(1, 1; 2+ \mu; i/x) }{x(1+\mu)} \\
&= i^{1+\mu} \, (-1 - ix)^{-\mu} \,  \dfrac{{}_2F_1(1, 1; 2+ \mu; i/x) }{x(1+\mu)}.
\end{align*}

Noting that $\lim_{x \searrow 0^+} (-1- i x)^{-\mu} = \exp (i \pi \mu)$,
\begin{align} \label{Jminusm01}
J^-_{[-1,0]} (\mu; x) &  = \exp (i \pi (1+ \mu +2\mu -1)/2) \, \log x + O(1) \quad & \text{by } \eqref{2F1infinityx} \nonumber \\
& = \exp (i \pi 3 \mu /2) \, \log x + O(1) .
\end{align}

Thus, combining equations \eqref{Jminus01} and \eqref{Jminusm01}, 
\begin{equation*}
J^-_{[-1,1]} (\mu; x) =  C_\mu \, \log x + O(1),  
\end{equation*} 
as $x \searrow 0^+$, where
\begin{equation*}
C_\mu = - \exp (- i \pi \mu/2) + \exp (i \pi 3 \mu /2).
\end{equation*} 
Of course, given $-1 < \mu < 1$, the coefficient $C_\mu \not = 0$ if and only if $\mu \not = 0$. \\
\end{proof}

\begin{remark}
On the other hand, it is easy to show that
\begin{align*}
J^+_{[0,1]} (\mu; x) & = (-i)^{1+\mu} (1 + ix)^{-\mu} \, \dfrac{ {}_2 F_1(1, 1; 2+\mu; i/x)}{x (1+\mu)} \\
&= - \exp (- i \pi \mu/2) \, \log x + O(1),
\end{align*}
and similarly,
\begin{align*}
J^+_{[-1,0]} (\mu; x) &= -i \, i^{\mu} (-1 + ix)^{-\mu} \, \dfrac{ {}_2 F_1(1, 1; 2+\mu; -i/x)}{x (1+\mu)} \\
&= \exp (- i \pi \mu/2) \, \log x + O(1).
\end{align*}
Therefore, $J^+_{[-1,1]} (\mu; x) = O(1)$, as $x \searrow 0^+$, confirming the result in Lemma \ref{InverFTmulog}. \\
\end{remark}